%% file: main_final.tex
\documentclass[3p]{elsarticle}
\usepackage[utf8]{inputenc}
\usepackage{amsmath,amssymb,amsfonts,amsthm}

\usepackage{pgfplots}
\pgfplotsset{compat=newest}
\pgfplotsset{plot coordinates/math parser=false}

\usepackage{hyperref}
\usepackage[capitalize,nameinlink]{cleveref}
\usepackage{dsfont}
\usepackage{todonotes}
\usepackage{multicol}
\usepackage{enumitem}
\usepackage{graphicx} 
\usepackage{mathtools}
\usepackage{tikz}
\usepackage{comment}
\graphicspath{ {images/} }
\allowdisplaybreaks
\input{definitions.tex}

\newtheorem{theo}{Theorem}[section]
\newtheorem{defi}[theo]{Definition}

\newtheorem{cor}[theo]{Corollary}
\newtheorem{lem}[theo]{Lemma}

\newtheorem{ass}[theo]{Assumption}
\newtheorem{rem}[theo]{Remark}
\newcommand{\keta}{\kappa}

\title{Systems of nonlocal conservation laws: well-posedness and the singular limit for a nonlocal generalized Aw-Rascle-Zhang model}

\author[<1>]{Debora Amadori}
\ead{<debora.amadori@univaq.it>}
\author[<2>]{Felisia Angela Chiarello}

\ead{<chiarello@lum.it>}

\author[<1>]{Gianmarco Cipollone}

\ead{<gianmarco.cipollone@graduate.univaq.it>}

\affiliation[<1>]{organization={University of L'Aquila, Department of Engineering and Information Science and Mathematics
(DISIM)},
addressline={Via Vetoio}, 
            city={L'Aquila},
            postcode={67100}, 
            country={Italy}  
            }

            \affiliation[<2>]{organization={Department of Engineering, University LUM},
addressline={S.S. 100 km 18},  
            city={Casamassima (BA)},
            postcode={70010}, 
            country={Italy}  
            }

 \author[<3>]{Xiaoqian Gong}

 \ead{<xgong@amherst.edu>}  
 \affiliation[<3>]{organization={Amherst College, Department of Mathematics},
            city={Amherst},
            postcode={MA 01002}, 
            country={USA}}

\author[<4>]{Alexander Keimer}

\affiliation[<4>]{organization={University of Rostock, Institute of Mathematics},
            addressline={Ulmenstraße 69}, 
            city={Rostock},postcode={18051}, 
            country={Germany}}
 \ead{<alexander.keimer@fau.de>}

\begin{document}

\begin{abstract}

In this paper, we study a system of nonlocal conservation laws motivated by traffic flow: a nonlocal version of the generalized Aw-Rascle-Zhang (GARZ) model. The nonlocality arises from downstream spatial averaging of the velocity by a one-sided kernel. We prove the existence and uniqueness of weak solutions for initial data of bounded variation via a fixed-point argument in the nonlocal velocity. We also establish stability with respect to the initial datum and an approximation of weak solutions by strong solutions in \(L^1\). Under additional, physically meaningful assumptions on the velocity and the initial datum, we obtain either a maximum principle for the density or invariant-region estimates. Finally, we study the singular limit as the nonlocal kernel converges to a Dirac distribution. 
Indeed, under additional assumptions, convergence to the unique local entropy solution can be proved.
Some numerical simulations are provided, and the paper concludes with a discussion of open problems.
\end{abstract}


\maketitle

\section{Introduction}
Models of traffic flow are fundamental to understanding, planning, and managing modern transportation networks, and their importance continues to grow as urbanization and population growth progress.
Nonlocal traffic models have attracted considerable attention in recent years because they can provide a realistic description of scenarios in which driving speeds change as a result of events downstream \cite{blandin2016well,chiarello2019non,chiarello,scialanga,friedrich2018godunov}.
Moreover, classical first-order ``local'' models, like the well-known Lighthill-Whitham-Richards model \cite{lwr_1,lwr_2} and its extensions, have no notion of deceleration or acceleration and typically stabilize over time owing to entropy decrease.
This makes them unsuitable for reproducing traffic phenomena like \textit{phantom jams} \cite{Ramadan2020}.

To overcome these limitations, second-order models have been introduced, such as the Aw-Rascle-Zhang model \cite{aw,zhang2002non}, which is characterized by two equations, representing the conservation of mass and momentum (i.e., the product of mass and velocity). This model allows for a variety of so-called fundamental diagrams as the flux is a function of the unknown velocity and density. 
This results in a varying maximum density, which is unrealistic because a road's capacity is an intrinsic property. For this reason, \citet{FanHertySeibold2014} introduced the so-called generalized Aw-Rascle-Zhang (GARZ) model, in which the velocity is a function of two variables, the density $\rho$ and the Lagrangian marker $\omega$.
In a related model studied in \cite{Hamori-Tan_2026}, the velocity depends on both the density and the nonlocal density according to a specific expression proposed in \cite{Sopasakis2006}. \citet{Hamori-Tan_2026} studied smooth solutions for that model and provided threshold criteria for the formation of shocks.

As discussed earlier, nonlocal versions of such models are often more suitable; this motivates the model considered here. This model can be written as
\begin{equation}
\begin{aligned}
\partial_{t}\rho+
 \partial_{x}\big(\cV[\rho,\omega](t,x)\rho\big)&=0,
  \\ \partial_{t}\omega+\cV[\rho,\omega](t,x)\partial_{x}\omega&=0,\\
 \cV[\rho,\omega](t,x)&
 \coloneqq \int_{\R}\keta(x-y)V(\rho(t,y),\omega(t,y))\dd y, 
 \end{aligned}
\label{eq:nonlocal_GARZ}
\end{equation}
supplemented with the initial condition
\begin{equation}\label{eq:init_data}
   (\rho(0,\cdot),\omega(0,\cdot))= (\rho_{0},\omega_{0})\in \big(TV(\R)\big)^{2},
\end{equation}
where \(TV\) denotes the space of \(L^{1}_{\textnormal{loc}}(\R)\) functions with bounded variation:
\begin{equation}
TV(\R)\coloneqq \big\{f\in L^{1}_{\textnormal{loc}}(\R): |f|_{TV(\R)}<\infty\big\} \text{ and } |f|_{TV(\R)}\coloneqq \sup_{\substack{\phi\in C^{1}_{\textnormal{c}}(\R)\\ \|\phi\|_{L^{\infty}(\R)}\leq 1}}\int_{\R}\phi'(x)f(x)\dd x.\label{defi:TV}
\end{equation}
The function $\rho(t,x)$ describes the density at the location \(x\in\R\) at time \(t\in\R_{\geq0}\), and $\omega(t,x),$ the \textit{Lagrangian marker}, represents the maximally allowed velocity for a single driver. It can be interpreted as an intrinsic property of either the initial density distribution \cite{chiarello2020micro, FanHertySeibold2014} or, on the microscopic level, the individual behavior of each traffic participant.

The function $V:\R^{2}\rightarrow\R$ represents the velocity dependent on the density \(\rho\) and the Lagrangian marker \(\omega\). The velocity is averaged downstream against the nonlocal kernel \(\kappa\), similar to the idea for a first-order nonlocal traffic model in \cite{friedrich2018godunov} in which the velocity is also spatially averaged, and models the look-ahead distance of drivers, determined by the chosen nonlocal kernel. 

In our analysis, particularly regarding the singular limit in \cref{sec:singular_limit}, an important property of this system is that the velocity in the transport equation is the same as in the conservation law, which enables us to consider a fixed-point problem solely in the nonlocal velocity. Indeed, if this nonlocal velocity is regular enough in space, we can use it to construct the solutions of the transport equation and the conservation law. 

The study of singular limits of nonlocal traffic models has become a central topic in the mathematical analysis of conservation laws with nonlocal interactions. The main concern is the convergence of solutions of nonlocal conservation laws to the entropy solution of corresponding local conservation laws when the nonlocal kernel approaches a Dirac distribution. This problem, commonly referred to as the singular limit problem, has attracted considerable attention in recent years owing to both its theoretical relevance and its importance in validating nonlocal traffic models. For scalar conservation laws in which the density is averaged, substantial progress has been made in \cite{bressan-shen_2019traffic,bressan-shen2021entropy,coclite2022general,ColomboCrippaMarconiSpinolo2021,Marconi2023,keimer42,pflug4,coclite2025singular,keimer2026nonlocal,coclite2023oleinik,Chiarello2024,chiarello2025pnorm} with several approaches, such as taking advantage of monotonicity-preserving solutions to \(TV\) compactness of the nonlocal operator, as well as recent advances using compensated compactness.

The singular limit has been studied for nonlocal conservation laws in which the nonlocality acts on the velocity in \cite{friedrich2022conservation,pflug4}. It was shown for nonlocal kernels of exponential type that the unique solution converges in a weak or strong sense, depending on the regularity of the velocity, to the entropy solution of the local conservation law when the nonlocal weight approaches a Dirac distribution.
Only \citet{MS2025-1,MS2025-2} appear to have addressed the singular limits of systems of nonlocal conservation laws in traffic modeling. 
They established existence, uniqueness, and Ole\u{i}nik estimates, together with the singular limit, for another nonlocal version of the GARZ model \cite{FanHertySeibold2014}, in which the nonlocal dependence is on the density only (and not in the velocity, as in our case), making the analysis more involved, particularly for the transport equation in \(\omega\).
Concerning the singular limit for systems outside the context of traffic modeling, in \cite{Coclite2026nonlocal} the authors proved that a \(2\times 2\) system of triangular structure converges to the corresponding local solution. However, the velocity they considered is linear (so that nonlocality in the velocity and the density coincide) and depends only on one of the quantities, whereas the other quantity satisfies a conservation law with the same nonlocal velocity. Thus, the system can first be studied solely for the convergence of the first component, simplifying the analysis. Furthermore, as the limit solution is not necessarily essentially bounded, convergence of solutions to the second equation is only shown in measure.

The paper is organized as follows.
We state in \cref{sec:GARZ} the assumptions required for the system in \cref{eq:nonlocal_GARZ}  to be well-posed, as well as the main results obtained. \Cref{sec:well-posedness} then shows, under rather weak assumptions, the existence and uniqueness of weak solutions to \cref{eq:nonlocal_GARZ} using a fixed-point argument in the nonlocal velocity and relying on characteristics. We also state stability estimates of the solution when perturbing the initial datum in \(L^{1}\) and show that smooth solutions remain smooth and can approximate weak solutions in \(L^{1}\). We then show a maximum principle on the density, given that the velocity satisfies some additional constraints, as well as an invariant region and lower and upper bounds on the density. These properties are needed in \Cref{sec:singular_limit}, which addresses the singular limit problem, that is, whether the solution of the system in \cref{eq:nonlocal_GARZ} converges to the local semigroup solution when the nonlocal weight converges to a Dirac distribution. We tackle this using compactness and derive uniform \(TV\) bounds in the case of an exponential kernel of the nonlocal velocity, which enables us to prove that in the limit, the solution will converge to a weak solution of the local system.
\Cref{sec:entropy_uniqueness} first characterizes the unique entropy solution of the local equation and then demonstrates that our limiting solution indeed coincides with that solution, showing that even in the case of a system with an admittedly simplified nonlocal term, we can expect convergence in the singular limit to the suitable local solution.
In \cref{sec:numerical_simulations}, we present numerical examples illustrating the well-posedness result and convergence in the singular limit with examples.
\Cref{sec:conclusion} concludes the paper by discussing a few open problems worth investigating.

\section{Assumptions and main results}
\label{sec:GARZ}

We first state our assumptions on the problem and preview the main results. 
For the purpose of obtaining existence and uniqueness of solutions on short time horizons, rather weak assumptions are required, which we state in the following.
\begin{ass}[assumptions for existence and uniqueness on a short time horizon]\label{ass:small_time_existence_uniqueness}
We assume that 
\begin{multicols}{3}
\begin{itemize}
\item \(\keta\in BV(\R)\coloneqq L^{1}\cap TV(\R)\)
\item \(V\in C^{1}(\R^{2})\)
\item \((\rho_{0},\omega_{0})\in \big(TV(\R)\big)^{2}\) 
\end{itemize}
\end{multicols}
\noindent with \(TV(\R)\) as in \cref{defi:TV}.
\end{ass}

Given this, we provide the definition of a weak solution to \crefrange{eq:nonlocal_GARZ}{eq:init_data}.

\begin{defi}[weak solution of the nonlocal hyperbolic initial value problem with a finite time horizon]
\label{defi:weak_solution} 
Let \cref{ass:small_time_existence_uniqueness} and \(T\in\R_{>0}\) hold. 
Then, we call \((\rho,\omega)\in \big(C\big([0,T];L^{1}_{\text{\loc}}(\R)\big)\cap L^{\infty}((0,T);L^{\infty}(\R))\big)^{2}\)
a weak solution to the Cauchy problem \crefrange{eq:nonlocal_GARZ}{eq:init_data} iff 
it holds for all \(\phi_{1},\phi_{2}\in W^{1,\infty}_{\textnormal{c}}((-42,T)\times\R)\) and \(\OT\coloneqq (0,T)\times\R\) that
     \begin{align}
         \iint_{\OT} \rho(t,x)\big(\partial_{t}\phi_{1}(t,x)+\partial_{x}\phi_{1}(t,x)\cV[\rho,\omega](t,x)\big)\dd x\dd t+\int_{\R}\rho_{0}(x)\phi_{1}(0,x)\dd x&=0\label{eq:weak_1}\\
        \iint_{\OT} \omega(t,x)\Big(\partial_{t}\phi_{2}(t,x)+\partial_{x}\big(\phi_{2}(t,x)\cV[\rho,\omega](t,x)\big)\Big)\dd x\dd t+\int_{\R}\omega_{0}(x)\phi_{2}(0,x)\dd x&=0,\label{eq:weak_2}
\intertext{supplemented by the nonlocal term for \((t,x)\in\OT\) (with a slight abuse of notation),}
\cV[\rho,\omega](t,x)\coloneqq \cV[\rho(t,\cdot),\omega(t,\cdot)](x)\coloneqq \kappa\ast V(\rho,\omega)(t,x) =  \int_{\R}\kappa(x-y)V(\rho(t,y),\omega(t,y))\dd y.&\label{eq:weak_3}
\end{align}
\end{defi}
\begin{rem}[on the definition of weak solutions and the reasonability of the assumptions] 
The term \((t,x)\mapsto \partial_{x}\big(\phi_2(t,x)\cV[\rho,\omega](t,x)\big)\big)\) in \cref{eq:weak_2} 
is well-defined as a function  in \(L^{\infty}((0,T);L^{\infty}(\R))\), as \(\phi_{2}\) is sufficiently smooth, and the remaining spatial derivative of the nonlocal velocity \(\cV\) in \cref{eq:weak_3} is Lipschitz continuous in space thanks to the required regularity of solutions belonging to \(L^{\infty}((0,T);L^{\infty}(\R))\).
The assumptions in \cref{ass:small_time_existence_uniqueness} are rather weak but the only reasonable choice when it comes to modeling traffic flow (for instance, \(\rho_{0}\geqq0\leqq\omega_{0}\) are not necessary). However, because we later obtain the existence and uniqueness of solutions on \textbf{small} time horizons without further assumptions, we choose only what is necessary and supplement the assumptions with those reasonable for traffic flow when necessary.
\end{rem}

In \cref{theo:existence_uniqueness_small_time}, we will obtain, as in \cref{ass:small_time_existence_uniqueness}, the existence and uniqueness of a weak solution on a small time horizon. As can be seen, the required degree of regularity is quite low. This can be understood in comparison with corresponding systems of local conservation laws. For those systems, if we assume that the initial data satisfy a Lipschitz condition, we can prove existence and uniqueness of solutions on small time horizons (until the Lipschitz constant blows up). The same is true for nonlocal conservation laws, except that because of the integral operator over the velocity function, the velocity remains Lipschitz continuous as long as solutions' \(L^{\infty}\) bounds do not explode. Thus, no additional Lipschitz regularity is required; only \(L^{\infty}\) bounds are needed. However, if no further assumptions are made, the \(L^{\infty}\) norm of the solution can blow up \cite[Examples 6.1, 1]{pflug}, which is why we later require the following further assumptions on the velocity, the kernel, and the initial datum, restricting the considered class of PDEs to reasonable traffic flow models.

\begin{ass}[reasonable assumptions in traffic flow modeling on the initial datum that guarantee existence and uniqueness on any finite time horizon]\label{ass:long_time_horizon}
In addition to \cref{ass:small_time_existence_uniqueness}, we assume that \(\rho_{0}\geqq0\leqq \omega_{0}\) on \(\R\); \(\supp(\kappa)\subset\R_{\leq 0}\), where \(\kappa\geq 0\) is monotonically increasing on \(\R_{\leq 0}\) and \(\|\kappa\|_{L^{1}(\R_{\leq0})}=1\);
and either of the following:
\begin{itemize}
    \item[a)] 
there exists \(\rho_{\max}\in\R_{>0}\) such that \(\|\rho_{0}\|_{L^{\infty}(\R)}\leqq\rho_{\max}\), \(0\leqq V \text{ on } [0,\rho_{\max}]\times [\underline{\omega},\overline{\omega}]\), and \(V(\rho_{\max},\cdot)\equiv 0 \text{ on } [\underline{\omega},\overline{\omega}]\),
with \(\underline{\omega}\coloneqq \essinf_{y}\omega_{0}(y)\) and \(\overline{\omega}\coloneqq \|\omega_{0}\|_{L^{\infty}(\R)}\);
\item[b)] the initial data satisfy 
\begin{equation}
V\big(\rho_{0}(x),\omega_{0}(x)\big)\geq 0\qquad x\in\R \text{ a.e.}\label{eq:admissibility_initial_data-bis}
\end{equation}
and one of the following two conditions holds:
\begin{itemize} 
\item[b1)] \(\partial_{1}V\leqq0\) and
\begin{equation}
\exists\ c,\underline{\rho}\in\R_{>0}:\ \sup_{b\in[\underline{\omega},\overline{\omega}]}\partial_{1}V(a,b)\leq -c\text{ for } \forall a\in\R_{>\underline{\rho}};\label{eq:invariant_domain_1-bis}
\end{equation}
\item[b2)] 
\(\partial_{1}V<0,\ \partial_{2}V\geqq0\) 
and 
\begin{equation}\label{condition-b2-on-V}
\essinf_{y\in\R}V(\rho_{0}(y),\omega_{0}(y))>\lim_{\rho\rightarrow\infty} V\big(\rho,\|\omega_{0}\|_{L^{\infty}(\R)}\big).    
\end{equation}
\end{itemize}
\end{itemize}
\end{ass}

\begin{rem}[reasonableness of \cref{ass:long_time_horizon}]
    The first case in \cref{ass:long_time_horizon} can easily be satisfied when choosing
    \[
        V(a,b)=v_{1}(a)v_{2}(b),\ a\in[0,\rho_{\max}],\ b\in \big[\underline{\omega},\overline{\omega}\big]
    \]
    with \(v_{1}(\rho_{\max})=0,\ v_{1}(0)=v_{\textnormal{free}}\in\R_{>0}\), \(v_{1}\) monotonically decreasing, and \(v_{2}\geqq0 \text{ on } [\underline{\omega},\overline{\omega}]\). However, typically for traffic, one would assume that \(v_{2}\) is monotonically increasing, meaning the higher the free-flow velocity, the higher the actual velocity. Concerning the assumptions on the kernel, any reasonable monotonically increasing function with \(L^{1}\) mass \(1\) will work, and for instance, one can pick \(\kappa(x)=\exp(x),\ x\in\R_{<0}\), as we will later use in \cref{eq:nonlocal_velocity_definition}.

    Concerning the second case, we need to choose the corresponding initial data to fit the velocity function.
    For instance, \(V(\rho,\omega)=\omega-\rho\) is acceptable, as long as we have
    \[
        \omega_{0}(x)\geq \rho_{0}(x),\ x\in\R \text{ a.e.};
    \]
    because \(\partial_{1}V\equiv -1\), \cref{eq:invariant_domain_1-bis} is also satisfied, and so is b2 (compare \cref{rem:invariant}).
    Altogether, we can choose arbitrary initial data provided that the condition \(V(\rho_{\max},\cdot)=0\) holds and the velocity function is positive. For more general velocity functions that are not zero at a given maximal density uniformly in \(\omega\), we can use the second assumption with the caveat that the initial data \(\rho_{0},\ \omega_{0}\) will need to be chosen so that the velocity is nonnegative and has a strictly monotonic dependence on \(\rho\). The assumptions b1 and b2 guarantee the uniform bound of the density that is needed to prove the existence and uniqueness of weak solutions on any finite time horizon. These assumptions will also be important in the analysis of the singular limit in \cref{sec:singular_limit}.
\end{rem}

Under these conditions, we will later obtain existence and uniqueness on every finite time horizon, among other things, in \cref{theo:maximum_principle} and \cref{theo:invariant_region}.

\section{Well-posedness of the nonlocal system}\label{sec:well-posedness}
This section concerns the well-posedness of the Cauchy problem 
\cref{eq:nonlocal_GARZ} given the assumptions in \cref{ass:small_time_existence_uniqueness} on a small time horizon. 
For the existence and uniqueness proof,  we require several estimates, which we will provide in the following and which are related to either characteristics of the dynamics or estimates of the nonlocal part of the velocity function.

We start with some \(L^{\infty}\) and \(TV\) bounds on the spatial derivative of the nonlocal operator assuming that \(\rho\) and \(\omega\) are given. These estimates demonstrate that one can work with characteristics to write down the solution, as they imply that the nonlocal velocity is Lipschitz in space if the solution is essentially bounded or if the spatial derivative is \(TV\) bounded. This is the key result for our later formulation in terms of a fixed-point problem, given in \cref{defi:fixed_point_mapping}.
\begin{lem}[\(L^{\infty}\) and \(TV\) estimates for \(\partial_2 \cV\)] \label{lem:TV_partial_2_cV}
Let \cref{ass:small_time_existence_uniqueness} hold and assume that there exist
\begin{equation}
(\rho,\omega)\in \Big(C\big([0,T];L^{1}_{\text{\loc}}(\R)\big)\cap    L^{\infty}\big((0,T);L^{\infty}(\R)\cap TV(\R)\big)\Big)^{2}.
\label{eq:ass_TV_estimate_partial_2_cV}
\end{equation}
Then, the nonlocal operator in \cref{eq:weak_3} satisfies the following for \(t\in[0,T]\) a.e.:
\begin{align}
\|\partial_{2}\cV[\rho,\omega](t,\cdot)\|_{L^{\infty}(\R)}&\leq |\kappa|_{TV(\R)}\Big(|V(0,0)|+\|\partial_1 V \|_{L^{\infty}(Q[\rho,\omega](t))} \|\rho(t, \cdot)\|_{L^{\infty}(\R)}\notag\\
&\qquad +  \|\partial_2 V \|_{L^{\infty}(Q[\rho,\omega](t))}\|\omega(t,\cdot)\|_{L^{\infty}(\R)} \Big)\label{eq:partial_{2}_infty_bound}\\
|\partial_2 \cV[\rho, \omega](t, \cdot)|_{TV(\R)} & \leq  |\kappa|_{TV(\R)}\Big(\|\partial_1 V \|_{L^{\infty}(Q[\rho,\omega](t))} |\rho(t, \cdot)|_{TV(\R)}  +  \|\partial_2 V \|_{L^{\infty}(Q[\rho,\omega](t))}|\omega(t,\cdot)|_{TV(\R)} \Big) \label{eq:partial_2_cV_TVbound}
\end{align}   
with   
\begin{equation}
    Q[\rho,\omega](t)\coloneqq \big[-\|\rho(t,\cdot)\|_{L^{\infty}(\R)},\|\rho(t,\cdot)\|_{L^{\infty}(\R)}\big]\times \big[-\|\omega(t,\cdot)\|_{L^{\infty}(\R)},\|\omega(t,\cdot)\|_{L^{\infty}(\R)}\big]\subset\R^{2}. \label{eq:ass_TV_estimate_definition_Q}
\end{equation}
\end{lem}

\begin{proof}

Because the kernel \(\kappa\) satisfies \(\kappa\in BV(\R)\), according to \cite[Theorem 14.9]{leoni}, we can approximate it by a smooth kernel \(\kappa_{\eps}\in C^{\infty}(\R)\cap W^{1,1}(\R)\) so that
\begin{equation}
\lim_{\eps\rightarrow0} \|\kappa-\kappa_{\eps}\|_{L^{1}(\R)}=0\ \wedge\ \lim_{\eps\rightarrow 0} \|\kappa_{\eps}'\|_{L^{1}(\R)}=|\kappa|_{TV(\R)}.\label{eq:BV_approximation}
\end{equation}
Let \(\cV_{\eps}[\rho,\omega]\) be the nonlocal velocity with smooth kernel, that is,
\[
\cV_{\eps}[\rho,\omega](t,x)\coloneqq \int_{\R}\kappa_{\eps}(x-y)V(\rho(t,y),\omega(t,y))\dd y,\ (t,x)\in (0,T)\times\R
\]
with \((\rho,\omega)\) as in \cref{eq:ass_TV_estimate_partial_2_cV}. Then, we obtain by differentiating, for \((t,x)\in\OT\),
\begin{align}
\partial_2 \cV_{\eps}[\rho,\omega](t,x)&= \int_{\R} \kappa_{\eps}'(x-y) V(\rho(t,y),\omega(t,y))\dd y\notag
\intertext{and thus, in the \(L^{\infty}\)-norm,}
\|\partial_2 \cV_{\eps}[\rho,\omega](t,\cdot)\|_{L^{\infty}(\R)}
&\leq |\kappa_{\eps}'|_{L^{1}(\R)} \|V(\rho(t,\cdot),\omega(t,\cdot))\|_{L^{\infty}(\R)}.\label{eq:some_one_time_reference}
\end{align}
Performing a Taylor approximation at \((0,0)\) yields
\begin{align*}
\eqref{eq:some_one_time_reference}&\leq|\kappa_{\eps}'|_{L^{1}(\R)}\Big(\|\partial_{1}V\|_{L^{\infty}(Q[\rho,\omega](t))}|\rho(t,\cdot)|_{L^{\infty}(\R)}+ \|\partial_{2}V\|_{L^{\infty}(Q[\rho,\omega](t))}|\omega(t,\cdot)|_{L^{\infty}(\R)}+ |V(0,0)|\Big).
\intertext{Taking the \(TV\) seminorm of \(\partial_{2}\cV_{\eps}\) gives}
|\partial_{2}\cV_{\eps}(t,\cdot)|_{TV(\R)}&\leq |\kappa_{\eps}'|_{L^{1}(\R)}|V(\rho(t,\cdot),\omega(t,\cdot))|_{TV(\R)}\\
&\leq |\kappa_{\eps}'|_{L^{1}(\R)}\Big(\|\partial_{1}V\|_{L^{\infty}(Q[\rho,\omega](t))}|\rho(t,\cdot)|_{TV(\R)}+ \|\partial_{2}V\|_{L^{\infty}(Q[\rho,\omega](t))}|\omega(t,\cdot)|_{TV(\R)}\Big)
\end{align*}
with \(Q[\rho,\omega](t)\) as in \cref{eq:ass_TV_estimate_definition_Q}. Letting \(\eps\rightarrow 0\) according to \cref{eq:BV_approximation} yields the claim.
\end{proof}

In the following, we will set up a fixed-point equation involving the characteristics; this is the reason for defining characteristics when the solution (or the nonlocal velocity) is given.

\begin{defi}[characteristics for the nonlocal equation]\label{defi:characteristics} 
Let \(T\in\R_{>0}\) and \(\cV\in L^{\infty}((0,T);W^{1,\infty}(\R))\) be given. The characteristics \(\xi_{\cV}(t_0, x_0; \cdot) : [0, T] \to \R\) starting at \((t_0, x_0) \in \Omega_T\) are defined as the unique solution of the following initial value problem:
\begin{align*}
    \partial_3{\xi}_{\cV}(t_0, x_0;t) &=\cV(t, \xi_{\cV}(t_0, x_0;t))\quad \quad \forall\, t \in [0, T],\\
    \xi_{\cV}(t_0, x_0;t_0) &= x_0,
\end{align*}
which, in integral form, reads
\begin{equation}
\xi_{\cV}(t_0, x_0;t)=x_0+\int_{t_0}^{t} \cV(s,\xi_{\cV}(t_0, x_0;s))\dd s \quad \quad \forall t \in [0, T].
\label{eq:char_xi}
\end{equation}
\end{defi}
These characteristics have specific properties that we require in the subsequent analysis. 

\begin{lem}[properties of the characteristics]\label{lem:characteristics_properties}\label{lem:properties_characteristics}
Let \(\cV\in L^{\infty}((0,T);W^{1,\infty}(\R))\) be given. Let \(\xi_{\cV}\) be the characteristics associated with \(\cV\), as defined in \cref{defi:characteristics}. 
Then, for almost all \((t, x, \tau) \in \Omega_T \times [0, T]\), it holds that
\begin{align*}
\xi_{\cV}(t, \xi_{\cV}(\tau, x; t); \tau) &= x, \\
\partial_1 \xi_{\cV}(t, x; \tau) + \cV(t, x)\partial_2 \xi_{\cV}(t, x; \tau) &=0, \\
\partial_2 \xi_{\cV} (t, \xi_{\cV}(\tau, x; t); \tau) &= \tfrac{1}{\partial_2 \xi_{\cV}(\tau, x; t)}.
\end{align*}
Furthermore, \(\xi_{\cV}\) is monotonically increasing in the second variable, and 
\[
\|\partial_{2}\xi_{\cV}(t,\cdot;\tau)\|_{L^{\infty}(\R)}\leq \exp\Big(T\|\partial_{2}\cV\|_{L^{\infty}((0,T);L^{\infty}(\R))}\Big).
\]
Moreover, \(\xi_{\cV}\) is Lipschitz continuous with respect to the first variable, and it indeed holds \(\forall t_{1}, t_{2} \in [0, T]\) that 
\begin{equation}
\begin{aligned}
\|\xi_{\cV}(t_{1},\cdot;\ast) - \xi_{\cV}(t_{2},\cdot;\ast)\|_{L^{\infty}((0,T);L^{\infty}(\R))} &\leq \|\cV\|_{L^{\infty}((0, T); L^{\infty}(\R))} \exp\big(T\|\partial_{2}\cV\|_{L^{\infty}((0, T); L^{\infty}(\R))}\big) |t_{1} - t_{2}|\\
\|\xi_{\cV}(\ast,\cdot;t_{1}) - \xi_{\cV}(\ast,\cdot;t_{2})\|_{L^{\infty}((0,T);L^{\infty}(\R))} &\leq \|\cV\|_{L^{\infty}((0, T); L^{\infty}(\R))} |t_{1} - t_{2}|.
\end{aligned}
\label{eq:characteristics_Lipschitz_time}
\end{equation}
\end{lem}
\begin{proof}
The proof can be found for a much more general case in \cite[Section 2]{Keimer2023}.
\end{proof}
Because we will follow a fixed-point approach, we need to bound the spatial derivatives of the characteristics in terms of the velocity \(\cV\), as follows.
\begin{lem}[\(TV\) estimation of  \(\partial_2 \xi_{\cV}\)] \label{lem:TV_estimate_partial_2_xi}
For a time horizon \(T\in\R_{>0}\), let \(\cV\) be in \(L^{\infty}((0,T);W^{1,\infty}(\R))\) so that \(\partial_{2}\cV\in L^{\infty}((0,T);TV(\R))\).
Then, the following estimate holds for the characteristics \(\xi_{\cV}\) associated with \(\cV\), as defined in \cref{defi:characteristics} for \((t, \tau)\in [0,T]^{2}\):
\[\big|\partial_{2}\xi_{\cV}(t,\cdot;\tau)\big|_{TV(\R)}\leq \|\partial_{2}\xi_{\cV}(t,\cdot;\tau)\|_{L^{\infty}(\R)}|\tau-t|\|\partial_{2}\cV\|_{L^{\infty}((0, T);TV(\R))}. \]
Furthermore, assuming that the nonlocal velocity is indeed an operator dependent on \((\rho,\omega)\) as suggested in \cref{eq:weak_3}, we let
\cref{ass:small_time_existence_uniqueness} hold and assume that
\[(\rho,\omega)\in \Big(C\big([0,T];L^{1}_{\text{\loc}}(\R)\big)^2\Big)\cap    \Big(L^{\infty}\big((0,T);TV(\R)\big)^{2}\Big).\] 
Then, for \(t \in [0, T]\) a.e., 
\begin{align*}
\big|\partial_{2}\xi_{\cV}(t,\cdot;0)\big|_{TV(\R)} & \leq   \|\partial_2 \xi_{\cV}(t, \cdot; 0)\|_{L^{\infty}(\R)} |\kappa|_{TV(\R)} \bigg(\|\partial_1 V \|_{L^{\infty}( Q[\rho,\omega](t))}\int_{0}^{t}  |\rho(s, \cdot)|_{TV(\R)}\dd s \nonumber \\
& \quad +  \|\partial_2 V \|_{L^{\infty}(Q[\rho,\omega](t))} \int_0^t |\omega(s,\cdot)|_{TV(\R)}  \dd s \bigg)
\end{align*}
with  \(Q[\rho, \omega](t)\) as in \cref{eq:ass_TV_estimate_definition_Q}.
\end{lem}
\begin{proof}
The first estimate follows from a straightforward computation, assuming that the velocity \(\cV\) is arbitrarily smooth.  The second follows from the first by applying the estimates in \cref{lem:TV_partial_2_cV} and \cref{eq:partial_2_cV_TVbound}.
\end{proof}
We also require a technical lemma about \(BV\) functions in composition with diffeomorphisms.
\begin{lem}[composition of a \(BV\) 
function with a diffeomorphism]\label{lem:TV_estimate_composition}
Let \(u \in BV_{\textnormal{loc}}(\R)\) so that \(u\in TV(\R).\) Furthermore, let \(s, \tilde{s} \in W^{1, \infty}(\R)\) and \(c \in \R_{>0}\) be given such that \(s'(x) \geq c \leq \tilde{s}'(x)\ \forall x \in \R\). Then, we have
\[\|u \circ s - u \circ \tilde{s}\|_{L^{1}(\R)} \leq |u|_{TV(\R)}\big\|s^{-1} - \tilde{s}^{-1}\big\|_{L^{\infty}(\R)}.\]
\end{lem}
\begin{proof}
    The proof can be found in \cite[Lemma 2.9]{KeimerMultilane2022}.
\end{proof}

Next, we show that the characteristics are Lipschitz continuous with respect to the velocity field \(\cV\) in a specific topology, as required for the later fixed-point argument starting in \cref{defi:fixed_point_mapping}.
\begin{lem}[Lipschitz continuity of characteristics with respect to the velocity field]\label{lem:stability_characteristics}
Let the two velocity fields \(\cV,\tilde{\cV}\in L^{\infty}((0,T);W^{1,\infty}(\R))\) be given, assume in addition that \(\partial_{2}\cV,\partial_{2}\tilde{\cV}  \in L^{\infty}((0,T);TV(\R))\), and denote by \(\xi_{\cV}\) and \(\xi_{\tilde{\cV}}\) the corresponding characteristics, as in \cref{defi:characteristics}.
Then, the following stability estimates hold for all \((t,\tau)\in[0,T]^{2}\):
\begin{align}
    \|\xi_{\cV}(t,\cdot;\tau)-\xi_{\tilde{\cV}}(t,\cdot;\tau)\|_{L^{\infty}(\R)}&\leq |t-\tau|\|\cV-\tilde{\cV}\|_{L^{\infty}((0,T);L^{\infty}(\R))}\exp\Big(T\|\partial_{2}\cV\|_{L^{\infty}((0,T);L^{\infty}(\R))}\Big) \label{xi_cV_L_infty}\\
    \|\partial_{2}\xi_{\cV}(t,\cdot;\tau)-\partial_{2}\xi_{\tilde{\cV}}(t,\cdot;\tau)\|_{L^{1}(\R)}&\leq \exp\big(|t - \tau|\|\partial_2 \cV\|_{L^{\infty}((0, T); L^{\infty}(\R))}\big) |t -\tau| \nonumber \\ & \cdot \Big(\|\partial_2 \cV\|_{L^{\infty}((0, T); TV(\R))} \|\xi_{\cV} - \xi_{\tilde{\cV}}\|_{L^{\infty}(\Omega_T \times [0, T])} \nonumber \\
    & \quad + \exp\big(|t - \tau|\|\partial_2 \cV\|_{L^{\infty}((0, T); L^{\infty}(\R))}\big) \|\cV - \tilde{\cV}\|_{L^{\infty}((0, T); TV(\R))} \Big).  \label{partial_2_xi_cV_L_1}
\end{align}
We also have the following time regularity:
\begin{equation}
\lim_{t\rightarrow \tau}\|\partial_{2}\xi_{\cV}(t,\cdot;0) - \partial_{2}\xi_{\cV}(\tau,\cdot;0)\|_{L^{1}_{\textnormal{loc}}(\R)}=0.\label{eq:time_regularity_partial_2_xi}
\end{equation}
\end{lem}

\begin{proof}
The claimed estimates follow by straightforward computations.  The time regularity estimate  \cref{eq:time_regularity_partial_2_xi} of the spatial derivative of \(\xi_{\cV}\) can be found in \cite[Proposition 32]{Keimer2023}.
\end{proof}

Being prepared with the previous estimates and results on characteristics and stability, we can move forward to defining a fixed-point mapping, which will be used to construct a weak solution and show its uniqueness.
First, we will briefly sketch the idea behind the fixed-point mapping and motivate its structure.

Suppose we are given a velocity function \(V(\rho,\omega)\), that is, we know the solution \((\rho,\omega)\) over the entire time horizon. Then, we can compute \(\cV[\rho,\omega]\) as defined in \cref{eq:weak_3}, and we know that it is Lipschitz continuous in the spatial variable because of the integral operator. With that in mind, we can define characteristics as in \cref{defi:characteristics} with \(\cV\coloneqq  \cV[\rho,\omega]\), and we can take advantage of the method of characteristics to write the solution in terms of these: 
\begin{align*}
    \rho(t,x)&=\rho_{0}(\xi_{\cV}(t,x;0))\partial_{2}\xi_{\cV}(t,x;0),&&(t,x)\in\OT  \text{ a.e.} \\
    \omega(t,x)&=\omega_{0}(\xi_{\cV}(t,x;0)), && (t,x)\in\OT \text{ a.e.}   \\
    \xi_{\cV}(t,x;\tau)&=x+\int_{t}^{\tau}\cV[\rho,\omega](s,\xi(t,x;s))\dd s.
\end{align*}
This yields the solution when the velocity field is given, leading to a fixed point in the nonlocal velocity, namely for \(k\in\N_{\geq0}\) given some initial guess for \((\rho^{(0)},\omega^{(0)})\) in a suitable space:
\begin{align*}
    \cV^{(k+1)}&\coloneqq \cV[\rho^{(k)},\omega^{(k)}]\\
    \rho^{(k+1)}(t,x)&=\rho_{0}(\xi_{\cV^{(k+1)}}(t,x;0))\partial_{2}\xi_{\cV^{(k+1)}}(t,x;0),&&(t,x)\in\OT  \text{ a.e.} \\
    \omega^{(k+1)}(t,x)&=\omega_{0}(\xi_{\cV^{(k+1)}}(t,x;0)), && (t,x)\in\OT \text{ a.e.}  
\end{align*}
Altogether, existence and uniqueness of solutions reduce to a fixed-point problem in the nonlocal velocity. This is made precise in the following definitions and lemmata.

\begin{defi}[the set of the fixed-point mapping \(\Omega(T)\)] \label{defi:set_Omega}
Let \cref{ass:small_time_existence_uniqueness} hold. Then, we define the following constants:
\begin{align*}
   \cV^{\infty}&\:\|\keta\|_{L^{1}(\R)}\|V\|_{L^{\infty}(Q_{0}(\omega_{0},\rho_{0}))},\\
   \cV^{1,\infty}&\:|\keta|_{TV(\R)}\|V\|_{L^{\infty}(Q_{0}(\omega_{0},\rho_{0}))}, \\
   \cV^{1,TV}&\: 42\|\keta\|_{BV(\R)}\|\partial_{2}V\|_{L^{\infty}(Q_{0}(\omega_{0},\rho_{0}))}|\omega_{0}|_{TV(\R)} \\
   &\quad +42\|\keta\|_{BV(\R)}\|\partial_{1}V\|_{L^{\infty}(Q_{0}(\omega_{0},\rho_{0}))}\big(|\rho_{0}|_{TV(\R)}+\|\rho_{0}\|_{L^{\infty}(\R)}\big),
\end{align*}
with \[Q_{0}(\omega_{0},\rho_{0})\coloneqq \big(-42\|\rho_{0}\|_{L^{\infty}(\R)},42\|\rho_{0}\|_{L^{\infty}(\R)}\big)\times \big(-\|\omega_{0}\|_{L^{\infty}(\R)},\|\omega_{0}\|_{L^{\infty}(\R)}\big)\subset\R^{2},\] 
and for \(T\in\R_{>0}\),
\begin{align*}
   \Omega(T)&\: \Big\{\cV\in L^{\infty}\big((0,T);W^{1,\infty}(\R)\big):\ \|\cV\|_{L^{\infty}((0,T);L^{\infty}(\R))}\leq \cV^{\infty}\ \wedge \ \|\partial_{2}\cV\|_{L^{\infty}((0,T);L^{\infty}(\R))}\leq \cV^{1,\infty}\\
    &\qquad\qquad \wedge \|\partial_{2}\cV\|_{L^{\infty}((0,T);TV(\R))}\leq \cV^{1,TV}\Big\}.
\end{align*}
\end{defi}
\begin{lem}[closedness of \(\Omega(T)\) in the respective topology]\label{lem:closed_set}
For each \(T\in\R_{>0},\)
the set \(\Omega(T)\) as in \cref{defi:set_Omega} is closed in the topology induced by the norm of \(L^{\infty}\big((0,T);L^{\infty}(\R)\cap TV(\R)\big)\).
\end{lem}
\begin{proof}
    This is a direct consequence of the uniform bounds in the definition of \(\Omega(T)\).
\end{proof}
Now, we define the key ingredient, the fixed-point mapping in the nonlocal velocity.
\begin{defi}[the fixed-point mapping for the nonlocal velocity on \(\Omega(T)\)] \label{defi:fixed_point_mapping}
Let \cref{ass:small_time_existence_uniqueness} hold, and for \(\Omega(T)\) as in \cref{defi:set_Omega}, we define the following mapping:
\begin{equation}
\cF:\begin{cases}
            \Omega(T)&\rightarrow L^{\infty}((0,T);L^{\infty}(\R))\\
            \cV&\mapsto \Big((t,x)\mapsto \int_{\R} \keta(x-\xi_{\cV}(0,y;t))V\big(\tfrac{\rho_{0}(y)}{\partial_{2}\xi_{\cV}(0,y;t)},\omega_{0}(y)\big)\partial_{2}\xi_{\cV}(0,y;t)\dd y\Big).
            \end{cases}
\label{eq:fixed_point_mapping}
\end{equation}
with the characteristics \(\xi_{\cV}\) as in \cref{defi:characteristics}.
\end{defi}
With \(\cF\) (later used as fixed-point mapping) defined, we obtain some properties, beginning with the self-mapping for a sufficiently small time horizon.
\begin{lem}[\(\cF\) self-mapping on \(\Omega(T)\)]\label{lem:self_mapping_small_time}
For a sufficiently small time horizon \(T\in\R_{>0},\) the map \(\cF\) defined in \cref{defi:fixed_point_mapping} is a self-mapping on \(\Omega(T)\), that is,
\[
\cF(\Omega(T))\subset\Omega(T).
\]
 \end{lem}
\begin{proof}
To show the self-mapping property, pick \(\cV\in \Omega(T)\) with \(T\in\R_{>0}\) arbitrary. Then, we need to derive \(L^{\infty}(L^{\infty})\) bounds on the mapping \(\cF\) and its spatial derivative, as well as bounds on \(\partial_{2}\cF\) in \(L^{\infty}(TV)\). Choosing a small enough time horizon will then give us the claimed self-mapping. 
We start with the \(L^{\infty}(L^{\infty})\) estimate on the mapping \(\cF\), with \(T^{*}\in\R_{>0}\) yet to be determined:
\begin{description}
    \item[\(L^{\infty}((0,T^{*});L^{\infty}(\R))\) estimate on the mapping \(\cF\):] Letting \(\cV\in \Omega(T^{*})\) be given and recalling the definition of the fixed-point mapping in \cref{defi:fixed_point_mapping}, we claim the following estimate given fixed \((t,x)\in\OT\): 
    \begin{align*}
        |\cF[\cV](t,x)|&=\bigg|\int_{\R} \keta(x-\xi_{\cV}(0,y;t))V\Big(\tfrac{\rho_{0}(y)}{\partial_{2}\xi_{\cV}(0,y;t)},\omega_{0}(y)\Big)\partial_{2}\xi_{\cV}(0,y;t)\dd y\bigg|\\
        &\leq \|V\|_{L^{\infty}(Q_{0}(\omega_{0},\rho_{0}))}\int_{\R} |\keta(x-\xi_{\cV}(0,y;t))\partial_{2}\xi_{\cV}(0,y;t)|\dd y\\
        &\leq \|V\|_{L^{\infty}(Q_{0}(\omega_{0},\rho_{0}))}\|\keta\|_{L^{1}(\R)}.
    \end{align*}
    where, in the last line, we used the fact that \(\partial_{2}\xi_{\cV}>0\), according to \cref{lem:properties_characteristics}, to make the substitution \(z(y)=\xi_{\cV}(0,y;t)\).
    This claim, which is similar to the estimate in \cref{defi:set_Omega}, can be justified as follows. It holds for \((t,y)\in (0,T^{*})\times\R\) and \(\cV\in \Omega(T^{*})\), by \cref{lem:properties_characteristics}, that
    \begin{equation}
    0\leq \tfrac{\rho_{0}(y)}{\partial_{2}\xi_{\cV}(0,y;t)}\leq \|\rho_{0}\|_{L^{\infty}(\R)} \exp\big(t\|\partial_{2}\cV\|_{L^{\infty}((0,T);L^{\infty}(\R))}\big)\leq \|\rho_{0}\|_{L^{\infty}(\R)} \e^{\cV^{1,\infty}t}\leq 42\|\rho_{0}\|_{L^{\infty}(\R)},\label{eq:one_time_reference_self_mapping_Q_0}
    \end{equation}
    for all \(t\in (0,T^{*})\) with \(T^{*}\) chosen small enough. Thus, we can estimate \(V\) as previously done.
    Making the estimate above uniform in \((t,x)\in (0,T^{*})\times\R\), we end up with
    \[
          \|\cF[\cV]\|_{L^{\infty}((0,T^{*});L^{\infty}(\R))}\leq \|V\|_{L^{\infty}(Q_{0}(\omega_{0},\rho_{0}))}\|\keta\|_{L^{1}(\R)}=\cV^{\infty}
    \]
    as required.
    \item[{\(L^{\infty}((0,T^{**});L^{\infty}(\R))\) estimate on \(\partial_2\cF[\cV]\):}] The time horizon \(T^{**}\in\R_{>0}\) on which this part of the self-mapping property holds will be determined in the following; however, as we also want the previously established self-mapping property in \(L^{\infty}(L^{\infty})\), we need  \(T^{**}\in(0,T^{*}]\). Picking \((t,x)\in (0,T^{**})\times\R\), we estimate the spatial derivative of \(\cF\) as follows. Because of the missing higher regularity of \(\kappa\), an approximation argument as in \cref{eq:BV_approximation} for \(\eps\in\R_{>0}\) yields
    \begin{align*}
        |\partial_{2}\cF_{\eps}[\cV](t,x)|        & =\bigg|\int_{\R} \keta_{\eps}'(x-\xi_{\cV}(0,y;t))V\Big(\tfrac{\rho_{0}(y)}{\partial_{2}\xi_{\cV}(0,y;t)},\omega_{0}(y)\Big)\partial_{2}\xi_{\cV}(0,y;t)\dd y\bigg|.\\
        \intertext{Estimating \(V\) uniformly, we have}
        |\partial_{2}\cF_{\eps}[\cV](t,x)|&\leq \|V\|_{L^{\infty}(Q_{0}(\omega_{0},\rho_{0}))}\int_{\R} |\keta_{\eps}'(x-\xi_{\cV}(0,y;t))\partial_{2}\xi_{\cV}(0,y;t)|\dd y\\
        &\leq |\kappa_{\eps}|_{TV(\R)} \|V\|_{L^{\infty}(Q_{0}(\omega_{0},\rho_{0}))}.
    \end{align*}
    However, for \(\eps\rightarrow0\), the last term converges to
    \[
\lim_{\eps\rightarrow0}|\kappa_{\eps}|_{TV(\R)} \|V\|_{L^{\infty}(Q_{0}(\omega_{0},\rho_{0}))}=|\keta|_{TV(\R)} \|V\|_{L^{\infty}(Q_{0}(\omega_{0},\rho_{0}))},
    \]
   where we have used the same argument as in \cref{eq:one_time_reference_self_mapping_Q_0}. This is indeed the definition of \(\cV^{1,\infty}\).
As there are no further restrictions on the time horizon required, we can choose \(T^{**}=T^{*}\). Making the estimate uniform in \((t,x)\in (0,T^{**})\times\R\), we obtain
\[
\|\partial_{2}\cF[\cV]\|_{L^{\infty}((0,T^{**});L^{\infty}(\R))}\leq \cV^{1,\infty}
\]
as required for the second part of the self-mapping property.
\item[{\(L^{\infty}((0,T^{***});TV(\R))\) estimate on  \(\partial_2\cF[\cV]\):}] Obtaining the total variation estimate on \(\partial_{2}\cF\) is more involved and will use the previously derived self-mapping properties. Once more, we have some freedom in choosing the time horizon \(T^{***}\in\R_{>0}\) small enough, and at least as small as \(T^{*}\). As we do not assume high regularity on \(\keta\) and other functions, it is, however, useful to first reformulate the expression \(\partial_{2}\cF\). Assuming for now enough regularity on \(y\mapsto V\Big(\tfrac{\rho_{0}(y)}{\partial_{2}\xi_{\cV}(0,y;t)},\omega_{0}(y)\Big)\) and using the same approximation result for the kernel as before and in \cref{eq:BV_approximation}, it holds for \((t,x)\in(0,T^{**})\times\R \text{ a.e.}\) that
\begin{align*}
    \partial_{2}\cF[\cV](t,x)& =\int_{\R} \keta_{\eps}'(x-\xi_{\cV}(0,y;t))V\Big(\tfrac{\rho_{0}(y)}{\partial_{2}\xi_{\cV}(0,y;t)},\omega_{0}(y)\Big)\partial_{2}\xi_{\cV}(0,y;t)\dd y.\\
        \intertext{An integration by parts to move the derivative to the velocity yields}
      \partial_{2}\cF[\cV](t,x)& =\int_{\R} \keta_{\eps}(x - \xi_{\cV}(0, y; t))\tfrac{\dd}{\dd y}V\Big(\tfrac{\rho_{0}(y)}{\partial_{2}\xi_{\cV}(0,y;t)},\omega_{0}(y)\Big)\dd y.
      \end{align*}
Now, we derive a \(TV\) estimate on \(\partial_2 \cF[\cV]\) for \(t\in(0,T^{***})\), with \(T^{***}\in (0,T^{**}]\) yet to be determined, and apply \cref{lem:TV_partial_2_cV}:
\begin{align*}
    |\partial_{2}\cF[\cV](t,\cdot)|_{TV(\R)} &\leq \int_{\R} \int_{\R} \bigg|\keta'_{\eps}(x-\xi_{\cV}(0,y;t))\tfrac{\dd}{\dd y}V\Big(\tfrac{\rho_{0}(y)}{\partial_{2}\xi_{\cV}(0,y;t)},\omega_{0}(y)\Big)\bigg|\dd y\dd x\\
    &=\int_{\R}\Big|\tfrac{\dd}{\dd y}V\Big(\tfrac{\rho_{0}(y)}{\partial_{2}\xi_{\cV}(0,y;t)},\omega_{0}(y)\Big)\Big|\dd y\int_{\R}|\kappa_{\eps}'(x-\xi_{\cV}(0,y;t))\big|\dd x \dd y\\
    &\leq |\kappa_{\eps}|_{TV(\R)}\Big|V\Big(\tfrac{\rho_{0}(\cdot)}{\partial_{2}\xi_{\cV}(0,\cdot;t)},\omega_{0}(\cdot)\Big)\Big|_{TV(\R)} \overset{\eps\rightarrow 0}{=}|\kappa|_{TV(\R)}\Big|V\Big(\tfrac{\rho_{0}(\cdot)}{\partial_{2}\xi_{\cV}(0,\cdot;t)},\omega_{0}(\cdot)\Big)\Big|_{TV(\R)}.
    \end{align*}
Using the chain rule and the fact that \(V\) is regular enough from \cref{ass:small_time_existence_uniqueness}, we obtain the following, recalling the definition of the constants in \cref{defi:set_Omega} as well as \cref{lem:properties_characteristics}:
\begin{align*}
&\Big|V\Big(\tfrac{\rho_{0}(\cdot)}{\partial_{2}\xi_{\cV}(0,\cdot;t)},\omega_{0}(\cdot)\Big)\Big|_{TV(\R)}\\
&\leq \|\partial_{1}V\|_{L^{\infty}(Q_{0}(\omega_{0},\rho_{0}))}\Big|\tfrac{\rho_{0}}{\partial_{2}\xi_{\cV}(0,\cdot;t)}\Big|_{TV(\R)}\!\!\!\!\!\!+\|\partial_{2}V\|_{L^{\infty}(Q_{0}(\omega_{0},\rho_{0}))}|\omega_{0}|_{TV(\R)}\\
&\leq \|\partial_{1}V\|_{L^{\infty}(Q_{0}(\omega_{0},\rho_{0}))}\Big(|\rho_{0}|_{TV(\R)}\big\|\tfrac{1}{\partial_{2}\xi_{\cV}(0,\cdot;t)}\big\|_{L^{\infty}(\R)} +\|\rho_{0}\|_{L^{\infty}(\R)}\big|\tfrac{1}{\partial_{2}\xi_{\cV}(0,\cdot;t)}\big|_{TV(\R)}\Big)\\
&\quad +\|\partial_{2}V\|_{L^{\infty}(Q_{0}(\omega_{0},\rho_{0}))}\big|\omega_{0}\big|_{TV(\R)}\\
&\leq \|\partial_{1}V\|_{L^{\infty}(Q_{0}(\omega_{0},\rho_{0}))}\exp(2\cV^{1,\infty}t)\Big(|\rho_{0}|_{TV(\R)} +\|\rho_{0}\|_{L^{\infty}(\R)}t\cV^{1,TV}\Big)\\
&\quad +\|\partial_{2}V\|_{L^{\infty}(Q_{0}(\omega_{0},\rho_{0}))}\big|\omega_{0}\big|_{TV(\R)},
\end{align*}
where the last line used \cref{lem:TV_estimate_partial_2_xi}.
Recalling the definition of $\cV^{1,TV}$ given in \cref{defi:set_Omega}, that is,
\begin{equation*}
\begin{split}
\cV^{1,TV}&= 42\|\keta\|_{BV(\R)}\|\partial_{2}V\|_{L^{\infty}(Q_{0}(\omega_{0},\rho_{0}))}\|\omega_{0}\|_{TV(\R)} \\
   &\quad +42\|\keta\|_{BV(\R)}\|\partial_{1}V\|_{L^{\infty}(Q_{0}(\omega_{0},\rho_{0}))}\big(|\rho_{0}|_{TV(\R)}+\|\rho_{0}\|_{L^{\infty}(\R)}\big),
   \end{split}
\end{equation*}
it becomes apparent that by choosing \(t\in (0,T^{***})\) small enough, we can obtain
\[
  |\partial_{2}\cF[\cV](t,\cdot)|_{TV(\R)}\leq \cV^{1,TV}
\]
and altogether, we have thus shown that \(\cF\) is a self-mapping on \(\Omega(T^{***}),\) that is,
\begin{equation}
    \cF[\Omega(T^{***})]\subset \Omega(T^{***}).
\end{equation}

\end{description}
\end{proof}
In the next lemma, we establish a type of Lipschitz continuity of \(\cF\) when changing its inputs.
\begin{lem}[estimating \(\cF\) for different velocities]\label{lem:cF_Lipschitz_velocities}
Consider the map \(\cF\) defined in \cref{defi:fixed_point_mapping} 
with the assumptions therein. 
Then, for \(T\in\R_{>0}\) sufficiently small that \(\cV\) is a self-mapping (i.e., using the time horizon established in \cref{lem:self_mapping_small_time}) and \(\cV,\tilde{\cV}\in\Omega(T)\), the following estimates hold:
 \begin{align*}
  &\|\mathcal{F}[\cV]-\mathcal{F}[\tilde{\cV}]\|_{L^{\infty}((0,T);L^{\infty}(\R))}\\
  &\leq \|\kappa\|_{BV(\R)}\Big(\|V\|_{L^{\infty}(Q(\rho_{0},\omega_{0}))}+\cV^{1,TV}T\big(\|V\|_{L^{\infty}(Q(\rho_{0},\omega_{0}))}+\|\partial_{1}V\|_{L^{\infty}(Q(\rho_{0},\omega_{0}))}\|\rho_{0}\|_{L^{\infty}(\R)}\big)\Big)\\
    &\quad \cdot T\e^{5\cV^{1,\infty}T}\|\cV-\tilde{\cV}\|_{L^{\infty}((0,T);L^{\infty}(\R))}\\
    &\quad +\|\kappa\|_{BV(\R)}\big(\|\partial_{1}V\|_{L^{\infty}(Q(\rho_{0},\omega_{0}))}\|\rho_{0}\|_{L^{\infty}(\R)}+ \|V\|_{L^{\infty}(Q(\rho_{0},\omega_{0}))}\big)T\e^{5\cV^{1,\infty}T}|\cV-\tilde{\cV}|_{L^{\infty}((0,T);TV(\R))},\\
  &\|\mathcal{F}[\cV]-\mathcal{F}[\tilde{\cV}]\|_{L^{\infty}((0,T);TV(\R))}\\
  &\leq |\kappa|_{TV(\R)}\bigg(\|\partial_{1}V\|_{L^{\infty}(Q(\rho_{0},\omega_{0}))}\Big(|\rho_{0}|_{TV(\R)}+\|\rho_{0}\|_{L^{\infty}(\R)}\cV^{1,TV}\Big)\exp\big(\cV^{1,\infty}T\big)+ \|\partial_{2}V\|_{L^{\infty}(Q(\rho_{0},\omega_{0}))}|\omega_{0}|_{TV(\R)}\bigg)\\
    &\qquad\qquad\cdot T\exp\big(\cV^{1,\infty}T\big)\|\cV-\tilde{\cV}\|_{L^{\infty}((0,T);L^{\infty}(\R))}\\
    &\quad +|\kappa|_{TV(\R)}\|\partial_{1}V\|_{L^{\infty}(Q(\rho_{0},\omega_{0}))}\|\rho_{0}\|_{L^{\infty}(\R)}\exp\big(\cV^{1,\infty}T\big)T|\cV-\tilde{\cV}|_{L^{\infty}((0,T);TV(\R))}
  \end{align*}
  with the involved constants as in \cref{defi:set_Omega}.
\end{lem}

\begin{proof}
    Letting \((t,x)\in\OT\), we estimate for \(\cV,\tilde{\cV}\in \Omega(T)\) and \(T\in\R_{>0}\) as in \cref{lem:self_mapping_small_time} (named \(T^{***}\) therein) the corresponding mapping in \cref{defi:fixed_point_mapping}:
        \begin{align*}
        |\cF[\cV](t,x)-\cF[\tilde{\cV}](t,x)|&\leq
        \bigg| \int_{\R}\keta(x-\xi_{\cV}(0,y;t))V\Big(\tfrac{\rho_{0}(y)}{\partial_{2}\xi_{\cV}(0,y;t)},\omega_{0}(y)\Big)\partial_{2}\xi_{\cV}(0,y;t)\dd y\\
        &\qquad\qquad -\int_{\R} \keta(x-\xi_{\tilde{\cV}}(0,y;t))V\Big(\tfrac{\rho_{0}(y)}{\partial_{2}\xi_{\tilde{\cV}}(0,y;t)},\omega_{0}(y)\Big)\partial_{2}\xi_{\tilde{\cV}}(0,y;t)\dd y  \bigg|.\\
        \intertext{Adding several zeros and using triangle inequalities yields}
        |\cF[\cV](t,x)-\cF[\tilde{\cV}](t,x)|&\leq \int_{\R} \Big|\big(\keta(x-\xi_{\cV}(0,y;t))-\keta(x-\xi_{\tilde{\cV}}(0,y;t))\big)V\Big(\tfrac{\rho_{0}(y)}{\partial_{2}\xi_{\cV}(0,y;t)},\omega_{0}(y)\Big)\Big|\partial_{2}\xi_{\cV}(0,y;t)\dd y\\
        &\quad+\!\!\int_{\R}\!\!\keta(x-\xi_{\tilde{\cV}}(0,y;t))\Big|V\Big(\tfrac{\rho_{0}(y)}{\partial_{2}\xi_{\cV}(0, y;t)},\omega_{0}(y)\Big)-V\Big(\tfrac{\rho_{0}(y)}{\partial_{2}\xi_{\tilde{\cV}}(0, y;t)},\omega_{0}(y)\Big)\Big|\partial_{2}\xi_{\cV}(0,y;t)\dd y\\
        &\quad+ \!\!\int_{\R}\!\!\keta(x-\xi_{\tilde{\cV}}(0,y;t))V\Big(\tfrac{\rho_{0}(y)}{\partial_{2}\xi_{\tilde{\cV}}(0, y;t)},\omega_{0}(y)\Big) \big|\partial_{2}\xi_{\cV}(0,y;t)-\partial_{2}\xi_{\tilde{\cV}}(0,y;t)\big|\dd y.
        \intertext{Because \(\cV,\tilde{\cV}\in \Omega(T)\) and as such \(\mathcal{F}: \Omega(T) \to \Omega(T)\) is a self-mapping on \(\Omega(T)\), we can deduce the following estimate, recalling the constants introduced in \cref{defi:set_Omega}:}
        |\cF[\cV](t,x)-\cF[\tilde{\cV}](t,x)|&\leq\|V\|_{L^{\infty}(Q(\rho_{0},\omega_{0}))}\exp(t\cV^{1,\infty})\int_{\R} \big|(\keta(x-\xi_{\cV}(0, y; t))-\keta(x-\xi_{\tilde{\cV}}(0,y;t)))|\dd y\\
        &\quad +\|\keta\|_{L^{\infty}(\R)}\exp(t\cV^{1,\infty})\int_{\R}\Big|V\Big(\tfrac{\rho_{0}(y)}{\partial_{2}\xi_{\cV}(0, y;t)},\omega_{0}(y)\Big)-V\Big(\tfrac{\rho_{0}(y)}{\partial_{2}\xi_{\tilde{\cV}}(0, y;t)},\omega_{0}(y)\Big)\Big|\dd y\\
        &\quad +\|\keta\|_{L^{\infty}(\R)}\|V\|_{L^{\infty}(Q(\rho_{0},\omega_{0}))}\|\partial_{2}\xi_{\cV}(0,\cdot;t)-\partial_{2}\xi_{\tilde{\cV}}(0,\cdot;t)\|_{L^{1}(\R)}.
        \intertext{Using \cref{lem:TV_estimate_composition} and \cref{lem:properties_characteristics} for the first term and the regularity of \(V\) in the second term, we have}
        |\cF[\cV](t,x)-\cF[\tilde{\cV}](t,x)|&\leq|\keta|_{TV(\R)}\|V\|_{L^{\infty}(Q(\rho_{0},\omega_{0}))}\exp(t\cV^{1,\infty})\|\xi_{\cV}(t,\cdot;0)-\xi_{\tilde{\cV}}(t,\cdot;0)\|_{L^{\infty}(\R)}\\
        &\quad +\|\keta\|_{L^{\infty}(\R)}\exp(t\cV^{1,\infty})\|\partial_{1}V\|_{L^{\infty}(Q(\rho_{0},\omega_{0}))}\|\rho_{0}\|_{L^{\infty}(\R)}\!\!\int_{\R}\Big|\tfrac{1}{\partial_{2}\xi_{\cV}(0,y;t)}-\tfrac{1}{\partial_{2}\xi_{\tilde{\cV}}(0,y;t)}\Big|\dd y\\
        &\quad +\|\keta\|_{L^{\infty}(\R)}\|V\|_{L^{\infty}(Q(\rho_{0},\omega_{0}))}\|\partial_{2}\xi_{\cV}(0,\cdot;t)-\partial_{2}\xi_{\tilde{\cV}}(0,\cdot;t)\|_{L^{1}(\R)}.
        \intertext{Exchanging the order of integration in the second term and using some estimates for the third term leads to}
        |\cF[\cV](t,x)-\cF[\tilde{\cV}](t,x)|&\leq|\keta|_{TV(\R)}\|V\|_{L^{\infty}(Q(\rho_{0},\omega_{0}))}\e^{t\cV^{1,\infty}}\|\xi_{\cV}(t,\cdot;0)-\xi_{\tilde{\cV}}(t,\cdot;0)\|_{L^{\infty}(\R)}\\
        &\quad +\|\keta\|_{L^{\infty}(\R)}\e^{3t\cV^{1,\infty}}\|\partial_{1}V\|_{L^{\infty}(Q(\rho_{0},\omega_{0}))}\|\rho_{0}\|_{L^{\infty}(\R)}\|\partial_{2}\xi_{\cV}(0,\cdot;t)-\partial_{2}\xi_{\tilde{\cV}}(0,\cdot;t)\|_{L^{1}(\R)}\\
        &\quad +\|\keta\|_{L^{\infty}(\R)}\|V\|_{L^{\infty}(Q(\rho_{0},\omega_{0}))}\|\partial_{2}\xi_{\cV}(0,\cdot;t)-\partial_{2}\xi_{\tilde{\cV}}(0,\cdot;t)\|_{L^{1}(\R)}\\
        &\leq \|\xi_{\cV}(t,\cdot;0)-\xi_{\tilde{\cV}}(t,\cdot;0)\|_{L^{\infty}(\R)}|\keta|_{TV(\R)}\|V\|_{L^{\infty}(Q(\rho_{0},\omega_{0}))}\e^{t\cV^{1,\infty}}\\
        &\quad +\|\partial_{2}\xi_{\cV}(0,\cdot;t)-\partial_{2}\xi_{\tilde{\cV}}(0,\cdot;t)\|_{L^{1}(\R)}\e^{3t\cV^{1,\infty}}\\
        &\qquad\qquad\cdot\big(\|\partial_{1}V\|_{L^{\infty}(Q(\rho_{0},\omega_{0}))}\|\rho_{0}\|_{L^{\infty}(\R)}+ \|V\|_{L^{\infty}(Q(\rho_{0},\omega_{0}))}\big)\|\keta\|_{L^{\infty}(\R)}.
        \intertext{Using the stability properties of the characteristics in \cref{lem:stability_characteristics}, we obtain}
       |\cF[\cV](t,x)-\cF[\tilde{\cV}](t,x)|&\leq t\|\cV-\tilde{\cV}\|_{L^{\infty}((0,T);L^{\infty}(\R))}\e^{2t\cV^{1,\infty}}|\keta|_{TV(\R)}\|V\|_{L^{\infty}(Q(\rho_{0},\omega_{0}))}\\
        &\quad +t\e^{4t\cV^{1,\infty}}\|\keta\|_{L^{\infty}(\R)}\big(\|\partial_{1}V\|_{L^{\infty}(Q(\rho_{0},\omega_{0}))}\|\rho_{0}\|_{L^{\infty}(\R)}+ \|V\|_{L^{\infty}(Q(\rho_{0},\omega_{0}))}\big)\\
        &\qquad\cdot\Big(\cV^{1,TV}\|\xi_{\cV}-\xi_{\tilde{\cV}}\|_{L^{\infty}((0,T)\times\R\times(0,T))}+|\cV-\tilde{\cV}|_{L^{\infty}((0,T);TV(\R))} \Big)\\
        &\leq t\|\cV-\tilde{\cV}\|_{L^{\infty}((0,T);L^{\infty}(\R))}\e^{2t\cV^{1,\infty}}|\keta|_{TV(\R)}\|V\|_{L^{\infty}(Q(\rho_{0},\omega_{0}))}\\
        &\quad +t\e^{5t\cV^{1,\infty}}\|\keta\|_{L^{\infty}(\R)}\big(\|\partial_{1}V\|_{L^{\infty}(Q(\rho_{0},\omega_{0}))}\|\rho_{0}\|_{L^{\infty}(\R)}+ \|V\|_{L^{\infty}(Q(\rho_{0},\omega_{0}))}\big)\\
        &\qquad\cdot\Big(\cV^{1,TV}t\|\cV-\tilde{\cV}\|_{L^{\infty}((0,T);L^{\infty}(\R))}+|\cV-\tilde{\cV}|_{L^{\infty}((0,T);TV(\R))} \Big).
        \end{align*}
        Making this uniform in \((t,x)\) for \(t\in[0,T]\), with \(T\in\R_{>0}\) sufficiently small, yields the first part of the claim.
\smallskip
For the second part, we require the \(TV\) estimate on the fixed-point mapping. Assuming once more that the kernel \(\kappa\) is approximated smoothly by \(\kappa_{\eps}\), as introduced in \cref{eq:BV_approximation}, we differentiate first with respect to the spatial component to obtain, for \((t,x)\in(0,T)\times\R\),
\begin{align*}
    \big|\partial_{2}\cF_{\eps}[\cV](t,x)-\partial_{2}\cF_{\eps}[\tilde{\cV}](t,x)\big|
& = \Big|\int_{\R} \keta_{\eps}'(x-\xi_{\cV}(0,y;t))V\Big(\tfrac{\rho_{0}(y)}{\partial_{2}\xi_{\cV}(0,y;t)},\omega_{0}(y)\Big)\partial_{2}\xi_{\cV}(0,y;t)\dd y\\
        &\qquad - \int_{\R} \keta_{\eps}'(x-\xi_{\tilde{\cV}}(0,y;t))V\Big(\tfrac{\rho_{0}(y)}{\partial_{2}\xi_{\tilde{\cV}}(0,y;t)},\omega_{0}(y)\Big)\partial_{2}\xi_{\tilde{\cV}}(0,y;t)\dd y\Big|.
    \intertext{Substituting \(z(y)=\xi_{\cV}(0,y;t)\) and \(z(y)=\xi_{\tilde{\cV}}(0,y;t)\) yields}
    \big|\partial_{2}\cF_{\eps}[\cV](t,x)-\partial_{2}\cF_{\eps}[\tilde{\cV}](t,x)\big|&\leq \bigg|\int_{\R}\!\!\! \keta_{\eps}'(x-z)\Big(V\big(\rho_{0}(\xi_{\cV}(t,z;0))\partial_{2}\xi_{\cV}(t,z;0),\omega_{0}(\xi_{\cV}(t,z;0))\big)\\
    &\qquad -V\big(\rho_{0}(\xi_{\tilde{\cV}}(t,z;0))\partial_{2}\xi_{\tilde{\cV}}(t,z;0),\omega_{0}(\xi_{\tilde{\cV}}(t,z;0))\big)\Big)\dd z\bigg|,
    \intertext{and using the fact that \(V\in W^{1,\infty}_{\textnormal{loc}}(\R^{2})\), we have}
    \big|\partial_{2}\cF_{\eps}[\cV](t,x)-\partial_{2}\cF_{\eps}[\tilde{\cV}](t,x)\big|&\leq \|\partial_{1}V\|_{L^{\infty}(Q(\rho_{0},\omega_{0}))}\int_{\R}\big|\keta_{\eps}'(x-y)||\rho_{0}(\xi_{\cV}(t,y;0))\partial_{2}\xi_{\tilde{\cV}}(t,y;0)\\
    &\qquad \qquad \qquad\qquad\qquad\qquad\qquad\qquad-\rho_{0}(\xi_{\tilde{\cV}}(t,y;0))\partial_{2}\xi_{\cV}(t,y;0)\big|\dd y\\
    &\quad + \|\partial_{2}V\|_{L^{\infty}(Q(\rho_{0},\omega_{0}))}\int_{\R}|\keta_{\eps}'(x-y)|\big|\omega_{0}(\xi_{\cV}(t,y;0))-\omega_{0}(\xi_{\tilde{\cV}}(t,y;0))\big|\dd y.
\end{align*}
Because we require the \(TV\) estimate, we obtain
\begin{align*}
    &\big|\cF[\cV](t,\cdot)-\cF[\tilde{\cV}](t,\cdot)\big|_{TV(\R)}\\
    & = \int_{\R}\big|\partial_{2}\cF[\cV](t,x)-\partial_{2}\cF[\tilde{\cV}](t,x)\big| \dd x\\
   &\leq \|\partial_{1}V\|_{L^{\infty}(Q(\rho_{0},\omega_{0}))}\int_{\R}\int_{\R}|\keta_{\eps}'(x-y)||\rho_{0}(\xi_{\cV}(t,y;0))\partial_{2}\xi_{\cV}(t,y;0)-\rho_{0}(\xi_{\tilde{\cV}}(t,y;0))\partial_{2}\xi_{\tilde{\cV}}(t,y;0)|\dd y\dd x\\
    &\quad + \|\partial_{2}V\|_{L^{\infty}(Q(\rho_{0},\omega_{0}))}\int_{\R} \int_{\R}|\keta_{\eps}'(x-y)||\omega_{0}(\xi_{\cV}(t,y;0))-\omega_{0}(\xi_{\tilde{\cV}}(t,y;0))|\dd y\dd x.
\end{align*}
Applying \cref{lem:TV_estimate_composition}, adding zeros, and exchanging the order of integration in the last two terms leads to
\begin{align*}
  & \big|\cF[\cV](t,\cdot)-\cF[\tilde{\cV}](t,\cdot)\big|_{TV(\R)}\\
  &\leq  |\kappa_{\eps}'|_{TV(\R)}\|\partial_{1}V\|_{L^{\infty}(Q(\rho_{0},\omega_{0}))}|\rho_{0}|_{TV(\R)}\|\xi_{\cV}(0,\cdot;t)-\xi_{\tilde{\cV}}(0,\cdot;t)\|_{L^{\infty}(\R)}\|\partial_{2}\xi_{\cV}(t,\cdot;0)\|_{L^{\infty}(\R)}\\
    &\quad +|\kappa_{\eps}'|_{TV(\R)}\|\partial_{1}V\|_{L^{\infty}(Q(\rho_{0},\omega_{0}))}\|\rho_{0}\|_{L^{\infty}(\R)}\|\partial_{2}\xi_{\cV}(t,\cdot;0)-\partial_{2}\xi_{\tilde{\cV}}(t,\cdot;0)\|_{L^{1}(\R)}\\
    &\quad +|\kappa_{\eps}'|_{TV(\R)}\|\partial_{2}V\|_{L^{\infty}(Q(\rho_{0},\omega_{0}))}|\omega_{0}|_{TV(\R)}\|\xi_{\cV}(0,\cdot;t)-\xi_{\tilde{\cV}}(0,\cdot;t)\|_{L^{\infty}(\R)}.
\end{align*}
Letting \(\eps\rightarrow 0\) and estimating \(\partial_{2}\xi_{\cV}\) by means of \cref{lem:properties_characteristics}, recalling the bounds in \cref{defi:set_Omega} guaranteed by \cref{lem:self_mapping_small_time}, we find that
\begin{align*}
    &\big|\cF[\cV](t,\cdot)-\cF[\tilde{\cV}](t,\cdot)\big|_{TV(\R)}\\
    &\leq |\kappa|_{TV(\R)}\Big(\|\partial_{1}V\|_{L^{\infty}(Q(\rho_{0},\omega_{0}))}|\rho_{0}|_{TV(\R)}\exp\big(t\cV^{1,\infty}\big)+ \|\partial_{2}V\|_{L^{\infty}(Q(\rho_{0},\omega_{0}))}|\omega_{0}|_{TV(\R)}\Big)\\
    &\qquad\qquad\cdot\|\xi_{\cV}(0,\cdot;t)-\xi_{\tilde{\cV}}(0,\cdot;t)\|_{L^{\infty}(\R)}\\
    &\quad +|\kappa|_{TV(\R)}\|\partial_{1}V\|_{L^{\infty}(Q(\rho_{0},\omega_{0}))}\|\rho_{0}\|_{L^{\infty}(\R)}\|\partial_{2}\xi_{\cV}(t,\cdot;0)-\partial_{2}\xi_{\tilde{\cV}}(t,\cdot;0)\|_{L^{1}(\R)}.
    \intertext{Applying the stability of the characteristics in \cref{lem:stability_characteristics} leads to}
    &\big|\cF[\cV](t,\cdot)-\cF[\tilde{\cV}](t,\cdot)\big|_{TV(\R)}\\
    &\leq |\kappa|_{TV(\R)}\Big(\|\partial_{1}V\|_{L^{\infty}(Q(\rho_{0},\omega_{0}))}\big(|\rho_{0}|_{TV(\R)}+\|\rho_{0}\|_{L^{\infty}(\R)}\cV^{1,TV}\big)\exp\big(t\cV^{1,\infty}\big)+ \|\partial_{2}V\|_{L^{\infty}(Q(\rho_{0},\omega_{0}))}|\omega_{0}|_{TV(\R)}\Big)\\
    &\qquad\qquad\cdot t\exp\big(t\cV^{1,\infty}\big)\|\cV-\tilde{\cV}\|_{L^{\infty}((0,T);L^{\infty}(\R))}\\
    &\quad +|\kappa|_{TV(\R)}\|\partial_{1}V\|_{L^{\infty}(Q(\rho_{0},\omega_{0}))}\|\rho_{0}\|_{L^{\infty}(\R)}\exp\big(t\cV^{1,\infty}\big)t|\cV-\tilde{\cV}|_{L^{\infty}((0,T);TV(\R))}.
\end{align*}
Making this uniform in \(t\in(0,T)\) yields the second estimate.
\end{proof}

Having established the self-mapping property in \cref{lem:self_mapping_small_time} and the Lipschitz continuity of \(\cF\) with respect to different velocities in \cref{lem:cF_Lipschitz_velocities}, the following fixed-point result can now be derived. 
\begin{theo}[unique fixed point of \(\cF\)]\label{theo:fixed_point_mapping_uniqueness}
There exists \(T\in\R_{>0}\) 
such that the mapping \(\cF\) in \cref{defi:fixed_point_mapping} has a unique fixed point in \(\Omega(T)\) (as in \cref{defi:set_Omega}), 
that is, 
\[
\exists !\ \cV^{*}\in\Omega(T):\quad \cF[\cV^{*}]\equiv \cV^{*} \text{ on } (0,T)\times\R.
\]
\end{theo}
\begin{proof}
    This is seen to be a direct consequence of \cref{lem:closed_set} together with the self-mapping property in \cref{lem:self_mapping_small_time} and \cref{lem:cF_Lipschitz_velocities} when recognizing that we can obtain by \cref{lem:self_mapping_small_time} a time horizon \(T^{***}\in\R_{>0}\) small enough that \(\cF\) is a self-mapping, and by \cref{lem:cF_Lipschitz_velocities} on an even smaller time horizon \(T^{****}\in (0,T^{***}]\) a contraction with respect to the \(L^{\infty}\big((0,T^{****});L^{\infty}\cap TV(\R)\big)\) contraction. Because \(\Omega(T^{****})\) is closed under \(L^{\infty}\big((0,T^{****});L^{\infty}\cap TV(\R)\big)\) (as it has uniform bounds), we can apply Banach's fixed-point theorem and obtain the existence and uniqueness of a fixed point of \(\cF\) in \(\Omega(T^{****})\) with a small enough \(T^{****}\in\R_{>0}\).
\end{proof}

\subsection{Existence and uniqueness for a small time horizon}
Now, we are well prepared to prove the existence and uniqueness of a weak solution to \cref{eq:nonlocal_GARZ,eq:init_data}. We will rely on the fixed point in \cref{theo:fixed_point_mapping_uniqueness}.

\begin{theo}[existence and uniqueness for a small time horizon]\label{theo:existence_uniqueness_small_time}
Let \cref{ass:small_time_existence_uniqueness} hold, that is,
\begin{multicols}{3}
\begin{itemize}
    \item \(\keta\in BV(\R)\)
\item \(V\in C^{1}(\R^{2};\R)\)
\item \((\rho_{0},\omega_{0})\in \big(TV(\R)\big)^{2}\).
\end{itemize}
\end{multicols}
\noindent Then, there exists \(T\in\R_{>0}\) 
such that \cref{eq:nonlocal_GARZ,eq:init_data} have a unique weak solution 
\begin{equation}
(\rho,\omega)\in \big(C\big([0,T];L^{1}_{\textnormal{loc}}(\R)\big)\cap L^{\infty}((0,T); TV(\R))\big)^{2}\label{eq:solution_minimal_regularity}
\end{equation}
in the sense of \cref{defi:weak_solution}. 
Moreover, the solution can be written as
\begin{equation}
\rho(t,x)=\rho_{0}(\xi_{\cV^{*}}(t,x;0))\partial_{2}\xi_{\cV^{*}}(t,x;0),\qquad \omega(t,x)=\omega_{0}(\xi_{\cV^{*}}(t,x;0)),\ (t,x)\in (0,T)\times \R \text{ a.e.}\label{eq:solution}
\end{equation}
with \(\cV^{*}\) the unique solution in \cref{theo:fixed_point_mapping_uniqueness} and \(\xi_{\cV^{*}}:(0,T)\times\R\times (0,T)\rightarrow\R\) the solution to the corresponding characteristic equation, as defined in \cref{defi:characteristics} for the velocity \(\cV^{*}\).
\end{theo}
\begin{proof}
    We need to show that 
\begin{enumerate}    
 \item \cref{eq:solution} is a weak solution in the sense of \cref{defi:weak_solution}, 
 
\item  that it has the required regularity in \cref{eq:solution_minimal_regularity}, and 
\item that the solution is unique.
\end{enumerate}

\begin{description}
    \item[1. \cref{eq:solution} is a solution:]
    Let \(\cV^{*}\) be the unique solution in \cref{theo:fixed_point_mapping_uniqueness} with \(T\in\R_{>0}\), so that the unique fixed-point \(\cV^{*}\) actually exists. Let \(\xi_{\cV^{*}}:(0,T)\times\R\times (0,T)\rightarrow\R\) be the solution to the corresponding characteristic equation in \cref{defi:characteristics} for the velocity \(\cV^{*}\). Then, \cref{eq:solution} is well-defined.
We will now show that such a \((\rho, \omega)\) is a weak solution to the nonlocal GARZ model in \cref{eq:nonlocal_GARZ} supplemented with the initial data in \cref{eq:init_data}, in the sense of  \cref{defi:weak_solution}. 

To this end, we take \(\phi_{1} \in W^{1,\infty}_{\textnormal{c}}((-42,T)\times\R)\) and plug the solution formula for \(\rho\) into the weak formulation in \cref{defi:weak_solution} to arrive at
    \begin{align*}
        & \iint_{\OT} \rho(t,x)\big(\partial_{1}\phi_{1}(t,x)+\partial_{2}\phi_{1}(t,x)\cV^{*}(t,x)\big)\dd x\dd t+\int_{\R}\rho_{0}(x)\phi_{1}(0,x)\dd x\\
        & =  \iint_{\OT} \rho_{0}(\xi_{\cV^{*}}(t,x;0))\partial_{2}\xi_{\cV^{*}}(t,x;0)\partial_{1}\phi_{1}(t,x)\dd x\dd t\\
        & \quad + \iint_{\OT} \rho_{0}(\xi_{\cV^{*}}(t,x;0))\partial_{2}\xi_{\cV^{*}}(t,x;0)\partial_{2}\phi_{1}(t,x)\cV^{*}(t,x)\dd x\dd t
         +\int_{\R}\rho_{0}(x)\phi_{1}(0,x)\dd x.
        \intertext{Substituting the characteristics in space by \(y(x)=\xi_{\cV^{*}}(t,x;0)\) for any given time \(t\in(0,T)\) gives}
        & \iint_{\OT} \rho(t,x)\big(\partial_{1}\phi_{1}(t,x)+\partial_{2}\phi_{1}(t,x)\cV^{*}(t,x)\big)\dd x\dd t+\int_{\R}\rho_{0}(x)\phi_{1}(0,x)\dd x\\
        & =  \iint_{\OT} \rho_{0}(y)\partial_{1}\phi_{1}(t,\xi_{\cV^{*}}(0,y;t))\dd y\dd t\\
        & \quad + \iint_{\OT} \rho_{0}(y)\partial_{2}\phi_{1}(t,\xi_{\cV^{*}}(0,y;t))\cV^{*}(t,\xi_{\cV^{*}}(0,y;t))\dd y\dd t+\int_{\R}\rho_{0}(x)\phi_{1}(0,x)\dd x.
        \intertext{Next, taking advantage of the definition of the characteristics in \cref{defi:characteristics} leads to}
        & \iint_{\OT} \rho(t,x)\big(\partial_{1}\phi_{1}(t,x)+\partial_{2}\phi_{1}(t,x)\cV^{*}(t,x)\big)\dd x\dd t+\int_{\R}\rho_{0}(x)\phi_{1}(0,x)\dd x\\
        &=   \int_{\R} \int_0^T \rho_{0}(y)\tfrac{\dd}{\dd t}\phi_1(t,\xi_{\cV^{*}}(0,y;t))\dd t\dd y +\int_{\R}\rho_{0}(x)\phi_{1}(0,x)\dd x,
        \intertext{applying the fundamental theorem of calculus yields}
        & = - \int_{\R} \rho_0(y)  \phi_1(0,\xi_{\cV^{*}}(0,y;0))\dd y  + \int_{\R}\rho_{0}(x)\phi_{1}(0,x)\dd x\\
          & = - \int_{\R} \rho_0(y)  \phi_1(0,y)\dd y  + \int_{\R}\rho_{0}(x)\phi_{1}(0,x)\dd x=0.
    \end{align*}
Thus, the given \(\rho\) is a weak solution.
Similarly, for \(\phi_{2} \in W^{1,\infty}_{\textnormal{c}}((-42,T)\times\R)\), we show that \(\omega\) is a solution by plugging it into the weak formulation in \cref{defi:weak_solution} and substituting the characteristics in space by \(y(x)=\xi_{\cV^{*}}(t,x;0)\) for any given time \(t\in(0,T)\):
    \begin{align*} 
    & \iint_{\OT} \omega(t,x)\big(\partial_{1}\phi_{2}(t,x)+\partial_{2}\big(\phi_{2}(t,x)\cV^{*}(t,x)\big)\big)\dd x\dd t+\int_{\R}\omega_{0}(x)\phi_{2}(0,x)\dd x\\
    &= \iint_{\OT}  \omega_{0}(y)\partial_{2}\xi_{\cV^{*}}(0,y;t)\Big(\partial_{1}\phi_{2}(t,\xi_{\cV^{*}}(0,y;t))+\partial_{2}\phi_{2}(t,\xi_{\cV^{*}}(0,y;t))\cV^{*}(t,\xi_{\cV^{*}}(0,y;t))\Big) \dd y\dd t \\
    & \quad + \iint_{\OT}\!\!\!\! \omega_{0}(y) \phi_{2}(t,\xi_{\cV^{*}}(0,y;t)) \partial_2 \cV^{*}(t,\xi_{\cV^{*}}(0,y;t)) \partial_2 \xi_{\cV^{*}}(0, y; t) \dd y\dd t +\int_{\R}\omega_{0}(x)\phi_{2}(0,x)\dd x\\
    &= \iint_{\OT}  \omega_{0}(y)\partial_{2}\xi_{\cV^{*}}(0,y;t)\tfrac{\dd}{\dd t}\phi_{2}(t,\xi_{\cV^{*}}(0,y;t))\dd y\dd t \\
    & \quad + \iint_{\OT} \omega_{0}(y) \phi_{2}(t,\xi_{\cV^{*}}(0,y;t)) \partial_2 \cV^{*}(t,\xi_{\cV^{*}}(0,y;t)) \partial_2 \xi_{\cV^{*}}(0, y; t) \dd y\dd t +\int_{\R}\omega_{0}(x)\phi_{2}(0,x)\dd x,
    \intertext{an integration by parts in time in the first term yields}
    &=\int_{\R} \omega_{0}(y)\Big[ \partial_{2}\xi_{\cV^{*}}(0,y;t)\phi_{2}(t,\xi_{\cV^{*}}(0,y;t)) \Big]_{t=0}^{t=T} \dd y\! -\!\iint_{\OT}\!\!\!\!  \omega_{0}(y)\partial_{t}\partial_{2}\xi_{\cV^{*}}(0,y;t)\phi_{2}(t,\xi_{\cV^{*}}(0,y;t))\dd y\dd t\\
    & \quad + \iint_{\OT} \omega_{0}(y) \phi_{2}(t,\xi_{\cV^{*}}(0,y;t)) \partial_2 \cV^{*}(t,\xi_{\cV^{*}}(0,y;t)) \partial_2 \xi_{\cV^{*}}(0, y; t) \dd y\dd t +\int_{\R}\omega_{0}(x)\phi_{2}(0,x)\dd x,
    \intertext{and using the fact that \(\phi_2(T,\cdot)=0\), as well as \(\partial_{t}\partial_{2}\xi_{\cV^{*}}(0,y;t)=\partial_{2}\cV^{*}[\rho,\omega](t,\xi_{\cV^{*}}(0,y;t))\partial_{2}\xi_{\cV^{*}}(0,y;t)\), which follows from the definition of the characteristics in \cref{defi:characteristics}, we have}
    & \iint_{\OT} \omega(t,x)\big(\partial_{1}\phi_{2}(t,x)+\partial_{2}\big(\phi_{2}(t,x)\cV^{*}(t,x)\big)\big)\dd x\dd t+\int_{\R}\omega_{0}(x)\phi_{2}(0,x)\dd x\\
    &= -\iint_{\OT}\!\!\!\!  \omega_{0}(y)\phi_{2}(t,\xi_{\cV^{*}}(0,y;t))\partial_{2}\cV^{*}[\rho,\omega](t,\xi_{\cV^{*}}(0,y;t))\partial_{2}\xi_{\cV^{*}}(0,y;t)\dd y\dd t\\
    &\quad-\int_{\R}\omega_{0}(y)\partial_{2}\xi_{\cV^{*}}(0,y;0)\phi_{2}(0,\xi_{\cV^{*}}(0,y;0))\dd y\\
    & \quad + \iint_{\OT} \omega_{0}(y) \phi_{2}(t,\xi_{\cV^{*}}(0,y;t)) \partial_2 \cV^{*}[\rho,\omega](t,\xi_{\cV^{*}}(0,y;t)) \partial_2 \xi_{\cV^{*}}(0, y; t) \dd y\dd t +\int_{\R}\omega_{0}(x)\phi_{2}(0,x)\dd x.
    \intertext{Now, using \(\partial_{2}\xi_{\cV^{*}}(0,y;0)=1\), as well as \(\xi_{\cV^{*}}(0,y;0)=y\), which is a consequence of \cref{lem:characteristics_properties}, we find that}
    &=-\int_{\R}\omega_{0}(y)\phi_{2}(0,y)\dd y+\int_{\R}\omega_{0}(x)\phi_{2}(0,x)\dd x=0.
\end{align*}
Concluding, we have shown that  the functions defined in \cref{eq:solution}  form a weak solution in the sense of \cref{defi:weak_solution}.

\item[2. Regularity of the weak solution:]

We start by proving the claimed regularity for \(\omega \).
Notice that \(\omega_0 \in L^{\infty}(\R)\), \(\xi_{\cV^{*}}(t, x; \cdot) \colon [0, T] \to \R\) for any \((t, x) \in \Omega_T\), and \(\omega(t,x)=\omega_{0}(\xi_{\cV^{*}}(t,x;0))\ \text{ for } (t,x)\in (0,T)\times \R \text{ a.e.}\) imply that \(\omega \in L^{\infty}((0,T);L^{\infty}(\R))\). 

Next, we claim that for \(t \in (0, T)\), \(\omega(t, \cdot) \in TV(\R)\) a.e. This is true as the composition of a function with bounded \(TV\) seminorm and a Lipschitz-continuous function is again \(TV\) bounded \cite{Josephy1981} (the result is for pointwise \(BV\); however, it carries over to the \(BV\) class considered here).

Now we claim that \( \omega \in C([0, T]; L^{1}_{\textnormal{loc}}(\R))\). For this, we require an approximation result and claim that for each compact \(K\subset\R\), \(\forall \varepsilon >0\), there exists \(\omega_{0, \varepsilon,K} \in C^{\infty}(\R)\) such that 
\begin{align}
\int_K |\omega_0 (x) - \omega_{0, \varepsilon,K}(x)| \dd x < \tfrac{\varepsilon}{3\exp\Big(T\|\partial_{2}\cV^{*}\|_{L^{\infty}((0,T);L^{\infty}(\R))}\Big)}. 
\label{eq:epsilon_omega}
\end{align}
Using this, we obtain \(\forall\, t_1, t_2 \in [0, T]\) that
\begin{align*}
   \int_{K}| \omega(t_1, x)  - \omega(t_2, x)|\dd x   & = \int_K |\omega_{0}(\xi_{\cV^{*}}(t_1,x;0)) -  \omega_{0}(\xi_{\cV^{*}}(t_2,x;0))| \dd x.
   \intertext{By approximating \(\omega_{0}\) smoothly and then adding zeros, we reach}
   \int_{K}| \omega(t_1, x)  - \omega(t_2, x)|\dd x & \leq \int_K |\omega_{0}(\xi_{\cV^{*}}(t_1,x;0)) -\omega_{0, \varepsilon,K}(\xi_{\cV^{*}}(t_1,x;0))|\dd x\\
   &\quad+\int_K |\omega_{0, \varepsilon,K}(\xi_{\cV^{*}}(t_1,x;0)) -\omega_{0, \varepsilon,K}(\xi_{\cV^{*}}(t_2,x;0))| \dd x \\
   & \quad +\int_{K}| \omega_{0, \varepsilon,K}(\xi_{\cV^{*}}(t_2,x;0))-\omega_{0}(\xi_{\cV^{*}}(t_2,x;0))| \dd x. 
   \intertext{By a substitution, we have}
   \int_{K}| \omega(t_1, x)  - \omega(t_2, x)|\dd x & \leq \int_K |\omega_{0}(y) -\omega_{0, \varepsilon,K}(y)| \partial_2 \xi_{\cV^{*}}(0, y; t_1)\dd y\\
   &\quad +\int_K |\omega_{0, \varepsilon,K}(\xi_{\cV^{*}}(t_1,x;0)) -\omega_{0, \varepsilon,K}(\xi_{\cV^{*}}(t_2,x;0)) \dd x \\
   & \quad +\int_{K}| \omega_{0, \varepsilon,K}(y)-\omega_{0}(y)| \partial_2 \xi_{\cV^{*}}(0, y; t_2)\dd y.\\
   \intertext{By \cref{lem:properties_characteristics},}
   \int_{K}| \omega(t_1, x)  - \omega(t_2, x)|\dd x & \leq 2\exp\Big(T\|\partial_{2}\cV^{*}\|_{L^{\infty}((0,T);L^{\infty}(\R))}\Big) \int_K |\omega_{0}(y) -\omega_{0, \varepsilon,K}(y)| \dd y\\
   &\quad+\int_K |\omega_{0, \varepsilon}(\xi_{\cV^{*}}(t_1,x;0)) -\omega_{0, \varepsilon,K}(\xi_{\cV^{*}}(t_2,x;0)) \dd x. \\
   \intertext{Then, by \cref{eq:epsilon_omega},}
   \int_{K}| \omega(t_1, x)  - \omega(t_2, x)|\dd x &\leq  \tfrac{2 \varepsilon }{3} + \int_K |\omega_{0, \varepsilon,K}(\xi_{\cV^{*}}(t_1,x;0)) -\omega_{0, \varepsilon,K}(\xi_{\cV^{*}}(t_2,x;0)) |\dd x.
\end{align*}
Taking advantage of \cref{lem:TV_estimate_composition} and \cref{lem:characteristics_properties}, we derive the following estimate:
\[
\int_K\!\! |\omega_{0, \varepsilon,K}(\xi_{\cV^{*}}(t_1,x;0)) -\omega_{0, \varepsilon,K}(\xi_{\cV^{*}}(t_2,x;0)) |\dd x\leq |\omega_{0, \varepsilon,K}|_{TV(\R)}\|\xi_{\cV^{*}}(0,\cdot;t_{1})-\xi_{\cV^{*}}(0,\cdot;t_{2})\|_{L^{\infty}(\R)}.
\]
However, when choosing a suitable mollification (standard mollifier), we obtain
\[
\sup_{\varepsilon\in\R_{>0}}|\omega_{0, \varepsilon,K}|_{TV(\R)}\leq |\omega_{0}|_{TV(\R)} .
\]
Using the definition of the characteristics in \cref{defi:characteristics}, 
we obtain, recalling \cref{defi:set_Omega}, that
\[
\lim_{t_{1}\rightarrow t_{2}}\|\xi_{\cV^{*}}(0,\cdot;t_{1})-\xi_{\cV^{*}}(0,\cdot;t_{2})\|_{L^{\infty}(\R)}\leq \lim_{t_{1}\rightarrow t_{2}}|t_{1}-t_{2}|\|\cV^{*}\|_{L^{\infty}((0,T);L^{\infty}(\R))}\leq  \lim_{t_{1}\rightarrow t_{2}}|t_{1}-t_{2}|\cV^{\infty}=0.
\]
Hence, it holds that
\(
\lim_{t_{1}\rightarrow t_{2}}\int_{K}\big| \omega(t_1, x)  - \omega(t_2, x)\big|\dd x \to 0,
\)
which gives the claimed regularity.

Now, we turn our attention to the proof of the claimed regularity of \(\rho\).
Notice that \(\rho_0 \in L^{\infty}(\R)\), \(\xi_{\cV^{*}}(t, x; \cdot) \colon [0, T] \to \R\) for any \((t, x) \in \Omega_T\), \(\partial_2\xi_{\cV^{*}}(\cdot, \cdot; 0) \colon \Omega_T \to \R \) is bounded by \cref{lem:properties_characteristics}, and \(\rho(t,x)=\rho_{0}(\xi_{\cV^{*}}(t,x;0))\partial_{2}\xi_{\cV^{*}}(t,x;0),\ (t,x)\in (0,T)\times \R \text{ a.e.}\) imply that \(\rho \in L^{\infty}((0,T);L^{\infty}(\R))\). 

Now, we need to show for \(t \in (0, T)\), \(\rho(t, \cdot) \in TV(\R)\) a.e.
Notice that for all \(\phi \in C_{\textnormal{c}}^1(\R)\) such that \(\|\phi\|_{L^{\infty}(\R)}\leq 1 \), it holds that
\begin{align*}
|\rho(t,\cdot)|_{TV(\R)}&\leq\|\rho_0\|_{L^{\infty}(\R)} |\partial_2 \xi_{\cV^{*}}(t, \cdot; 0)|_{TV(\R)} + \exp\Big(T\|\partial_{2}\cV^{*}\|_{L^{\infty}((0,T);L^{\infty}(\R))}\Big)|\rho_0|_{TV(\R)}.
\end{align*}
Because \(\rho_{0}\in TV(\R)\) and by \cref{lem:TV_estimate_partial_2_xi}, we have, recalling \cref{defi:set_Omega}, that
 \[
\big|\partial_{2}\xi_{\cV^{*}}(t,\cdot;0)\big|_{TV(\R)}\leq \|\partial_{2}\xi_{\cV}(t,\cdot;0)\|_{L^{\infty}(\R)}t\|\partial_{2}\cV\|_{L^{\infty}((0, T);TV(\R))}\leq T\exp(T\cV^{1,\infty})\cV^{1,TV},
 \]
which implies
\begin{align*}
    |\rho|_{L^{\infty}((0,T);TV(\R))}\leq \|\rho_{0}\|_{L^{\infty}(\R)}T\exp(T\cV^{1,\infty})\cV^{1,TV}+|\rho_{0}|_{TV(\R)}\exp(T\cV^{1,\infty}).
\end{align*}
Next, we prove that \( \rho \in C\big([0, T]; L^{1}_{\textnormal{loc}}(\R)\big)\). 

First, for any \(t \in (0, T)\) and any compact \(K \subset \R\), 
\begin{align*}
    \int_K |\rho(t, x)| \dd x &= \int_K |\rho_{0}(\xi_{\cV^{*}}(t,x;0))\partial_{2}\xi_{\cV^{*}}(t,x;0)| \dd x. \\
    \intertext{Setting \(y(x) = \xi_{\cV^{*}}(t,x;0)\) gives us}
    \int_K |\rho(t, x)| \dd x &= \int_{\{\xi_{\cV^{*}}(t,x;0)\colon x \in K\}}|\rho_{0}(y)|\dd y \leq \|\rho_0\|_{L^{\infty}(\R)} \lambda(\{\xi_{\cV^{*}}(t,x;0)\colon x \in K\}) < \infty,
\end{align*}
where \(\lambda(\{\xi_{\cV^{*}}(t,x;0)\colon x \in K\})\) represents the Lebesgue measure of the set \(\{\xi_{\cV^{*}}(t,x;0)\colon x \in K\}\), which proves that the solution is integrable for almost all \(t\in(0,T)\).

Second, for any \(t_1, t_2 \in [0, T]\), 
\begin{align*}
    &\int_{K} |\rho(t_1, x) - \rho(t_2, x)| \dd x \\
    &=  \int_{K} |\rho_{0}(\xi_{\cV^{*}}(t_1,x;0))\partial_{2}\xi_{\cV^{*}}(t_1,x;0) - \rho_{0}(\xi_{\cV^{*}}(t_2,x;0))\partial_{2}\xi_{\cV^{*}}(t_2,x;0)|\dd x \\
    &\leq \int_{K} |\rho_{0}(\xi_{\cV^{*}}(t_1,x;0))\partial_{2}\xi_{\cV^{*}}(t_1,x;0) - \rho_{0}(\xi_{\cV^{*}}(t_1,x;0))\partial_{2}\xi_{\cV^{*}}(t_2,x;0)| \dd x \\
    & \quad+\int_{K} |\rho_{0}(\xi_{\cV^{*}}(t_1,x;0))\partial_{2}\xi_{\cV^{*}}(t_2,x;0) -\rho_{0}(\xi_{\cV^{*}}(t_2,x;0))\partial_{2}\xi_{\cV^{*}}(t_2,x;0)|\dd x \\
    &= \int_{K} |\rho_{0}(\xi_{\cV^{*}}(t_1,x;0))| |\partial_{2}\xi_{\cV^{*}}(t_1,x;0) - \partial_{2}\xi_{\cV^{*}}(t_2,x;0)| \dd x \\
    &\quad +\int_{K} |\rho_{0}(\xi_{\cV^{*}}(t_1,x;0))-\rho_{0}(\xi_{\cV^{*}}(t_2,x;0))|\partial_{2}\xi_{\cV^{*}}(t_2,x;0)\dd x \\
    & \leq\|\rho_0\|_{L^{\infty}(\R)}\int_{K}  |\partial_{2}\xi_{\cV^{*}}(t_1,x;0) - \partial_{2}\xi_{\cV^{*}}(t_2,x;0)| \dd x \\
    &\quad +\exp\Big(T\|\partial_{2}\cV^{*}\|_{L^{\infty}((0,T);L^{\infty}(\R))}\Big) \int_{K} |\rho_{0}(\xi_{\cV^{*}}(t_1,x;0))-\rho_{0}(\xi_{\cV^{*}}(t_2,x;0))|\dd x. 
\end{align*}
Observe that 
\begin{align*}\int_{K} |\rho_{0}(\xi_{\cV^{*}}(t_1,x;0))-\rho_{0}(\xi_{\cV^{*}}(t_2,x;0))|\dd x \to 0 \text{ as } t_1 \to t_2
\label{limit_1}
\end{align*}
for the same reason as in \cref{eq:epsilon_omega} when mollifying \(\rho_{0}\) and using \cref{lem:characteristics_properties}, namely, \cref{eq:characteristics_Lipschitz_time} together with \cref{lem:TV_estimate_composition}.
Moreover, from \cref{eq:time_regularity_partial_2_xi} in \cref{lem:stability_characteristics}, we obtain
\begin{align*}
    \lim_{t_{1}\rightarrow t_{2}}\int_{K}  |\partial_{2}\xi_{\cV^{*}}(t_1,x;0) - \partial_{2}\xi_{\cV^{*}}(t_2,x;0)| \dd x=0
\end{align*}
so that we have proven the claimed regularity.
    \item[3. Uniqueness of the weak solution:] We follow the argument in \cite{wang} and \cite{pflug} for uniqueness of solutions.
Let two weak solutions \[(\rho, \omega),\ (\tilde{\rho}, \tilde{\omega}) \in \Big(C\big([0,T];L^{1}_{\textnormal{loc}}(\R)\big)\cap L^{\infty}((0,T);L^{\infty}(\R)\cap TV(\R))\Big)^{2} \] in the sense of \cref{defi:weak_solution} be given. Define, for \((t,x)\in\OT\),
\begin{align*}
    \tilde{\cV}(t,x)=\int_{\R}\keta(x-y)V(\tilde{\rho}(t,y),\tilde{\omega}(t,y))\dd y
\end{align*}
and denote by \(\xi_{\tilde{\cV}} \colon (0, T) \times \R \times (0, T) \to \R\) the solution to the corresponding characteristic equation in \cref{defi:characteristics} for the velocity \(\tilde{\cV}\). Let \(\psi_{0,1}, \psi_{0,2} \in C_c^1([-42, T+1]\times\R)\) be given, and define 
\begin{align}
\psi_1 (t,x) &= \psi_{0, 1}(\xi_{\tilde{\cV}}(t,x;T)), &&(t, x) \in \Omega_T \\ 
\psi_2(t,x) &= \psi_{0, 2}(\xi_{\tilde{\cV}}(t,x;T))\partial_2 \xi_{\tilde{\cV}}(t, x;T), &&(t, x) \in \Omega_T.
\intertext{Notice that for \((t,x)\in \OT\), \(\psi_1\) satisfies the following a.e.:}
    \partial_t \psi_1(t, x) &= \psi_{0, 1}'(\xi_{\tilde{\cV}}(t,x;T)) \partial_1 \xi_{\tilde{\cV}}(t, x;T)\notag\\
    \partial_x \psi_1(t,x) &= \psi_{0, 1}'(\xi_{\tilde{\cV}}(t,x;T)) \partial_2 \xi_{\tilde{\cV}}(t, x;T)\notag
\end{align}
which implies that
\begin{align*}
\partial_t \psi_1(t, x) + \tilde{\cV}(t, x) \partial_x \psi_1 (t, x) & = \psi_{0, 1}'(\xi_{\tilde{\cV}}(t,x;T)) \partial_1 \xi_{\tilde{\cV}}(t, x;T) + \tilde{\cV}(t, x)\psi_{0,1}'(\xi_{\tilde{\cV}}(t,x;T)) \partial_2 \xi_{\tilde{\cV}}(t, x;T)\\
& = \psi_{0, 1}'(\xi_{\tilde{\cV}}(t,x;T)) (\partial_1 \xi_{\tilde{\cV}}(t, x;T) + \tilde{\cV}(t, x) \partial_2 \xi_{\tilde{\cV}}(t,x;T))=0,
\end{align*}
where we have used \cref{lem:characteristics_properties}.
In addition, it holds by \cref{lem:characteristics_properties} that \(\psi_1(T, x) = \psi_{0, 1}(\xi_{\tilde{\cV}}(T,x;T)) = \psi_{0, 1}(x)\ \forall x\in\R\).
Furthermore, it can be shown that it holds for \((t,x)\in\OT\) that a.e.,
\begin{align*}
\partial_{t}\psi_2(t,x)+\partial_{x}\big(\psi_2(t,x)\tilde{\cV}(t,x)\big)=0
\end{align*}
and 
\begin{align}
    \psi_2(T, x) = \psi_{0, 2}(\xi_{\tilde{\cV}}(T,x;T))\partial_2 \xi_{\tilde{\cV}}(T, x;T) = \psi_{0, 2}(x)\partial_2 \xi_{\tilde{\cV}}(T,x;T)= \psi_{0, 2}(x),
\end{align}
where we have again used \cref{lem:properties_characteristics}.

Since \(\psi_1, \psi_2 \in W_\textnormal{c}^{1, \infty}((-42, T) \times \R)\) and  \((\tilde{\rho}, \tilde{\omega}) \in \Big(C\big([0,T];L^{1}_{\textnormal{loc}}(\R)\big)\cap L^{\infty}((0,T);L^{\infty}(\R)\cap TV(\R))\Big)^{2} \) is a weak solution (as defined in \cref{defi:weak_solution}) to \cref{eq:nonlocal_GARZ} and \cref{eq:init_data}, we can plug in \(\psi_{1}\) as a test function in the weak formulation and obtain
\begin{align*}
        & \iint_{\OT} \tilde{\rho}(t,x)\big(\partial_{t}\psi_1(t,x)+\partial_{x}\psi_1(t,x)\tilde{\cV}(t,x)\big)\dd x\dd t+\int_{\R}\rho_{0}(x)\psi_1(0,x)\dd x=\int_{\R}\tilde{\rho}(T,x)\psi_1(T,x)\dd x,
        \intertext{which, following the previous analysis, is equivalent to}
        & \int_{\R}\rho_{0}(x)\psi_{0, 1}(\xi_{\tilde{\cV}}(0, x;T))\dd x - \int_{\R}\tilde{\rho}(T,x)\psi_1(T,x)\dd x=0.\\
        \intertext{After a substitution, we have}
        & \int_{\R}\rho_{0}(x)\psi_{0, 1}(\xi_{\tilde{\cV}}(0, x;T))\dd x - \int_{\R}\tilde{\rho}(T,x)\psi_{0,1}(x)\dd x=0,
        \intertext{and after another substitution, taking advantage of \cref{lem:properties_characteristics}, we arrive at}
  &  \int_{\R}\tilde{\rho}(T,x)\psi_{0,1}(x)\dd x=\int_{\R}\rho_{0}(\xi_{\tilde{\cV}}(T, y;0))\partial_2 \xi_{\tilde{\cV}}(T, y;0) \psi_{0, 1}(y) \dd y.
\end{align*}

Since \(\psi_{0,1} \in W_\textnormal{c}^{1,\infty}([-42, T+1]\times\R)\) was arbitrary, by the fundamental Theorem of the calculus of variations, we have, for almost every \(y \in \R \) and \(\tau \in [0, T]\), 
\begin{align}
    \tilde{\rho}(\tau,y) & = \rho_{0}(\xi_{\tilde{\cV}}(\tau, y;0))\partial_2 \xi_{\tilde{\cV}}(\tau, y;0). 
\end{align}
However, this means that any weak solution \(\rho\) can be written in terms of characteristics. 

We repeat the process for \(\omega\) and obtain, according to the weak solution for \(\omega\) in \cref{defi:weak_solution}, when plugging in \(\psi_{2}\) as a test function,
\begin{align*}
     & \iint_{\OT} \tilde{\omega}(t,x)\big(\partial_{t}\psi_2(t,x)+\partial_{x}\big(\psi_2(t,x)\tilde{\cV}(t,x)\big)\big)\dd x\dd t+\int_{\R}\omega_{0}(x)\psi_2(0,x)\dd x=\int_{\R}\tilde{\omega}(T, x)\psi_2(T,x)\dd x.
     \intertext{This implies, according to the previous computations, that}
     & \int_{\R}\tilde{\omega}(T, x)\psi_2(T,x)\dd x = \int_{\R}\omega_{0}(x)\psi_2(0,x)\dd x,\\
     \intertext{which is equivalent to}
     & \int_{\R}\tilde{\omega}(T, x)\psi_{0,2}(x)\dd x = \int_{\R}\omega_{0}(x)\psi_{0, 2}(\xi_{\tilde{\cV}}(0, x;T))\partial_2 \xi_{\tilde{\cV}}(0, x;T)\dd x.\\
     \intertext{Owing to \cref{lem:properties_characteristics}, a substitution leads to}
     & \int_{\R}\tilde{\omega}(T, x)\psi_{0,2}(x)\dd x = \int_{\R}\omega_{0}(\xi_{\tilde{\cV}}(T, y;0))\psi_{0, 2}(y)\dd y.
\end{align*}
Again, because \(\psi_{0,2} \in W_{\textnormal{c}}^{1,\infty}([-42, T+1]\times\R)\) was arbitrary and using the fundamental lemma of calculus of variations, we have, for almost every \(y \in \R \) and \(\tau \in [0, T]\), 
\begin{align}
    \tilde{\omega}(\tau,y) & = \omega_{0}(\xi_{\tilde{\cV}}(\tau, y;0)).
\end{align}
Thus, we have shown that any weak solution \(\omega\) can be represented in terms of suitable characteristics.

Hence, for \((t, x) \in \Omega_T\), we can compute the nonlocal velocity, and we have
\begin{align*}
 \tilde{\cV}(t,x) & =\int_{\R}\keta(x-y)V(\tilde{\rho}(t,y),\tilde{\omega}(t,y))\dd y \\
& =\int_{\R}\keta(x-y)V(\rho_{0}(\xi_{\tilde{\cV}}(T, y;0))\partial_2 \xi_{\tilde{\cV}}(T, y;0),\omega_{0}(\xi_{\tilde{\cV}}(T, y;0)))\dd y.
\intertext{After substitution, this becomes}
\tilde{\cV}(t,x) &=\int_{\R}\keta(x-\xi_{\tilde{\cV}}(0, z;t))V\big(\tfrac{\rho_{0}(z)}{\partial_{2}\xi_{\tilde{\cV}}(0,z;t)},\omega_{0}(z)\big)\partial_2 \xi_{\tilde{\cV}}(0, z;t) \dd z.
\end{align*}
However, the right-hand side constitutes the fixed-point mapping in \cref{defi:fixed_point_mapping} and thus, we obtain
\[
\tilde{\cV}=\cF[\tilde{\cV}] \text{ on }\OT.
\]
By \cref{theo:fixed_point_mapping_uniqueness}, such a fixed point is unique, and thus, we have a unique \(\cV^{*}\), and by the solution formula in terms of characteristics, \((\rho,\omega)\) is unique. 
\end{description}
This completes the proof of \cref{theo:existence_uniqueness_small_time}.
\end{proof}
In the following, we obtain a property that is helpful in proving results. We start with the following lemma. 
\begin{lem}[continuity of {\(\cV\)} with respect to initial conditions]\label{lem:continuity_cV_initial_data}
Assume that \cref{ass:small_time_existence_uniqueness} holds.
   Let \((\rho_{0,i}, \omega_{0,i}) \in TV(\R;\R_{\geq 0})^2\) for \( i \in\{1,2\} \) satisfy the following bounds:  
\[
\|\rho_{0,i}\|_{L^{\infty}(\R)} \leq \rho_{0, \infty}, \, \|\omega_{0,i}\|_{L^{\infty}(\R)} \leq \omega_{0, \infty}, \, |\rho_{0,i}|_{TV(\R)} \leq \rho_{0, TV}, \, |\omega_{0,i}|_{TV(\R)} \leq \omega_{0, TV},
\]  
where \(\rho_{0, \infty}, \omega_{0, \infty}, \rho_{0, TV}, \omega_{0, TV} \in\R_{>0}\) are some given constants. Without loss of generality, we set  
\begin{equation}
Q(\rho_{0, \infty}, \omega_{0, \infty}) = (-42\rho_{0, \infty}, 42\rho_{0, \infty}) \times (-\omega_{0,\infty}, \omega_{0, \infty}) \subset \R^2\label{eq:Q_0}
\end{equation}
and define the corresponding fixed-point mapping set \( \Omega(T) \) and other related constants as in \cref{defi:set_Omega}, using \(Q(\rho_{0, \infty}, \omega_{0, \infty})\) with \(T>0\) sufficiently small. 
For \(i\in\{1,2\}\), let \( \cV_{i} \) be the unique solution to the corresponding fixed-point problem in \cref{theo:fixed_point_mapping_uniqueness}, and let \( \xi_{\cV_{i}}:(0,T)\times\R\times (0,T)\rightarrow\R \) be the solution to the corresponding characteristic equation in \cref{defi:characteristics} for the velocity \( \mathcal{V}_{i} \). Denote the corresponding weak solutions by \[(\rho_i, \omega_i) \in \Big(C\big([0,T];L^{1}_{\textnormal{loc}}(\R)\big) \cap L^{\infty}((0,T);TV(\R))\Big)^2,\] in the sense of \cref{defi:weak_solution}.  
Then, for every \( \varepsilon\in\R_{>0} \), there exists \( \delta\in\R_{>0} \) such that   \(\forall t \in [0, T]\),
\[
\|\rho_{0,1} - \rho_{0,2}\|_{L^1(\R)} + \|\omega_{0,1} - \omega_{0,2}\|_{L^1(\R)} \leq \delta \implies \|\mathcal{V}_1(t, \cdot) - \mathcal{V}_2(t, \cdot)\|_{L^{\infty}(\R)} + \|\mathcal{V}_1(t, \cdot) - \mathcal{V}_2(t, \cdot)\|_{TV(\R)} \leq \varepsilon.
\]
\end{lem}
\begin{proof}
We begin by examining the term \(\|\cV_1(t, \cdot) - \cV_2(t, \cdot)\|_{L^{\infty}(\R)}\), \(\forall\, t\in [0,T]\). Notice that \( \forall \, (t, x) \in [0, T] \times \R\), according to \cref{defi:fixed_point_mapping}, when splitting the integrals into different regions (compare the similar estimate in the proof of \cref{lem:cF_Lipschitz_velocities}), it holds that
\begin{align*}
    & |\cV_1(t,x) - \cV_2(t, x)| \\
     & = \bigg|\int_{\R} \keta(x-\xi_{\cV_1}(0,y;t))V\Big(\tfrac{\rho_{0, 1}(y)}{\partial_{2}\xi_{\cV_1}(0,y;t)},\omega_{0, 1}(y)\Big)\partial_{2}\xi_{\cV_1}(0,y;t)\dd y \\
    & \quad - \int_{\R} \keta(x-\xi_{\cV_2}(0,y;t))V\Big(\tfrac{\rho_{0, 2}(y)}{\partial_{2}\xi_{\cV_2}(0,y;t)},\omega_{0, 2}(y)\Big)\partial_{2}\xi_{\cV_2}(0,y;t)\dd y\bigg|.\\
        \intertext{After adding several zeros and applying triangle inequalities, this becomes}
     & |\cV_1(t,x) - \cV_2(t, x)| \\
     & \leq \int_{\R} \Big| \keta(x-\xi_{\cV_1}(0,y;t)) - \keta(x-\xi_{\cV_2}(0,y;t))\Big|V\Big(\tfrac{\rho_{0, 1}(y)}{\partial_{2}\xi_{\cV_1}(0,y;t)},\omega_{0, 1}(y)\Big)\partial_{2}\xi_{\cV_1}(0,y;t)\dd y\\
      & \quad +  \int_{\R} |\keta(x-\xi_{\cV_2}(0,y;t))|\Big|V\Big(\tfrac{\rho_{0, 1}(y)}{\partial_{2}\xi_{\cV_1}(0,y;t)},\omega_{0, 1}(y)\Big) -V\Big(\tfrac{\rho_{0, 2}(y)}{\partial_{2}\xi_{\cV_1}(0,y;t)},\omega_{0, 1}(y)\Big)\Big|\partial_{2}\xi_{\cV_1}(0,y;t)\dd y\\
        & \quad +  \int_{\R} |\keta(x-\xi_{\cV_2}(0,y;t))|\Big| V\Big(\tfrac{\rho_{0, 2}(y)}{\partial_{2}\xi_{\cV_1}(0,y;t)},\omega_{0, 1}(y)\Big) - V\Big(\tfrac{\rho_{0, 2}(y)}{\partial_{2}\xi_{\cV_2}(0,y;t)},\omega_{0, 1}(y)\Big) \Big| \partial_{2}\xi_{\cV_1}(0,y;t)\dd y \\
          & \quad + \int_{\R} |\keta(x-\xi_{\cV_2}(0,y;t))| \Big| V\Big(\tfrac{\rho_{0, 2}(y)}{\partial_{2}\xi_{\cV_2}(0,y;t)},\omega_{0, 1}(y)\Big) - V\Big(\tfrac{\rho_{0, 2}(y)}{\partial_{2}\xi_{\cV_2}(0,y;t)},\omega_{0, 2}(y)\Big) \Big| \partial_{2}\xi_{\cV_1}(0,y;t)\dd y \\
          & \quad + \int_{\R} \Big|\keta(x-\xi_{\cV_2}(0,y;t))V\Big(\tfrac{\rho_{0, 2}(y)}{\partial_{2}\xi_{\cV_2}(0,y;t)},\omega_{0, 2}(y)\Big)\Big| \Big| \partial_{2}\xi_{\cV_1}(0,y;t) - \partial_{2}\xi_{\cV_2}(0,y;t)\Big| \dd y.
          \intertext{Using the fact that \(\kappa\) is essentially bounded under \cref{ass:small_time_existence_uniqueness} as a consequence of the assumed \(BV\) regularity, \cref{lem:TV_estimate_composition} as well as the bounds on \(\partial_{2}\xi\) as given in \cref{lem:properties_characteristics}, the bound on the spatial derivative of \(\cV\) supplied in \cref{lem:self_mapping_small_time}, and the definition of \(\Omega(T)\) in \cref{defi:set_Omega}, we find that}
          & |\cV_1(t,x) - \cV_2(t, x)| \\
          &\leq \|V\|_{L^{\infty}(Q(\rho_{0, \infty}, \omega_{0, \infty}))}\exp(t\cV^{1, \infty})\int_{\R} \Big| \keta(x-\xi_{\cV_1}(0,y;t)) - \keta(x-\xi_{\cV_2}(0,y;t))\Big| \dd y \\
      & \quad +  \|\keta\|_{L^{\infty}(\R)} \exp(t \cV^{1,\infty})\int_{\R} \Big|V\Big(\tfrac{\rho_{0, 1}(y)}{\partial_{2}\xi_{\cV_1}(0,y;t)},\omega_{0, 1}(y)\Big) -V\Big(\tfrac{\rho_{0, 2}(y)}{\partial_{2}\xi_{\cV_1}(0,y;t)},\omega_{0, 1}(y)\Big)\Big|\dd y\\
        & \quad +   \|\keta\|_{L^{\infty}(\R)} \exp(t \cV^{1,\infty}) \int_{\R} \Big| V\Big(\tfrac{\rho_{0, 2}(y)}{\partial_{2}\xi_{\cV_1}(0,y;t)},\omega_{0, 1}(y)\Big) - V\Big(\tfrac{\rho_{0, 2}(y)}{\partial_{2}\xi_{\cV_2}(0,y;t)},\omega_{0, 1}(y)\Big) \Big| \dd y \\
          & \quad +  \|\keta\|_{L^{\infty}(\R)} \exp(t \cV^{1,\infty}) \int_{\R} \Big| V\Big(\tfrac{\rho_{0, 2}(y)}{\partial_{2}\xi_{\cV_2}(0,y;t)},\omega_{0, 1}(y)\Big) - V\Big(\tfrac{\rho_{0, 2}(y)}{\partial_{2}\xi_{\cV_2}(0,y;t)},\omega_{0, 2}(y)\Big) \Big| \dd y \\
          & \quad +  \|\keta\|_{L^{\infty}(\R)} \|V\|_{L^{\infty}(Q(\rho_{0, \infty}, \omega_{0, \infty}))} \int_{\R} \Big| \partial_{2}\xi_{\cV_1}(0,y;t) - \partial_{2}\xi_{\cV_2}(0,y;t)\Big| \dd y.
          \intertext{Using the local Lipschitz continuity of \(V\) with respect to both arguments, we have}
          & |\cV_1(t,x) - \cV_2(t, x)| \\
          & \leq |\kappa|_{TV(\R)}\|\xi_{\cV_{1}}(t,\cdot;0)-\xi_{\cV_{2}}(t,\cdot;0)\|_{L^{\infty}(\R)}\|V\|_{L^{\infty}(Q(\rho_{0, \infty}, \omega_{0, \infty}))}\exp(t\cV^{1, \infty})\\
      & \quad +  \|\keta\|_{L^{\infty}(\R)} \exp(2t \cV^{1,\infty}) 
       \|\partial_1 V\|_{L^{\infty}(Q(\rho_{0, \infty}, \omega_{0, \infty}))} 
       \int_{\R} \big|\rho_{0, 1}(y)-\rho_{0, 2}(y)\big|\dd y\\
      & \quad +   \|\keta\|_{L^{\infty}(\R)} \exp(3t \cV^{1,\infty})  \|\partial_1 V\|_{L^{\infty}(Q(\rho_{0, \infty}, \omega_{0, \infty}))}\|\rho_{0,2}\|_{L^{\infty}(\R)}\int_{\R} \Big| \partial_{2}\xi_{\cV_1}(0,y;t) - \partial_{2}\xi_{\cV_2}(0,y;t) \Big| \dd y \\
     & \quad +  \|\keta\|_{L^{\infty}(\R)} \exp(t \cV^{1,\infty})  \|\partial_2 V\|_{L^{\infty}(Q(\rho_{0, \infty}, \omega_{0, \infty}))} \int_{\R} \big| \omega_{0, 1}(y) - \omega_{0, 2}(y) \big| \dd y \\
     & \quad +  \|\keta\|_{L^{\infty}(\R)} \|V\|_{L^{\infty}(Q(\rho_{0, \infty}, \omega_{0, \infty}))} \int_{\R} \big| \partial_{2}\xi_{\cV_1}(0,y;t) - \partial_{2}\xi_{\cV_2}(0,y;t)\big| \dd y.
     \intertext{By applying the stability result of the characteristics in \cref{lem:stability_characteristics}, we obtain}
      & |\cV_1(t,x) - \cV_2(t, x)| \\
      & \leq t |\kappa|_{TV(\R)}\|V\|_{L^{\infty}(Q(\rho_{0, \infty}, \omega_{0, \infty}))}\exp(2t\cV^{1, \infty})\|\cV_{1}(t,\cdot)-\cV_{2}(t,\cdot)\|_{L^{\infty}(\R)}\\
      & \quad +  \|\keta\|_{L^{\infty}(\R)} \exp(2t \cV^{1,\infty}) 
       \|\partial_1 V\|_{L^{\infty}(Q(\rho_{0, \infty}, \omega_{0, \infty}))}\|\rho_{0,1}-\rho_{0,2}\|_{L^{1}(\R)}\\
      & \quad +  \|\keta\|_{L^{\infty}(\R)} \exp(t \cV^{1,\infty})  \|\partial_2 V\|_{L^{\infty}(Q(\rho_{0, \infty}, \omega_{0, \infty}))} \|\omega_{0, 1}- \omega_{0, 2}\|_{L^{1}(\R)}\\
      &\quad +\|\kappa\|_{L^{\infty}(\R)}\big(\exp(3t \cV^{1,\infty})  \|\partial_1 V\|_{L^{\infty}(Q(\rho_{0, \infty}, \omega_{0, \infty}))}\|\rho_{0,2}\|_{L^{\infty}(\R)}+\|V\|_{L^{\infty}(Q(\rho_{0, \infty}, \omega_{0, \infty}))}\big)\\
      &\qquad \cdot\exp\big(t\cV^{1,\infty}\big) t\Big(\cV^{1,TV} \|\xi_{\cV_{1}} - \xi_{\cV_{2}}\|_{L^{\infty}(\Omega_T \times [0, T])} + \exp\big(t\cV^{1,\infty}\big) \|\cV_{1} - \cV_{2}\|_{L^{\infty}((0, T); TV(\R))} \Big).
      \intertext{Applying \cref{lem:stability_characteristics} once more leads to}
      & |\cV_1(t,x) - \cV_2(t, x)| \\
      &\leq t |\kappa|_{TV(\R)}\|V\|_{L^{\infty}(Q(\rho_{0, \infty}, \omega_{0, \infty}))}\exp(2t\cV^{1, \infty})\|\cV_{1}(t,\cdot)-\cV_{2}(t,\cdot)\|_{L^{\infty}(\R)}\\
      & \quad +  \|\keta\|_{L^{\infty}(\R)} \exp(2t \cV^{1,\infty}) 
       \|\partial_1 V\|_{L^{\infty}(Q(\rho_{0, \infty}, \omega_{0, \infty}))}\|\rho_{0,1}-\rho_{0,2}\|_{L^{1}(\R)}\\
      & \quad +  \|\keta\|_{L^{\infty}(\R)} \exp(t \cV^{1,\infty})  \|\partial_2 V\|_{L^{\infty}(Q(\rho_{0, \infty}, \omega_{0, \infty}))} \|\omega_{0, 1}- \omega_{0, 2}\|_{L^{1}(\R)}\\
      &\quad +\|\kappa\|_{L^{\infty}(\R)}\big(\exp(3t \cV^{1,\infty})  \|\partial_1 V\|_{L^{\infty}(Q(\rho_{0, \infty}, \omega_{0, \infty}))}\|\rho_{0,2}\|_{L^{\infty}(\R)}+\|V\|_{L^{\infty}(Q(\rho_{0, \infty}, \omega_{0, \infty}))}\big)\\
      &\qquad \cdot\exp\big(2t\cV^{1,\infty}\big) t\Big(\cV^{1,TV} t\|\cV_{1}-\cV_{2}\|_{L^{\infty}((0,T);L^{\infty}(\R))}+ \|\cV_{1} - \cV_{2}\|_{L^{\infty}((0, T); TV(\R))} \Big).
\end{align*}
As can be seen, the term on the right-hand side involving \(\|\cV-\tilde{\cV}\|_{L^{\infty}((0,T);L^{\infty}(\R))}\) can be absorbed into the same term on the left-hand side (after taking the supremum over \(x\in\R\)), owing to the coefficient \(t\), which we can choose to be small enough that only the \(L^{1}\) distance of the initial data remains. However, one further term also remains, namely, \(\|\cV - \tilde{\cV}\|_{L^{\infty}((0, T); TV(\R))}\), which necessitates us to go further with our analysis and look also into the following total variation estimate.

For this estimate, we again follow the strategy of the proof of \cref{lem:cF_Lipschitz_velocities}. We find, for a smooth kernel \(\kappa_{\eps}\) as in \cref{eq:BV_approximation}, that
\begin{align*}
   & |\partial_2 \cV_1(t, x) - \partial_2 \cV_2(t, x)| \\
    &\leq \bigg|\int_{\R} \keta_{\eps}'(x-\xi_{\cV_1}(0,y;t))V\Big(\tfrac{\rho_{0,1}(y)}{\partial_{2}\xi_{\cV_1}(0,y;t)},\omega_{0,1}(y)\Big)\partial_{2}\xi_{\cV_1}(0,y;t)\dd y\\
    &\qquad\qquad - \int_{\R} \keta_{\eps}'(x-\xi_{\cV_2}(0,y;t))V\Big(\tfrac{\rho_{0,2}(y)}{\partial_{2}\xi_{\cV_2}(0,y;t)},\omega_{0,2}(y)\Big)\partial_{2}\xi_{\cV_2}(0,y;t)\dd y\bigg|.\\
     \intertext{Setting \(z(y) = \xi_{\cV_1}(0, y; t)\) and \(z(y) = \xi_{\cV_2}(0, y; t)\) separately in the two integrals and adding zeros yields}
     & |\partial_2 \cV_1(t, x) - \partial_2 \cV_2(t, x)| \\
    &\leq \|\partial_2 V\|_{L^{\infty}(Q(\rho_{0, \infty}, \omega_{0, \infty}))} \int_{\R} \big|\keta_{\eps}'(x-z)  \big|\big| \omega_{0,1}(\xi_{\cV_1}(t, z; 0))- \omega_{0,1}(\xi_{\cV_2}(t, z; 0)) \big| \dd z \\
     &\quad +\|\partial_2 V\|_{L^{\infty}(Q(\rho_{0, \infty}, \omega_{0, \infty}))} \int_{\R} \big|\keta_{\eps}'(x-z)  \big|\big| \omega_{0,1}(\xi_{\cV_2}(t, z; 0))- \omega_{0,2}(\xi_{\cV_2}(t, z; 0)) \big|  \dd z \\
      &\quad +\|\partial_1 V\|_{L^{\infty}(Q(\rho_{0, \infty}, \omega_{0, \infty}))} \int_{\R} \big|\keta_{\eps}'(x-z)  \big|\big| \rho_{0,1}(\xi_{\cV_1}(t, z; 0))\partial_2 \xi_{\cV_1}(t, z; 0)- \rho_{0,2}(\xi_{\cV_1}(t, z; 0))\partial_2 \xi_{\cV_1}(t, z; 0) \big|  \dd z \\
     &\quad +\|\partial_1 V\|_{L^{\infty}(Q(\rho_{0, \infty}, \omega_{0, \infty}))} \int_{\R} \big|\keta_{\eps}'(x-z)  \big|\big| \rho_{0,2}(\xi_{\cV_1}(t, z; 0))\partial_2 \xi_{\cV_1}(t, z; 0)- \rho_{0,2}(\xi_{\cV_2}(t, z; 0))\partial_2 \xi_{\cV_1}(t, z; 0) \big|  \dd z \\
      &\quad +\|\partial_1 V\|_{L^{\infty}(Q(\rho_{0, \infty}, \omega_{0, \infty}))} \int_{\R} \big|\keta_{\eps}'(x-z)  \big| \big| \rho_{0,2}(\xi_{\cV_2}(t, z; 0))\partial_2 \xi_{\cV_1}(t, z; 0)- \rho_{0,2}(\xi_{\cV_2}(t, z; 0))\partial_2 \xi_{\cV_2}(t, z; 0) \big|  \dd z.
\end{align*}
Then, for \(\forall t \in [0, T]\), after applying the \(TV\) seminorm, which comes down to computing the \(L^{1}\)-norm of \(\partial_{2}\cV_{1}-\partial_{2}\cV_{2}\) because \(\cV\) is Lipschitz in the spatial variable, we have
\begin{align*}
    & |\cV_1(t, \cdot) - \cV_2(t, \cdot)|_{TV(\R)} = \int_{\R} |\partial_2 \cV_1(t, x) - \partial_{2}\cV_2(t,x)|\dd x \\
    &\leq \|\partial_2 V\|_{L^{\infty}(Q(\rho_{0, \infty}, \omega_{0, \infty}))} \iint_{\R^{2}} \big|\keta_{\eps}'(x-z)  \big|\big| \omega_{0,1}(\xi_{\cV_1}(t, z; 0))- \omega_{0,1}(\xi_{\cV_2}(t, z; 0)) \big|  \dd z \dd x \\
     &\ +\|\partial_2 V\|_{L^{\infty}(Q(\rho_{0, \infty}, \omega_{0, \infty}))} \iint_{\R^{2}} \big|\keta_{\eps}'(x-z)  \big|\big| \omega_{0,1}(\xi_{\cV_2}(t, z; 0))- \omega_{0,2}(\xi_{\cV_2}(t, z; 0)) \big|  \dd z \dd x \\
      &\ +\|\partial_1 V\|_{L^{\infty}(Q(\rho_{0, \infty}, \omega_{0, \infty}))}\!\!\! \iint_{\R^{2}}\!\!\!\! \big|\keta_{\eps}'(x-z)  \big|\big| \rho_{0,1}(\xi_{\cV_1}(t, z; 0))\partial_2 \xi_{\cV_1}(t, z; 0)- \rho_{0,2}(\xi_{\cV_1}(t, z; 0))\partial_2 \xi_{\cV_1}(t, z; 0) \big| \dd z \dd x \\
     &\ +\|\partial_1 V\|_{L^{\infty}(Q(\rho_{0, \infty}, \omega_{0, \infty}))} \!\!\!\iint_{\R^{2}}\!\!\!\! \big|\keta_{\eps}'(x-z)  \big|\big| \rho_{0,2}(\xi_{\cV_1}(t, z; 0))\partial_2 \xi_{\cV_1}(t, z; 0)- \rho_{0,2}(\xi_{\cV_2}(t, z; 0))\partial_2 \xi_{\cV_1}(t, z; 0) \big| \dd z \dd x \\
      &\ +\|\partial_1 V\|_{L^{\infty}(Q(\rho_{0, \infty}, \omega_{0, \infty}))} \!\!\!\iint_{\R^{2}}\!\!\!\!\big|\keta_{\eps}'(x-z)  \big| \big| \rho_{0,2}(\xi_{\cV_2}(t, z; 0))\partial_2 \xi_{\cV_1}(t, z; 0)- \rho_{0,2}(\xi_{\cV_2}(t, z; 0))\partial_2 \xi_{\cV_2}(t, z; 0) \big| \dd z \dd x.
      \intertext{Now, by changing the order of integration and applying \cref{lem:TV_estimate_composition}, we can see that}
      & |\cV_1(t, \cdot) - \cV_2(t, \cdot)|_{TV(\R)} \\
      &\leq \|\partial_2 V\|_{L^{\infty}(Q(\rho_{0, \infty}, \omega_{0, \infty}))} |\kappa_{\eps}|_{TV(\R)}|\omega_{0,1}|_{TV(\R)}\|\xi_{\cV_1}(0, \cdot; t)- \xi_{\cV_2}(0, \cdot; t)\|_{L^{\infty}(\R)}\\
     &\ +\|\partial_2 V\|_{L^{\infty}(Q(\rho_{0, \infty}, \omega_{0, \infty}))}\exp\big(t\cV^{1,\infty}\big)|\kappa_{\eps}|_{TV(\R)}\|\omega_{0,1}-\omega_{0,2}\|_{L^{1}(\R)}\\
      &\ +\|\partial_1 V\|_{L^{\infty}(Q(\rho_{0, \infty}, \omega_{0, \infty}))}\exp\big(2t\cV^{1,\infty}\big)|\kappa_{\eps}|_{TV(\R)} \|\rho_{0,1}-\rho_{0,2}\|_{L^{1}(\R)}\\
     &\ +\|\partial_1 V\|_{L^{\infty}(Q(\rho_{0, \infty}, \omega_{0, \infty}))}\exp\big(t\cV^{1,\infty}\big)|\kappa_{\eps}|_{TV(\R)}|\rho_{0,2}|_{TV(\R)}\|\xi_{\cV_{1}}(0,\cdot;t)-\xi_{\cV_{2}}(0,\cdot;t)\|_{L^{\infty}(\R)}\\
      &\ +\|\partial_1 V\|_{L^{\infty}(Q(\rho_{0, \infty}, \omega_{0, \infty}))}|\kappa_{\eps}|_{TV(\R)}\|\rho_{0,2}\|_{L^{\infty}(\R)}\|\partial_{2}\xi_{\cV_{1}}(t,\cdot;0)-\partial_{2}\xi_{\cV_{2}}(t,\cdot;0)\|_{L^{1}(\R)}.
      \intertext{By rearranging terms, applying the stability results on the characteristics in \cref{lem:stability_characteristics}, and passing to the limit \(\eps\rightarrow 0\), we have}
      & |\cV_1(t, \cdot) - \cV_2(t, \cdot)|_{TV(\R)} \\
      &\leq|\kappa|_{TV(\R)}\Big(\|\partial_2 V\|_{L^{\infty}(Q(\rho_{0, \infty}, \omega_{0, \infty}))} |\omega_{0,1}|_{TV(\R)}+\|\partial_1 V\|_{L^{\infty}(Q(\rho_{0, \infty}, \omega_{0, \infty}))}\exp\big(t\cV^{1,\infty}\big)|\rho_{0,2}|_{TV(\R)}\Big)\\
      &\qquad\cdot t\|\cV_{1}-\cV_{2}\|_{L^{\infty}((0,t);L^{\infty}(\R))}\exp\big(t\cV^{1,\infty}\big)\\
     &\quad +\|\partial_2 V\|_{L^{\infty}(Q(\rho_{0, \infty}, \omega_{0, \infty}))}\exp\big(t\cV^{1,\infty}\big)|\kappa|_{TV(\R)}\|\omega_{0,1}-\omega_{0,2}\|_{L^{1}(\R)}\\
      &\quad +\|\partial_1 V\|_{L^{\infty}(Q(\rho_{0, \infty}, \omega_{0, \infty}))}\exp\big(2t\cV^{1,\infty}\big)|\kappa|_{TV(\R)} \|\rho_{0,1}-\rho_{0,2}\|_{L^{1}(\R)}\\
      &\quad +\|\partial_1 V\|_{L^{\infty}(Q(\rho_{0, \infty}, \omega_{0, \infty}))}|\kappa|_{TV(\R)}\|\rho_{0,2}\|_{L^{\infty}(\R)}\cdot t\exp\big(t\cV^{1,\infty}\big)\\
      &\qquad\cdot\Big(\cV^{1,\infty}\|\xi_{\cV_{1}}-\xi_{\cV_{2}}\|_{L^{\infty}((0,t)\times\R\times(0,t))}+\exp\big(t\cV^{1,\infty}\big)|\cV_{1}-\cV_{2}|_{L^{\infty}((0,t);TV(\R))}\Big).
      \intertext{Applying \cref{lem:stability_characteristics} once more leads to}
      & |\cV_1(t, \cdot) - \cV_2(t, \cdot)|_{TV(\R)} \\
      &\leq|\kappa|_{TV(\R)}\Big(\|\partial_2 V\|_{L^{\infty}(Q(\rho_{0, \infty}, \omega_{0, \infty}))} |\omega_{0,1}|_{TV(\R)}+\|\partial_1 V\|_{L^{\infty}(Q(\rho_{0, \infty}, \omega_{0, \infty}))}\exp\big(t\cV^{1,\infty}\big)|\rho_{0,2}|_{TV(\R)}\Big)\\
      &\qquad\cdot t\|\cV_{1}-\cV_{2}\|_{L^{\infty}((0,t);L^{\infty}(\R))}\exp\big(t\cV^{1,\infty}\big)\\
     &\quad +\|\partial_2 V\|_{L^{\infty}(Q(\rho_{0, \infty}, \omega_{0, \infty}))}\exp\big(t\cV^{1,\infty}\big)|\kappa|_{TV(\R)}\|\omega_{0,1}-\omega_{0,2}\|_{L^{1}(\R)}\\
      &\quad +\|\partial_1 V\|_{L^{\infty}(Q(\rho_{0, \infty}, \omega_{0, \infty}))}\exp\big(2t\cV^{1,\infty}\big)|\kappa|_{TV(\R)} \|\rho_{0,1}-\rho_{0,2}\|_{L^{1}(\R)}\\
      &\quad +\|\partial_1 V\|_{L^{\infty}(Q(\rho_{0, \infty}, \omega_{0, \infty}))}|\kappa|_{TV(\R)}\|\rho_{0,2}\|_{L^{\infty}(\R)}\cdot t\exp\big(2t\cV^{1,\infty}\big)\\
      &\qquad\cdot\Big(t\|\cV_{1}-\cV_{2}\|_{L^{\infty}((0,t);L^{\infty}(\R))}+|\cV_{1}-\cV_{2}|_{L^{\infty}((0,t);TV(\R))}\Big).
\end{align*}
When we combine both estimates, namely, the estimate in terms of \(\|\cV_{1}-\cV_{2}\|_{L^{\infty}((0,t);L^{\infty}(\R))}\) and the total variation estimate, we obtain, structurally,
\begin{align}
&\|\cV_{1}-\cV_{2}\|_{L^{\infty}((0,t);L^{\infty}(\R))}+|\cV_{1}-\cV_{2}|_{L^{\infty}((0,t);TV(\R))}\label{eq:compensation_0}\\
&\leq C(t)\big(\|\cV_{1}-\cV_{2}\|_{L^{\infty}((0,t);L^{\infty}(\R))}+|\cV_{1}-\cV_{2}|_{L^{\infty}((0,t);TV(\R))}\big)\label{eq:compensation_1}\\
&\quad+ D\big(\|\rho_{0,1}-\rho_{0,2}\|_{L^{1}(\R)}+\|\omega_{0,1}-\omega_{0,2}\|_{L^{1}(\R)}\big)\label{eq:compensation_2}
\end{align}
with \(D\in\R_{>0}\) a properly chosen constant independent of \(\rho_{0,i}\) and \(\omega_{0,i}\), for \(i\in\{1,2\}\), and a continuous function \(C:[0,T]\rightarrow\R\) for a \(T\in\R_{>0}\) with \(\lim_{t\rightarrow 0} C(t)=0\).
Thus, by choosing the time horizon \(t\in\R_{>0}\) small enough, one can absorb \cref{eq:compensation_1} with \cref{eq:compensation_0} and show that the right-hand side becomes small, that is,
\[
(1-C(t))\Big(\|\cV_{1}-\cV_{2}\|_{L^{\infty}((0,t);L^{\infty}(\R))}+|\cV_{1}-\cV_{2}|_{L^{\infty}((0,t);TV(\R))}\Big)\leq D\big(\|\rho_{0,1}-\rho_{0,2}\|_{L^{1}(\R)}+\|\omega_{0,1}-\omega_{0,2}\|_{L^{1}(\R)}\big).
\]
Therefore, we can make the left-hand side arbitrarily small when the \(L^{1}\) distance on the right-hand side is small, given that
\(C(t)<1\), which can be achieved by choosing a small enough time.
This concludes the proof.
\end{proof}
\begin{rem}[Gr\"onwall vs.\ compensation method]
   Note that we could have instead worked with a Gr\"onwall argument in the proof of \cref{lem:continuity_cV_initial_data} to derive the claimed convergence. However, this would have necessitated a refinement of the estimates in \cref{lem:stability_characteristics}, and because we were at that point only interested in the stability on a small time horizon, we chose to use compensating terms for a small enough time, as in \crefrange{eq:compensation_0}{eq:compensation_2}.
\end{rem}
Having derived the stability estimate in \cref{lem:continuity_cV_initial_data} of the nonlocal velocity \(\cV\) in terms of the initial data in \(L^{1}(\R)\), we now analyze the stability of solutions in the following.
\begin{theo}[continuity of \(\rho\) and \(\omega\) with respect to initial conditions]\label{theo:stability}
    Let the assumptions in \cref{lem:continuity_cV_initial_data} hold. Then, for every \(\varepsilon\in  \R_{>0}\), there exists \(\delta\in\R_{>0}\) such that 
\begin{align*}
    \|\rho_{0, 1} - \rho_{0, 2}\|_{L^1(\R)} + \|\omega_{0, 1} - \omega_{0,2}\|_{L^1(\R)} \leq \delta \implies \|\rho_1 - \rho_2\|_{C([0, T]; L^1(\R))}+\|\omega_1 - \omega_2\|_{C([0, T]; L^1(\R))} \leq \varepsilon.
\end{align*}
\end{theo}
\begin{proof}
Let \((\rho_{1},\omega_{1})\in C([0,T];L^{1}_{\textnormal{loc}}(\R))^{2}\ni (\rho_{2},\omega_{2})\) denote the solution for the corresponding initial datum.
Using the solution formula in \cref{eq:solution}, we can write the \(L^{1}\) distance of the solution as follows for \(t\in[0,T]\) (we do not yet know whether this difference is finite, but our estimates later will prove that it is):
\begin{align*}
&\|\omega_1(t, \cdot) - \omega_2(t, \cdot)\|_{L^{1}(\R)}\\
& = \int_{\R} |\omega_{0,1}(\xi_{\cV_{1}}(t,x;0)) - \omega_{0,2}(\xi_{\cV_2}(t,x;0))| \dd x .
\intertext{After adding zeros and using the triangular inequality, we have}
    &\|\omega_1(t, \cdot) - \omega_2(t, \cdot)\|_{L^{1}(\R)}\\
    & \leq  \int_{\R} |\omega_{0,1}(\xi_{\cV_{1}}(t,x;0)) - \omega_{0,1}(\xi_{\cV_{2}}(t,x;0)) | \dd x + \int_{\R}|\omega_{0,1}(\xi_{\cV_{2}}(t,x;0))- \omega_{0,2}(\xi_{\cV_2}(t,x;0))| \dd x.
    \intertext{By applying \cref{lem:TV_estimate_composition} and performing an integration by substituting \(y = \xi_{\cV_2}(t, x; 0)\), we obtain}
    &\|\omega_1(t, \cdot) - \omega_2(t, \cdot)\|_{L^{1}(\R)}\\
    & \leq |\omega_{0,1}|_{TV(\R)} \|\xi_{\cV_1}(0, \cdot; t)-\xi_{\cV_2}(0, \cdot; t)\|_{L^{\infty}(\R)} + \exp(t \cV^{1, \infty}) \|\omega_{0,1} - \omega_{0,2}\|_{L^1(\R)}.
    \intertext{Finally, using \cref{xi_cV_L_infty}, we arrive at}
    & \leq \exp(t \cV^{1, \infty}) \|\omega_{0,1} - \omega_{0,2}\|_{L^1(\R)} + |\omega_{0,1}|_{TV(\R)} t \exp\big(t\cV^{1, \infty}\big)\|\cV_1-\cV_2\|_{L^{\infty}((0,T);L^{\infty}(\R))}.
\end{align*}
Thus, for the specified time horizon, when choosing a small \(L^{1}\)-norm of the initial data, both summands become arbitrarily small, 
the latter due to \cref{lem:continuity_cV_initial_data}.
We continue with \(\rho_{1}-\rho_{2}\) and obtain, again using the solution formula in \cref{eq:solution},
\begin{align*}
    & \int_{\R}|\rho_1(t, x) - \rho_2(t, x)| \dd x
     = \int_{\R} |\rho_{0,1}(\xi_{\cV_1}(t,x;0))\partial_{2}\xi_{\cV_1}(t,x;0) - \rho_{0,2}(\xi_{\cV_2}(t,x;0))\partial_{2}\xi_{\cV_2}(t,x;0)| \dd x.
     \intertext{Adding zeros and using the triangular inequality yields}
     &\int_{\R}|\rho_1(t, x) - \rho_2(t, x)| \dd x\\
     & \leq \int_{\R} |\rho_{0,1}(\xi_{\cV_1}(t,x;0))\partial_{2}\xi_{\cV_1}(t,x;0) - \rho_{0,1}(\xi_{\cV_1}(t,x;0))\partial_{2}\xi_{\cV_2}(t,x;0) | \dd x  \\
     & \quad +\int_{\R} | \rho_{0,1}(\xi_{\cV_1}(t,x;0))\partial_{2}\xi_{\cV_2}(t,x;0) -\rho_{0,1}(\xi_{\cV_2}(t,x;0))\partial_{2}\xi_{\cV_2}(t,x;0)| \dd x  \\
     & \quad + \int_{\R} | \rho_{0,1}(\xi_{\cV_2}(t,x;0))\partial_{2}\xi_{\cV_2}(t,x;0) -\rho_{0,2}(\xi_{\cV_2}(t,x;0))\partial_{2}\xi_{\cV_2}(t,x;0) | \dd x.
     \intertext{By again using \cref{lem:characteristics_properties} and \cref{lem:TV_estimate_composition}, as well as making a substitution in the last term, we find that}
     &\int_{\R}|\rho_1(t, x) - \rho_2(t, x)| \dd x\\
     & \leq \|\rho_{0,1}\|_{L^{\infty}(\R)} |\partial_{2}\xi_{\cV_1}(t, \cdot; 0) - \partial_{2}\xi_{\cV_2}(t, \cdot; 0)|_{TV(\R)} + \exp(t \cV^{1, \infty})|\rho_{0,1}|_{TV(\R)}\|\xi_{\cV_1}(0, \cdot; t) - \xi_{\cV_2}(0, \cdot; t)\|_{L^{\infty}(\R)}\\
     & \quad + \|\rho_{0,1} - \rho_{0,2}\|_{L^1(\R)}.
      \intertext{Using \cref{xi_cV_L_infty} and \cref{partial_2_xi_cV_L_1} from \cref{lem:stability_characteristics}, we obtain}
      &\int_{\R}|\rho_1(t, x) - \rho_2(t, x)| \dd x\\
     & \leq \|\rho_{0,1} - \rho_{0,2}\|_{L^1(\R)} + \exp(t \cV^{1, \infty})|\rho_{0,1}|_{TV(\R)}t\|\cV_1-\cV_2\|_{L^{\infty}((0,T);L^{\infty}(\R))}\exp\big(t\cV^{1, \infty}\big)\\
      & \quad + \|\rho_{0,1}\|_{L^{\infty}(\R)} \exp\big(t\cV^{1, \infty}\big)t \Big(\cV^{1, TV} t\|\cV_1-\cV_2\|_{L^{\infty}((0,T);L^{\infty}(\R))}+ \exp\big(t\cV^{1, \infty}\big) \|\cV_1 - \cV_2\|_{L^{\infty}((0, T); TV(\R))} \Big).
\end{align*}
This right-hand side can be made small by choosing a small \(L^{1}\) distance of the initial data and again applying \cref{lem:continuity_cV_initial_data}.

Altogether, we obtain the continuity of the solution with respect to a change of the initial data in \(L^{1}(\R)\). This concludes the proof.
\end{proof}
The stated stability result also enables us to prove---as is common for nonlocal conservation laws of a specific class \cite[Lemma 29]{Keimer2023}---that smooth solutions can approximate any weak solution and, in particular, that smoothness is conserved (depending on the regularity of the velocity function) \cite[Theorem 3.1]{coclite2022general}. This is detailed in the following.
\begin{cor}[stability of solutions with respect to the initial data/preservation of 
smoothness] \label{cor:Stability_of_solutions}
Let \cref{ass:small_time_existence_uniqueness} hold, and let \((\phi_{\eps})_{\eps\in\R_{>0}}\subset C^{\infty}(\R)\) be a standard mollifier. Define the mollified (smooth) initial data \(\omega_{0}^{\eps}\equiv \phi_{\eps}\ast \omega_{0}\) and \(\rho_{0}^{\eps}\equiv \phi_{\eps}\ast \rho_{0}\). Denote by \((\rho_{\eps},\omega_{\eps})\in \big(C([0,T];L^{1}_{\textnormal{loc}}(\R))\big)^{2}\) the corresponding solutions for the mollified initial data and by \((\rho,\omega)\in C([0,T];L^{1}_{\textnormal{loc}}(\R))^{2}\) the corresponding solution for the initial data \(\rho_{0},\omega_{0}\in TV(\R)\). Then, there exists a time horizon \(T\in\R_{>0}\) (independent of \(\eps\in\R_{>0}\)) such that
\[
\lim_{\eps\rightarrow 0}\left(\|\rho-\rho_{\eps}\|_{C([0,T];L^{1}(\R))}+\|\omega-\omega_{\eps}\|_{C([0,T];L^{1}(\R))}\right)=0.
\]
Furthermore, it holds that \(\rho_{\eps},\omega_{\eps}\in C^{1}([0,T]\times\R)\); for
\(V\in C^{\infty}(\R^{2})\), we even have \(\rho_{\eps},\omega_{\eps}\in C^{\infty}([0,T]\times\R)\).
\end{cor}
\begin{proof}
    The first part of the approximation can be seen to be a direct consequence of \cref{theo:stability} when recalling that for initial data \(\rho_{0},\omega_{0}\in TV(\R)\), it holds that
    \[
        \lim_{\eps\rightarrow0}\|\rho_{0}-\rho_{0}^{\eps}\|_{L^{1}(\R)}=0=\lim_{\eps\rightarrow0}\|\omega_{0}-\omega_{0}^{\eps}\|_{L^{1}(\R)}.
    \]
    It remains to show that the solution is smooth enough. By \cref{eq:solution}, we can write the solution for \((t,x)\in (0,T)\times\R\) a.e.\ as
    \[
\rho(t,x)=\rho_{0}(\xi_{\cV^{*}}(t,x;0))\partial_{2}\xi_{\cV^{*}}(t,x;0),\qquad \omega(t,x)=\omega_{0}(\xi_{\cV^{*}}(t,x;0)),\ 
    \]
    with the characteristics as defined in \cref{defi:characteristics} and the nonlocal \(\cV^{*}\) being the solution to the unique fixed-point equation in \cref{theo:fixed_point_mapping_uniqueness}.
Because we have proven that \(\partial_{2}\xi_{\cV^{*}}(\cdot,\ast;0)\in L^{\infty}((0,T);W^{1,\infty}(\R))\), it is apparent that when assuming a higher regularity on \(\omega_{0}\), \(\omega\) is Lipschitz continuous in space and time when applying the chain rule. This is not as clear for \(\rho\) as its solution formula already contains \(\partial_{2}\xi_{\cV^{*}}\), for which we have so far only proved that it admits some \(TV\) regularity. However, for \(\rho\) to be Lipschitz, we would require \(\partial_{2}\xi_{\cV^{*}}\) to be Lipschitz. As \(\xi\) is defined through the unique fixed point \(\cV^{*},\) we require such a fixed point to be smoother if the datum is smoother. This can be seen from \cref{eq:fixed_point_mapping} as the second spatial derivative of the right-hand side of the fixed-point equation involves derivatives of \(V,\ \rho_{0},\ \omega_{0}\), and \(x\mapsto \partial_{2}\xi_{\cV^{*}}(0,x;t)\), of which the latter exists and is bounded when \(\partial_{2}\cV\) is Lipschitz.

Thus, for more regular data, the solutions are Lipschitz in space and time; as such, they are strong solutions. Higher regularity and continuity follow by similar arguments.
\end{proof}
\subsection{Maximum and minimum principle, global existence}
So far, our assumptions on the involved datum, particularly on the velocity \(V\) and the kernel \(\kappa\) in \cref{ass:small_time_existence_uniqueness}, have been generic, allowing only an existence and uniqueness result on small time horizons.
In this section, we will restrict those data so that we can obtain the existence and uniqueness of solutions for arbitrarily large time horizons. The assumptions, which we already detailed in \cref{ass:long_time_horizon}, will also lead to invariant regions for the solution, a property that is particularly reasonable in traffic flow modeling and that enables us to obtain uniform bounds on the solution in a certain topology, so that we can work with a time-clustering argument to extend the solution to any finite time horizon.
\begin{theo}[maximum and minimum principle given the first part of \cref{ass:long_time_horizon}]\label{theo:maximum_principle}
Let the first part of \cref{ass:long_time_horizon} hold; that is, in addition to \cref{ass:small_time_existence_uniqueness}, there exists a \(\rho_{\max}\in\R_{>0}\) such that
\begin{multicols}{2}
    \begin{itemize}
    \item \((\rho_{0},\omega_{0})\in 
    \left(TV(\R;\R_{\geq0})\right)^{2},\ \rho_0\leqq \rho_{\max}\text{ on }\R\)
    \item \(0\leqq V \text{ on }[0,\rho_{\max}]\times [0,\|\omega_{0}\|_{L^{\infty}(\R)}]
    \)
    \item \(V(\rho_{\max},\cdot)\equiv 0 \text{ on } 
  [0,\|\omega_{0}\|_{L^{\infty}(\R)}]
    \)
    \item \( \supp(\kappa)\subset\R_{\leq 0}:\) \(\kappa\geq 0\) monotonically increasing on \(\R_{\leq 0}\),\ \(\|\kappa\|_{L^{1}(\R_{\leq0})}=1\).
\end{itemize}
\end{multicols}\noindent
\smallskip
Then, for any \(T\in\R_{>0}\) there exists a unique weak solution (in the sense of \cref{defi:weak_solution})
\[(\rho,\omega)\in \Big(C\big([0,T];L^{1}_{\textnormal{loc}}(\R)\big)\cap L^{\infty}((0,T);L^{\infty}(\R)\cap TV(\R))\Big)^{2}\]
to the Cauchy problem  
\cref{eq:nonlocal_GARZ}-\cref{eq:init_data}. Even more, the solution satisfies the following:
\begin{description}
\item[\(L^{\infty}\) bound:] \(
0\leqq \rho\leqq \rho_{\max},\qquad 0\leqq \omega\leqq \|\omega_{0}\|_{L^{\infty}(\R)}\ \text{ on } (0,T)\times\R, 
\)
\item[\(TV\) bound:] for a.e.\ \(t \in [0, T]\),  \(|\omega(t, \cdot)|_{TV(\R)}= |\omega_0|_{TV(\R)}\) and
\(Q_{\max}\coloneqq [0,\rho_{\max}]\times[0,\|\omega_{0}\|_{L^{\infty}(\R)}]\subset\R^{2}\)
\begin{align*}
   |\rho(t, \cdot)|_{TV(\R)} & \leq \Big(|\rho_{0}|_{TV(\R)}+t\rho_{\max}|\kappa|_{TV(\R_{<0})}\|\partial_{2}V\|_{L^{\infty}(Q_{\max})}|\omega_{0}|_{TV(\R)}\Big)\\
  &\quad \cdot\exp\Big(t|\kappa|_{TV(\R_{<0})}\big(\|V\|_{L^{\infty}(Q_{\max})}+\|\partial_{1}V\|_{L^{\infty}(Q_{\max})}\rho_{\max}\big)\Big).
\end{align*}
\end{description}

In addition, if \(V\) is monotonically decreasing in its first argument and increasing in its second argument, that is, if \(\partial_{1}V\leqq 0,\ \partial_{2}V\geqq0\), and \(\omega_{0}\) is monotonically decreasing, then we have that \(x\mapsto \omega(t,x)\) is nonincreasing for almost all times \(t\in[0,T],\) and we obtain a lower bound on the density \(\rho\):
\begin{equation}\label{eq:rho-lower-bound}
\essinf_{y\in\R} \rho_{0}(y)\leq\rho(t,x),\quad (t,x)\in\OT \text{ a.e.}    
\end{equation}
\end{theo}

\begin{proof}
Because of the solution formula in \cref{eq:solution}, we have for \(\omega_0 \in L^{\infty}(\R; \R_{\geq 0}),\) \((t,x)\in (0,T^{*})\times \R\) and \(\omega(t,x)=\omega_{0}(\xi_{\cV^{*}}(t,x;0))\), that \(0\leqq \omega\leqq \|\omega_{0}\|_{L^{\infty}(\R)}\) a.e.\ on \((0,T^{*})\times\R\)
as long as the solution exists, so for now only on the claimed short enough time horizon \(T^{*}\in\R_{>0}\), whose existence is guaranteed by \cref{theo:existence_uniqueness_small_time}.
The lower bound on the density \(\rho\) by zero follows immediately from the solution formula \cref{eq:solution}, as we proved that
\[
\rho(t,x)=\rho_{0}(\xi_{\cV^{*}}(t,x;0))\partial_{2}\xi_{\cV^{*}}(t,x;0),\ (t,x)\in (0,T^{*})\times\R \text{ a.e.,}
\]
and we have \(\partial_{2}\xi_{\cV^{*}}>0\) by construction and \(\rho_{0}\geqq0\) by assumption.

Concerning the upper bound, we smooth the initial data \((\rho_{0},\omega_{0}) \) with a standard mollifier parameterized by \(\eps\in\R_{>0}\), and we obtain by \cref{cor:Stability_of_solutions} that the corresponding solution \((\rho_{\eps},\omega_{\eps})\in C^{1}([0,T^{*}]\times\R)^{2}\) is continuously differentiable in space and time.
Then, we use Danskin's theorem and assume that \(\rho_{\eps}(t,x)\) is maximal at \(\tilde{x}\in\R\)  at a given time \(t\in[0,T]\). Then, when also smoothing the nonlocal kernel, as was suggested in \cref{eq:BV_approximation}, that for \(\delta\in\R_{>0}\),
\begin{align*}
    \partial_{t}\rho_{\eps}(t,x)&=-\cV[\rho_{\eps},\omega_{\eps}](t,x)\partial_{x}\rho_{\eps}(t,x)-\rho_{\eps}(t,x)\partial_{x}\cV[\rho_{\eps},\omega_{\eps}](t,x).
    \intertext{Evaluating this at \(x=\tilde{x}\), given that we are at a maximal point of \(\rho_{\eps}(t,\cdot)\) so that \(\partial_{2}\rho_{\eps}(t,\tilde{x})=0\), yields}
    \partial_{t}\rho_{\eps}(t,\tilde{x})&=-\rho_{\eps}(t,\tilde{x})\Big(\partial_{x}\int_{x}^{\infty}\kappa_{\delta}(x-y)V(\rho_{\eps}(t,y),\omega_{\eps}(t,y))\dd y\Big)\bigg|_{x=\tilde{x}}\\
    &=-\rho_{\eps}(t,\tilde{x})\bigg(\int_{\tilde{x}}^{\infty}\keta'_{\delta}(\tilde{x}-y)V(\rho_{\eps}(t,y),\omega_{\eps}(t,y))\dd y-\keta_{\delta}(0) V(\rho_{\eps}(t,\tilde{x}),\omega_{\eps}(t,\tilde{x}))\bigg).
    \end{align*}
If \(\rho_{\eps}(t,\tilde{x})<\rho_{\max}\), there is nothing to do, so we assume that \(\rho_{\eps}(t,\tilde{x})=\rho_{\max}\). However, \cref{ass:long_time_horizon} then implies that \(V(\rho_{\max},\cdot)\equiv0\), so
    \begin{align*}
    \partial_{t}\rho_{\eps}(t,\tilde{x})
    &=-\rho_{\eps}(t,\tilde{x})\int_{\tilde{x}}^{\infty}\keta_{\delta}'(\tilde{x}-y)V(\rho_{\eps}(t,y),\omega_{\eps}(t,y))\dd y.
    \intertext{Because \(\kappa_{\delta}'\geq 0\) by the monotonicity of \(\kappa\),\ \(\rho_{\eps}\geqq0\) by the previous argument, and \(V\geqq 0\) by assumption, we have that}
    \partial_{t}\rho_{\eps}(t,\tilde{x})&\leq 0
\end{align*}
uniformly in \(\delta\in\R_{>0}\) and \(\eps\in\R_{>0}\). Thus, we let \(\delta,\eps\rightarrow 0\) and find that the maximum cannot increase further in the unregularized case either. This proves the claimed upper bound on the density.

Now, we look at the total variation bound on \(\omega\). This bound is immediate, as we have a composition of a \(TV\) function with a Lipschitz-continuous function so that the \(TV\) bound does not change  (this becomes clear when realizing that the definition of \cref{defi:TV} is stable when replacing continuously differentiable functions by Lipschitz-continuous functions).
Concerning the total variation of \(\rho,\) we can directly address the conservation law in its strong form by \cref{cor:Stability_of_solutions} and obtain, for a smoothed solution parameterized by \(\eps\in\R_{>0}\),
\begin{align*}
    \partial_{t}\int_{\R}|\partial_{x}\rho_{\eps}(t,x)|\dd x&=-\int_{\R}|\partial_{x}\rho_{\eps}(t,x)|\partial_{x}\cV[\rho_{\eps},\omega_{\eps}](t,x)\dd x-\int_{\R}\sgn(\partial_{x}\rho_{\eps}(t,x))\rho_{\eps}(t,x)\partial_{x}^{2}\cV[\rho_{\eps},\omega_{\eps}](t,x)\dd x.
    \intertext{Taking advantage of the uniform bounds on \(\rho_{\eps},\omega_{\eps}\), as derived earlier, we have}
    \partial_{t}\int_{\R}|\partial_{x}\rho_{\eps}(t,x)|\dd x&\leq \|\partial_{2}\rho_{\eps}(t,\cdot)\|_{L^{1}(\R)}|\kappa|_{TV(\R_{<0})} \|V\|_{L^{\infty}([0,\rho_{\max}]\times [0,\|\omega_{0}\|_{L^{\infty}(\R)}])}\\
    &\quad + \|\rho_{\eps}(t,\cdot)\|_{L^{\infty}(\R)}|\kappa|_{TV(\R_{<0})}\|\partial_{1}V\|_{L^{\infty}([0,\rho_{\max}]\times[0,\|\omega_{0}\|_{L^{\infty}(\R)}])}\|\partial_{2}\rho_{\eps}(t,\cdot)\|_{L^{1}(\R)}\\
    &\quad +\|\rho_{\eps}(t,\cdot)\|_{L^{\infty}(\R)}|\kappa|_{TV(\R_{<0})}\|\partial_{2}V\|_{L^{\infty}([0,\rho_{\max}]\times[0,\|\omega_{0}\|_{L^{\infty}(\R)}])}\|\partial_{2}\omega_{\eps}(t,\cdot)\|_{L^{1}(\R).}
    \intertext{Then, using the previously derived upper bounds on \(\rho\) and thus \(\rho_{\eps}\), as well as the \(TV\)-seminorm of \(\omega\), yields}
    &\leq \|\partial_{2}\rho_{\eps}(t,\cdot)\|_{L^{1}(\R)}|\kappa|_{TV(\R_{<0})}\Big(\|V\|_{L^{\infty}([0,\rho_{\max}]\times [0,\|\omega_{0}\|_{L^{\infty}(\R)}])}\\
    &\qquad\qquad\qquad\qquad\qquad\qquad+\|\partial_{1}V\|_{L^{\infty}([0,\rho_{\max}]\times[0,\|\omega_{0}\|_{L^{\infty}(\R)}])}\rho_{\max}\Big)\\
    &\quad +\rho_{\max}|\kappa|_{TV(\R_{<0})}\|\partial_{2}V\|_{L^{\infty}([0,\rho_{\max}]\times[0,\|\omega_{0}\|_{L^{\infty}(\R)}])}\|\omega_{0,\eps}'\|_{L^{1}(\R)}.
\end{align*}
Applying Gr\"onwall's inequality, we obtain, still for \(\eps\in\R_{>0}\) and \(t\in[0,T]\),
\begin{align*}
  \|\partial_{2}\rho_{\eps}(t,\cdot)\|_{L^{1}(\R)}&\leq \Big(\|\rho_{0,\eps}'\|_{L^{1}(\R)}+t\rho_{\max}|\kappa|_{TV(\R_{<0})}\|\partial_{2}V\|_{L^{\infty}([0,\rho_{\max}]\times[0,\|\omega_{0}\|_{L^{\infty}(\R)}])}\|\omega_{0,\eps}'\|_{L^{1}(\R)}\Big)\\
  &\quad \cdot\exp\Big(t|\kappa|_{TV(\R_{<0})}\big(\|V\|_{L^{\infty}([0,\rho_{\max}]\times [0,\|\omega_{0}\|_{L^{\infty}(\R)}])}+\|\partial_{1}V\|_{L^{\infty}([0,\rho_{\max}]\times[0,\|\omega_{0}\|_{L^{\infty}(\R)}])}\rho_{\max}\big)\Big),
  \intertext{and letting \(\eps\rightarrow 0\) results in}
  |\rho(t,\cdot)|_{TV(\R)}&\leq \Big(|\rho_{0}|_{TV(\R)}+t\rho_{\max}|\kappa|_{TV(\R_{<0})}\|\partial_{2}V\|_{L^{\infty}([0,\rho_{\max}]\times[0,\|\omega_{0}\|_{L^{\infty}(\R)}])}|\omega_{0}|_{TV(\R)}\Big)\\
  &\quad \cdot\exp\Big(t|\kappa|_{TV(\R_{<0})}\big(\|V\|_{L^{\infty}([0,\rho_{\max}]\times [0,\|\omega_{0}\|_{L^{\infty}(\R)}])}+\|\partial_{1}V\|_{L^{\infty}([0,\rho_{\max}]\times[0,\|\omega_{0}\|_{L^{\infty}(\R)}])}\rho_{\max}\big)\Big).
\end{align*}
This is the claimed \(TV\) estimate.
Concerning the minimum principle for \(\rho\) in the case in which \(\omega_{0}\) is monotonically decreasing, we recall that thanks to \cref{eq:solution}, \(\omega\in C([0,T];L^{1}_{\textnormal{loc}}(\R))\) can be written as
\[
\omega(t,x)=\omega_{0}(\xi_{\cV^{*}}(t,x;0))
\]
where \(\xi_{\cV^{*}}\) are the corresponding characteristics as in \cref{defi:characteristics} for \(\cV^{*}\), as the unique solution to the fixed-point mapping in \cref{theo:fixed_point_mapping_uniqueness}.
Because \(\partial_{2}\xi_{\cV^{*}}>0\) by construction and \(\omega_{0}\) is monotonically decreasing, the composition \(\omega\coloneqq \omega_{0}\circ \xi_{\cV^{*}}(t,\cdot;0)\) is also monotonically decreasing, meaning that monotonicity in \(\omega\) is conserved.

Next, we turn our attention to the lower bound for the density.
Following the same argument as for the maximum principle, at a minimal point \(\tilde{x}\in\R\), that is, at a point where \(\essinf_{y\in\R}\rho(t,y)=\rho(t,\tilde{x})\), we have that
\begin{align*}
     \partial_{t}\rho_{\eps}(t,\tilde{x})&=-\rho_{\eps}(t,\tilde{x})\Big(\int_{\tilde{x}}^{\infty}\keta_{\delta}'(\tilde{x}-y)V(\rho_{\eps}(t,y),\omega_{\eps}(t,y))\dd y-\kappa_{\delta}(0) V(\rho_{\eps}(t,\tilde{x}),\omega_{\eps}(t,\tilde{x}))\Big).
\end{align*}
Hence, for the minimum principle to hold, it would suffice that
\[
\int_{\tilde{x}}^{\infty}\keta_{\delta}'(\tilde{x}-y)V(\rho_{\eps}(t,y),\omega_{\eps}(t,y))\dd y-\kappa_{\delta}(0) V(\rho_{\eps}(t,\tilde{x}),\omega_{\eps}(t,\tilde{x}))\leqq 0
\]
because \(\rho_{\eps}(t,\tilde{x})\geq 0\), so the sign of this term would make the time derivative at the minimal point nonnegative, and the solution at a minimum could only increase.
Estimating this term further yields
\begin{align*}
    &\int_{\tilde{x}}^{\infty}\keta_{\delta}'(\tilde{x}-y)V(\rho_{\eps}(t,y),\omega_{\eps}(t,y))\dd y-\kappa_{\delta}(0) V(\rho_{\eps}(t,\tilde{x}),\omega_{\eps}(t,\tilde{x}))\\
    &\leq \int_{\tilde{x}}^{\infty}\keta_{\delta}'(\tilde{x}-y)V(\rho_{\eps}(t,y),\omega_{\eps}(t,\tilde{x}))\dd y-\kappa_{\delta}(0) V(\rho_{\eps}(t,\tilde{x}),\omega_{\eps}(t,\tilde{x})),
    \intertext{where we have used the facts that \(\kappa_{\delta}'\geqq 0\), \(V\) is monotonically increasing in the second component, and \(\omega(t,\cdot)\) is decreasing. Because \(V\) is monotonically decreasing in its first component, we have}
&\int_{\tilde{x}}^{\infty}\keta_{\delta}'(\tilde{x}-y)V(\rho_{\eps}(t,y),\omega_{\eps}(t,y))\dd y-\kappa_{\delta}(0) V(\rho_{\eps}(t,\tilde{x}),\omega_{\eps}(t,\tilde{x}))\\
    &\leq \int_{\tilde{x}}^{\infty}\keta_{\delta}'(\tilde{x}-y)V(\rho_{\eps}(t,\tilde{x}),\omega_{\eps}(t,\tilde{x}))\dd y-\kappa_{\delta}(0) V(\rho_{\eps}(t,\tilde{x}),\omega_{\eps}(t,\tilde{x}))\\
    &=V(\rho_{\eps}(t,\tilde{x}),\omega_{\eps}(t,\tilde{x}))\int_{\tilde{x}}^{\infty}\keta_{\delta}'(\tilde{x}-y)\dd y-\kappa_{\delta}(0) V(\rho_{\eps}(t,\tilde{x}),\omega_{\eps}(t,\tilde{x}))\\
    &=0.
\end{align*}
As this holds uniformly in \(\delta,\eps\in\R_{>0}\), we can first let \(\delta\rightarrow 0\) to obtain the inequality for the considered \(BV\) kernel class, and we can then let \(\eps\rightarrow0\) to obtain the estimate for the nonsmoothed initial data of only the \(TV\) class. Altogether, the minimal value of \(\rho\) can only increase, and we obtain the claimed lower bound.

Finally, because of the uniformity of the \(L^{\infty}\) estimates with respect to time and the given \(TV\) estimate, which is uniform for a given end time \(T\in\R_{>0}\) by a time patching argument, we obtain the existence and uniqueness of solutions on any finite time horizon, as well as the stability estimates in \cref{cor:Stability_of_solutions}.
\end{proof}
\begin{rem}[transport equation, \(L^{\infty}\) and \(TV\) bounds, and more]
The obtained \(L^{\infty}\) and \(TV\) bounds and the monotonicity on the solution \(\omega\) of a transport equation are not surprising: they are an intrinsic property of such a transport equation as long as the velocity field is smooth enough, that is, if it has a solution.
Furthermore, the construction of the solution in terms of a fixed-point argument and the obtained maximum principle show not only that one can obtain a solution \(\rho\) on any time horizon \(T\in\R_{>0}\) but that solutions \(\tilde{\rho}\) found when considering a smaller time horizon \(T^{*}<T\) are identical on the smaller time horizon, that is, \(\rho(t,\cdot)\equiv \tilde{\rho}(t,\cdot)\text{ on }\R \text{ for } t\in[0,T^{*}]\).
\end{rem}
In the following, we provide another condition on the dependence of the initial data on the velocity \(V\) so that the solution stays within a certain invariant region. The key idea here is to prove bounds at first not on \(\rho\) but on the velocity where one has, due to the nonlocal operator, the advantage that a maximal point for the velocity will also be greater than or equal to its nonlocal averaging.
\begin{theo}[invariant region]\label{theo:invariant_region}
Let \cref{ass:small_time_existence_uniqueness} hold, and assume that \(\partial_{1}V\leqq 0\) and \(\supp(\kappa)\subset\R_{\leq 0}\) such that \(\kappa\geq 0\) is monotonically increasing on \(\R_{\leq 0}\), with \(\|\kappa\|_{L^{1}(\R_{\leq0})}=1\).
Let \((\rho_{0},\omega_{0})\in TV(\R)^{2}, \text{ with } \rho_{0}\geqq0\leqq\omega_{0}\), and define 
\begin{equation}\label{def:V-max-min}
\essinf_{x\in\R}V(\rho_{0}(x),\omega_{0}(x))\eqqcolon V_{\min},\qquad V_{\max}\coloneqq \esssup_{x\in\R} V(\rho_{0}(x),\omega_{0}(x)).
\end{equation}
Assume that \(V_{\min}\in\R_{\geq 0}\).
Then, it holds for a small enough \(T\in\R_{>0}\) that the solution 
\[
(\rho,\omega)\in \big(C([0,T];L^{1}_{\loc}(\R))\cap L^{\infty}((0,T);TV(\R))\big)^{2}
\] satisfies
\begin{equation}
V_{\min}\leq V(\rho(t,x),\omega(t,x))\leq V_{\max}, \quad (t,x)\in[0,T]\times\R \text{ a.e.}\label{eq:invariant_velocity}
\end{equation}
and
\[
\essinf_{y\in\R}\omega_{0}(y)\leq \omega(t,x)\leq \|\omega_{0}\|_{L^{\infty}(\R)}.
\]
Furthermore, we have the following:
\begin{enumerate}
\item \label{item:1:invariant} If \cref{ass:long_time_horizon} b and b1 hold, that is, if for \(\Omega\coloneqq \big[\essinf_{x\in\R}\omega_{0}(x),\|\omega_{0}\|_{L^{\infty}(\R)} \big]\subset\R\)\,, 
\begin{equation}
\exists\, c,\underline{\rho}\in\R_{>0}:\ \sup_{b\in\Omega}\partial_{1}V(a,b)\leq -c \ \forall a\in\R_{>\underline{\rho}},\label{eq:invariant_domain}
\end{equation}
then the solution \((\rho,\omega)\in\big(C([0,T];L^{1}_{\loc}(\R))\cap L^{\infty}((0,T);TV(\R))\big)^{2}\) exists on any finite time horizon \(T\in\R_{>0}\), it satisfies \cref{eq:invariant_velocity},
and \(\rho\) obeys the following \(L^{\infty}\) bound:
\begin{equation}
\begin{aligned}
    0&\leq \rho(t,x)\leq \rho_{\max}\coloneqq \Big(\tfrac{\sup_{b\in\Omega}V(\underline{\rho},b)}{c}\Big)^{+}+\underline{\rho},
\end{aligned}
 \qquad (t,x)\in (0,T)\times\R.
 \label{eq:rho_max_invariant_domain}
\end{equation}
    It also obeys the same \(TV\) bounds as in \cref{theo:maximum_principle}. 
   
\item \label{item:2:invariant} If \(\partial_{2}V\geqq 0\) and \(\partial_{1}V<0\), we have the existence of the inverse function \(\underline{V}^{-1}\)\! to
\[
\R\ni\rho\mapsto \underline{V}(\rho)\coloneqq V\Big(\rho,\essinf_{y\in\R}\omega_{0}(y)\Big)\in\big(-\infty,V\big(0,\essinf_{y\in\R}\omega_{0}(y)\big)\big],
\]
and in the case that 
\begin{equation}
\essinf_{y\in\R}\omega_{0}(y)>V_{\max},\label{eq:as_lower_bound_density}
\end{equation}
we have, for small enough \(T\in\R_{>0}\), the lower bound
\begin{equation}\label{eq:positive_lower_bound_density}
\rho(t,x)\geq  \rho_{\min}\eqqcolon \underline{V}^{-1}(V_{\max})>0,\ (t,x)\in\OT \text{ a.e.}    
\end{equation}
We also have the existence of the inverse function \(\overline{V}^{-1}\)\! to
\[
\R_{\geq0}\ni\rho\mapsto\overline{V}(\rho)\coloneqq  V\big(\rho,\|\omega_{0}\|_{L^{\infty}(\R)}\big)\in \ \big(-\infty,V\big(0,\|\omega_{0}\|_{L^{\infty}(\R)}\big)\big],
\]
and if \(V_{\min}>\lim_{\rho\rightarrow\infty}V\big(\rho,\|\omega_{0}\|_{L^{\infty}(\R)}\big)\) (this is indeed \cref{ass:long_time_horizon} b2), then we have the following upper bound on the density:
\begin{equation}
0\leq\rho(t,x)\leq \rho_{\max}\coloneqq\overline{V}^{-1}(V_{\min}),\ (t,x)\in\OT \text{ a.e.}\label{eq:upper_bound_solution}
\end{equation}
In the case that \cref{eq:upper_bound_solution} holds, the solution \((\rho,\omega)\) also exists for all \(T\in\R_{>0}\), and it satisfies the claimed bounds, including \cref{eq:rho_max_invariant_domain}.
\end{enumerate}
\end{theo}
\begin{proof}
    We recall from \cref{cor:Stability_of_solutions} that one can approximate the solution \((\rho,\omega)\) by a strong solution \((\rho_{\eps},\omega_{\eps})\),\ \(\eps\in\R_{>0}\), on a given small time horizon \(T_{1}\in\R_{>0}\) provided by \cref{theo:existence_uniqueness_small_time}. Then, we look at the time derivative of the ``local'' quantity \(t\mapsto V(\rho_{\eps}(t,x),\omega_{\eps}(t,x))\) for a specifically chosen \(x\in\R\) and obtain
    \begin{align*}
        \tfrac{\dd}{\dd t}V(\rho_{\eps}(t,x),\omega_{\eps}(t,x))&=\partial_{1}V(\rho_{\eps}(t,x),\omega_{\eps}(t,x))\partial_{t}\rho_{\eps}(t,x)+\partial_{2}V(\rho_{\eps}(t,x),\omega_{\eps}(t,x))\partial_{t}\omega_{\eps}(t,x).
        \intertext{Using the strong form for the dynamics \cref{eq:nonlocal_GARZ} in \(\omega\), we have}
        \tfrac{\dd}{\dd t}V(\rho_{\eps}(t,x),\omega_{\eps}(t,x))&=-\partial_{1}V(\rho_{\eps}(t,x),\omega_{\eps}(t,x))\Big(\partial_{x}\cV[\rho_{\eps},\omega_{\eps}](t,x)\rho_{\eps}(t,x)+ \cV[\rho_{\eps},\omega_{\eps}](t,x)\partial_{x}\rho_{\eps}(t,x)\Big)\\
        &\qquad-\partial_{2}V(\rho_{\eps}(t,x),\omega_{\eps}(t,x))\cV[\rho_{\eps},\omega_{\eps}](t,x)\partial_{x}\omega_{\eps}(t,x)\\
        &=-\tfrac{\dd}{\dd x} V(\rho_{\eps}(t,x),\omega_{\eps}(t,x)) \cV[\rho_{\eps},\omega_{\eps}](t,x)\\
        &\qquad-\partial_{1}V(\rho_{\eps}(t,x),\omega_{\eps}(t,x))\partial_{x}\cV[\rho_{\eps},\omega_{\eps}](t,x)\rho_{\eps}(t,x).
        \intertext{Inserting, at a given time \(t\in\R_{>0}\), a spatial location \(\tilde{x}\in\R\) where \(V(\rho_{\eps},\omega_{\eps})\) is minimal, implying in particular that \(\tfrac{\dd}{\dd x}V(\rho_{\eps}(t,x),\omega_{\eps}(t,x))|_{x=\tilde{x}}=0\), leads to} 
\tfrac{\dd}{\dd t}V(\rho_{\eps}(t,\tilde{x}),\omega_{\eps}(t,\tilde{x}))       &= -\partial_{1}V(\rho_{\eps}(t,\tilde{x}),\omega_{\eps}(t,\tilde{x}))\partial_{2}\cV[\rho_{\eps},\omega_{\eps}](t,\tilde{x})\rho_{\eps}(t,\tilde{x}).
    \end{align*}
    As \(\partial_{1}V\leqq 0\) by assumption, it suffices to show that \(\partial_{2}\cV[\rho,\omega](t,\tilde{x})\geqq 0\) as the time derivative then is nonnegative so that the invariant region is according to \cite{Danskin1967} 
respected. However, it holds that when approximating the nonlocal term by a smooth kernel \(\kappa_{\delta},\ \delta\in\R_{>0},\) in the way described in \cref{eq:BV_approximation}
    \begin{align*}
        \partial_{2}\cV[\rho_{\eps},\omega_{\eps}](t,\tilde{x})&=\tfrac{\dd}{\dd\tilde{x}}\int_{\tilde{x}}^{\infty}\kappa_{\delta}(\tilde{x}-y)V(\rho_{\eps}(t,y),\omega_{\eps}(t,y))\dd y\\
        &=\int_{\tilde{x}}^{\infty}\kappa_{\delta}'(\tilde{x}-y)V(\rho_{\eps}(t,y),\omega_{\eps}(t,y))\dd y- \kappa_{\delta}(0)V(\rho_{\eps}(t,\tilde{x}),\omega_{\eps}(t,\tilde{x}))
        \intertext{and as \(V(\rho_{\eps},\omega_{\eps})\) is minimal at \(x=\tilde{x}\) and \(\kappa_{\delta}'\geqq0\) due to \(\kappa\) monotonically increasing}
        &\geq V(\rho_{\eps}(t,\tilde{x}),\omega_{\eps}(t,\tilde{x}))\int_{\tilde{x}}^{\infty}\kappa_{\delta}'(\tilde{x}-y)\dd y-\kappa_{\delta}(0)V(\rho_{\eps}(t,\tilde{x}),\omega_{\eps}(t,\tilde{x}))\\
        &=0.
    \end{align*}
    This inequality is uniform in \(\eps,\delta,\) and thus implies not only that smooth solutions obey the invariant domain but also weaker \(TV-\) solutions.
The same can be obtained for the upper bound, where the last inequality is reversed and we look at the \(V(\rho_{\eps}(t,\tilde{x}),\omega_{\eps}(t,\tilde{x}))\) which is maximal.

    The upper bound on \(\rho\) is a direct consequence of the invariance of \(V(\rho,\omega)\) together with \cref{eq:invariant_domain}. To see this, we recall that for \((t,x)\in \OT \,\text{a.e.}\) assuming that \(\rho(t,x)\geqq \underline{\rho}\) it holds that
    \begin{align*}
    -V(\underline{\rho},\omega(t,x))&\leq V\big(\rho(t,x),\omega(t,x)\big)
    -V(\underline{\rho},\omega(t,x));
    \intertext{
    performing a worst case estimate on the left-hand side in \(\omega\) by recalling \(\Omega=\big(\essinf_{x\in\R}\omega_{0}(x),\|\omega_{0}\|_{L^{\infty}(\R)}\big)\) and a Taylor-expansion on the right-hand side}
    -\sup_{b\in\Omega}V(\underline{\rho},b)&\leq \partial_{1}V\big(\tilde{\rho},\omega(t,x)\big)\big(\rho(t,x)-\underline{\rho}\big)
    \intertext{and as \(\partial_{1}V\leqq 0\) and \(\rho(t,x)\geq \underline{\rho}\) we can furthermore estimate with a \(\tilde{\rho}\in (\underline{\rho},\infty)\) properly chosen}
    -\sup_{b\in\Omega}V(\underline{\rho},b)&\leq \sup_{b\in\Omega}\partial_{1}V\big(\tilde{\rho},b\big)\big(\rho(t,x)-\underline{\rho}\big)
    \end{align*}
    from which we conclude thanks to \cref{eq:invariant_domain}
    \[
\rho(t,x)\leq \underline{\rho} +\tfrac{-\sup_{b\in\Omega}V(\underline{\rho},b)}{\sup_{b\in\Omega}\partial_{1}V(\tilde{\rho},b)}\leq \underline{\rho} +\tfrac{\sup_{b\in\Omega}V(\underline{\rho},b)}{c}.
    \]
    For \(\rho(t,x)<\underline{\rho}\), the given bound still applies when replacing
    \(\tfrac{\sup_{b\in\Omega}V(\underline{\rho},b)}{c}\) by \(\Big(\tfrac{\sup_{b\in\Omega}V(\underline{\rho},b)}{c}\Big)^{+}\).
    
    The arguments for the bounds on \(\omega\) are analogous to the proof of \cref{theo:maximum_principle}, and finally, the \(TV\) estimates as well. This enables us also to use a time clustering argument to extend the solution to any finite time horizon.
\end{proof}
\begin{rem}[\cref{ass:long_time_horizon} and the relation to the assumptions in \cref{theo:invariant_region}]\label{rem:invariant-asm2.4}
The result in \cref{theo:invariant_region} addresses existence and uniqueness of solutions on any finite time horizon for the assumptions stated in \cref{ass:long_time_horizon}. However, it contains more as it also proves that \(V(\rho,\omega)\) stays within \(V_{\min},V_{\max}\) and that there is an explicit lower bound on the density given that \cref{eq:as_lower_bound_density} holds. Such a lower bound is later required in the singular limit analysis and the uniqueness of entropy solutions in \cref{sec:entropy_uniqueness}, however, this is not required for the existence and uniqueness of solutions to the nonlocal system of conservation laws as we have automatically also the lower bound zero.
\end{rem}
\begin{rem}[invariant domains for specific choices of \(V\)]\label{rem:invariant}
For \((\rho_{0},\omega_{0})\in TV(\R)^{2}\),\ \(\omega_{0},\rho_{0}\geqq 0\) we set as before \(\overline{\omega}\coloneqq \|\omega_{0}\|_{L^{\infty}(\R)},\ \underline{\omega}\coloneqq \essinf_{y\in\R}\omega_{0}(y)\).
\begin{description}
\item[linear velocity:] Consider for instance for \(\alpha\in\R_{>0}\) the velocity \[V(\rho,\omega)=\omega-\alpha\rho\] such that the initial data satisfy
\[
V(\rho_{0}(x),\omega_{0}(x))=\omega_{0}(x)-\alpha\rho_{0}(x)\geq0\quad \forall x\in\R,
\] 
(compare also the later \cref{e:initialdatum_1} in \cref{sec:numerical_simulations} where we use such velocity to derive numerical solutions). Then, all the assumptions in \cref{theo:invariant_region}, for having the velocity staying nonnegative, are satisfied. 

Even more, we have uniform bounds on the solutions as \(\partial_{1}V=-\alpha\) and we can choose \(\underline{\rho}=0\) so that
\begin{equation}
0\leq \rho(t,x)\leq \tfrac{\|\omega_{0}\|_{L^{\infty}(\R)}}{\alpha},\ (t,x)\in\OT \text{ a.e.}\label{eq:a_one_time_label}
\end{equation}
when applying \cref{item:1:invariant}.

In the given case, one can also apply \cref{item:2:invariant}. For the upper bound, we set
\[
V_{\min}\coloneqq \essinf_{y\in\R} V(\rho_{0}(y),\omega_{0}(y))=\essinf_{y\in\R} \big(\omega_{0}(y)-\alpha\rho_{0}(y)\big),
\]
compute \(\overline{V}^{-1}(y)=\tfrac{\overline{\omega}-y}{\alpha}\) and obtain
\[
\rho(t,x)\leq \overline{V}^{-1}(V_{\min})=\tfrac{\overline{\omega}-\essinf_{y\in\R} V(\rho_{0}(y),\omega_{0}(y))}{\alpha}=\tfrac{\overline{\omega}-\essinf_{y\in\R} [\omega_{0}(y)-\alpha\rho_{0}(y)]}{\alpha}\leq \tfrac{\overline{\omega}}{\alpha},\ (t,x)\in \OT \text{ a.e.}
\]
due to \(V_{\min}\coloneqq \essinf_{y\in\R}V(\rho_{0}(y),\omega_{0}(y))\geqq 0\) which is identical to the upper bound in \cref{eq:a_one_time_label}. 

Considering the lower bound, we set \(V_{\max}\coloneqq \|V(\rho_{0},\omega_{0})\|_{L^{\infty}(\R)}\), 
compute \(\underline{V}^{-1}(y)=\tfrac{\underline{\omega}-y}{\alpha}\) and have
\begin{align*}
\rho(t,x)&\geq \underline{V}^{-1}(V_{\max})=\tfrac{\underline{\omega}-\esssup_{y\in\R}V(\rho_{0}(y),\omega_{0}(y))}{\alpha}=\tfrac{\underline{\omega}-\esssup_{y\in\R}[\omega_{0}(y)-\alpha \rho_{0}(y)]}{\alpha} \\
&\geq \tfrac{\underline{\omega}-\overline{\omega}}{\alpha}+\essinf_{y\in\R}\rho_{0}(y),\ (t,x)\in\OT\text{ a.e.}
\end{align*}
Thus, whenever \(\alpha\cdot\essinf_{y\in\R}\rho_{0}(y)<\overline{\omega}-\underline{\omega}\), we have a lower bound, bounded away from zero.

Note that here, as a consequence of \cref{eq:invariant_velocity} and the bounds on $\omega$, the values of $(\rho(t,x),\omega(t,x))$ belong to a polygonal region (a parallelogram if $V_{\max}\le \underline{\omega}$). In particular, for this region, $\rho\in [\rho_{\min}, \rho_{\max}]$ with
\begin{equation}\label{eq:rho-min-linear-V}
    \rho_{\min}=\left(\tfrac{\underline{\omega}-V_{\max}}{\alpha} \right)_+ \text{ and }  \rho_{\max}=\tfrac{\overline{\omega}-V_{\min}}{\alpha}.
\end{equation}
The quantity $\left(\underline{\omega}-V_{\max}\right)$ has no sign, while $V_{\min}\ge 0$ implies $\left(\overline{\omega} -V_{\min}\right)\ge 0$. 
\item[nonlinear multiplicative velocity:]
For the velocity function (see \cref{sim:velocity_2} in \cref{sec:numerical_simulations} where it is used for numerical simulations),
\[
V(\rho,\omega)=\tfrac{\omega}{1+\rho}\geq 0,
\]
we have by \cref{theo:invariant_region} that the velocity remains nonnegative.

As a consequence of \cref{eq:invariant_velocity} and the bounds on $\omega$, the values of $\rho(t,x)$ belong to $[\rho_{\min}, \rho_{\max}]$ with 
\begin{equation*}
    \rho_{\min}=\left(\tfrac{\underline{\omega}}{V_{\max}}-1 \right)_+\,,
    \qquad  
    \rho_{\max}=\tfrac{\overline{\omega}}{V_{\min}}-1\,
\end{equation*}
with \((x)_{+}\coloneqq \max\{0,x\},\ x\in\R\).
Observe that the quantity $\left(\underline{\omega}-V_{\max}\right)$ has no sign, while $V_{\min}\ge 0$ implies $\left(\overline{\omega} -V_{\min}\right)\ge 0$.

Even more, it holds that \(\partial_{1}V<0\) and \(\partial_{2}V\geqq 0\) so that \cref{item:2:invariant} applies. We then compute \(\underline{V}(\rho)=\tfrac{\underline{\omega}}{1+\rho}\).
Next, let us pick \(1\leqq \rho_{0}\in TV(\R)\) so that 
\[
V_{\max}=\esssup_{x\in\R}V(\rho_{0}(x),\omega_{0}(x))=\esssup_{x\in\R}\tfrac{\omega_{0}(x)}{1+\rho_{0}(x)}\leq \tfrac{\overline{\omega}}{2}.\]
For a lower bound to hold, we require that
\[
\underline{\omega}>V_{\max}.
\]
Thus, it suffices that \[\underline{\omega}>\tfrac{\overline{\omega}}{2}\]
which we can assume for \(\omega_{0}\). As \(\underline{V}^{-1}(z)=\tfrac{\underline{\omega}}{z} -1\), we obtain as lower bound
\[
\rho(t,x)\geq \underline{V}^{-1}(V_{\max})= \tfrac{\underline{\omega}}{V_{\max}}-1>\tfrac{\underline{\omega}}{\tfrac{\overline{\omega}}{2}}-1>0,\ (t,x)\in \OT \ \text{a.e.}
\]
For the upper bound on \(\rho\), we require \(V_{\min}\) or a lower bound on it
\[
V_{\min}= \essinf_{x\in\R}V(\rho_{0}(x),\omega_{0}(x))=\essinf_{x\in\R}\tfrac{\omega_{0}(x)}{1+\rho_{0}(x)}\geq \tfrac{\underline{\omega}}{1+\|\rho_{0}\|_{L^{\infty}(\R)}}.
\]
Computing the inverse function \(\overline{V}^{-1}(y)=\tfrac{\overline{\omega}}{y}-1\), we have
\[
\rho(t,x)\leq \overline{V}^{-1}(V_{\min})=\tfrac{\overline{\omega}}{\tfrac{\underline{\omega}}{1+\|\rho_{0}\|_{L^{\infty}(\R)}}}-1=\tfrac{\overline{\omega}}{\underline{\omega}}\big(1+\|\rho_{0}\|_{L^{\infty}(\R)}\big)-1,\ (t,x)\in\OT \text{ a.e. }
\]
\end{description}
\end{rem}
We later require that the nonlocal velocity vanishes for \(x\rightarrow \pm \infty\) which is what the following will provide.
\begin{lem}[\(\partial_{2}\cV\) vanishing at \(\pm\infty\)]\label{lem:partial_2cV_vanishing_pm_infty}
Let \cref{ass:long_time_horizon} hold and assume
\(\rho,\omega \in C^{1}(\R)\cap TV(\R)\cap L^{\infty}(\R).
\)

Then it holds that
\[
\lim_{x\rightarrow\pm\infty}\partial_{x} \cV[\rho,\omega](x)=\lim_{x\rightarrow\pm\infty} \partial_x \int_{x}^{\infty}\kappa(x-y)V(\rho(y),\omega(y))\dd y=0.
\]
\end{lem}
\begin{proof}
This is a consequence of the assumed \(TV\) bounds on the solution, the monotonicity of the nonlocal kernel and that it vanishes at \(-\infty\).
\end{proof}
\begin{rem}[solution at \(\pm\infty\)]\label{rem:values-at-infty}
Thanks to \cref{lem:partial_2cV_vanishing_pm_infty} and the solution formula in \cref{eq:solution}, we also obtain that for \(t\in[0,T] \text{ a.e.}\) it holds that
\[
\lim_{x\rightarrow\pm\infty}\rho(t,x)=\rho_{\pm\infty},\ \lim_{x\rightarrow\pm\infty}\omega(t,x)=\omega_{\pm\infty}
\]
where
\[
\rho_{\pm\infty}\coloneqq \lim_{x\rightarrow\pm\infty}\rho_{0}(x),\ \omega_{\pm\infty}\coloneqq \lim_{x\rightarrow\pm\infty}\omega_{0}(x).
\]
Thanks to the assumed \(TV\) bounds on \(\omega_{0},\rho_{0},\) the latter limits actually exist.
\end{rem}
\section{The singular limit: convergence to a local weak solution}\label{sec:singular_limit}
In this section, we study the limit when the nonlocal kernel \(\kappa\) is of exponential type, see \cref{eq:nonlocal_velocity_definition} below,
that was introduced in \cite{bressan-shen2021entropy} and exploited in several papers (\cite{coclite2022general,Chiarello2024,friedrich2022conservation}).

The strategy is to first derive uniform \(TV\) bounds of the nonlocal velocity with exponential kernel, uniformly with respect to the scaling parameter $\eta>0$.
To this end, we identify a PDE which is a nonlocal transport equation in the velocity and which enables us to derive such \(TV\) bounds under a few further restrictions on the velocity. This then enables us to pass to the limit and to demonstrate that the limit is a weak solution of the corresponding local GARZ model. To this end, we follow the line of the proof in \cite{coclite2022general,Chiarello2024,friedrich2022conservation}. 
In \cref{sec:entropy_uniqueness}, we will show that it is even entropy admissible. However, uniqueness of the local entropy solution then follows only when the density is bounded away from zero, which is what we then also assume in the nonlocal equation, taking advantage of \cref{theo:maximum_principle} or  of \cref{theo:invariant_region}, Item 2, first part.

\begin{lem}[transport equation satisfied by the nonlocal operator]\label{lem:nonlocal_transport_equation}
Let \(T\in\R_{>0}\) be given so that there exists a unique weak solution to \cref{eq:nonlocal_GARZ}.
Define for \((t,x)\in\OT\) 
\begin{equation}
    \cV_{\eta}(t,x)\coloneqq \tfrac{1}{\eta}\int_{x}^{\infty}\exp\big(\tfrac{x-y}{\eta}\big)V(\rho_{\eta}(t,y),\omega_{\eta} (t,y))\dd y,
    \label{eq:nonlocal_velocity_definition}
\end{equation}
that is, assume that the nonlocal velocity in \cref{eq:nonlocal_GARZ} is equipped with the exponential kernel and the parameter \(\eta\in\R_{>0}\).
Then, \(\cV_{\eta}\) satisfies
\begin{equation} \label{eq:nonlocal_identity}
\partial_{x}\cV_{\eta}(t,x)=\tfrac{1}{\eta}\cV_{\eta}(t,x)-\tfrac{1}{\eta}V(\rho_{\eta}(t,x),\omega_{\eta}(t,x)),\qquad (t,x)\in (0,T)\times\R
\end{equation}
as well as the following transport equation with nonlocal source:
\begin{equation}
\begin{split}
  &  \partial_t \cV_{\eta}(t,x)+ \cV_{\eta}(t,x) \partial_x \cV_{\eta}(t,x)\\
  &= -\tfrac{1}{\eta} \int_x^\infty \exp\big(\tfrac{x-y}{\eta}\big) \partial_1 V(\rho_{\eta}(t,y),\omega_{\eta}(t,y))\partial_{y}\cV_{\eta}(t,y)\rho_{\eta}(t,y)        \dd y
- \int_x^\infty \exp\big(\tfrac{x-y}{\eta}\big) \big(\partial_y \cV_{\eta}(t,y) \big)^2 \dd y 
\\
    &= -\tfrac{1}{\eta} \int_x^\infty \exp\big(\tfrac{x-y}{\eta}\big)\partial_{y}\cV_{\eta}(t,y)\Big( \partial_1 V(\rho_{\eta}(t,y),\omega_{\eta}(t,y))\rho_{\eta}(t,y)  +\eta\partial_{y}\cV_{\eta}(t,y)\Big)      \dd y
    \end{split}
    \label{eq:entirely_nonlocal}
\end{equation} 
supplemented by the initial condition:
        \begin{equation}
\cV_{\eta}(0,x)=\tfrac{1}{\eta} \int_x^\infty\exp(\tfrac{x-y}{\eta}) V\left(\rho_{0}(y), \omega_0(y)\right) \dd y,\qquad x\in\R.\label{eq:lem:nonlocal_transport_equation_initial_datum}
        \end{equation}
The \(\rho_{\eta}\) showing up in \cref{eq:entirely_nonlocal} is thereby the nonlocal weak solution in the sense of \cref{defi:weak_solution} with the mentioned and outlined in \cref{eq:nonlocal_velocity_definition} (one-sided) exponential kernel.
\end{lem}
\begin{proof} 
Starting with \cref{eq:nonlocal_GARZ}, we use \cref{cor:Stability_of_solutions} and compute
for \((t,x)\in\OT\)
\begin{align*}
    \partial_{t}\cV_{\eta}(t,x)&=\tfrac{1}{\eta}\int_{x}^{\infty}\exp\big(\tfrac{x-y}{\eta}\big)\tfrac{\dd}{\dd t}V(\rho_{\eta}(t,y),\omega_{\eta}(t,y))\dd y\\
    &=\tfrac{1}{\eta}\int_{x}^{\infty}\exp\big(\tfrac{x-y}{\eta}\big)\partial_{1}V(\rho_{\eta}(t,y),\omega_{\eta}(t,y))\partial_{t}\rho_{\eta}(t,y)\dd y\\
    &\qquad +\tfrac{1}{\eta}\int_{x}^{\infty}\exp\big(\tfrac{x-y}{\eta}\big)\partial_{2}V(\rho_{\eta}(t,y),\omega_{\eta}(t,y))\partial_{t}\omega_{\eta}(t,y)\dd y
    \intertext{and plugging in the corresponding PDEs as in \cref{eq:nonlocal_GARZ} yields}
    &=-\tfrac{1}{\eta}\int_{x}^{\infty}\exp\big(\tfrac{x-y}{\eta}\big)\partial_{1}V(\rho_{\eta}(t,y),\omega_{\eta}(t,y))\tfrac{\dd}{\dd y}\big(\cV_{\eta}(t,y)\rho_{\eta}(t,y)\big)\dd y\\
    &\quad-\tfrac{1}{\eta}\int_{x}^{\infty}\exp\big(\tfrac{x-y}{\eta}\big)\partial_{2}V(\rho_{\eta}(t,y),\omega_{\eta}(t,y))\cV_{\eta}(t,y)\partial_{y}\omega_{\eta}(t,y)\dd y\\
    \intertext{rearranging terms}
    &=-\tfrac{1}{\eta}\int_{x}^{\infty}\exp\big(\tfrac{x-y}{\eta}\big)\partial_{1}V(\rho_{\eta}(t,y), \omega_{\eta}(t,y))\partial_{y}\cV_{\eta}(t,y)\rho_{\eta}(t,y)\dd y\\
    &\quad-\tfrac{1}{\eta}\int_{x}^{\infty}\exp\big(\tfrac{x-y}{\eta}\big)\tfrac{\dd}{\dd y}\big(V(\rho_{\eta}(t,y),\omega_{\eta}(t,y))\big)\cV_{\eta}(t,y)\dd y\\
    \intertext{and an integration by parts in the second term yields}
    &=-\tfrac{1}{\eta}\int_{x}^{\infty}\exp\big(\tfrac{x-y}{\eta}\big)\partial_{1}V(\rho_{\eta}(t,y),\omega_{\eta}(t,y))\partial_{y}\cV_{\eta}(t,y)\rho_{\eta}(t,y)\dd y\\
    &\quad+\tfrac{1}{\eta}\int_{x}^{\infty}\exp\big(\tfrac{x-y}{\eta}\big)V(\rho_{\eta}(t,y),\omega_{\eta}(t,y))\partial_{y}\cV_{\eta}(t,y)\dd y\\
    &\quad-\tfrac{1}{\eta^{2}}\int_{x}^{\infty}\exp\big(\tfrac{x-y}{\eta}\big)V(\rho_{\eta}(t,y),\omega_{\eta}(t,y))\cV_{\eta}(t,y)\dd y\\
    &\quad+\tfrac{1}{\eta}V(\rho_{\eta}(t,x),\omega_{\eta}(t,x))\cV_{\eta}(t,x).
    \intertext{Taking next advantage of \cref{eq:nonlocal_identity}, that is, \(V(\rho_{\eta},\omega_{\eta})\equiv\cV_{\eta}-\eta\partial_{2}\cV_{\eta}\) in the second and third term}
    &=-\tfrac{1}{\eta}\int_{x}^{\infty}\exp\big(\tfrac{x-y}{\eta}\big)\partial_{1}V(\rho_{\eta}(t,y),\omega_{\eta}(t,y))\partial_{y}\cV_{\eta}(t,y)\rho_{\eta}(t,y)\dd y\\
    &\quad+\tfrac{1}{\eta}\int_{x}^{\infty}\exp\big(\tfrac{x-y}{\eta}\big)\cV_{\eta}(t,y)\partial_{y}\cV_{\eta}(t,y)\dd y-\int_{x}^{\infty}\exp\big(\tfrac{x-y}{\eta}\big)\big(\partial_{y}\cV_{\eta}(t,y)\big)^{2}\dd y\\
    &\quad-\tfrac{1}{\eta^{2}}\int_{x}^{\infty}\exp\big(\tfrac{x-y}{\eta}\big)\cV_{\eta}(t,y)^{2}\dd y+\tfrac{1}{\eta}\int_{x}^{\infty}\exp\big(\tfrac{x-y}{\eta}\big)\cV_{\eta}(t,y)\partial_{y}\cV_{\eta}(t,y)\dd y\\
    &\quad+\tfrac{1}{\eta}\cV_{\eta}(t,x)^{2}-\cV_{\eta}(t,x)\partial_{x}\cV_{\eta}(t,x),\\
    \intertext{and by performing another integration by parts in the second term, we find that}
    \partial_{t}\cV_{\eta}(t,x)&=-\tfrac{1}{\eta}\int_{x}^{\infty}\exp\big(\tfrac{x-y}{\eta}\big)\partial_{1}V(\rho_{\eta}(t,y),\omega_{\eta}(t,y))\partial_{y}\cV_{\eta}(t,y)\rho_{\eta}(t,y)\dd y\\
    &\quad-\int_{x}^{\infty}\exp\big(\tfrac{x-y}{\eta}\big)\big(\partial_{y}\cV_{\eta}(t,y)\big)^{2}\dd y-\cV_{\eta}(t,x)\partial_{x}\cV_{\eta}(t,x).
\end{align*}
This is the claimed identity.
\end{proof}
\subsection{Uniform bounds on the total variation}\label{subsec:4.1}
The identity established in \cref{lem:nonlocal_transport_equation}
allows us to prove total variation bounds of the nonlocal velocity \(\cV_{\eta}\) as in \cref{eq:nonlocal_velocity_definition}, uniformly in the nonlocal scaling $\eta\in\R_{>0}$.
\begin{theo}[spatial total variation bound uniform in $\eta$]\label{theo:total_variation_bound}
Let \cref{ass:long_time_horizon} hold, and assume that 
the kernel is exponential, that is, the nonlocal velocity is given by \cref{eq:nonlocal_velocity_definition}. Assume further that
\begin{equation}
\partial_{1}V(\rho_{\eta}(t,x),\omega_{\eta}(t,x))\rho_{\eta}(t,x)+\eta \partial_{x}\cV_{\eta}(t,x)\leqq 0\qquad \forall (t,x)\in \OT.
\label{eq:requirement_TV_diminishing}
\end{equation}
Then $\cV_{\eta}$, 
as in \cref{eq:nonlocal_velocity_definition}, satisfies the following total variation bound:
\begin{align}\label{eq:bound-on-TV-V}
   |\cV_{\eta}|_{L^{\infty}((0,T);TV(\R))}\leq |V(\rho_{0}, \omega_0)|_{TV(\R)} < \infty.
\end{align}
\end{theo}
\begin{proof} By \cref{cor:Stability_of_solutions}, we can consider $(\rho_{\eta},\omega_{\eta})$ to be smooth.
Using the identity on the velocity in \cref{eq:entirely_nonlocal}, we compute the spatial derivative of $\partial_{t}\cV_{\eta}$ and obtain, for \((t,x)\in\OT\),
\begin{align*}
    \partial_t \partial_{x}&\cV_{\eta}(t,x)\\
    &=- \big(\partial_x \cV_{\eta}(t,x)\big)^{2}-\cV_{\eta}(t,x)\partial_{x}^{2}\cV_{\eta}(t,x)\\
    &\quad-\tfrac{1}{\eta^{2}} \int_x^\infty \exp\big(\tfrac{x-y}{\eta}\big) \partial_1 V(\rho_{\eta}(t,y),\omega_{\eta}(t,y))\partial_{y}\cV_{\eta}(t,y)\rho_{\eta}(t,y)        \dd y\\
    &\quad+\tfrac{1}{\eta}\partial_{1}V(\rho_{\eta}(t,x),\omega_{\eta}(t,x))\partial_{x}\cV_{\eta}(t,x)\rho_{\eta}(t,x)- \tfrac{1}{\eta}\int_x^\infty \!\!\!\!\exp\big(\tfrac{x-y}{\eta}\big) \big(\partial_y \cV_{\eta}(t,y) \big)^2 \dd y +\big(\partial_{x}\cV_{\eta}(t,x)\big)^{2}\\
    &=- \cV_{\eta}(t,x)\partial_{x}^{2}\cV_{\eta}(t,x)-\tfrac{1}{\eta^{2}} \int_x^\infty \exp\big(\tfrac{x-y}{\eta}\big) \partial_1 V(\rho_{\eta}(t,y),\omega_{\eta}(t,y))\partial_{y}\cV_{\eta}(t,y)\rho_{\eta}(t,y)        \dd y\\
    &\qquad+\tfrac{1}{\eta}\partial_{1}V(\rho_{\eta}(t,x),\omega_{\eta}(t,x))\partial_{x}\cV_{\eta}(t,x)\rho_{\eta}(t,x)- \tfrac{1}{\eta}\int_x^\infty \exp\big(\tfrac{x-y}{\eta}\big) \big(\partial_y \cV_{\eta}(t,y) \big)^2 \dd y \\
    &=- \cV_{\eta}(t,x)\partial_{x}^{2}\cV_{\eta}(t,x)-\tfrac{1}{\eta^{2}} \int_x^\infty\!\!\!\! \exp\big(\tfrac{x-y}{\eta}\big) \partial_{y}\cV_{\eta}(t,y)\Big(\partial_1 V(\rho_{\eta}(t,y),\omega_{\eta}(t,y))\rho_{\eta}(t,y)+\eta \partial_{y}\cV_{\eta}(t,y)\Big) \dd y\\
    &\qquad+\tfrac{1}{\eta}\partial_{1}V(\rho_{\eta}(t,x),\omega_{\eta}(t,x))\partial_{x}\cV_{\eta}(t,x)\rho_{\eta}(t,x).
\end{align*} 
Estimating the total variation then yields
\begin{align*}
   &\partial_{t}|\cV_{\eta}(t,\cdot)|_{TV(\R)}\\
   &=\int_{\R}\sgn(\partial_{x}\cV_{\eta}(t,x))\partial_{t}\partial_{x}\cV_{\eta}(t,x)\dd x\\
   &=-\int_{\R}\sgn(\partial_{x}\cV_{\eta}(t,x))\partial_{x}^{2}\cV_{\eta}(t,x)\cV_{\eta}(t,x)\dd x+\tfrac{1}{\eta}\int_{\R}|\partial_{x}\cV_{\eta}|\rho_{\eta}(t,x)\partial_{1}V(\rho_{\eta}(t,x),\omega_{\eta}(t,x))\dd x\\
    &\quad-\tfrac{1}{\eta^{2}}\int_{\R}\sgn(\partial_{x}\cV_{\eta}(t,x)) \!\int_x^\infty\!\!\!\! \exp\big(\tfrac{x-y}{\eta}\big) \partial_{y}\cV_{\eta}(t,y)\Big(\partial_1 V(\rho_{\eta}(t,y),\omega_{\eta}(t,y))\rho_{\eta}(t,y)+\eta \partial_{y}\cV_{\eta}(t,y)\Big) \dd y
    \intertext{and an integration by parts taking advantage of \cref{lem:partial_2cV_vanishing_pm_infty} so that the boundary terms at \(\pm\infty\) vanish}
    &=\int_{\R}|\partial_{x}\cV_{\eta}(t,x)|\partial_{x}\cV_{\eta}(t,x)\dd x +\tfrac{1}{\eta}\int_{\R}|\partial_{x}\cV_{\eta}|\rho_{\eta}(t,x)\partial_{1}V(\rho_{\eta}(t,x),\omega_{\eta}(t,x))\dd x\\
    &\quad-\tfrac{1}{\eta^{2}}\int_{\R}\sgn(\partial_{x}\cV_{\eta}(t,x)) \!\int_x^\infty\!\!\!\! \exp\big(\tfrac{x-y}{\eta}\big) \partial_{y}\cV_{\eta}(t,y)\Big(\partial_1 V(\rho_{\eta}(t,y),\omega_{\eta}(t,y))\rho_{\eta}(t,y)+\eta \partial_{y}\cV_{\eta}(t,y)\Big) \dd y\dd x
    \intertext{and rearranging terms}
    &=\tfrac{1}{\eta}\int_{\R}|\partial_{x}\cV_{\eta}(t,x)|\big(\rho_{\eta}(t,x)\partial_{1}V(\rho_{\eta}(t,x),\omega_{\eta}(t,x))+\eta \partial_{x}\cV_{\eta}(t,x)\big)\dd x\\
    &\quad-\tfrac{1}{\eta^{2}}\int_{\R}\sgn(\partial_{x}\cV_{\eta}(t,x)) \!\int_x^\infty\!\!\!\! \exp\big(\tfrac{x-y}{\eta}\big) \partial_{y}\cV_{\eta}(t,y)\Big(\partial_1 V(\rho_{\eta}(t,y),\omega_{\eta}(t,y))\rho_{\eta}(t,y)+\eta \partial_{y}\cV_{\eta}(t,y)\Big) \dd y\dd x
    \intertext{and as we assume that \cref{eq:requirement_TV_diminishing} holds, that is \(\rho_{\eta}\partial_{1}V(\rho_{\eta},\omega_{\eta})+\eta \partial_{2}\cV_{\eta}\leqq 0\)}
    &\leq\tfrac{1}{\eta}\int_{\R}|\partial_{x}\cV_{\eta}(t,x)|\big(\rho_{\eta}(t,x)\partial_{1}V(\rho_{\eta}(t,x),\omega_{\eta}(t,x))+\eta \partial_{x}\cV_{\eta}(t,x)\big)\dd x\\
    &\quad-\tfrac{1}{\eta^{2}}\int_{\R}\int_x^\infty\!\!\!\! \exp\big(\tfrac{x-y}{\eta}\big) |\partial_{y}\cV_{\eta}(t,y)|\Big(\partial_1 V(\rho_{\eta}(t,y),\omega_{\eta}(t,y))\rho_{\eta}(t,y)+\eta \partial_{y}\cV_{\eta}(t,y)\Big) \dd y\dd x
    \intertext{it yields after another exchange of the order of integration}
    &=\tfrac{1}{\eta}\int_{\R}|\partial_{x}\cV_{\eta}(t,x)|\big(\rho_{\eta}(t,x)\partial_{1}V(\rho_{\eta}(t,x),\omega_{\eta}(t,x))+\eta \partial_{x}\cV_{\eta}(t,x)\big)\dd x\\
    &\quad-\tfrac{1}{\eta^{2}}\int_{\R} |\partial_{y}\cV_{\eta}(t,y)|\Big(\rho_{\eta}(t,y)\partial_1 V(\rho_{\eta}(t,y),\omega_{\eta}(t,y))+\eta \partial_{y}\cV_{\eta}(t,y)\Big)\int_{-\infty}^{y}\exp\big(\tfrac{x-y}{\eta}\big)\dd x \dd y=0.
\end{align*}
Thus, the map $t\mapsto |\cV_{\eta}(t,\cdot)|_{TV(\R)}$ is nonincreasing. As \(V\in W^{1,\infty}_{\textnormal{loc}}(\R^{2})\) and by \cref{ass:long_time_horizon} on $(\rho_0,\omega_0)$, the quantity \(|V(\rho_{0},\omega_{0})|_{TV(\R)}\) is finite. Therefore we conclude that \cref{eq:bound-on-TV-V} holds.
\end{proof}

\begin{rem}[comments on \cref{theo:total_variation_bound} and, in particular, on assumption \cref{eq:requirement_TV_diminishing}]\label{rem:5.3}
~
\begin{itemize}
\item
The assumption in \cref{eq:requirement_TV_diminishing} may appear rather restrictive and remains implicit, as it still requires knowledge of \((\rho_{\eta},\omega_{\eta} )\). Nevertheless, it can be reformulated in a more explicit manner. Indeed, by virtue of \cref{eq:nonlocal_identity,eq:invariant_velocity}, a sufficient condition for \cref{eq:requirement_TV_diminishing} to hold is
\begin{align}
\partial_{1}V(\rho_{\eta},\omega_\eta)\rho_{\eta}+\eta \partial_{x}\cV_{\eta}&=\partial_{1}V(\rho_{\eta},\omega_{\eta})\rho_{\eta}+\cV_{\eta}-V(\rho_{\eta},\omega_{\eta})\notag\\
&  \leqq ~\partial_{1}V(\rho_{\eta},\omega_{\eta})\rho_{\eta}+V_{\max}-V(\rho_{\eta},\omega_{\eta})  
~{\leqq} ~0\quad \text{ on } (0,T)\times\R\label{eq:sufficent_condition-TV_bounds}
\end{align}
with \(V_{\max}\) as in \cref{def:V-max-min}. Given the available bounds on \(\rho_\eta\) and \(\omega_{\eta}\), condition \cref{eq:sufficent_condition-TV_bounds} is easier to verify in terms of the initial data; see the Examples 1 and 2 here below.

\item In the case of \cref{eq:requirement_TV_diminishing} one can also show that \(\cV_{\eta}\) remains monotone as long as the initial nonlocal velocity is monotone. This could also be used to prove compactness under the more restrictive assumptions on the nonlocal velocity being monotone. However, as such result is rather restrictive, we leave it as a remark.
\end{itemize}
\end{rem}

Below we present several examples of choices of $V$ and of initial data for which \cref{ass:long_time_horizon} 
is satisfied as well as condition \cref{eq:sufficent_condition-TV_bounds}.
\begin{description}
\item[Example 1---Linear velocity] Recalling the notation in \cref{rem:invariant}, namely that \(\overline{\omega}\coloneqq \|\omega_{0}\|_{L^{\infty}(\R)},\ \underline{\omega}\coloneqq \essinf_{y\in\R}\omega_{0}(y)\), we obtain 
when choosing a linear velocity function as in \cref{rem:invariant}, 
\(V(\rho,\omega)=\omega-\alpha \rho\) 
for \(\alpha\in\R_{>0}\) and \((\rho,\omega)\in\R^{2}\), and compatible initial data according to \cref{rem:invariant} \[V_{\max}=\|V(\rho_{0},\omega_{0})\|_{L^{\infty}(\R)}=\esssup_{x\in\R} \big(\omega_{0}(x)-\alpha\rho_{0}(x)\big)\leq \overline{\omega}-\alpha \essinf_{x\in\R} \rho_{0}(x)\] and furthermore by \cref{eq:sufficent_condition-TV_bounds}
\begin{equation}
-\alpha \rho_{\eta}+\overline{\omega}-\alpha \essinf_{x\in\R} \rho_{0}(x)-\omega_{\eta}+\alpha \rho_{\eta}\leqq 0\Longleftarrow \overline{\omega}-\underline{\omega}-\alpha \essinf_{x\in\R}\rho_{0}(x)\leqq 0 \label{eq:0815}
\end{equation}
meaning that whenever \(\overline{\omega}-\underline{\omega}\leq \alpha\essinf_{x\in\R}\rho_{0}(x)\) holds, this is satisfied.

Under the assumptions
\begin{equation}
\underline{\omega}> V_{\max},\qquad \tfrac{\overline{\omega} + V_{\max}}2 \le \underline{\omega}, \label{eq:0816}
\end{equation}
from \cref{eq:rho-min-linear-V} we have
\[  
\alpha\essinf_{x\in\R}\rho_{0}(x)\ge {\underline{\omega}-V_{\max}} > \overline{\omega}-\underline{\omega}\,.
\]
The above condition \cref{eq:0816} is nonempty and guarantees the positive lower bound on the density as well as \cref{eq:sufficent_condition-TV_bounds}, see \cref{graph:linear_velocity}.
\begin{figure}[htbp]
\centering
\begin{tikzpicture}[xscale=2]
    \fill[blue!20] (1.05,2)  -- (1.8,2) -- (3,0) -- (2.25,0);
    \draw[->,thick] (0,0) -- (0,5.5);
    \draw[->,thick] (0,0) -- (3.25,0);
    \draw (3.25,-0.25) node {$\rho$}; 
    \draw (-0.15,5.5) node {$V$};
    \draw[-] (2.25,0) -- (0,3.75);
    \draw[-] (3,0) -- (0,5);
    \draw (-0.15,3.75) node {$\underline\omega$};
    \draw (-0.15,5) node {$\bar\omega$};
    \draw (-0.25,2) node {$V_{\max}$};
    \draw[dashed] (0,2) -- (1.8,2);
    \draw[dashed] (1.05,2) --(1.05,-0.1);
    \draw (1.25,-0.25) node {$\rho_{\inf}$};
    \draw (3,-0.25) node {$\rho_{\sup}$};
   \end{tikzpicture}
\caption{Graph of a region, shaded in violet, in which 
\cref{eq:sufficent_condition-TV_bounds} is satisfied for the linear velocity \(V(\rho,\omega) = \omega-\alpha\rho\), \(\alpha\in\R_{>0}\) as stated in \crefrange{eq:0815}{eq:0816}.
}\label{graph:linear_velocity}
\end{figure}
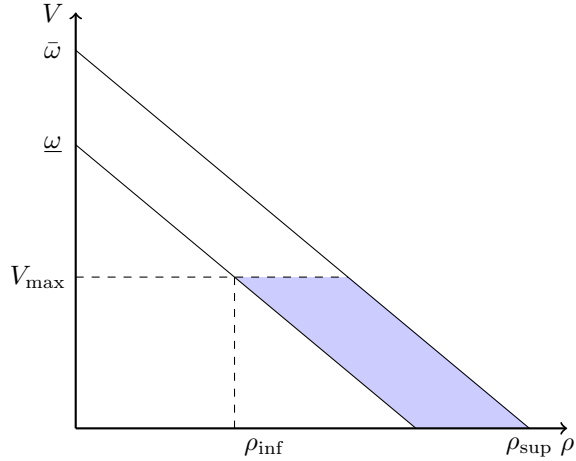
\item[Example 2---Additively separable nonlinear velocity] A nonlinear velocity function that is additively separable consists of
\(V(\rho,\omega) = \omega - p(\rho)\), where \((\rho,\omega)\in\R^{2}\) and \(p:\R\rightarrow\R\) is a function yet to be defined. To this end, let
\begin{align*}
 \bullet\ &V:\Omega\to \R\text{ with }  \Omega\coloneqq\{(\rho,\omega)\in\R_{\geq0}^{2}:\ \omega\leq \overline{\omega}\}\\
\bullet\ &\forall\,\omega \ge 0\
\exists\ R(\omega)\in\R_{>0}:\  V(R(\omega) ,\omega)=0 \\
\bullet\ & p\in C^{2}(\R):\ p(0) = 0\ \wedge\ p'>0\ \wedge\ p''>0
\end{align*}
for \(\overline{\omega}\coloneqq \|\omega_{0}\|_{L^{\infty}(\R)}\) as above. Under these assumptions, one has 
\begin{equation*}
  V(\rho,\omega) \geq 0 \qquad \forall \, (\rho,\omega)\in \Omega_{\underline{\rho},\overline{\omega}}\coloneqq\{(\rho,\omega):\ \omega\in [\underline{\omega},\overline{\omega}],\
    { \underline{\rho}\le \rho\le R(\omega)}
     \}
\end{equation*}
for $\underline{\rho} \in (0,\bar\rho)$, where $\bar\rho$ is the only value such that $p(\bar \rho) = \overline{\omega}$, that is, $\bar\rho = R(\overline{\omega})$. 
In this way, the set $\Omega_{\underline{\rho},\overline{\omega}}$ is nonempty and has a nonempty interior.
Then, \cref{eq:requirement_TV_diminishing} and, in particular, \cref{eq:sufficent_condition-TV_bounds} boil down to
\begin{align*}     
\partial_{1}V(\rho_{\eta},\omega_{\eta})\rho_{\eta}+V_{\max}-V(\rho_{\eta},\omega_{\eta}) &= - p'(\rho_{\eta})\rho_{\eta} + \overline{\omega} - p(\underline{\rho} )- \omega_{\eta}+ p(\rho_{\eta}) \\
     &\le p(\rho_{\eta}) - p'(\rho_{\eta})\rho_{\eta}+ \left(\overline{\omega} - \underline{\omega} - p(\underline{\rho} )\right)  .
\end{align*}
The above quantity is nonpositive  on $\Omega_{\underline{\rho},\bar\omega}$ whenever
\begin{equation*}
  \overline{\omega} - \underline{\omega} \le \min_{\rho\in [{\underline{\rho}}, \overline{\rho}]}  \left\{ p'(\rho)\rho - p(\rho)\right\} + p(\underline{\rho} ).
\end{equation*}
 Since $p$ is convex, the function \(\R_{\geq0}\ni\rho\mapsto p'(\rho)\rho-p(\rho)\) increases monotonically, and the minimum is reached at \(\underline \rho.\) Therefore, if
\begin{equation}
    \overline{\omega} - \underline{\omega} \le   p'({\underline{\rho}}){\underline{\rho}} \,,
    \label{eq:sufficient_condition-TV_bounds_reformulated}
\end{equation}
then \cref{eq:sufficent_condition-TV_bounds} is satisfied.
If we take, more specifically, \(p(\rho)\coloneqq \alpha\rho^{\gamma}\) for \(\rho\in\R_{\geq0}\), with \(\alpha\in\R_{>0}\) and \(\gamma\in\R_{\ge 1}\), then \cref{eq:sufficient_condition-TV_bounds_reformulated} results in
\begin{equation*}
 \overline{\omega} - \underline{\omega} \le   \alpha \gamma {\underline{\rho}}^\gamma 
\end{equation*}
from which, for \(\gamma=1\), the linear velocity case \cref{eq:0815} is recovered.

Performing the same computations in \cref{rem:invariant}, we obtain
\begin{equation}\label{eq:example_nonlinear}
    \underline{\rho} = \left( \tfrac{\underline{\omega}-V_{\max}}{\alpha} \right)_+^{\frac{1}{\gamma}},
\end{equation}
where $(x)_+ \coloneqq \max\{x,0\}$ for all $x\in\R$.
The quantity in \cref{eq:example_nonlinear} is nonzero under the assumption $\underline{\omega}>V_{\max}$. Moreover, if 
\begin{equation*}
    \underline{\omega}>\tfrac{\bar\omega+\gamma V_{\max}}{\gamma+1},
\end{equation*}
then the condition \cref{eq:sufficent_condition-TV_bounds} is satisfied with $\underline{\rho}$ as in \cref{eq:example_nonlinear}. Notice that when $\gamma=1$, the quantities found in the nonlinear case $V(\rho,\omega)=\omega-p(\rho)$ coincide with those in the linear case; see \cref{eq:0816}.

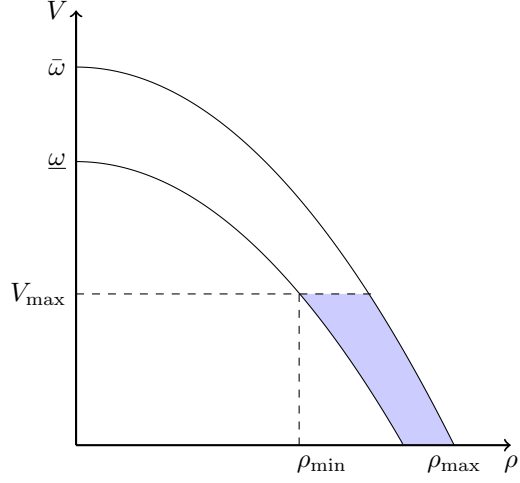
\begin{figure}[htbp]
\centering
\begin{tikzpicture}
    \fill[blue!20] (2.95,2) -- (3.87,2) -- plot[domain=3.87:5, samples=120,smooth](\x,{5 - 0.2*\x^2}) -- (5,0) -- (4.3,0) -- plot[domain=4.3:2.95, samples=120,smooth](\x,{3.75 - 0.2*\x^2}) -- cycle;
    \draw[->,thick] (0,0) -- (0,5.75);
    \draw[->,thick] (0,0) -- (5.75,0);
    \draw (5.75,-0.25) node {$\rho$}; 
    \draw (-0.25,5.75) node {$V$};
    \draw plot[domain=0:4.32,samples=120,smooth]
(\x,{3.75 - 0.2*\x^2});

\draw plot[domain=0:5,samples=120,smooth]
(\x,{5 - 0.2*\x^2});

    \draw (-0.25,3.75) node {$\underline\omega$};
    \draw (-0.25,5) node {$\bar\omega$};
    \draw (-0.5,2) node {$V_{\max}$};
    \draw[dashed] (0,2) -- (4.0,2);
    \draw[dashed] (2.95,2) -- (2.95,-0.1);
    \draw (3.25,-0.25) node {$\rho_{\min}$};
    \draw (5,-0.25) node {$\rho_{\max}$};    
\end{tikzpicture}
\caption{Graph of a region, shaded in violet, in which \cref{eq:sufficent_condition-TV_bounds} holds 
for the nonlinear velocity $V(\rho,\omega) = \omega-\alpha\rho^\gamma$, with $\alpha\in\R_{>0}$ and $\gamma\in\R_{\ge 1}$.
}\label{graph:nonlinear_velocity}
\end{figure}
\end{description}

\subsection{Compactness and convergence to a weak solution}\label{Sec:4}
So far, we have established spatial \(TV\) bounds of \(\cV_{\eta}\) under additional constraints on the velocity, but we require also equi-integrability in time to obtain compactness in a suitable space. This is established in the following.
\begin{theo}[compactness of \(\cV_\eta\) as defined in \cref{eq:nonlocal_velocity_definition}] \label{theo:compacteness}
Let either the assumptions in
\begin{itemize}
    \item \cref{theo:invariant_region} with the existence of a \(\rho_{\max}\in\R_{>0}\) or
    \item \cref{theo:maximum_principle}
\end{itemize}
together with \cref{eq:requirement_TV_diminishing} hold.
Then, the set (of nonlocal velocities)
\((\cV_\eta)_{\eta>0 
} \subset C\big([0,T]; L^{1}_{\textnormal{loc}}(\R)\big)\) as in \cref{eq:nonlocal_velocity_definition}
is relatively compact in $C\big([0,T];L^{1}_{\textnormal{loc}}(\R)\big),$ that is, \(\forall\, \Omega\subset\R\) open and bounded, it holds
\begin{equation*}
    \big\{\cV_{\eta}\big|_{[0,T]\times\Omega}\in C\big([0,T];L^{1}(\Omega)\big),\,\eta>0 
    \big\} \overset{\textnormal{c}}{\hookrightarrow} C\big([0,T]; L^1(\Omega)\big).
\end{equation*}
\end{theo}
\begin{proof}
For $t\in[0,T]$ define
\begin{equation*}
    F(t)\coloneqq\{\cV_\eta(t,\cdot)\in L^1_\textnormal{loc}(\R),\,\eta>0 
    \}.
\end{equation*} 
Applying the result in \cite[Theorem 14.39]{leoni}, we can see the set $F(t)$ is compact in $L^1_{\textnormal{loc}}(\R)$ because of the spatial total variation bound in \cref{theo:total_variation_bound} uniform with respect to $\eta\in\R_{>0}$ and the uniform bounds. Now, we prove that the set $(\cV_{\eta})_{\eta\in\R_{>0}}$ is uniformly equicontinuous in time; we can assume that we have sufficiently regular solutions, thanks to \cref{cor:Stability_of_solutions}. Using the dynamics of \(\cV_{\eta}\) in \cref{eq:entirely_nonlocal}, we thus estimate for $t_1, t_2 \in [0,T]$, 
such that $t_{1}\leq t_{2}$,
\begin{align*}
    \big\|\cV_{\eta}(t_{1},\cdot)-\cV_{\eta}(t_{2},\cdot) \big\|_{L^{1}(\R)}&\leq\int_{t_{1}}^{t_{2}}\int_{\R}|\partial_{t}\cV_{\eta}(t,x)|\dd x\dd t\\
    &\overset{\text{\cref{eq:entirely_nonlocal}}}{\leq}\int_{t_{1}}^{t_{2}}\int_{\R}\cV_{\eta}(t,x)\big|\partial_{x}\cV_{\eta}(t,x)\big|\dd x\dd t\\
    &\quad+\int_{t_{1}}^{t_{2}}\int_{\R}\tfrac{1}{\eta} \int_x^\infty \exp\big(\tfrac{x-y}{\eta}\big) \big|\partial_1 V(\rho_{\eta}(t,y),\omega_{\eta}(t,y))\big|\big|\partial_{y}\cV_{\eta}(t,y)\big|\rho_{\eta}(t,y)\dd y\dd x\dd t\\
    &\quad +\int_{t_{1}}^{t_{2}}\int_{\R}\int_{x}^{\infty}\exp\big(\tfrac{x-y}{\eta}\big)\big(\partial_{y}\cV_{\eta}(t,y)\big)^{2}\dd y\dd x\dd t.
    \intertext{By changing the order of integration and multiplying by \(\tfrac{\eta}{\eta}\), we have}
    \big\|\cV_{\eta}(t_{1},\cdot)-\cV_{\eta}(t_{2},\cdot) \big\|_{L^{1}(\R)}&\leq V_{\max}(t_{2}-t_{1})|\cV_{\eta}|_{L^{\infty}((0,T);TV(\R))}\\
    &\quad +\|\rho_{\eta}\|_{L^{\infty}((0,T);L^{\infty}(\R))}\|\partial_{1}V\|_{L^{\infty}}\int_{t_{1}}^{t_{2}}\int_{\R}\big|\partial_{y}\cV_{\eta}(t,y)\big|\tfrac{1}{\eta}\int_{-\infty}^{y}\exp\big(\tfrac{x-y}{\eta}\big)\dd x\dd y\dd t\\
    &\quad +\int_{t_{1}}^{t_{2}}\int_{\R}\big|\eta\partial_{y}\cV_{\eta}(t,y)\big|\big|\partial_{y}\cV_{\eta}(t,y)\big|\tfrac{1}{\eta}\int_{-\infty}^{y}\exp\big(\tfrac{x-y}{\eta}\big)\dd x\dd y\dd t.
    \intertext{Then, because \(|\eta\partial_{y}\cV_{\eta}(t,y)|\leq 2V_{\max}\), according to \cref{eq:nonlocal_identity}, we obtain}
    \big\|\cV_{\eta}(t_{1},\cdot)-\cV_{\eta}(t_{2},\cdot) \big\|_{L^{1}(\R)}&\leq V_{\max}(t_{2}-t_{1})|\cV_{\eta}|_{L^{\infty}((0,T);TV(\R))} +2V_{\max}(t_{2}-t_{1})|\cV_{\eta}|_{L^{\infty}((0,T);TV(\R))}\\
    &\quad+\|\rho_{\eta}\|_{L^{\infty}((0,T);L^{\infty}(\R))}\|\partial_{1}V\|_{L^{\infty}}(t_{2}-t_{1})|\cV_{\eta}|_{L^{\infty}((0,T);TV(\R))}
\end{align*}
with
\begin{align*}
    V_{\max}&\coloneqq \|V\|_{L^{\infty}((0,\rho_{\max})\times(0,\|\omega_{0}\|_{L^{\infty}(\R)}))},
    \\
    \|\partial_{1}V\|_{L^{\infty}}&\coloneqq \|\partial_{1}V\|_{L^{\infty}((0,\rho_{\max})\times(0,\|\omega_{0}\|_{L^{\infty}(\R)}))},
\end{align*}
where $\rho_{\max}$ is as follows:
\begin{enumerate}
\item in the first case of \cref{ass:long_time_horizon}, \(\rho_{\max}\in\R_{>0}\), so that \(V(\rho_{\max},\cdot)\equiv 0\);
\item in the second case of \cref{ass:long_time_horizon}, \(\rho_{\max}\) is as in \cref{eq:rho_max_invariant_domain} or \cref{eq:upper_bound_solution}.
\end{enumerate}
The quantities \(V_{\max}\) and \(\|\partial_{1}V\|_{L^{\infty}}\) are uniformly bounded in \(\eta\) thanks to the maximum principle in \cref{theo:maximum_principle} 
and the invariant region in \cref{theo:invariant_region}, so that \( \big\|\cV_{\eta}(t_{1},\cdot)-\cV_{\eta}(t_{2},\cdot) \big\|_{L^{1}(\R)}\) 
converges uniformly in \(\eta\) to zero for \(t_{1}\rightarrow t_{2}\), resulting in the equi-continuity in time. 
Thanks to \cite[Section 6, Theorem 3]{Simon1986}, we can infer the claimed compactness from these estimates.
\end{proof}

In the following, we investigate the convergence to solutions of the corresponding system of \textbf{local} conservation laws 
when \(\eta\rightarrow 0\). However, we cannot converge to the ``local'' version of \cref{defi:weak_solution} because in the limit, 
we would have a transport equation in \(w\), which would require higher regularity. However, both the nonlocal and local systems can also 
be written in conservative form (see, e.g., \cite[Eq.3.1, 3.2]{chiarello2020micro}), avoiding the problem in \cref{defi:weak_solution}.

\begin{lem}[weak solution of the nonlocal dynamics in conservative form]\label{lem:weak_solution_nonlocal_conservative}
    Let \((\rho,\omega)\in \big(C\big([0,T];L^{1}_{\text{\loc}}(\R)\big)\cap L^{\infty}((0,T);L^{\infty}(\R))\big)^{2}\) be a weak solution as in \cref{defi:weak_solution}.
    Then, the solution in the ``new'' variables \[
    (\rho,\rho\omega)\in \big(C([0,T];L^{1}_{\textnormal{loc}}(\R))\big)^{2}
    \] is a weak solution of the following problem:
    \begin{equation}
    \begin{aligned}
            \partial_{t}\rho
    +\partial_{x}\big(\cV[\rho,\omega](t,x)\rho\big)&=0,&& (t,x)\in\OT,\\
    \partial_{t}\big(\rho\omega)+\partial_{x}\big(\cV[\rho,\omega](t,x)\rho 
    \omega\big)&=0, && (t,x)\in\OT,\\
    \rho(0,x)&=\rho_{0}(x),&& x\in\R,\\
    \rho(0,x)\omega(0,x)&=\rho_{0}(x)\omega_{0}(x),&& x\in\R,\\
     \cV[\rho,\omega](t,x)&
 \coloneqq \int_{\R}\keta(x-y)V(\rho(t,y),\omega(t,y))\dd y, &&(t,x)\in \OT,
    \end{aligned}
    \label{eq:GARZ_strong_conservative}
    \end{equation}
    and is unique. Furthermore, for a weak solution of \cref{eq:GARZ_strong_conservative}, it must hold
        that \((\rho,\omega)\in C([0,T];L^{1}_{\text{\loc}}(\R))^{2}\) and that for all \(\phi_{1},\phi_{2}\in W^{1, \infty}_{\textnormal{c}}((-42,T)\times\R)\), the following is satisfied:
    \begin{align}
        \iint_{\OT} \rho(t,x)\big(\partial_{t}\phi_{1}(t,x)+\partial_{x}\phi_{1}(t,x)\cV[\rho,\omega](t,x)\big)\dd x\dd t+\int_{\R}\rho_{0}(x)\phi_{1}(0,x)\dd x&=0,\label{eq:conservation_weak_1}\\
        \iint_{\OT} \rho(t,x)\omega(t,x)\big(\partial_{t}\phi_{2}(t,x)+\partial_{x}\phi_{2}(t,x)\cV[\rho,\omega](t,x)\big)\dd x\dd t+\int_{\R}\rho_{0}(x)\omega_{0}(x)\phi_{2}(0,x)\dd x&=0.\label{eq:conservation_weak_2}
    \end{align}
    The reverse is also true: If there exists a weak solution in the sense of \crefrange{eq:conservation_weak_1}{eq:conservation_weak_2}, it is also a weak solution in the sense of \cref{defi:weak_solution}.

\end{lem}
\begin{proof}
    This follows along the lines of the proof of \cref{theo:existence_uniqueness_small_time}. We know by \cref{eq:solution} that the solution can be written as
    \[
\rho(t,x)=\rho_{0}(\xi_{\cV^{*}}(t,x;0))\partial_{2}\xi_{\cV^{*}}(t,x;0),\ \omega(t,x)=\omega_{0}(\xi_{\cV^{*}}(t,x;0)),\ (t,x)\in (0,T)\times \R
    \]
    with \(\cV^{*}\) being the unique solution of the fixed-point problem in \cref{theo:fixed_point_mapping_uniqueness}.
    Plugging this into \cref{eq:conservation_weak_2} (\cref{eq:conservation_weak_1} holds as it is identical to the weak formulation of \(\rho\) in \cref{defi:weak_solution}) yields, for \(\phi_{2}\in W^{1,\infty}_{\textnormal{c}}((-42,T)\times\R)\),
    \begin{align*}
        &\iint_{\OT} \rho(t,x)\omega(t,x)\big(\partial_{t}\phi_{2}(t,x)+\partial_{x}\phi_{2}(t,x)\cV[\rho,\omega](t,x)\big)\dd x\dd t\\
        &=\iint_{\OT} \rho_{0}(\xi_{\cV^{*}}(t,x;0))\partial_{2}\xi_{\cV^{*}}(t,x;0)\omega_{0}(\xi_{\cV^{*}}(t,x;0))\big(\partial_{t}\phi_{2}(t,x)+\partial_{x}\phi_{2}(t,x)\cV[\rho,\omega](t,x)\big)\dd x\dd t.
        \intertext{Substituting \(\xi_{\cV^{*}}(t,x;0)=y\) results in}
        &\iint_{\OT} \rho(t,x)\omega(t,x)\big(\partial_{t}\phi_{2}(t,x)+\partial_{x}\phi_{2}(t,x)\cV[\rho,\omega](t,x)\big)\dd x\dd t\\
        &=\iint_{\OT} \rho_{0}(y)\omega_{0}(y)\tfrac{\dd}{\dd t} \phi(t,\xi_{\cV^{*}}(0,y;t))\dd y\dd t=\int_{\R}\rho_{0}(y)\omega_{0}(y)\int_{0}^{T}\tfrac{\dd}{\dd t} \phi_{2}(t,\xi_{\cV^{*}}(0,y;t))\dd t\dd y.
        \intertext{Using properties of the characteristics in \cref{lem:properties_characteristics}, as well as the test function being zero at \(t=T\), yields}
        &\iint_{\OT} \rho(t,x)\omega(t,x)\big(\partial_{t}\phi_{2}(t,x)+\partial_{x}\phi_{2}(t,x)\cV[\rho,\omega](t,x)\big)\dd x\dd t\\
        &=-\int_{\R}\rho_{0}(y)\omega_{0}(y)\phi_{2}(t,\xi_{\cV^{*}}(0,y;0))\dd y=-\int_{\R}\rho_{0}(x)\omega_{0}(x)\phi_{2}(0,x)\dd x,
    \end{align*}
    as postulated. The converse implication follows by some integration by parts and the fundamental lemma of the calculus of variations.
\end{proof}
Next, we state the definition of a weak solution for the corresponding system of local conservation laws.

\begin{defi}[weak solution---local conservation laws]\label{defi:weak_solution_local}
Let the nonlocal dynamics in \cref{lem:weak_solution_nonlocal_conservative} be given. Then, the system
    \begin{equation}
    \begin{aligned}
            \partial_{t}\rho
    +\partial_{x}\big(V(\rho,\omega)\rho\big)&=0,&& (t,x)\in\OT,\\
    \partial_{t}\big(\rho\omega)+\partial_{x}\big(V(\rho,\omega)\rho \omega\big)&=0, && (t,x)\in\OT,\\
        \left(\rho(0,x),\omega(0,x)\right) &=\left(\rho_{0}(x), \omega_{0}(x)\right), && x\in\R  
       \end{aligned}
    \label{eq:GARZ__local_strong-1}
   \end{equation}
is referred to as the corresponding \textbf{local} system of conservation laws.
We define \((\rho,\omega)\in C\big([0,T];L^{1}_{\text{\loc}}(\R)\big)^{2}\cap L^{\infty}(\OT)^{2}\) to be a weak solution to the Cauchy problem \cref{eq:GARZ__local_strong-1} if,    
    \(\forall\phi_{1},\phi_{2}\in W^{1,\infty}_{\textnormal{c}}((-42,T)\times\R)\), it holds that
    \begin{align}
        \iint_{\OT} \rho(t,x)\big(\partial_{t}\phi_{1}(t,x)+\partial_{x}\phi_{1}(t,x)V(\rho(t,x),\omega(t,x))\big)\dd x\dd t+\int_{\R}\rho_{0}(x)\phi_{1}(0,x)\dd x&=0\label{eq:conservation_local_weak_1}\\
        \iint_{\OT} \rho(t,x)\omega(t,x)\big(\partial_{t}\phi_{2}(t,x)+\partial_{x}\phi_{2}(t,x)V(\rho(t,x),\omega(t,x))\big)\dd x\dd t+\int_{\R}\rho_{0}(x)\omega_{0}(x)\phi_{2}(0,x)\dd x&=0.\label{eq:conservation_local_weak_2}
    \end{align}
\end{defi}
If we assume \((\rho_{0},\omega_{0})\in TV(\R)^{2}\) and we are given suitable assumptions on the function $V$, the problem \cref{eq:GARZ__local_strong-1} admits weak solutions in the sense of \crefrange{eq:conservation_local_weak_1}{eq:conservation_local_weak_2}, and moreover, entropy solutions can be studied following the classical theory as in \cite{Bressan2000,Dafermos2016}; we refer to \cref{sec:entropy_uniqueness} for further discussion.

The following \cref{cor:convergence_weak_solution} shows that solutions to the nonlocal system converge, up to a subsequence, in \(C\big([0,T];L^{1}_{\textnormal{loc}}(\R)\big)^2\) to a weak solution of the corresponding 
local system when \(\eta\) tends to zero, given the assumptions of this section, particularly those in \cref{theo:total_variation_bound}. It also proves the existence of weak solutions for the local system of conservation laws
and thus provides an alternative proof of the existence of weak solutions for the local system of conservation laws \cref{eq:GARZ__local_strong-1} under the assumptions of \cref{theo:total_variation_bound}.

Because the compactness we have derived so far was on the velocity function, we also require an assumption on the velocity implying that 
we can infer the strong convergence of \(\rho_{\eta}\) in \(L^{1}\). 

\begin{theo}[convergence to a weak local solution]\label{cor:convergence_weak_solution}
Let the assumptions of \cref{theo:total_variation_bound} be given. Assume in addition that the velocity \(V\) satisfies the following: if 
\(V(\rho_{\eta},\omega_{\eta})\overset{\eta\rightarrow 0}{\longrightarrow} \cV^{*}\in C\big([0,T];L^{1}_{\text{\loc}}(\R)\big)\) and \(\omega_{\eta}\overset{\eta\rightarrow 0}{\longrightarrow} \omega^{*}\in C\big([0,T];L^{1}_{\textnormal{loc}}(\R)\big)\) on a subsequence, 
we can infer \(\rho^{*}\in C\big([0,T];L^{1}_{\textnormal{loc}}(\R)\big)\) so that
\[
\rho_{\eta}\rightarrow \rho^{*} \text{ in } C\big([0,T];L^{1}_{\textnormal{loc}}(\R)\big)\ \text{ and, as such, }\ \cV^{*}\equiv V(\rho^{*},\omega^{*})
\]
along a suitable subsequence.

Then, \((\rho_{\eta},\omega_{\eta})\in C\big([0,T];L^{1}_{\textnormal{loc}}(\R)\big)^{2}\) converge on a subsequence to a weak solution \((\rho^{*},\omega^{*})\) 
of the corresponding local system of conservation laws in \(C\big([0,T];L^{1}_{\textnormal{loc}}(\R)\big)\); 
that is, $(\rho^{*},\omega^{*})\in C([0,T];L^{1}_{\textnormal{loc}}(\R))^{2}$ satisfy \cref{defi:weak_solution_local} and
\[
\lim_{k\rightarrow \infty}\|\rho_{\eta_{k}}-\rho^{*}\|_{C([0,T];L^{1}(\Omega))}+\|\omega_{\eta_{k}}-\omega^{*}\|_{C([0,T];L^{1}(\Omega))}=0\ \quad\forall\ \text{open and bounded } \Omega\subset\R.
\]
In addition, the following \(TV\) bound on \(\cV^{*}\) holds:
\begin{equation}
    |\cV^{*}(t,\cdot)|_{TV(\R)}\leq |V(\rho_{0},\omega_{0})|_{TV(\R)},\ t\in[0,T].\label{eq:TV_bound_V_star}
\end{equation}
Finally, \begin{itemize}
    \item if \cref{ass:long_time_horizon}(a) holds, we obtain with the constants defined therein,
\[
0\leq \rho^{*}\leq \rho_{\max},\quad \underline{\omega}\leq\omega^{*}\leq \overline{\omega} \text{ on } (0,T)\times\R,
\]
and for \(\partial_{1}V\leqq 0,\ \partial_{2}V\geqq 0\) and \(\omega_{0}\) monotonically decreasing, we have the lower bound
\[
\essinf_{y\in\R}\rho_{0}(y)\leq \rho^{*} \text{ on } (0,T)\times\R;
\]
\item if \cref{ass:long_time_horizon}(b) holds, with the constants defined in \cref{theo:invariant_region}, we have that for the limit \(\cV^{*}\), \cref{eq:invariant_velocity} is satisfied, along with, depending on the additional assumptions on \(V\), either \cref{eq:rho_max_invariant_domain} on \(\rho^{*}\) or \cref{eq:positive_lower_bound_density} and \cref{eq:upper_bound_solution}.
\end{itemize} 
\end{theo}
\begin{proof}
We apply \cref{theo:compacteness} to the nonlocal velocity, and we know that for any open and bounded \(\Omega\subset\R\), 
there exists a subsequence \((\eta_{k})_{k\in\N}\subset\R\) and \(\cV^{*}\in C([0,T];L^{1}(\Omega))\) such that
\[
\lim_{k\rightarrow \infty}\|\cV_{\eta_{k}}-\cV^{*}\|_{C([0,T];L^{1}(\Omega))}=0.
\]
Furthermore, we know that the limit satisfies the same \(TV\) bound as the convergent sequence; that is, the bound in \cref{eq:bound-on-TV-V} is satisfied, resulting in \cref{eq:TV_bound_V_star}.
Moreover, we have
\[|\omega_{\eta}(t,\cdot)|_{TV(\R)}\leq |\omega_{0}|_{TV(\R)}\]
uniformly in \(\eta\) as it is a transport equation, so that following the reasoning in the proof of \cref{theo:compacteness},
there exists \(\omega^{*}\in C([0,T];L^{1}(\Omega))\) with
\begin{equation}
\lim_{n\rightarrow\infty} \|\omega_{\eta_{k_{n}}}-\omega^{*}\|_{C([0,T];L^{1}(\Omega))}=0,
\label{eq:convergence_weak_solution_1}
\end{equation}
where \(\big(\eta_{k_{n}}\big)_{n\in\N}\) is a subsequence of \((\eta_{k})_{k\in\N}\). 
Then, by the assumption on the velocity, we can infer \(\rho^{*}\in C([0,T];L^{1}(\Omega))\), so that
\begin{equation}
\lim_{n\rightarrow\infty} \|\rho_{\eta_{k_{n}}}-\rho^{*}\|_{C([0,T];L^{1}(\Omega))}=0\label{eq:convergence_weak_solution_2}
\end{equation}
and thus
\begin{equation}
\lim_{n\rightarrow\infty} \|V(\rho_{\eta_{k_{n}}},\omega_{\eta_{k_{n}}})-V(\rho^{*},\omega^{*})\|_{C([0,T];L^{1}(\Omega))}=0=\lim_{n\rightarrow\infty}\|V(\rho_{\eta_{k_{n}}},\omega_{\eta_{k_{n}}})- \cV^{*}\|_{C([0,T];L^{1}(\Omega))}.  
\label{eq:convergence_weak_solution_3}
\end{equation}
This is a direct consequence of the strong convergence of \(\rho_{\eta}\) and \(\omega_{\eta}\), their essential boundedness from 
\cref{theo:maximum_principle}, and the local Lipschitz-continuity of \(V\) in both arguments as postulated in \cref{ass:small_time_existence_uniqueness}.
The stated bounds on \(\cV^{*}\), \(\rho^{*}\), and \(\omega^{*}\) are immediate consequences of the results in \cref{theo:maximum_principle} and \cref{theo:invariant_region}.

\smallskip Now, let us look into the weak formulation for \(\rho_{\eta}\) in \cref{eq:conservation_weak_1} first. Refining the notation, 
it holds for all \(\eta\in\R_{>0}\) and \(\phi_{1}\in W^{1,\infty}_{\textnormal{c}}((-42,T)\times\R)\) that
    \[   \iint_{\OT} \rho_{\eta}(t,x)\big(\partial_{t}\phi_{1}(t,x)+\partial_{x}\phi_{1}(t,x)\cV[\rho_{\eta},\omega_{\eta}](t,x)\big)\dd x\dd t+\int_{\R}\rho_{0}(x)\phi_{1}(0,x)\dd x =0.
    \]
On subsequences in the limit \(\eta\rightarrow 0\) and recalling that \(\phi\) is compactly supported, we obtain, thanks to \crefrange{eq:convergence_weak_solution_1}{eq:convergence_weak_solution_3} and the uniform maximum principle in \cref{theo:maximum_principle},
    \[   \iint_{\OT} \rho^{*}(t,x)\big(\partial_{t}\phi_{1}(t,x)+\partial_{x}\phi_{1}(t,x)V(\rho^{*}(t,x),\omega^{*}(t,x))\big)\dd x\dd t+\int_{\R}\rho_{0}(x)\phi_{1}(0,x)\dd x =0,
    \]
which is indeed the first part of the definition of a weak solution for the corresponding local system of conservation laws in \cref{eq:conservation_local_weak_1}.

Next, we move to the second part, that is, considering
 \(\omega_{\eta}\cdot \rho_{\eta}\) in \cref{eq:conservation_weak_2}. From this, it holds for all \(\eta\in\R_{>0}\) and \(\phi_{2}\in W^{1,\infty}_{\textnormal{c}}((-42,T)\times\R)\) that
\[
\iint_{\OT} \rho_{\eta}(t,x)\omega_{\eta}(t,x)\big(\partial_{t}\phi_{2}(t,x)+\partial_{x}\phi_{2}(t,x)\cV[\rho_{\eta},\omega_{\eta}](t,x)\big)\dd x\dd t+\int_{\R}\rho_{0}(x)\omega_{0}(x)\phi_{2}(0,x)\dd x=0.
\]
On subsequences in the limit \(\eta\rightarrow 0\) and recalling that \(\phi\) is compactly supported, we obtain, thanks to \crefrange{eq:convergence_weak_solution_1}{eq:convergence_weak_solution_3} and the uniform maximum principle in \cref{theo:maximum_principle}, that
 \[
\iint_{\OT} \rho^{*}(t,x)\omega^{*}(t,x)\big(\partial_{t}\phi_{2}(t,x)+\partial_{x}\phi_{2}(t,x)V(\rho^{*}(t,x),\omega^{*}(t,x))\big)\dd x\dd t+\int_{\R}\rho_{0}(x)\omega_{0}(x)\phi_{2}(0,x)\dd x=0,
\]
which is the second part of the definition of a weak solution to the corresponding system of local conservation laws 
\cref{eq:conservation_local_weak_2}.
\end{proof}
\begin{rem}[the implicit assumption in \cref{cor:convergence_weak_solution} on the velocity \(V\)] \label{rem:4.8}
In \cref{cor:convergence_weak_solution}, we have assumed that we can infer from the strong convergence 
of \(V(\rho_{\eta},\omega_{\eta})\rightarrow V^{*}\in C([0,T];L^{1}(\Omega))\), where \(\Omega\subset\R\) is open and bounded, 
and the strong convergence of \(\omega_{\eta}\rightarrow \omega^{*}\in C([0,T];L^{1}(\Omega))\), the strong convergence of \(\rho_{\eta}\) 
and can identify as such the limit \(V^{*}\equiv V(\rho^{*},\omega^{*})\). For this, it suffices to postulate the strict monotonicity 
of \(x\mapsto V(x,y)\ \forall y\in\big[\underline{\omega},\overline{\omega}\big]\) (see also \cite[Prop. B.1]{Amadori-Gosse-Guerra_2004}). However, this is not especially restrictive, as we already assume from a modeling perspective that \(\partial_{1}V\leqq 0\), which is in line with the traffic modeling law that the higher the density, the lower the velocity should be. Even more, recalling \cref{theo:invariant_region}, we require in \cref{item:1:invariant} only that \(\partial_{1}V<0\). In \cref{item:2:invariant}, \(\partial_{1}V<0\) is already assumed. Concerning \cref{theo:maximum_principle}, we would indeed need to postulate in addition that \(\partial_{1}V<0\). Concluding, the additional assumption does not restrict the applicability. 

In the previous proof, a weak or weak* convergence of \(\rho_{\eta}\) in \(L^{p},\ p\in[1,\infty]\), would also suffice, but because of the nonlinearity of the flux, we would not be able to infer convergence to the ``correct'' velocity function, that is, the velocity \(V(\rho^{*},\omega^{*})\), where \(\rho^{*}\) denotes the corresponding weak or weak* limit.
\end{rem}

\begin{rem}[more general kernels]
It might be possible to generalize these results to kernels as suggested in \cite{Marconi2023} or \cite{keimer42}; however, as our aim was to obtain a \textbf{general} convergence result in the system's case, we are not focusing on that in this contribution.
\end{rem}

\begin{rem}[constant Lagrangian marker]
Let us assume that \(\omega_{0}\equiv \textnormal{const.}\) Then, according to \cref{theo:existence_uniqueness_small_time} and in particular \cref{eq:solution}, \(\omega(t,x)\equiv \textnormal{const.}\) on any time horizon, so that the dynamics in \(\rho_{\eta}\) reduce to
\[
\partial_{t}\rho_{\eta}+ \partial_{x}\big(\rho_{\eta}\cV_{\eta}(t,x)\big)=0
\]
with \(\cV_{\eta}(t,x)=\tfrac{1}{\eta}\int_{x}^{\infty}\exp(\tfrac{x-y}{\eta})V(\rho_{\eta}(t,y),\textnormal{const.})\dd y,\ (t,x)\in\OT.\) For this problem class, we could still invoke the previous results, but we also know from \cite{friedrich2022conservation}, which considers a scalar nonlocal (in the velocity) conservation law, that there exists a solution on any finite time horizon as long as \(V'\leqq 0\) (this is an abuse of notation, but what is meant is that \(V\) then depends only on the first argument since only the first argument remains, so \(V'\) makes sense), and that in this specific case, a maximum principle holds. This is almost in line with the result in \cref{theo:maximum_principle} except that one requires a maximal density \(\rho_{\max}\), so that \(V(\rho_{\max},\textnormal{const.})=0\), and no sign on \(V'\) is imposed.

The fact that in the given case of \(\omega_{0}\equiv \textnormal{const.}\), a minimum principle and with that a classical maximum and minimum principle holds, is due to the second part of \cref{theo:maximum_principle} under the assumption that \(\omega_{0}\) is decreasing, which, in particular, is satisfied for a constant datum.

Considering the \(TV\) bounds---uniform in \(\eta\in\R_{>0}\)---of \(\rho_{\eta}\), we require \cref{eq:requirement_TV_diminishing}, which, in this specific case of \(\omega_{0}\equiv \textnormal{const.}\), becomes
\begin{equation}
\partial_{1}V(\rho_{\eta},\textnormal{const.})\rho_{\eta}+\eta\partial_{2}\cV_{\eta}\leqq 0.\label{eq:nonlocal_velocity_TV}
\end{equation}
This can be found in \cite[Proposition 4.3, (4.4)]{friedrich2022conservation} in the form
\[
V'(x)x-V(x)+V\Big(\essinf_{x\in\R}\rho_{0}(x)\Big)\leqq 0\ \forall x\in \big[\essinf_{x\in\R}\rho_{0}(x),\|\rho_{0}\|_{L^{\infty}(\R)}\big],
\]
which is---when using the nonlocal identity in \cref{eq:nonlocal_identity} for \cref{eq:nonlocal_velocity_TV}---identical to \cref{eq:nonlocal_velocity_TV}.
Thus, one can state that the present work contains, (almost) as a special case, the convergence for \(\eta\rightarrow 0\) to a local weak solution established in \cite{friedrich2022conservation}.
\end{rem}

\section{Entropy admissibility and uniqueness of the limit}\label{sec:entropy_uniqueness}
In this section, we address the entropy admissibility and uniqueness of the limit $(\rho,\omega)$
established in \cref{cor:convergence_weak_solution} and \cref{lem:weak_solution_nonlocal_conservative}.
We consider the local system in conserved variables,
\begin{equation}
\begin{aligned}
    \partial_t\rho+\partial_x\left(\rho V\left(\rho, \omega
        \right)\right)&=0, &&(t,x)\in (0,T)\times\R,\\
        \partial_t q+\partial_x\left(q V\left(\rho,\omega
        \right)\right)&=0, &&(t,x)\in (0,T)\times\R,\\
    q&=\rho\omega, && \text{ on } (0,T)\times\R,\\
    \big(\rho(0,\cdot),\omega(0,\cdot)\big)&\equiv (\rho_{0},\omega_{0}) &&\text{ on } \R.
    \end{aligned}
    \label{eq:local_system}
\end{equation}
We start by summarizing the assumptions on the velocity field needed in this section.
\begin{ass}[assumptions on the velocity field] 
\label{ass:K_V_2}
We assume that
\begin{align}\label{S5-V1}
& V\in C^{2}(\R^{2}_{\geq0},\R)\\[1mm]
& \forall\,\omega \in\R_{\geq0}:\qquad ~~~V(0,\omega)=\omega\,,\quad\, \exists\ R(\omega) \in \R_{\geq0}: \ V(R(\omega) ,\omega)=0 \label{S5-V2}
\\[1mm] &
\forall\,(\rho,\omega)\in \R_{\geq0}^{2}:\qquad 
\left(\rho V (\rho,\omega)\right)_{\rho\rho} <0\,,\qquad V_\omega(\rho,\omega)\ge 0\,.
\label{S5-V3}
\end{align}
\end{ass}
\begin{rem}[about the assumptions on \(V\)]\label{rem:5.1}~
\begin{itemize}
\item[(a)] The concavity condition in \cref{S5-V3}\((1)\) implies that $\rho\mapsto V (\rho,\omega)$ is strictly decreasing with $\partial_{1}V<0$; see, for instance, \cite[Lemma 1]{FanHertySeibold2014}.
In  particular, it follows that
\begin{equation*}
    \forall\, \omega\in\R_{>0}:\  V (\rho,\omega)\begin{cases} >0& \mbox{ if }\ 0\le \rho< R(\omega)\\
    <0 & \mbox{ if }\ \rho>R(\omega)\end{cases}
\end{equation*}
and that
\begin{equation}\label{V-positive}
  \forall\, \overline{\omega}\in\R_{>0}:\   V(\rho,\omega) \geq 0 \qquad \forall \, (\rho,\omega)\in \Omega_{\overline{\omega}}\coloneqq\big\{(\rho,\omega)\in\R_{\geq0}^{2}:\ \omega\in [0,\overline{\omega}]\ \wedge\ 0\le \rho\le R(\omega)\big\}.
\end{equation}

\item[(b)] We assume $\partial_{2}V\ge 0$ instead of $\partial_{2}V> 0$ as in \cite{FanHertySeibold2014}. However, notice that, given \cref{S5-V2}, it holds that $\partial_{2}V (0,\omega)> 0$ for every $\omega\in \R_{\ge0}$.

\item[(c)] Notice that conditions \cref{S5-V3} imply that \(R\) is differentiable and that $R'(\omega) = - \tfrac {\partial_{2}V(R(\omega),\omega)}{\partial_{1}V(R(\omega),\omega)}\ge 0$ for \(\omega\in\R_{\geq0}\).
This is satisfied given \cref{ass:long_time_horizon}(a) and (b2), 
and in particular by the classical Aw--Rascle--Zhang system, in which $V(\rho,\omega)=\omega - p(\rho)$ \(\forall (\rho,\omega)\in\R_{\geq0}^{2}\) with $p'>0$.

\item[(d)] If the initial data $(\rho_0,\omega_0)$
are such that $V_{\min}\in\R_{\geq 0}$ as in \cref{def:V-max-min},
then \cref{ass:long_time_horizon}(b2) holds true. 
Indeed,
the inequality 
\begin{equation*}
    V_{\min} > \lim_{\rho\to\infty} V\big(\rho,\|\omega_0\|_{L^\infty(\R)}\big),
\end{equation*}
holds because the limit on the right-hand side is negative.
Furthermore, for every $\omega>0$, the map $\rho\mapsto V(\rho,\omega)$ is strictly decreasing and vanishes at $\rho=R(\omega)$. Therefore, $ \lim_{\rho\to\infty} V(\rho,\omega)$ is either negative and finite or $-\infty$.
\end{itemize}
\end{rem}
\begin{lem}[strict hyperbolicity and Riemann invariants for \cref{eq:local_system}]\label{lem:local_system_strict_hyperbolicity}
    Let \cref{ass:K_V_2} hold. 
    Then, the system of (local) conservation laws \cref{eq:local_system} is strictly hyperbolic on \(\R_{>0}\times\R_{> 0}\) with Riemann invariants $\R_{>0}\times\R_{>0}\ni (\rho,\omega)\mapsto V\left(\rho,\omega\right),\ \omega$.
\end{lem}
\begin{proof}
    Define the flux of the system of conservation laws in \cref{eq:local_system} as follows:
    \begin{equation}
        \bF:\begin{cases}
                \R_{>0}\times\R_{>0}&\rightarrow\R^{2},\\
                (\rho,q)&\mapsto 
        \begin{pmatrix}
            \rho \,V\left(\rho,\tfrac{q}{\rho}\right)\\
            q\, V\left(\rho,\tfrac{q}{\rho}\right)
        \end{pmatrix},
        \end{cases}
        \label{defi:flux}
    \end{equation}
    with \(V\) as in \cref{ass:K_V_2}.
    Then, we compute the Jacobian matrix and obtain
    \begin{equation*}
        D\bF(\rho,q) =
        \begin{pmatrix}
            V\left(\rho,\tfrac{q}{\rho}\right) + \rho \tfrac{d}{d\rho}V\left(\rho,\tfrac{q}{\rho}\right) & \rho\tfrac{\dd}{\dd q} V\left(\rho,\tfrac{q}{\rho}\right)\\
            q\tfrac{\dd}{\dd\rho} V\left(\rho,\tfrac{q}{\rho}\right) & V\left(\rho,\tfrac{q}{\rho}\right) + q\tfrac{\dd}{\dd q} V\left(\rho,\tfrac{q}{\rho}\right) 
        \end{pmatrix},\quad (\rho,q)\in\R_{>0}\times\R_{\geq0}.
    \end{equation*}
    The eigenvalues of \(D\bF(\rho,q)\) can be computed for \((\rho,q)\in\R_{>0}\times\R_{>0}\) as
    \begin{equation*}
        \lambda_1(\rho,q) = V\left(\rho,\tfrac{q}{\rho}\right) + \rho\partial_1 V\left(\rho,\tfrac{q}{\rho}\right),\qquad\lambda_2 = V\left(\rho,\tfrac{q}{\rho}\right),
    \end{equation*}
    and the corresponding eigenvectors are
    \begin{equation*}
        \br_1(\rho,q)\coloneqq (\rho,q)^{\top},\qquad \br_2(\rho,q) \coloneqq \left( -\tfrac{\dd}{\dd q} V\left(\rho,\tfrac{q}{\rho}\right), \tfrac{\dd}{\dd\rho}V\left( \rho,\tfrac{q}{\rho}\right) \right)^{\top}.
    \end{equation*}
    Given $\rho\in\R_{>0}$ and \(\partial_{1}V<0\), we obtain
    \begin{equation*}
        \lambda_1(\rho,q) < \lambda_2(\rho,q), 
    \end{equation*}
    from which the strict hyperbolicity follows.
    Moreover,  thanks to \cref{S5-V3}, it holds for \((\rho,q)\in\R_{>0}\times\R_{\geq0}\) that
    \begin{equation}\label{eq:char-fields}
        \nabla\lambda_1(\rho,q)\circ \br_1 = \partial^2_\rho\left(\rho V\left(\rho,\omega\right)\right)\big|_{\omega=\frac{q}{\rho}} <0,\qquad \nabla\lambda_2(\rho,q)\circ \br_2=0,
    \end{equation}
    and so the first characteristic field is genuinely 
    nonlinear, whereas the second one is linearly degenerate. 
    Thanks to the equality in \cref{eq:char-fields}, we have that $v_2(\rho,\omega)=V(\rho,\omega)$ is the \(2\)-Riemann invariant, and by straightforward computations, the \(1\)-Riemann invariant is $v_1(\rho,\omega)=\omega$.
\end{proof}
\begin{rem}[Temple system]
  The system of conservation laws in \cref{eq:local_system} is a Temple system on \(\R_{>0}\times\R_{>0}\) \cite{Temple01}.
\end{rem}
Next, we provide the definition of an entropy--entropy flux pair for the system in \cref{eq:local_system} and the corresponding definition of entropy weak solutions. We then prove the existence of a strictly convex entropy for \cref{eq:local_system}.
\begin{defi}[entropy--entropy flux for \cref{eq:local_system}]\label{defi:entropy-entropy flux}
The functions $\alpha$, $\beta \in C^2\big(\R_{>0}^{2};\R\big)$
  are called an entropy--entropy flux pair for the system in \cref{eq:local_system}
  if $\alpha$ is convex and it holds that
    \begin{equation*}
        \nabla\alpha(\rho,q)^{\top}D\bF(\rho,q)=\nabla\beta(\rho,q)^{\top}\quad \forall (\rho,q)\in\R_{>0}\times\R_{>0},
    \end{equation*}
    where \(\bF\) is the flux associated with \cref{eq:local_system} and defined in the proof of \cref{lem:local_system_strict_hyperbolicity}, namely in \cref{defi:flux}.
    \end{defi}
\begin{defi}[entropy weak solution]\label{defi:entropy-weak-soltion}
    Let $(\rho,\omega)$ be a weak solution of the Cauchy problem \cref{eq:local_system} in the sense of \cref{defi:weak_solution_local}. Then, we call $(\rho,q=\rho\omega)$ an entropy weak solution to \cref{eq:local_system} if the entropy inequality
    \begin{equation}\label{eq:entropy inequality}
        \partial_t \alpha(\rho,q) + \partial_x \beta(\rho,q) \leqq 0
    \end{equation}
    holds in the sense of distributions on $(0,T)\times\R$ for every entropy--entropy flux pair $(\alpha,\beta)$ defined as in \cref{defi:entropy-entropy flux}.
\end{defi}
\begin{lem}[a strictly convex entropy--entropy flux pair for \cref{eq:local_system}]\label{lem:convex_entropy}
Assume \cref{ass:K_V_2}, 
and let $0< \rho_{\min} < \rho_{\max}$ and $0< q_{\min} < q_{\max}$ be fixed constant values. 
Then, the system in \cref{eq:local_system}
admits an entropy--entropy flux pair $(\alpha,\beta)$ with $\alpha$ 
as constructed in the proof, satisfying the convexity condition
\begin{equation}\label{eq:convexity_entropy}
    \alpha(\rho,q)\geq\alpha(\tilde \rho,\tilde q)+\nabla\alpha(\tilde \rho,\tilde q)\circ
(\rho-\tilde \rho,q-\tilde q)+\tfrac{c_0}{2}\|(\rho-\tilde\rho,q-\tilde q)\|^2,
\end{equation}
for some $c_0\in\R_{>0}$ and for every $(\rho,q)$, $(\tilde\rho,\tilde q)\in[\rho_{\min},\rho_{\max}]\times [q_{\min},q_{\max}]$.
\end{lem}
\begin{proof}
By straightforward computations, we can verify that the following pair $(\alpha, \beta)$ is an entropy--entropy flux pair for \cref{eq:local_system}:
\begin{align} \label{def:alpha}
\alpha(\rho,q) & =
 \rho h\left(  \tfrac{q}{\rho}\right)+ \rho\int_{\rho}^{\bar\rho}   V 
\left(z,\tfrac{q}{\rho}\right)\tfrac{1}{z^2}\dd z,\\
\beta(\rho,q) 
    & = -\tfrac{V^2}{2}+V\rho\left( h\left( \tfrac{q}{\rho}\right)+\int_{\rho}^{\bar\rho} V\left(z,\tfrac{q}{\rho}\right)\tfrac{1}{z^2}\dd z \right), \label{def:beta}
\end{align}
for a specific
$h\in C^2(\R;\R)$ deduced in the following, with $\bar \rho\in[\rho_{\min},\rho_{\max}]$ fixed.

We claim that $h$ can be chosen so that \cref{eq:convexity_entropy} holds.
\smallskip
First, we write the Hessian matrix $\bH[\alpha](\rho,q)$ of $\alpha$ for \((\rho,q)\in \R_{>0}\times\R_{>0}\):
\begin{equation*} \bH[\alpha](\rho, q) = 
    \begin{bmatrix}
        \partial_\rho^2\alpha & \partial_\rho\partial_q\alpha\\
        \partial_\rho\partial_q\alpha & \partial_q^2 \alpha
    \end{bmatrix},
\end{equation*}
with
\begin{align}
    \partial_\rho\alpha(\rho,q) & =h\left( \tfrac{q}{\rho} \right)-h'\left(\tfrac{q}{\rho}\right)\tfrac{q}{\rho}+\int_\rho^{\bar\rho}V\left(z,\tfrac{q}{\rho}\right)\tfrac{1}{z^2}\dd z-\tfrac{q}{\rho}\int_\rho^{\bar\rho}\partial_2V\left( z,\tfrac{q}{\rho}\right)\tfrac{1}{z^2}\dd z-\tfrac{1}{\rho}V\left(\rho,\tfrac{q}{\rho}\right),\nonumber\\
    \partial_q\alpha(\rho,q) & = h'\left( \tfrac{q}{\rho} \right)+\int_{\rho}^{\bar\rho}\partial_2V\left(z,\tfrac{q}{\rho} \right)\tfrac{1}{z^2}\dd z,\nonumber\\
    \partial_\rho^2\alpha(\rho,q) & =h''\left(\tfrac{q}{\rho}\right)\tfrac{q^2}{\rho^3}+2\tfrac{q}{\rho^3}\partial_2V\left(\rho,\tfrac{q}{\rho}\right)+\tfrac{q^2}{\rho^3}\int_{\rho}^{\bar\rho}\partial_2^2V\left(z,\tfrac{q}{\rho}\right)\tfrac{1}{z^2}\dd z-\tfrac{1}{\rho}\partial_1V\left(\rho,\tfrac{q}{\rho}\right),\label{eq:partial-derivatives-entropy}\\
    \partial_\rho\partial_q\alpha(\rho, q) & = -h''\left( \tfrac{q}{\rho}\right)\tfrac{q}{\rho^2}-\partial_2V\left( \rho,\tfrac{q}{\rho}\right)\tfrac{1}{\rho^2}-\tfrac{q}{\rho^2}\int_\rho^{\bar\rho}\partial_2^2V\left( z,\tfrac{q}{\rho}\right)\tfrac{1}{z^2}\dd z,\nonumber\\
    \partial_q^2\alpha(\rho, q) & = h''\left( \tfrac{q}{\rho} \right)\tfrac{1}{\rho}+\tfrac{1}{\rho}\int_\rho^{\bar\rho}\partial_2^2V\left( z,\tfrac{q}{\rho}\right)\tfrac{1}{z^2}\dd z.\nonumber
\end{align}

Recalling that $\partial_1V<0$ thanks to \cref{rem:5.1}, 
we obtain that
\begin{align}
    \label{eq:determinante}\det\big(\bH[\alpha](\rho,\tfrac{q}{\rho})\big)
    & = -\tfrac{1}{\rho^2}\partial_1V\left(\rho,\tfrac{q}{\rho}\right)\left(
    h''\left( \tfrac{q}{\rho}\right)+\int_\rho^{\bar\rho}\partial_2^2V\left(z,\tfrac{q}{\rho}\right)\tfrac{1}{z^2}\dd z \right)-\tfrac{1}{\rho^4}\left(\partial_2V\left(\rho,\tfrac{q}{\rho}\right)
    \right)^2
    \\
    & = \tfrac{1}{\rho^2}\big|\partial_1V\left(\rho,\tfrac{q}{\rho}\right)\big| 
    \, h''\left( \tfrac{q}{\rho}\right)- G\left(\rho,\tfrac{q}{\rho}\right)\notag
\end{align}
with 
\[
G(\rho,\omega)\eqqcolon
\tfrac{1}{\rho^2}\partial_1V\left(\rho,\omega
\right)\int_\rho^{\bar\rho}\partial_2^2V\left(z,\omega 
\right)\tfrac{1}{z^2}\dd z 
    +\tfrac{1}{\rho^4}\left(\partial_2V\left(\rho,\omega 
    \right)
   \right)^2\,,\qquad \omega=\tfrac{q}{\rho}\,.
\]
Since $G$ is bounded on the set $[\rho_{\min},\rho_{\max}]\times [q_{\min},q_{\max}]$, we can choose a quadratic $h$,
\begin{equation}\label{def:h}
h(y)\coloneqq \bar C y^2,\ y\in\R,
\end{equation}
with $\bar C$ large enough that \cref{eq:determinante} is strictly positive on \([\rho_{\min},\rho_{\max}]\times [q_{\min},q_{\max}]\).

Under the same condition on $\bar C$, we obtain that \(\partial_q^2\alpha >0\). Therefore,
we conclude that $\alpha$ is strictly convex and that it satisfies \cref{eq:convexity_entropy}, with $c_0$ defined as
\begin{equation*}
    c_0\coloneqq\min \big\{\min\{\mu_1(\rho,q),\mu_2(\rho,q)\}; \ (\rho,q)\in [\rho_{\min},\rho_{\max}] \times [q_{\min},q_{\max}]\big\}
\end{equation*}
and $\mu_1$ and $\mu_2$ the eigenvalues of \(\bH[\alpha]\). 
\end{proof}

\begin{theo}[uniqueness of the entropy weak solution]\label{theo:uniqueness-of-entropy-weak-solution}
Assume \cref{ass:K_V_2}. 

Fix the constant values $M,\omega_{\infty},\rho_{\infty}\in\R_{>0}$ such that $0<\rho_\infty<R(\omega_\infty)$, and define $V_\infty\coloneqq V(\rho_\infty,\omega_\infty)$. 
Let $a_i,$ and $b_i$ for  $i\in\{1,2\}$ be given such that $a_i\le b_i,\ i\in\{1,2\}$, and 

\begin{multicols}{3}
\begin{enumerate}
    \item  \(a_ib_i \le 0 \) for \(i\in\{1,2\}\) \label{item:1}
    \item \(|a_2| \le V_\infty\) \label{item:2}
    \item \(b_2+ |a_1| < \omega_\infty - V_\infty\) \label{item:3}
\end{enumerate}
\label{cond-on-aibi}
\end{multicols}
\noindent Define
\begin{equation}\label{DM}
    \mathcal{D}_M\coloneqq\left\{ (\rho,\omega):\R\to E;\quad
    x\mapsto (v_1,v_2)(\rho(x),\omega(x))
    \in L^1\big(\R;\R^2\big),\quad |v_1|_{TV(\R)}+|v_2|_{TV(\R)}\leq M \right\},
\end{equation}
with
\begin{equation*}
    E\coloneqq\left\{ (\rho,\omega)\in\R_{>0}\times\R: v_i(\rho,\omega)\in[a_i,b_i]\quad i=1,2 \right\}
\end{equation*}
and
\begin{equation*}
    v_1(\rho,\omega)=\omega-\omega_\infty,\qquad v_2(\rho,\omega)=V(\rho,\omega)-V_\infty.
\end{equation*}
Let $(\rho,\omega)$ be an entropy weak solution to the Cauchy problem \cref{eq:local_system}
with initial data $(\rho_0,\omega_0)\in \mathcal{D}_M$ and $t\mapsto (\rho,\omega)(t,\cdot)\in \mathcal{D}_M$ for all $t>0$.
Then,
$(\rho,\omega)$ is the unique entropy weak solution to the Cauchy problem in \cref{eq:local_system} that takes values in $\mathcal{D}_M$. 

\end{theo}

\begin{proof} We start by making several observations:
\begin{itemize}
    \item for every function $(\rho,\omega)\in \mathcal{D}_M$, 
\begin{equation*}
\lim_{x\to\pm\infty}\rho(x)=\rho_\infty, \qquad\lim_{x\to\pm\infty}\omega(x)=\omega_\infty,
\end{equation*}

and $V_\infty = V(\rho_\infty,\omega_\infty) < V(0,\omega_\infty) = \omega_\infty$, so the right-hand side of \cref{item:3} is positive;
\item for every $(\rho,\omega)\in E$,
\begin{equation*}
\omega\in [a_1+\omega_\infty, b_1 +\omega_\infty]
\,,
\qquad 
V(\rho,\omega)\in [a_2+V_\infty, b_2+V_\infty]
\,;
\end{equation*}
\item \cref{item:2} in \cref{cond-on-aibi} implies that $V(\rho,\omega)\ge 0$ for all $(\rho,\omega)\in E$;
\item \cref{item:3} in \cref{cond-on-aibi} implies that $\omega\ge \underline\omega\coloneqq a_1+\omega_\infty >0 $ for all $(\rho,\omega)\in E$ and that
\begin{equation*}
    \max_{(\rho,\omega)\in E}V(\rho,\omega)\eqqcolon  V_{\max} <  \underline\omega.
\end{equation*}
Since $\partial_{1}V<0$, the function
\begin{equation*}
    \rho\mapsto \bar V_\omega (\rho)\coloneqq V(\rho,\omega),\ \rho\in\R_{>0}
\end{equation*}
is invertible for every $\omega\in[a_1+\omega_{\infty},b_1+\omega_{\infty}]$. Denoting by $v\mapsto V^{-1} (v,\omega)$ its inverse, from the fact that $V_{\max} < \underline\omega$, we conclude that there exists a positive quantity 
\[
\rho_{\min}\coloneqq V^{-1}(V_{\max},\underline\omega)=V^{-1}(b_2+V_{\infty},a_1+\omega_{\infty})>0
\] 
such that $\rho \ge \rho_{\min}>0$ for any $(\rho,\omega)\in E$.
\end{itemize}

\par\smallskip
The rest of the proof follows the strategy of \cite[Theorem 1.1]{Bressan-Guerra_2024} for the system in \cref{eq:local_system}, relying on the proof in \cref{lem:convex_entropy} that it admits a strictly convex entropy.
Specifically, we use the semigroup constructed in \cite{Baiti1997} 
on the domain $\mathcal{D}_M$ as above. 
Applying \cite[Theorem 1]{Baiti1997} ensures the existence of a unique continuous semigroup on $\mathcal{D}_M$.
We refer to \cite[Definition 9.1]{Bressan2000} for the definition of a standard Riemann semigroup:
\[
S: [0,\infty)\times \mathcal{D}_M \mapsto \mathcal{D}_M\,,\qquad (\rho(t,\cdot), \omega(t,\cdot)) = S_t (\rho_0,\omega_0) \in \mathcal{D}_M\quad \forall\, t\ge 0\,.
\]
Thanks to \cref{lem:convex_entropy},
the proof of Theorem 1.1 in \cite{Bressan-Guerra_2024} 
can be adapted to our case, indeed the fact that the system is endowed with a strictly convex entropy and the property of the trajectories of the semigroup on \(\mathcal{D}_M\) that can be characterized as ``viscosity solutions'' yield that the proof of \cite[Theorem 1.1]{Bressan-Guerra_2024} holds true for our semigroup \(S\). Then we 
conclude that every entropy weak solution $(\rho, \omega)$ to \cref{eq:local_system} that takes values within the domain of the semigroup $\mathcal{D}_M$ coincides with a semigroup trajectory. This solution is also unique as a consequence of the uniqueness of the semigroup.
\end{proof}
See also \cite{Cheng2025} for the strategy of \cite[Theorem 1.1]{Bressan-Guerra_2024} applied in a similar context for the one-dimensional isothermal Euler system. For Temple systems with genuinely nonlinear characteristic fields, the uniqueness of solutions was studied in \cite{DG91,H94} using Holmgren's principle.

Before we turn our attention to the next theorem on the convergence of the nonlocal solution to the unique local one, 
recall the nonlocal problem \cref{eq:nonlocal_GARZ}-\cref{eq:init_data} together with \cref{eq:nonlocal_velocity_definition}:
\begin{equation}
    \begin{aligned}
        \partial_t \rho + \partial_x (\mathcal{V}_\eta\rho)&=0, && (t,x)\in (0,T)\times\R \\
        \partial_t (\rho\omega) + \partial_x ( \mathcal{V}_\eta\rho\omega)&=0, && (t,x)\in (0,T)\times\R,\\
    \mathcal{V}_\eta(t,x)&= \tfrac{1}{\eta}\int_x^{\infty} \e^{\frac{x-y}{\eta}}V(\rho(t,y),\omega(t,y))\dd y && (t,x)\in (0,T)\times\R,\\
    \rho(0,\cdot)&\equiv\rho_{0}, && \text{ on } \R,\\
    \omega(0,\cdot)&\equiv\omega_{0},&&\text{ on } \R.
    \end{aligned}
    \label{S5:NL}
\end{equation}
We denote by \((\rho_\eta,\omega_\eta)\in C([0,T];L^{1}_{\textnormal{loc}}(\R))^{2}\) the corresponding solution.
In the next theorem, we show that if the sequence $\{(\rho_{\eta},\omega_{\eta})\}_{\eta>0}$ satisfies \cref{eq:requirement_TV_diminishing}, then any limit obtained by compactness is an entropy weak solution to \cref{eq:local_system}.

\begin{theo}[convergence to the entropy solution]\label{theo:singular_limit_problem}
Under the assumptions of \cref{theo:uniqueness-of-entropy-weak-solution}, 
let $(\rho_\eta,\omega_\eta)$,  
where $\eta>0$, be the weak solution to \cref{S5:NL} in the sense of \cref{defi:weak_solution}. 
Assume that the sequence  $\{(\rho_{\eta},\omega_{\eta})\}_{\eta>0}$ satisfies \cref{eq:requirement_TV_diminishing}, and let $(\rho^*,\omega^*)$ be the limit in \(C([0,T];L^{1}_{\textnormal{loc}}(\R))\) of 
$(\rho_{\eta_k},\omega_{\eta_k})$ along a subsequence $\eta_k\to 0$ as $k\to\infty$.
Then, $(\rho^*,\omega^*)$ is an entropy weak solution with respect to the entropy--entropy flux pairs $(\alpha,\beta)$ of the form \cref{def:alpha}--\cref{def:beta}.
\end{theo}

\begin{proof}
We consider the system in \cref{eq:local_system} in the conserved variables, with $q=\rho\omega$,
and we prove that the limit function $(\rho^*,q^*)$ obtained by passing to the limit $\eta\to 0$ is an entropy weak solution in the sense of \cref{defi:entropy-weak-soltion}. We observe that a convergent subsequence $(\rho_{\eta_k},\omega_{\eta_k})$ exists thanks to \cref{cor:convergence_weak_solution}, \cref{rem:4.8}, 
\cref{ass:K_V_2}, and \cref{rem:5.1}(a).
To prove \cref{eq:entropy inequality},
we apply an argument similar to \cite{bressan-shen2021entropy}, generalizing it to the system case. Let $(\rho_\eta,\omega_\eta)$ be a weak entropy solution of the system in \cref{eq:nonlocal_GARZ}. By rewriting this system, adding and subtracting the local flux in each equation, we obtain
\begin{equation}
\begin{aligned}
     \partial_t\rho+\partial_x\left(\rho V\right)&=\partial_x\left( \rho  V- \rho \mathcal{V}_\eta \right),\\
        \partial_t q+\partial_x\left(q V\right)&=\partial_x\left( q V  - q\mathcal{V}_\eta  \right),
\end{aligned}    
\end{equation}
where we denote $(\rho_\eta,q_\eta=\rho_\eta\omega_\eta)$ 
as $(\rho, q)$ as soon as $\eta>0$ is fixed. After multiplying the first equation
by $\partial_\rho\alpha$ and the second one by $\partial_q\alpha$ and then summing them, we have
\begin{align*}
    & \partial_\rho\alpha\,\partial_t\rho+\partial_\rho\alpha\,\partial_x\left(\rho V\right)+\partial_q\alpha\,\partial_t q+\partial_q\alpha\,\partial_x\left(q V\right)\\
    & = \partial_\rho\alpha\,\partial_x\left( \rho V- \rho \mathcal{V}_\eta\right)+\partial_q\alpha\,\partial_x\left( q V- q\mathcal{V}_\eta \right)    \\
    &= \underbrace{[\partial_\rho\alpha\cdot\partial_x \rho  + \partial_q\alpha\cdot\partial_x q]}_{=\partial_x\alpha}  \left( V- \mathcal{V}_\eta \right) + 
\left[\rho\partial_\rho\alpha + q\partial_q\alpha\right] \partial_x \left(V- \mathcal{V}_\eta \right) 
\\
    &= \partial_x[\alpha \left( V- \mathcal{V}_\eta \right)]   +
\left[\rho\partial_\rho\alpha + q\partial_q\alpha -\alpha\right] \partial_x \left(V- \mathcal{V}_\eta \right)\,.
\end{align*}
Thanks to the properties of $(\alpha,\beta)$, one has that
\begin{equation*}
\partial_\rho\alpha\cdot\partial_x\left(\rho V\right) + \partial_q\alpha\cdot\partial_x\left(q V\right) 
    = \partial_\rho \beta \cdot \partial_x \rho + \partial_q \beta \cdot\partial_x q=
    \partial_x\beta 
\end{equation*}
and then
\begin{align*}
     \partial_t \alpha +\partial_x\beta  &= \partial_\rho\alpha\,\partial_x\left( \rho V- \rho \mathcal{V}_\eta \right) +\partial_q \alpha\,\partial_x\left( q V- q\mathcal{V}_\eta \right) \\
    &=\underbrace{[\partial_\rho\alpha \cdot \rho_x + \partial_q \alpha\cdot q_x]}_{=\partial_x\alpha}\,\left( V -\mathcal{V}_\eta \right)+\left[\rho \partial_\rho\alpha + q \partial_q \alpha \right]\,\partial_x\left( V-\mathcal{V}_\eta\right).
\end{align*}
Multiplying by a test function $\varphi \in W^{1,\infty}_{\textnormal{c}}((0,\infty)\times\R;\R_{\ge0})$ and integrating in space and time, we obtain
\begin{align}
    &-\int_0^{\infty}\int_\R [\alpha\varphi_t+\beta\varphi_x]\dd x\dd t  = \int_0^{\infty}\int_\R \partial_x\alpha[V-\mathcal{V}_\eta]\varphi\dd x\dd t + \int_0^{\infty}\int_\R[\rho\partial_\rho\alpha+q\partial_q\alpha]\partial_x[V-\mathcal{V}_\eta]\varphi\dd x\dd t\nonumber\\
    &\qquad  = -\int_0^{\infty}\!\!\int_\R\alpha\partial_x[V-\mathcal{V}_\eta]\varphi\dd x\dd t - \int_0^{\infty}\!\!\int_\R\alpha[V-\mathcal{V}_\eta]\partial_x\varphi\dd x\dd t + \int_0^{\infty}\!\!\int_\R [\rho\partial_\rho\alpha+q\partial_q\alpha]\partial_x[V-\mathcal{V}_\eta]\varphi\dd x\dd t\nonumber\\
    & \qquad = \int_0^{\infty}\int_\R [\rho\partial_\rho\alpha+q\partial_q\alpha-\alpha]\partial_x[V-\mathcal{V}_\eta]\varphi\dd x\dd t - \int_0^{\infty}\int_\R\alpha[V-\mathcal{V}_\eta]\partial_x\varphi\dd x\dd t.\label{eq:proof-singular-limit}
\end{align}
Recalling the partial derivatives of $\alpha$ in \cref{eq:partial-derivatives-entropy},
\begin{align*}
    \partial_\rho\alpha(\rho,q)
    & =h\left(\tfrac{q}{\rho}\right) -\rho h'\left(\tfrac{q}{\rho}\right)\tfrac{q}{\rho^2}+\int_{\rho}^{\bar\rho}V\left( z,\tfrac{q}{\rho} \right)\tfrac{1}{z^2}\,dz-V\left(\rho,\tfrac{q}{\rho}\right)\tfrac{1}{\rho}- \tfrac{q}{\rho}\int_{\rho}^{\bar\rho}\partial_2V\left( z,\tfrac{q}{\rho}\right)\tfrac{1}{z^2}\,dz,\\
  \partial_q\alpha(\rho,q) 
    & = \rho h'\left(\tfrac{q}{\rho}\right)\tfrac{1}{\rho}+\int_{\rho}^{\bar\rho}\partial_2V\left(z,\tfrac{q}{\rho}\right)\tfrac{1}{z^2}\,dz,  
\end{align*}
we obtain
\begin{equation*}
    \rho\partial_\rho\alpha+q\partial_q\alpha-\alpha = - V,
\end{equation*}
and using this identity and \cref{eq:nonlocal_identity} in \cref{eq:proof-singular-limit}, we have
\begin{align*}
    & - \int_0^{\infty}\int_\R [\alpha\varphi_t+\beta\varphi_x]\dd x\dd t = - \int_0^{\infty}\int_\R V \partial_x[V-\mathcal{V}_\eta]\varphi\dd x\dd t - \int_0^{\infty}\int_\R\alpha [V-\mathcal{V}_\eta]\partial_x\varphi\dd x\dd t\\
    & \qquad  = -\tfrac{1}{2}\int_0^{\infty}\int_\R \partial_x(V^2)\varphi\dd x\dd t + \int_0^{\infty}\int_\R V\partial_x(\mathcal{V}_\eta)\varphi\dd x\dd t - \int_0^{\infty}\int_\R \alpha[V-\mathcal{V}_\eta]\partial_x\varphi\dd x\dd t\\
    & \qquad = \tfrac{1}{2} \int_0^{\infty}\int_\R \partial_x ( \mathcal{V}_\eta^2 - V^2)\varphi\dd x\dd t - \eta \int_0^{\infty}\int_\R (\partial_x\mathcal{V}_\eta)^2\varphi\dd x\dd t - \int_0^{\infty}\int_\R \alpha [V-\mathcal{V}_\eta]\partial_x\varphi\dd x\dd t\\
    & \qquad \le \tfrac{1}{2} \int_0^{\infty}\int_\R \partial_x ( \mathcal{V}_\eta^2 - V^2)\varphi\dd x\dd t - \int_0^{\infty}\int_\R \alpha [V-\mathcal{V}_\eta]\partial_x\varphi\dd x\dd t\\
    &\qquad =  -\tfrac{1}{2} \int_0^{\infty}\int_\R ( \mathcal{V}_\eta^2 - V^2)\partial_x\varphi\dd x\dd t - \int_0^{\infty}\int_\R \alpha [V-\mathcal{V}_\eta]\partial_x\varphi\dd x\dd t.
\end{align*}
Both terms in the last line go to zero when $\eta\to0$. To conclude, we consider a sequence of nonlocal solutions $(\rho_{\eta},\omega_{\eta})$. We know there exists a subsequence of $\{\eta_k\}_{k\in\N}$ such that, as $\lim_{k\to\infty}\eta_k=0$, $(\rho_{\eta_k},\omega_{\eta_k})$ converges to a limit solution $(\rho^*,\omega^*)$. Thanks to the computations above, we obtain
\begin{align*}
    \int_0^{\infty}\!\!\!\int_\R \left(\alpha(\rho_{\eta_k},\omega_{\eta_k})\partial_t\varphi+\beta(\rho_{\eta_k},\omega_{\eta_k})\partial_x\varphi\right)\dd t\dd x\geq \tfrac{1}{2}\!\int_0^\infty\!\!\!\!\int_\R (\mathcal{V}_{\eta_k}^2-V^2)\partial_x\varphi\dd x\dd t +\!\! \int_0^\infty\!\!\!\!\int_\R \alpha(V-\mathcal{V}_{\eta_k})\partial_x\varphi\dd x\dd t.
\end{align*}
By letting $\eta_k\to0$ and using the fact that the right-hand side goes to zero as proven in \cref{cor:convergence_weak_solution}, we find that
\begin{equation*}
    \int_0^{\infty}\!\!\! \int_\R \alpha(\rho^*,\omega^*)\partial_t\varphi+\beta(\rho^*,\omega^*)\partial_x\varphi\dd t\dd x \geq 0.
\end{equation*}
\end{proof}

To conclude in the next corollary, we prove that the sequence of nonlocal solutions \{\((\rho_\eta,\omega_\eta)\}_{\eta>0}\) admits a unique limit \((\rho^*,\omega^*)\) that coincides with a semigroup trajectory \(t\mapsto S_t(\rho_0,\omega_0)\). 
\begin{cor}\label{final-cor}
    Under the assumptions of \cref{theo:uniqueness-of-entropy-weak-solution}, let $ \mathcal{D}_M$ be a domain of the form \cref{DM}
 and let $(\rho_{0},\omega_{0})\in \mathcal{D}_M$. 
Assume that the sequence  $\{(\rho_{\eta},\omega_{\eta})\}_{\eta>0}$ satisfies \cref{eq:requirement_TV_diminishing}. Then, as $\eta\to 0$, the sequence $(\rho_{\eta},\omega_{\eta})$ converges in \(C([0,T];L^{1}_{\textnormal{loc}}(\R))\) to the solution of the semigroup trajectory $t\mapsto S_t(\rho_0,\omega_0)$.
\end{cor}
\begin{proof}
In the following, we consider a domain $\mathcal{D}_M$ as defined in the statement of \cref{theo:uniqueness-of-entropy-weak-solution}. Assume that $(\rho_{0},\omega_{0})\in\mathcal{D}_M$; as a consequence, 
\[
\lim_{x\to\pm\infty}\rho_0(x)=\rho_\infty, \qquad\lim_{x\to\pm\infty}\omega_0(x)=\omega_\infty
\]
for some positive values $\rho_\infty$ and $\omega_\infty$. By \cref{rem:values-at-infty}, the solution $t\mapsto (\rho_\eta,\omega_\eta)(t,\cdot)$ admits the same states at $\pm\infty$ as the initial data, and so \((\rho_\eta,\omega_\eta)\in\mathcal{D}_M\) for any \(t\in[0,T]\). The sequence \(\{(\rho_\eta,\omega_\eta)\}_{\eta>0}\) admits a subsequence that converges in \(C([0,T];L^1_{\text{loc}}(\R))\) to the limit \((\rho^*,\omega^*)\), that is an entropy weak solution to \cref{eq:local_system} with respect to the entropy--entropy flux pairs \((\alpha,\beta)\) as in \cref{def:alpha,def:beta}. Then, the limit \((\rho^*,\omega^*)\) satisfies the entropy inequality \cref{eq:entropy inequality} and takes values in the domain \(\mathcal{D}_M\) of the semigroup \(S\). Thus, thanks to \cref{theo:uniqueness-of-entropy-weak-solution} the limit \((\rho^*,\omega^*)\) coincides with the semigroup trajectory \(t\mapsto S_t(\rho_0,\omega_0)\), and the uniqueness of the limit follows by the uniqueness of the semigroup.
\end{proof}

\section{Numerical simulations}\label{sec:numerical_simulations}
In this section, we present some numerical simulations regarding the singular limit.
In particular, we discretize the nonlocal system in conservative form, that is,
\begin{equation}\label{eul-coord}
    \begin{cases}
        \partial_t\rho+\partial_x\left(\rho\, \cV_\eta\right)=0,\\
        \partial_t q+\partial_x\left(q\,\cV_\eta\right)=0,
    \end{cases}
\end{equation}
through an adapted Lax--Friedrichs numerical scheme as in \cite{chiarello}. We consider the exponential kernel function and set the spatial step to $\Delta x=10^{-4},$ the Courant-Friedrichs-Lewy (CFL) condition to \(0.1\), implying as time step $\Delta t=0.1\cdot\Delta x$. 
For the first two simulations, the exact solutions of the local system \cref{eq:local_system} are computed using the Riemann solver for the Aw--Rascle--Zhang model presented in \cite{RosiniAnnales}, whereas for the third simulation the exact solution is obtained analytically.
For all plots in this section, the final time is set to $T = 1$.
For the first simulation (\crefrange{fig:ex1}{fig:ex3}), let us consider 
\begin{equation}\label{eq:arz_vel}
    V(\rho,\omega)=\omega-6\rho,
\end{equation}
    as in \cite{chiarello2020micro}, and the initial data for \(x\in\R\)
\begin{equation}
\label{e:initialdatum_1}
\rho_0(x)=0.05,\
q_0(x)=\begin{cases}
    { 0.0175}&\hbox{if } x<0,\\
    0.04&\hbox{if } x\geq0,
\end{cases}\
v_0(x)=\begin{cases}
    0.05&\hbox{if } x<0,\\
    0.5&\hbox{if } x\geq0,
\end{cases}\
\omega_0(x)=\begin{cases}
    0.35&\hbox{if } x<0,\\
    0.8&\hbox{if } x\geq0
\end{cases}
\end{equation}
with $v_0(x)=V\left(\rho_0(x),\omega_0(x)\right)$.
Starting from a constant initial density and an increasing Riemann-type initial Lagrangian marker, we observe the formation of a rarefaction wave on the left and a contact discontinuity on the right; after a certain time, the density reaches the vacuum. Moreover, the solution remains between zero and the maximum of the initial data.
\begin{figure}[!ht]
\centering
\includegraphics[width=0.3\textwidth]{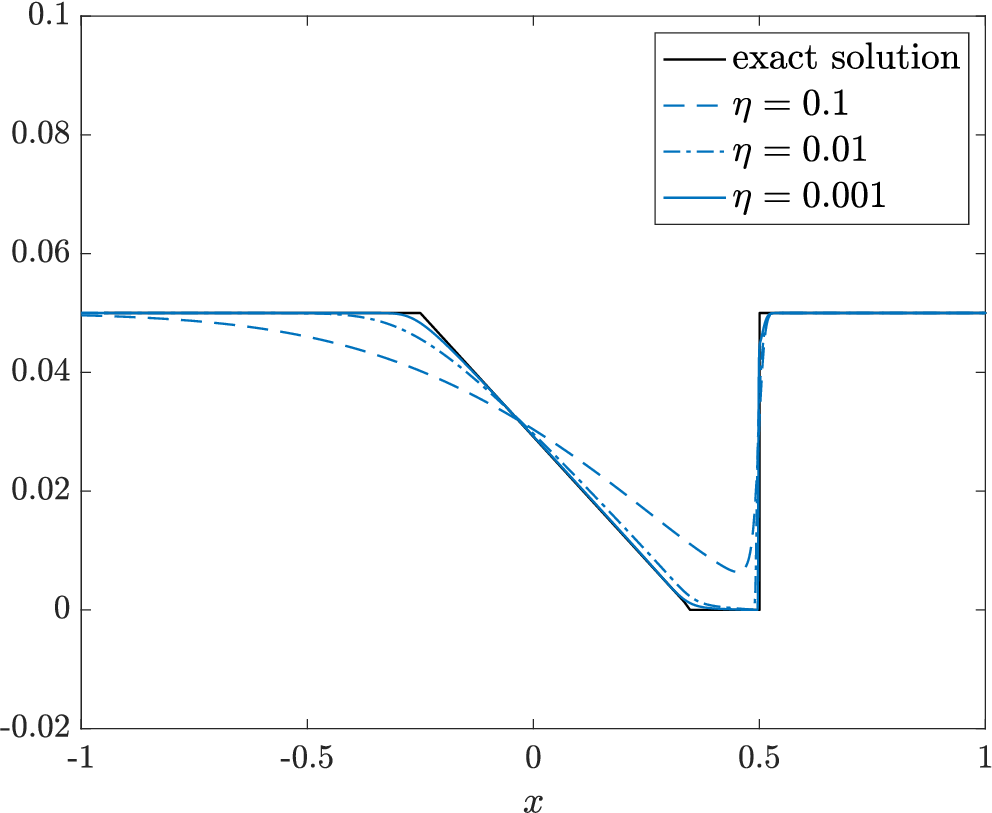}
\hfil
\includegraphics[width=0.3\textwidth]{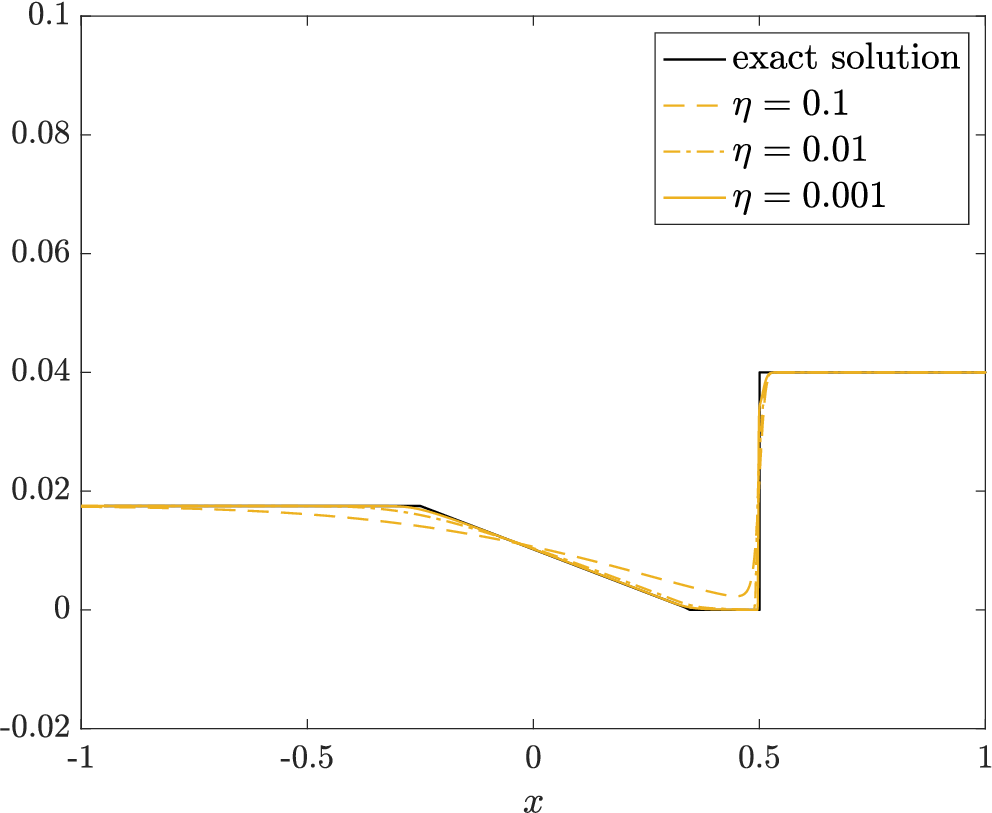}
\caption{Numerical simulations for the initial data in \cref{e:initialdatum_1} and the velocity function \cref{eq:arz_vel}. Left: density $\rho$ for different values of $\eta.$ Right: $q$ for different values of $\eta.$
} 
\label{fig:ex1}
\end{figure}
\begin{figure}[!ht]
\centering
\includegraphics[scale=0.25, trim={6mm 0mm 31mm 0mm}, clip]{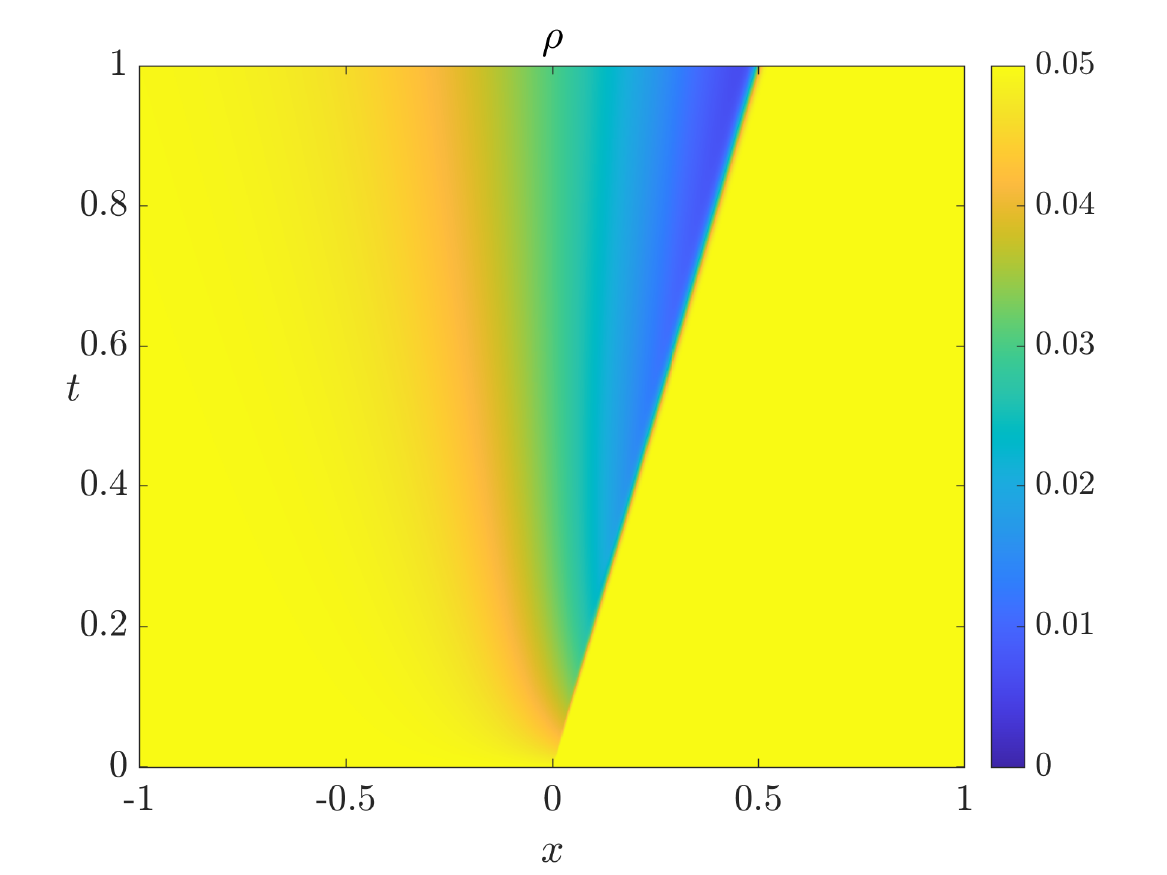}
\includegraphics[scale=0.25, trim={26mm 0mm 31mm 0mm}, clip]{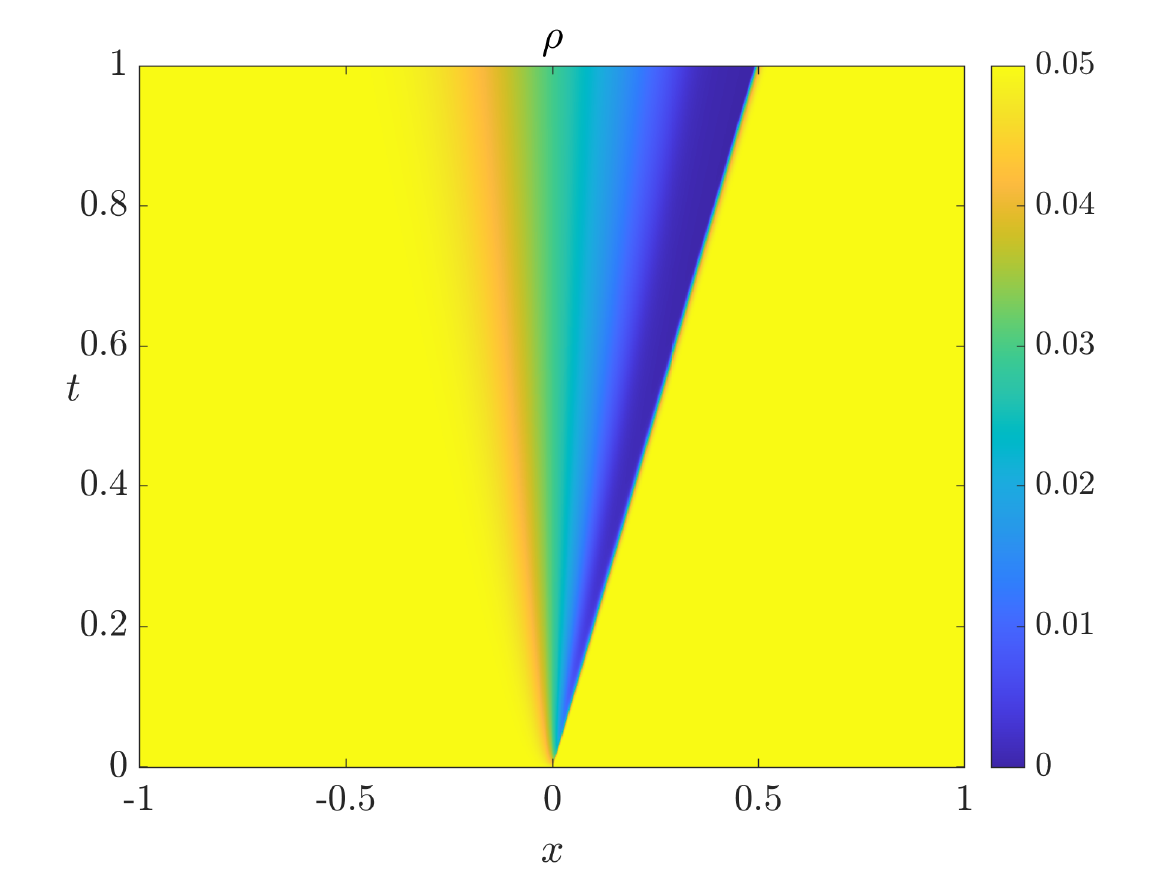}
\includegraphics[scale=0.25, trim={26mm 0mm 31mm 0mm}, clip]{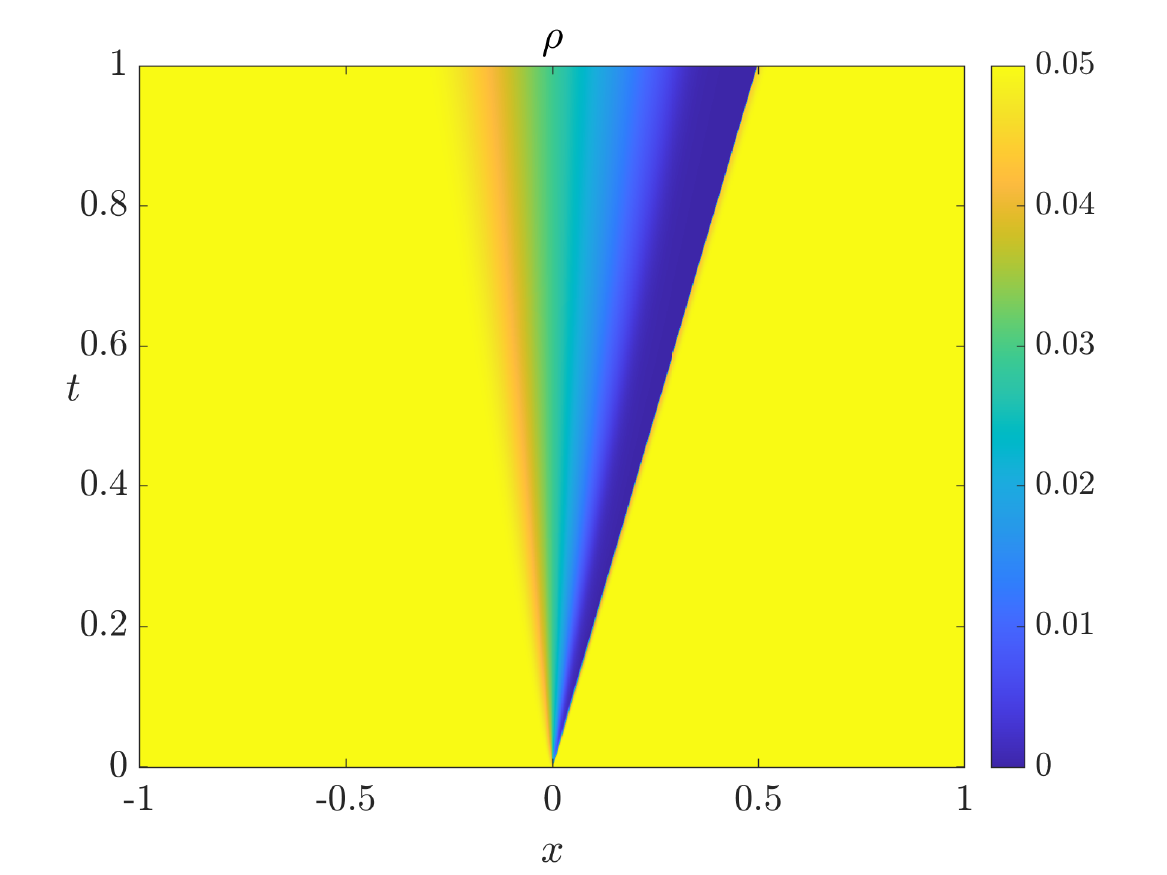}
\includegraphics[scale=0.25, trim={26mm 0mm 6mm 0mm}, clip]{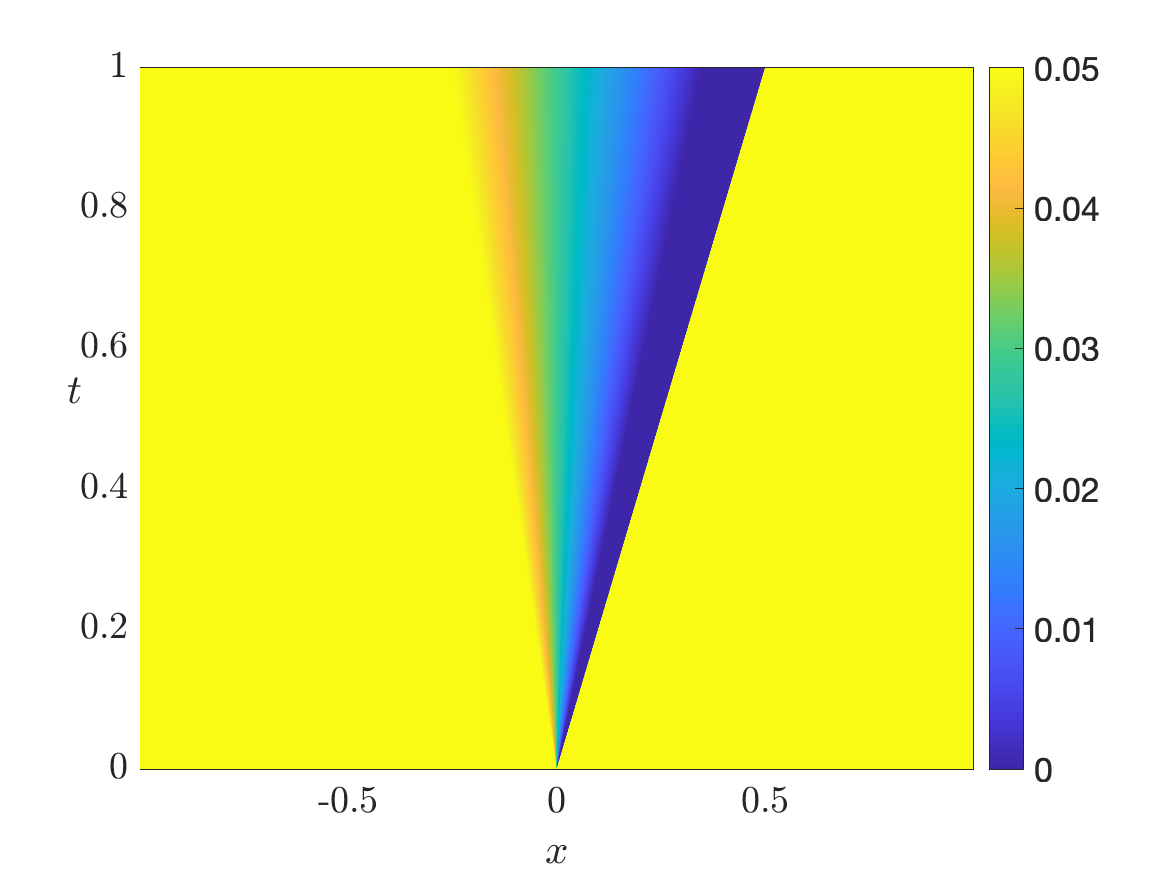}
\caption{$(t,x)$-plots of the density $\rho$ for different values of $\eta.$ 
From left to right, $\eta= 0.1$,\ $\eta=0.01$,\ $\eta=0.001$
followed by the exact solution.
}
\label{fig:ex2}
\end{figure}
\begin{figure}[!ht]
\centering
\includegraphics[scale=0.25, trim={6mm 0mm 31mm 0mm}, clip]{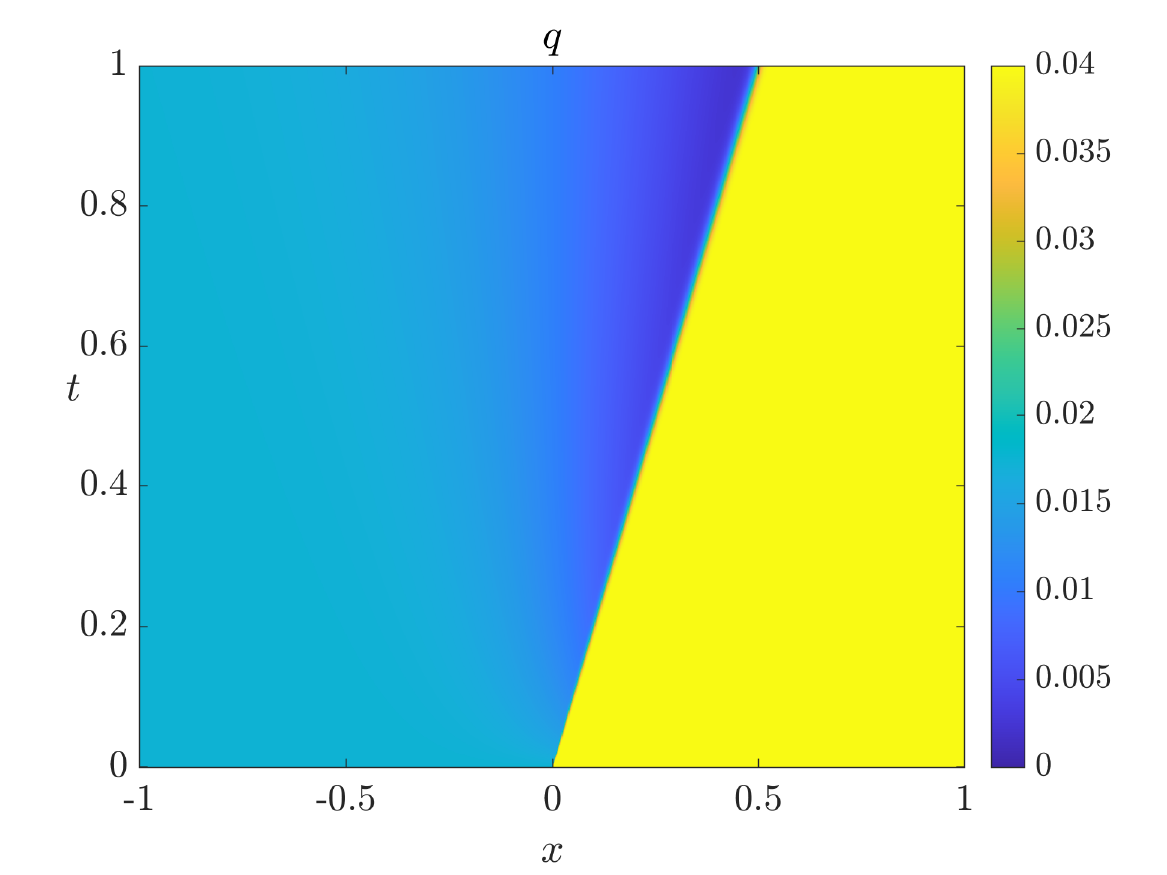}
\includegraphics[scale=0.25, trim={26mm 0mm 31mm 0mm}, clip]{New_Figures/eta1e-1_fast_q.pdf}
\includegraphics[scale=0.25, trim={26mm 0mm 31mm 0mm}, clip]{New_Figures/eta1e-1_fast_q.pdf}
\includegraphics[scale=0.25, trim={26mm 0mm 6mm 0mm}, clip]{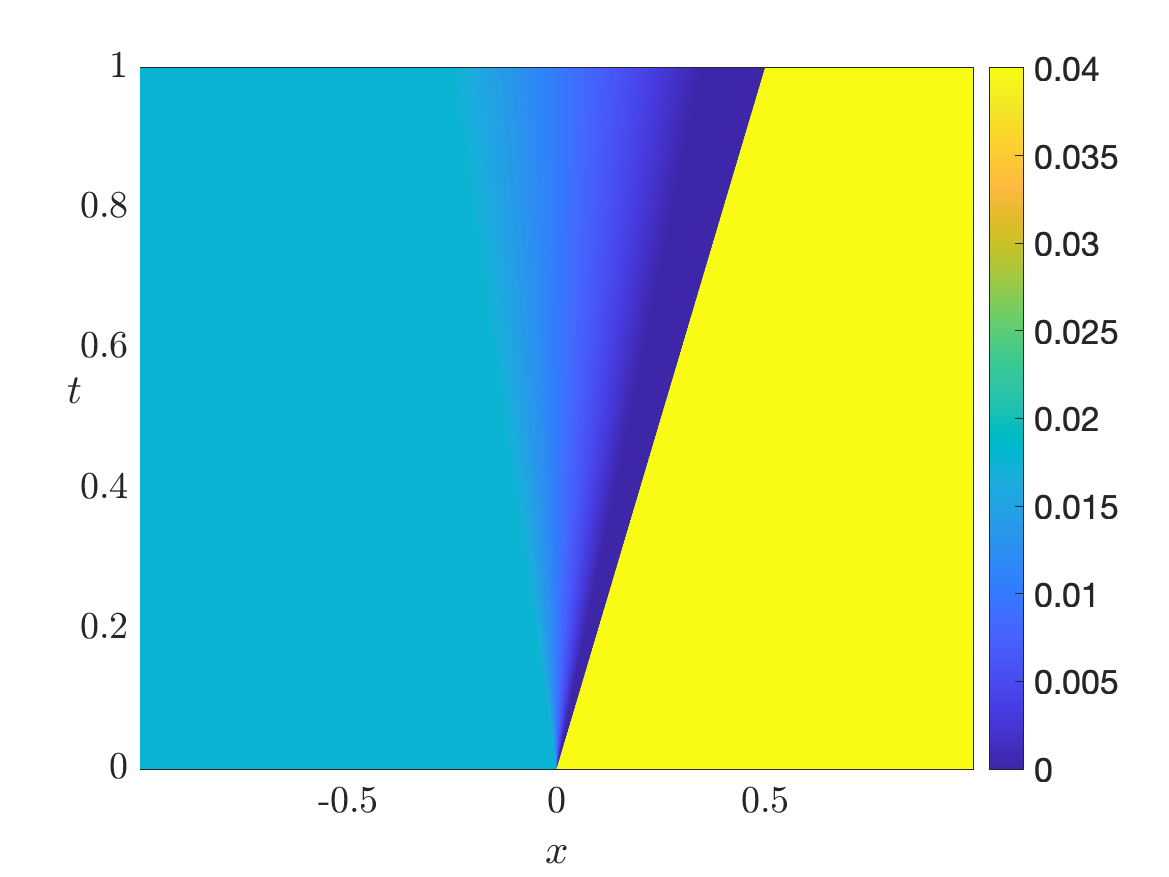}
\caption{$(t,x)$-plots of the momentum $q$ for different values of $\eta.$ 
From left to right, $\eta= 0.1$,\ $\eta=0.01$,\ $\eta=0.001$
followed by the exact solution.
}
\label{fig:ex3}
\end{figure}

For the second simulation (\crefrange{fig:ex4}{fig:ex7}), 
let us consider the same velocity function \cref{eq:arz_vel} and 
\begin{equation}
\label{e:si_C_in}
\begin{aligned}
\rho_0(x)&=0.05
&
q_0(x)&=\begin{cases}
    0.04&\hbox{if } x<0,\\
     0.0175
    &\hbox{if } x\geq0,
\end{cases}\\
v_0(x)&=\begin{cases}
    0.5&\hbox{if } x<0,\\
    0.05&\hbox{if } x\geq0,
\end{cases}&
\omega_0(x)&=\begin{cases}
    0.8&\hbox{if } x<0,\\
    0.35&\hbox{if } x\geq0.
\end{cases}
\end{aligned}
\end{equation}
In this case, starting from a constant initial density and a decreasing Riemann-type initial Lagrangian marker, we observe the formation of a shock and a contact discontinuity, and the solution remains bounded in an invariant region.
\begin{figure}[!ht]
\centering
\includegraphics[scale=0.3, trim={0mm 0mm 0mm 0mm}, clip]{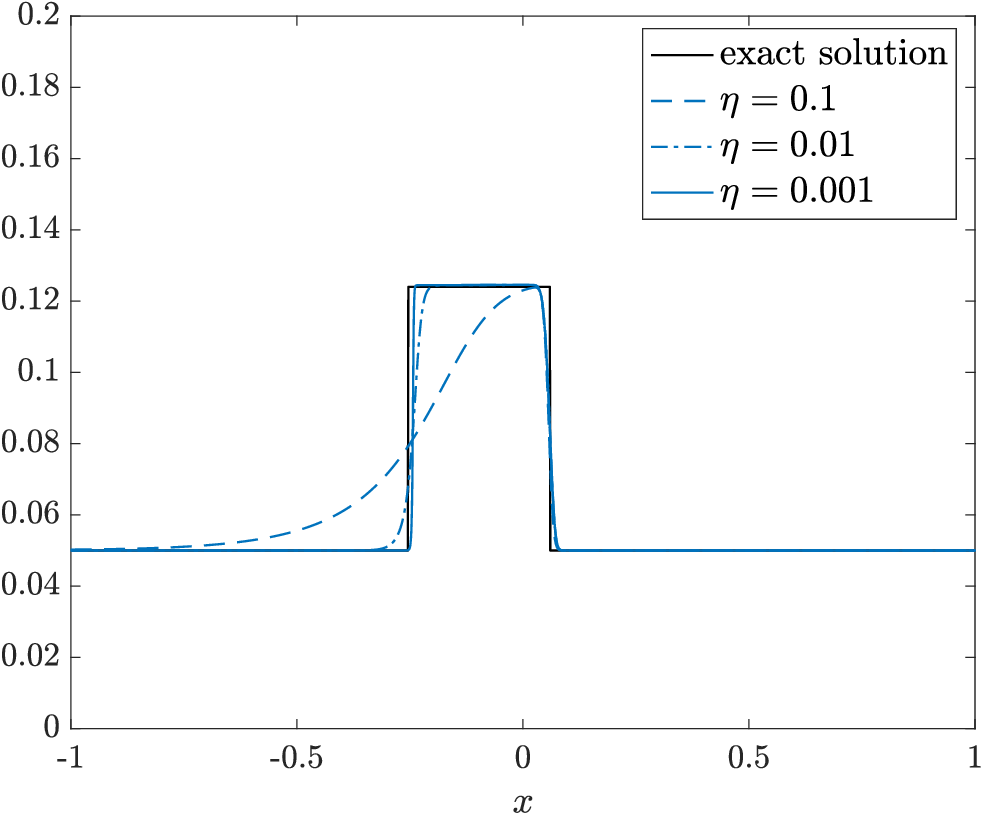}
\hfil
\includegraphics[scale=0.3, trim={0mm 0mm 0mm 0mm}, clip]{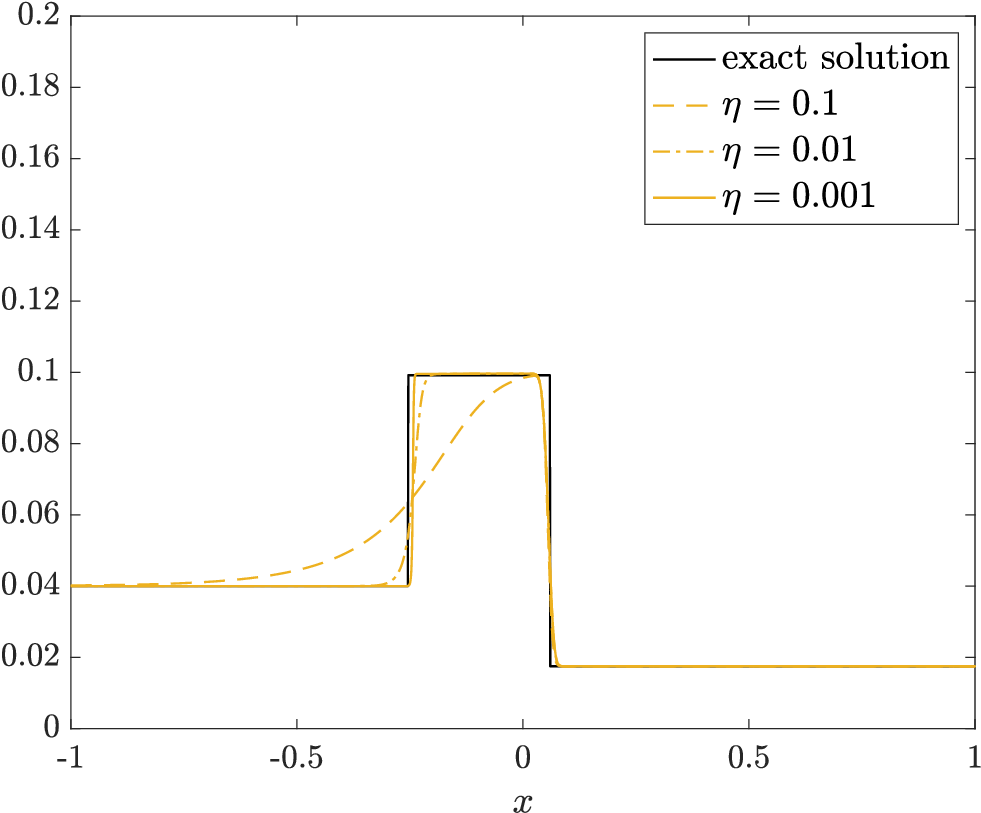}
\hfil
\includegraphics[scale=0.3, trim={0mm 0mm 0mm 0mm}, clip]{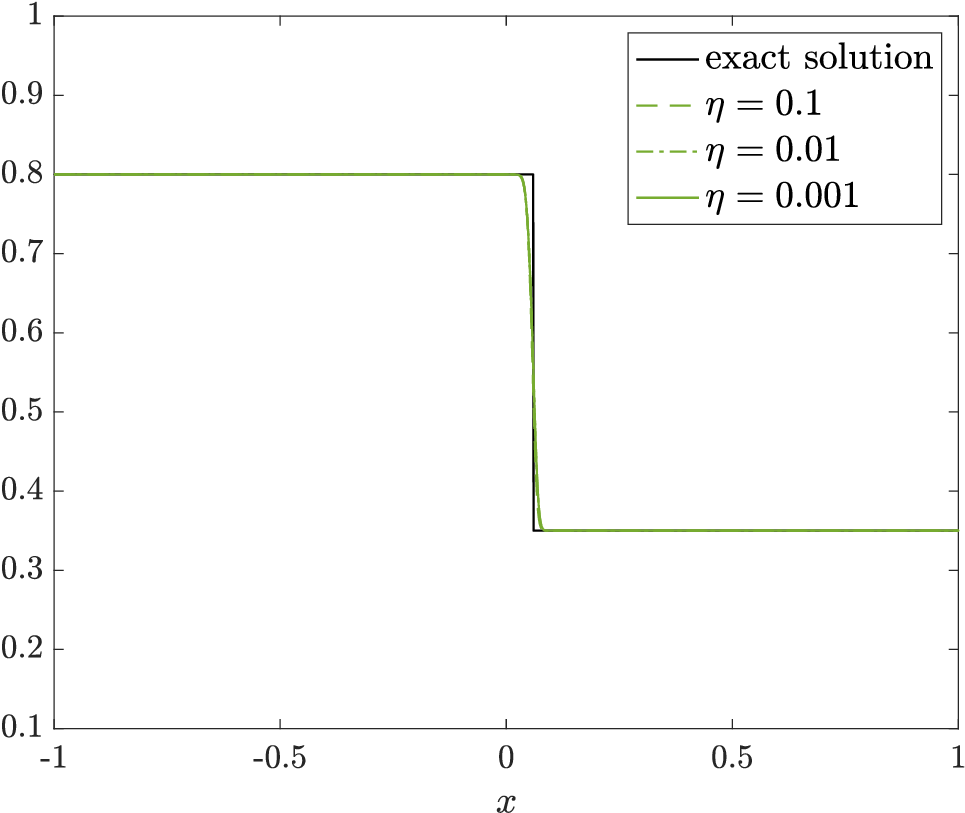}
\caption{Numerical simulations for the initial data in \cref{e:si_C_in} and velocity function \cref{eq:arz_vel}. From the left: density $\rho$, momentum $q$, and  $\omega$ for different values of $\eta.$
}
\label{fig:ex4}
\end{figure}
\begin{figure}[!ht]
\centering
\includegraphics[scale=0.25, trim={6mm 0mm 31mm 0mm}, clip]{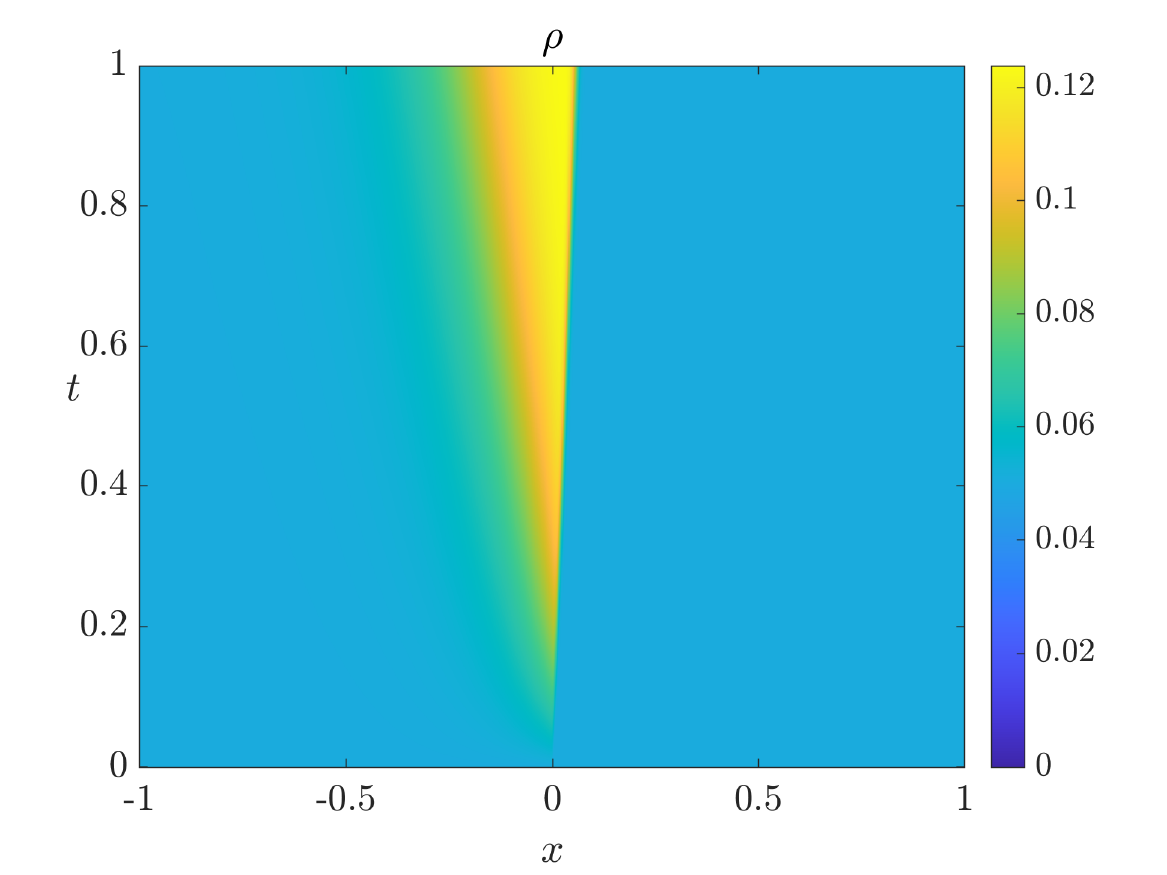}
\includegraphics[scale=0.25, trim={26mm 0mm 31mm 0mm}, clip]{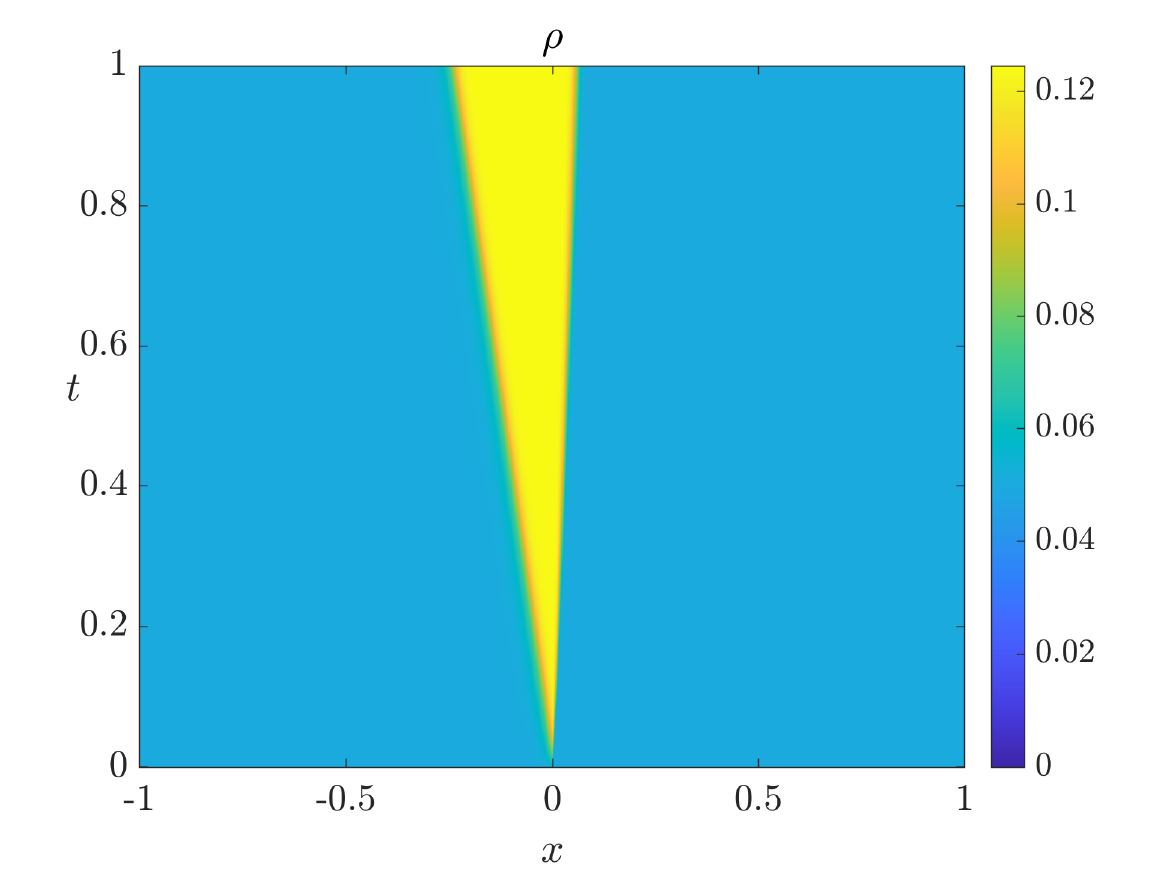}
\includegraphics[scale=0.25, trim={26mm 0mm 31mm 0mm}, clip]{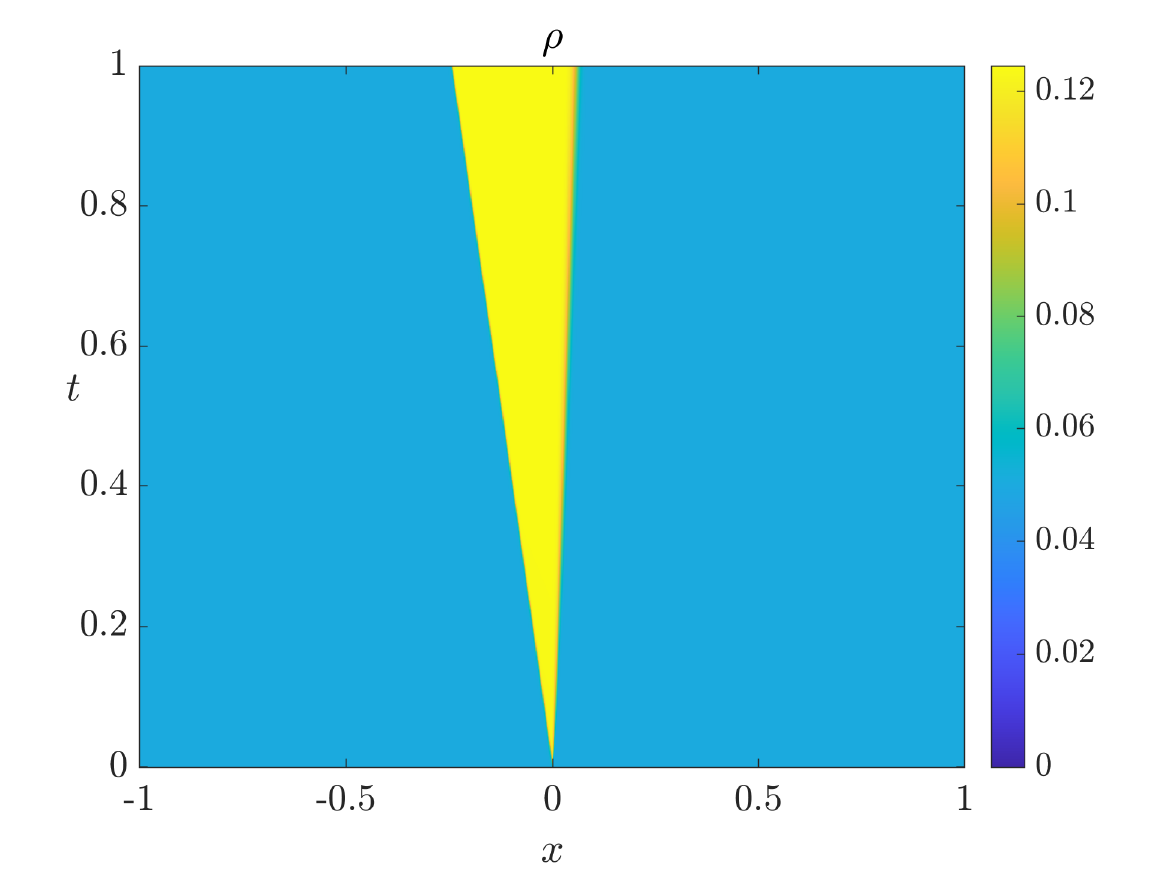}
\includegraphics[scale=0.25, trim={26mm 0mm 6mm 0mm}, clip]{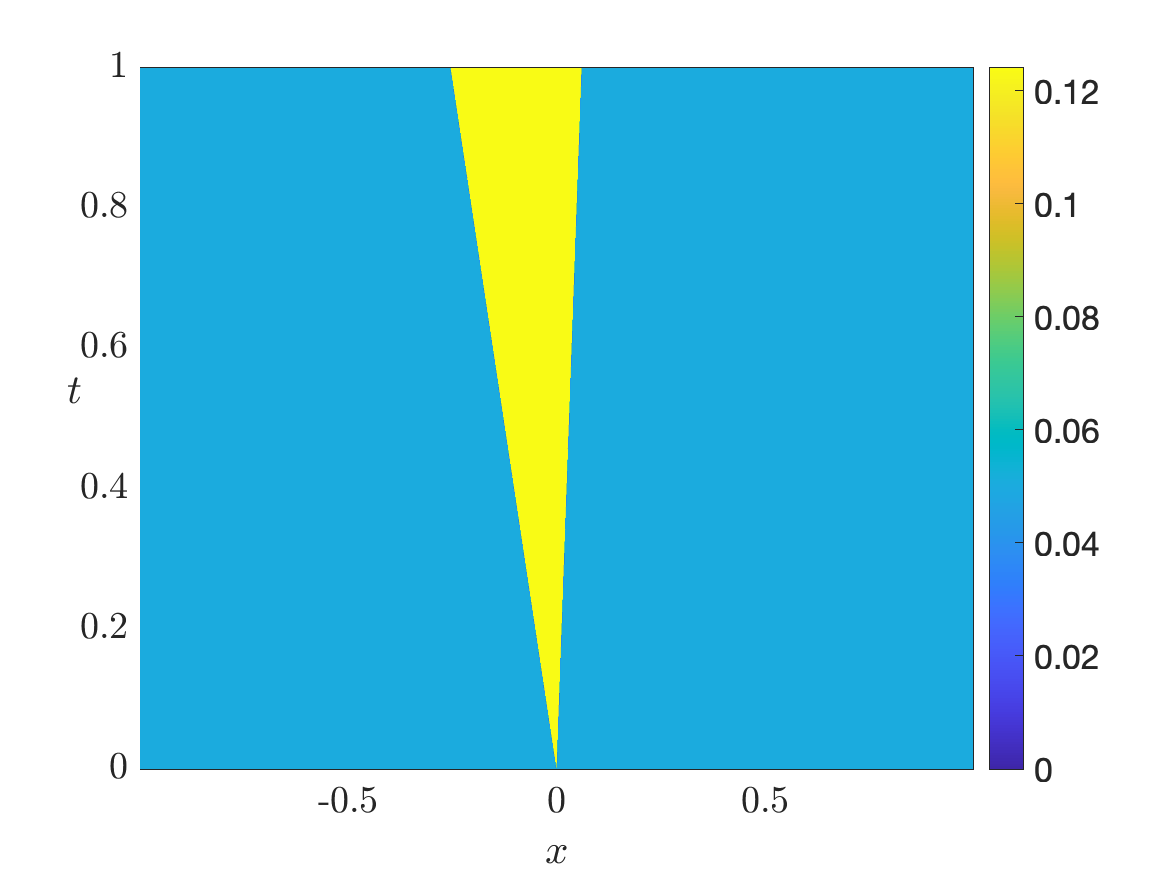}
\caption{$(t,x)$-plots of the density $\rho$ for different values of $\eta.$ 
From left to right, $\eta= 0.1$,\ $\eta=0.01$,\ $\eta=0.001$
followed by the exact solution. 
}
\label{fig:ex5}
\end{figure}
\begin{figure}[!ht]
\centering
\includegraphics[scale=0.25, trim={6mm 0mm 31mm 0mm}, clip]{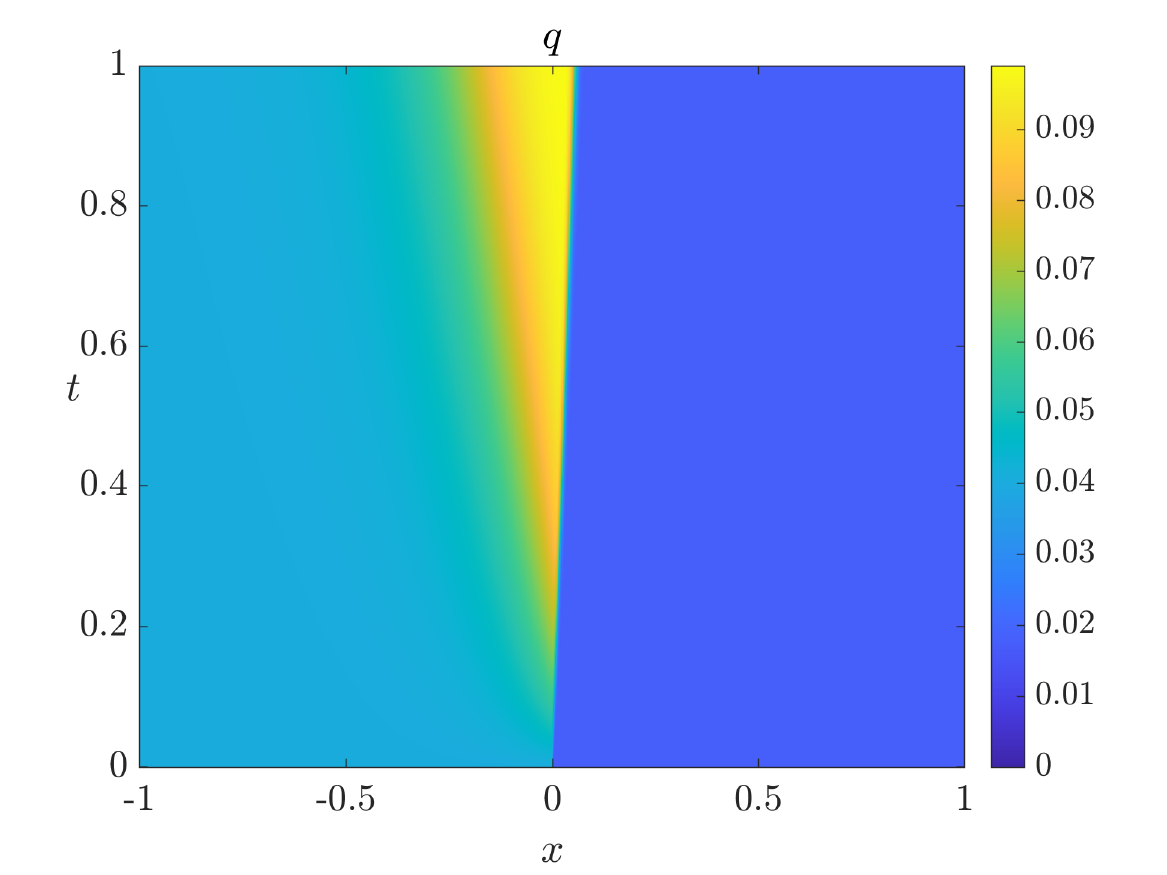}
\includegraphics[scale=0.25, trim={26mm 0mm 31mm 0mm}, clip]{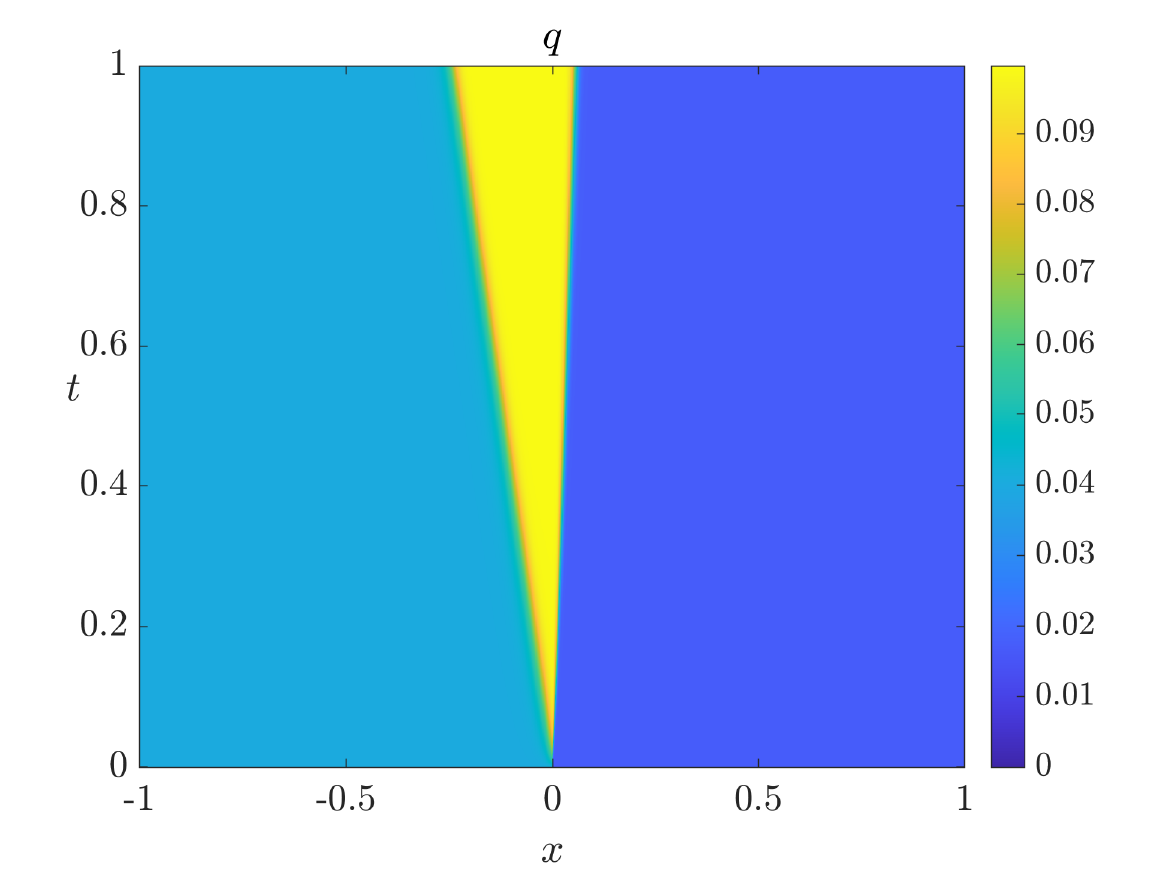}
\includegraphics[scale=0.25, trim={26mm 0mm 31mm 0mm}, clip]{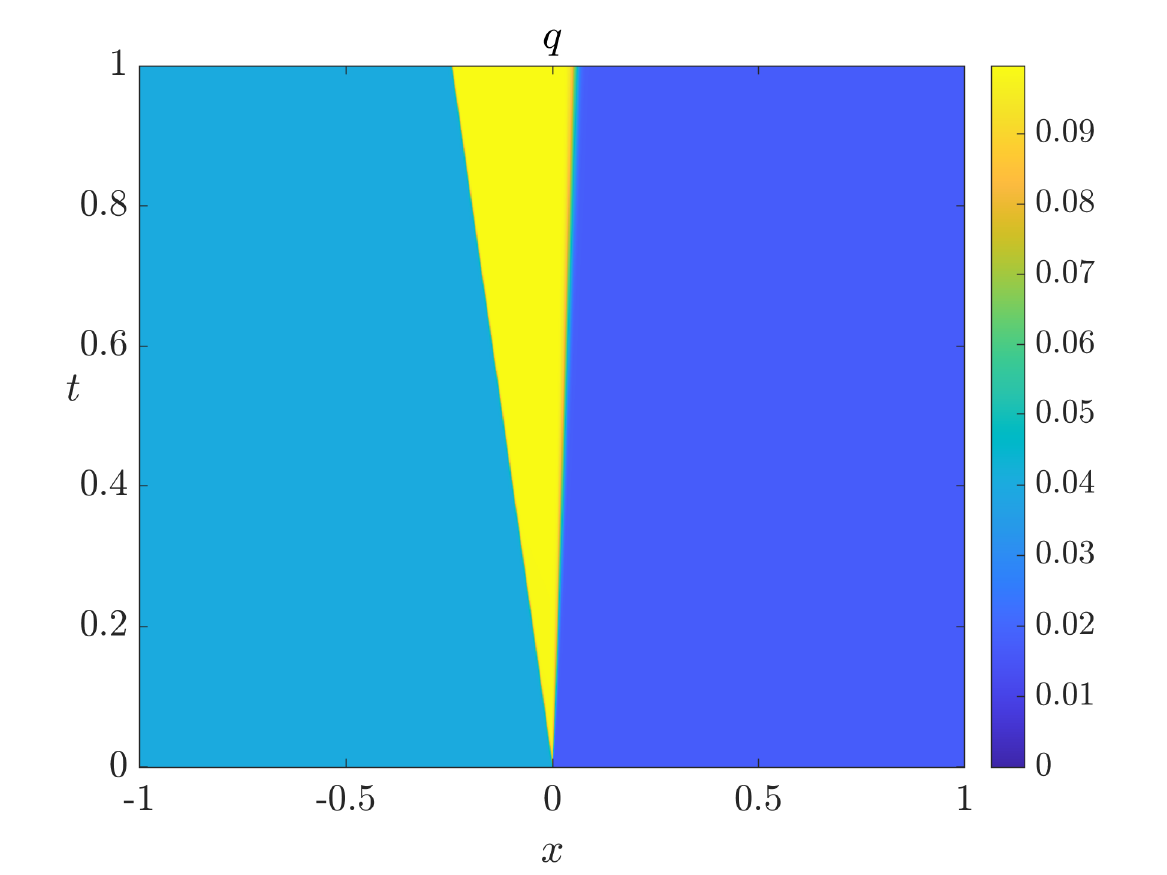}
\includegraphics[scale=0.25, trim={26mm 0mm 6mm 0mm}, clip]{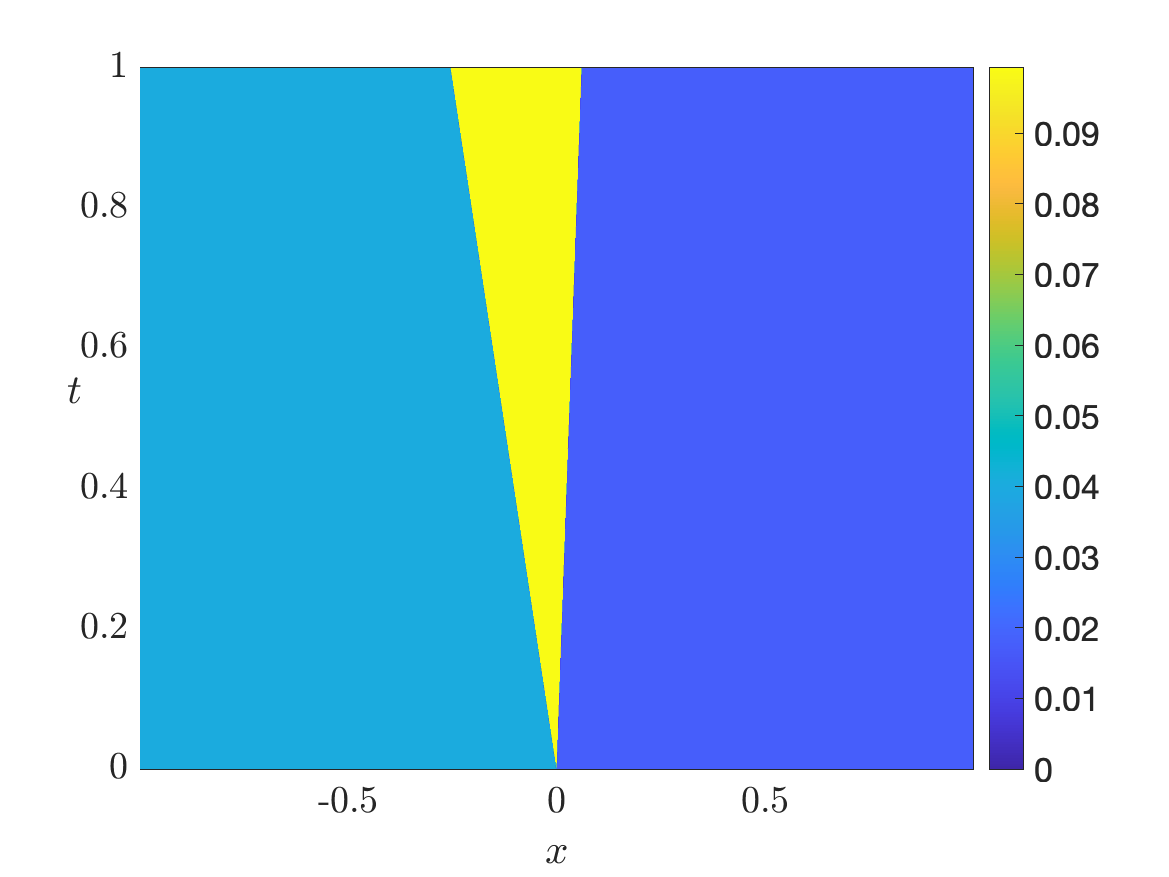}
\caption{$(t,x)$-plots of the momentum $q$ for different values of $\eta.$ 
From left to right, $\eta= 0.1$,\ $\eta=0.01$,\ $\eta=0.001$
followed by the exact solution. 
}
\label{fig:ex6}
\end{figure}
\begin{figure}[!ht]
\centering
\includegraphics[scale=0.25, trim={6mm 0mm 31mm 0mm}, clip]{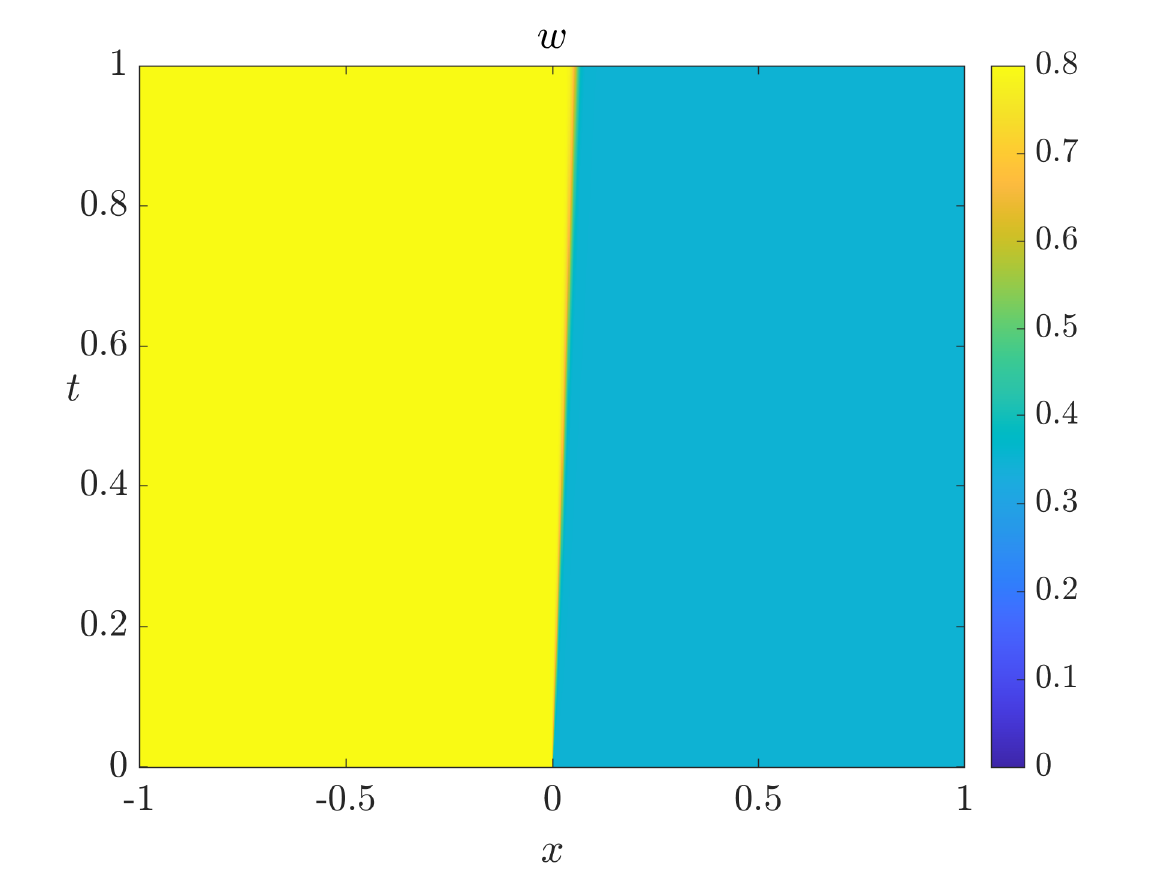}
\includegraphics[scale=0.25, trim={26mm 0mm 31mm 0mm}, clip]{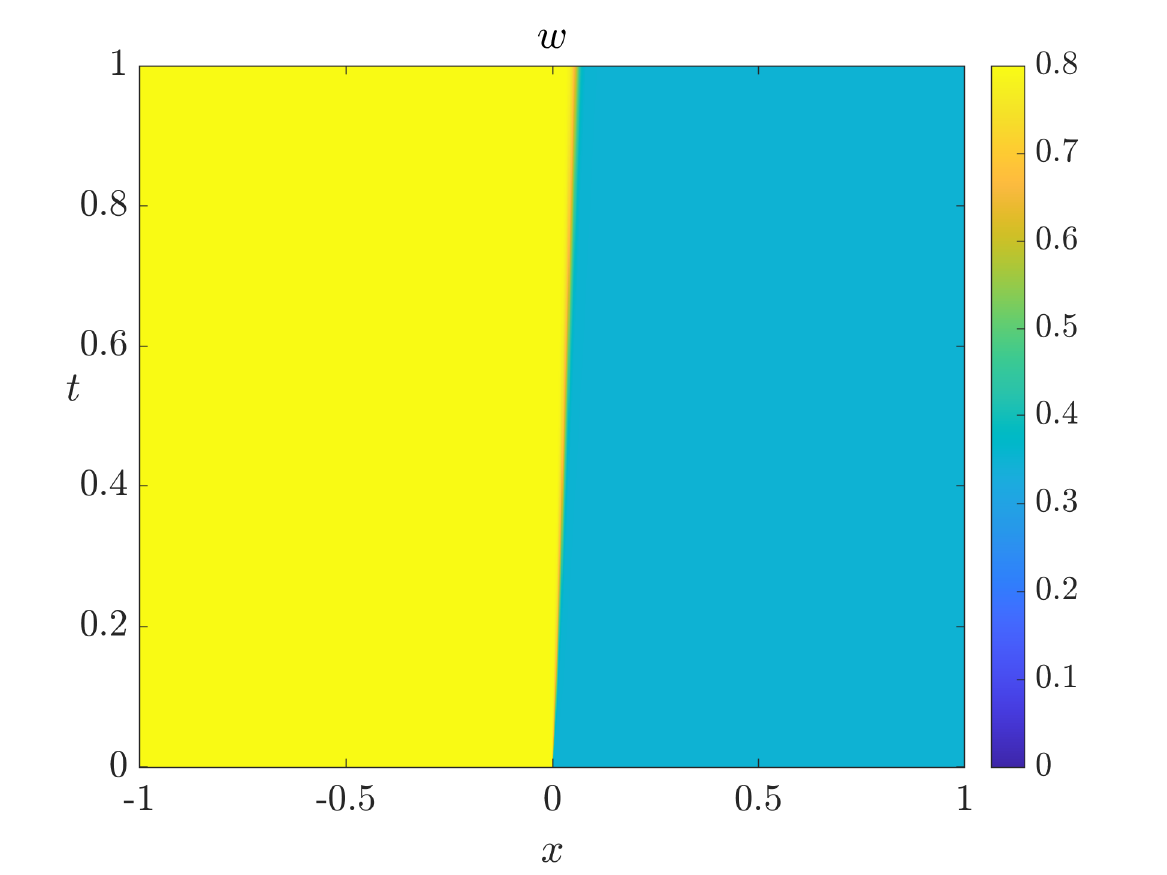}
\includegraphics[scale=0.25, trim={26mm 0mm 31mm 0mm}, clip]{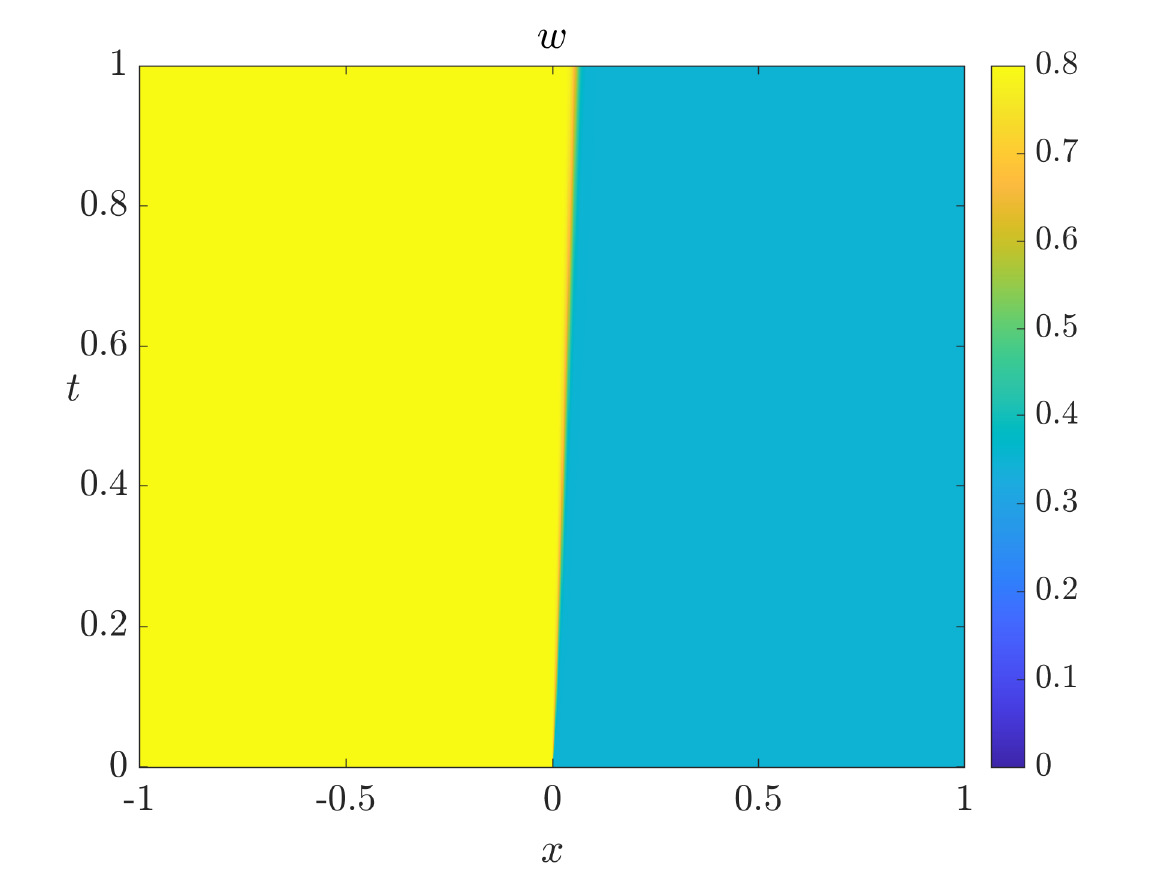}
\includegraphics[scale=0.25, trim={26mm 0mm 6mm 0mm}, clip]{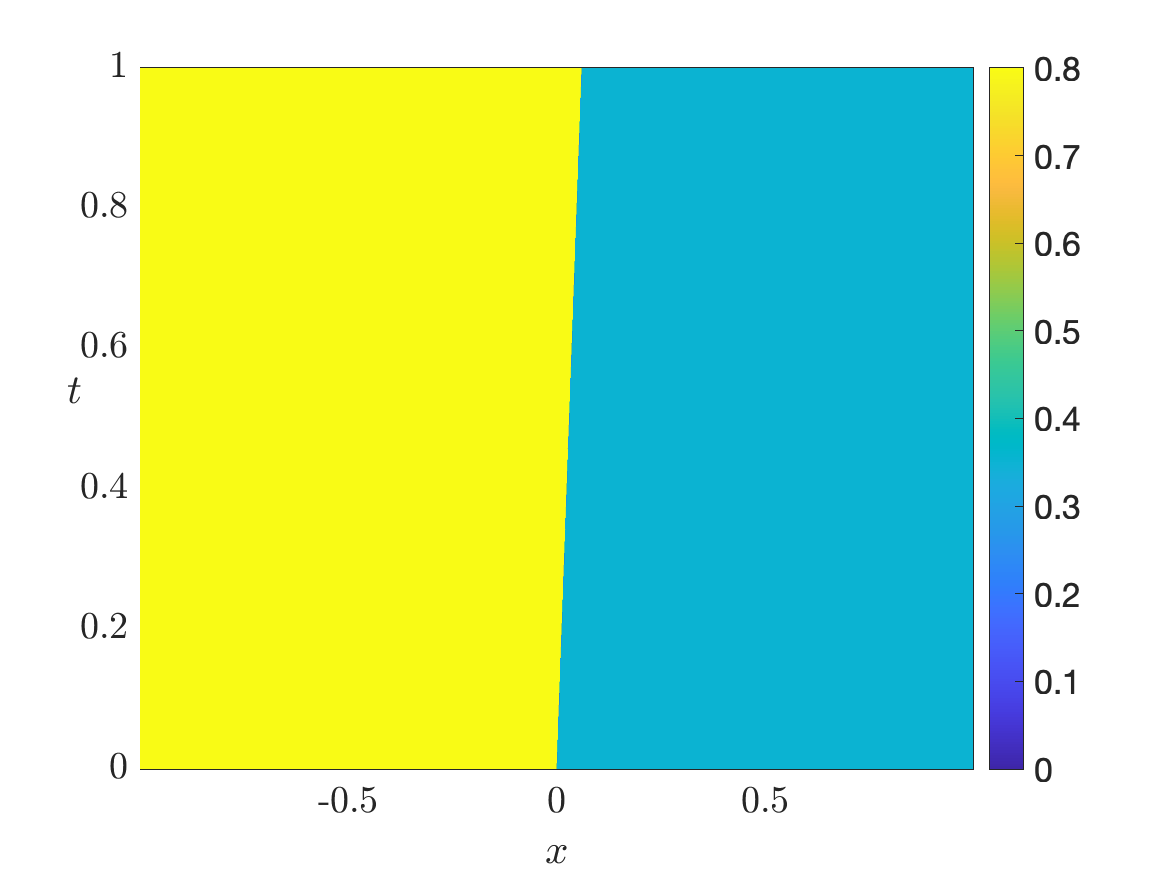}
\caption{$(t,x)$-plots of the Lagrangian marker $\omega$ for different values of $\eta.$ 
From left to right, $\eta= 0.1$,\ $\eta=0.01$,\ $\eta=0.001$
followed by the exact solution. 
}
\label{fig:ex7}
\end{figure}

For the last simulation (see \cref{fig:ex3bis,fig:ex8,fig:ex9}), 
we consider 
\begin{equation}\label{sim:velocity_2}
    V(\rho,\omega)=   \tfrac{\omega}{1+\rho},
\end{equation}
which satisfies the assumptions \cref{S5-V1} and \cref{S5-V3} and the first condition in \cref{S5-V2}, although $V$ does not vanish at any value of $\rho$.
Let us set the initial data \cref{e:initialdatum_1} as follows:
\begin{equation*}
\rho_0(x)=\bar \rho\coloneqq 0.05\quad \forall\, x\in\R\,,\qquad 
\omega_0(x)=\begin{cases}
   \omega_\ell \coloneqq 0.35&\hbox{if } x<0,\\
   \omega_r \coloneqq 0.8&\hbox{if } x\geq0\,.\qquad 
\end{cases}
\end{equation*}
Let ${v}_\ell$ and ${v}_r$ be the corresponding values at $t=0$ for $x<0$ and $x>0$ respectively.
The solution consists of a rarefaction up to $\rho=0$ followed by a contact discontinuity with speed ${v}_r$. Here, $\lambda_1(\rho,\omega) = V+\rho\partial_\rho V = \frac{\omega}{(1+\rho)^2}$.
With the notation
\[
s_1 = \lambda_1(\bar \rho,\omega_\ell) = \tfrac{\omega_\ell}{(1+\bar \rho)^2}\,,\qquad s_2 = \lambda_1(0,\omega_\ell)=\omega_\ell \,,\qquad s_3 = v_r > \omega_\ell\,,
\]
the solution (the $\rho$ component) is given by
\[
\rho(t,x) = \begin{cases}
    \bar \rho & x\le s_1 t\\
    \left(\omega_{\ell} \frac{t}{x}\right)^{-1/2} -1 & s_1t <x\le s_2 t\\
0 & s_2t <x < s_3 t\\
\bar \rho & x>s_3 t\,.
\end{cases}
\]
In this case, as in the first simulation, we can see that the density reaches the vacuum in finite time and is confined between zero and the maximum of the initial data. 
\begin{figure}[!ht]
\centering
\includegraphics[width=0.3\textwidth]{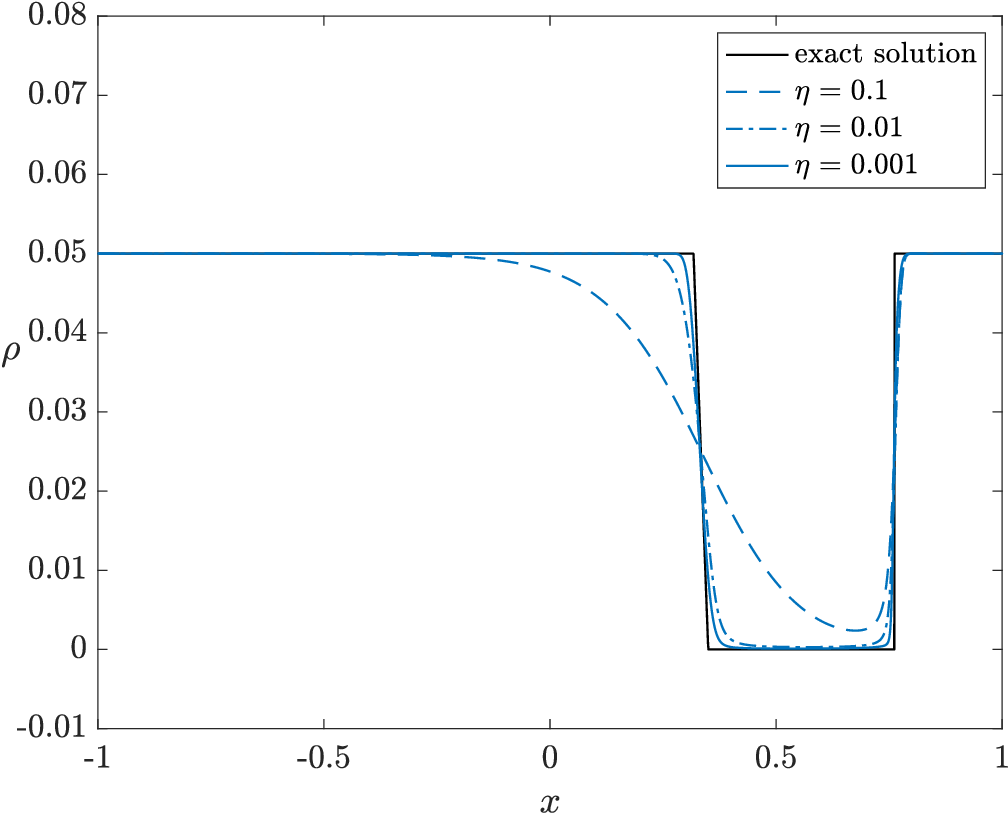}
\hfil
\includegraphics[width=0.3\textwidth]{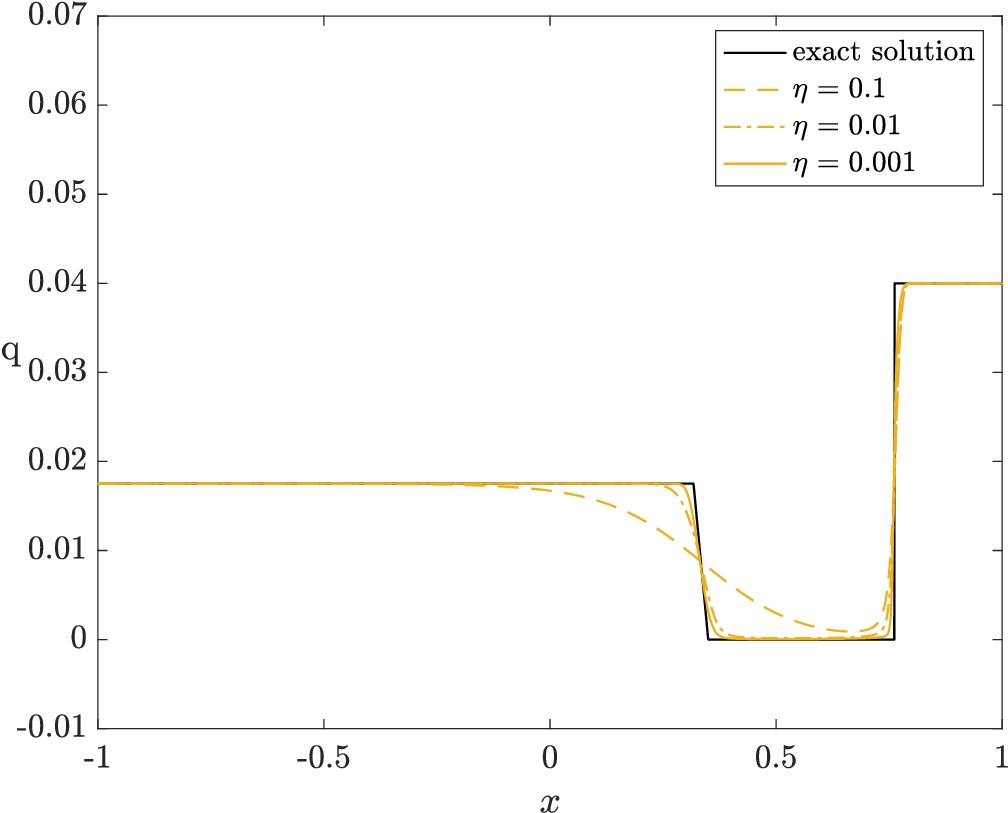}
\caption{Numerical simulations for the initial data in \cref{e:initialdatum_1} and velocity \cref{sim:velocity_2}. Left: density $\rho$ for different values of $\eta.$ Right: $q$ for different values of $\eta$. 
} 
\label{fig:ex3bis}
\end{figure}
\begin{figure}[!ht]
\centering
\includegraphics[scale=0.25, trim={6mm 0mm 31mm 0mm}, clip]{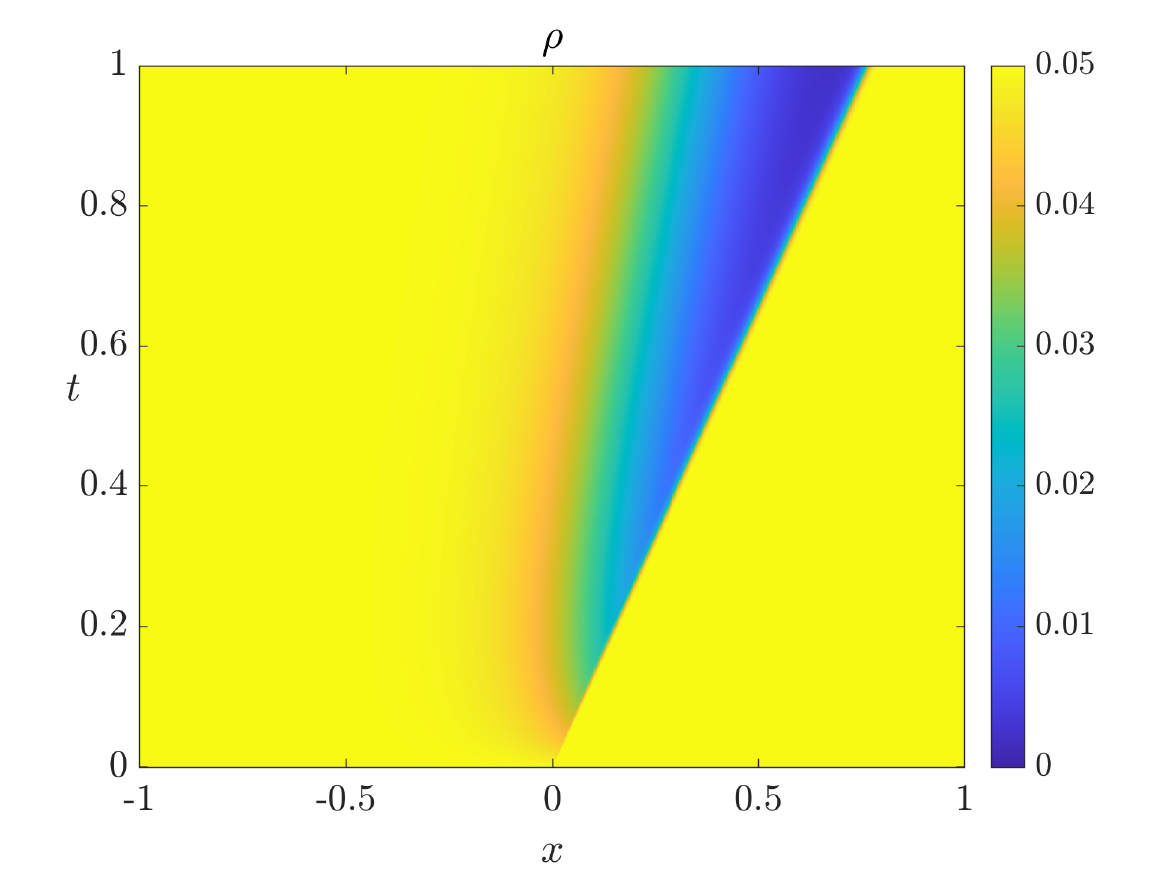}
\includegraphics[scale=0.25, trim={26mm 0mm 31mm 0mm}, clip]{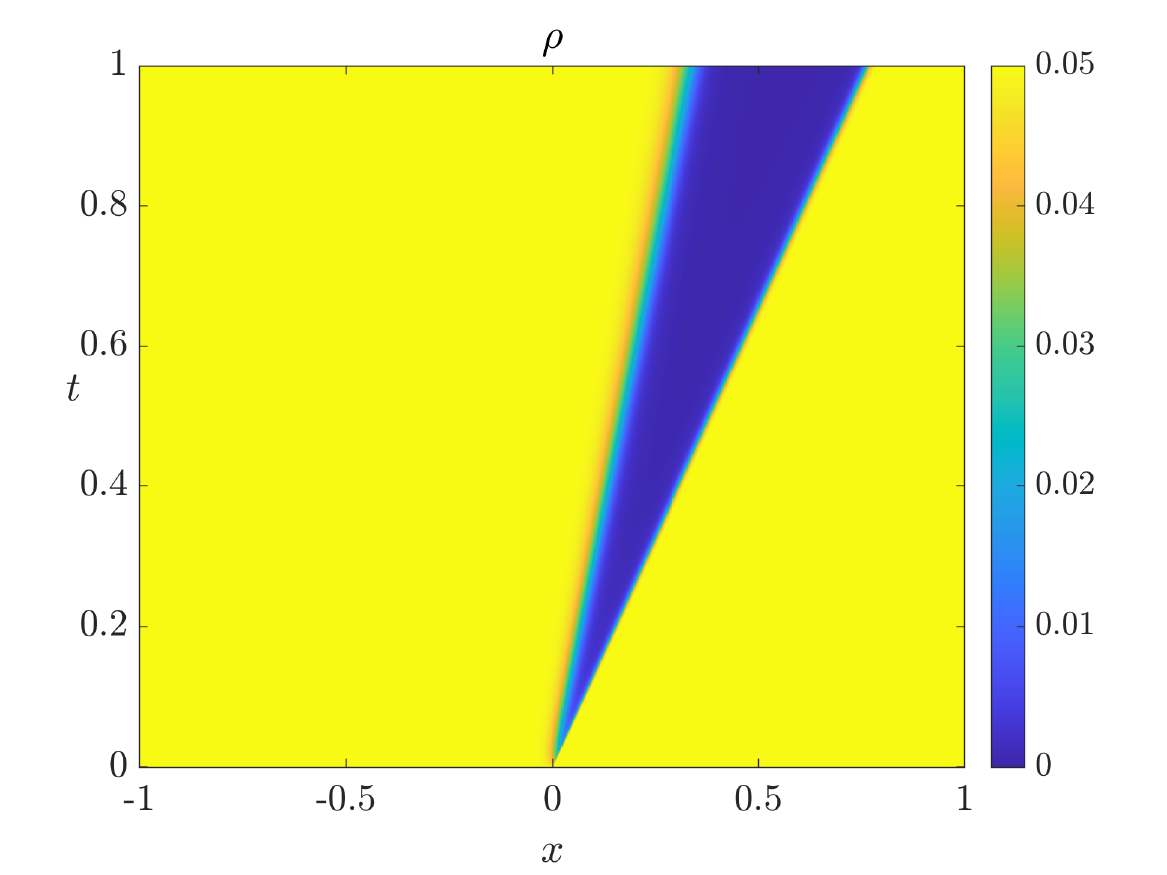}
\includegraphics[scale=0.25, trim={26mm 0mm 31mm 0mm}, clip]{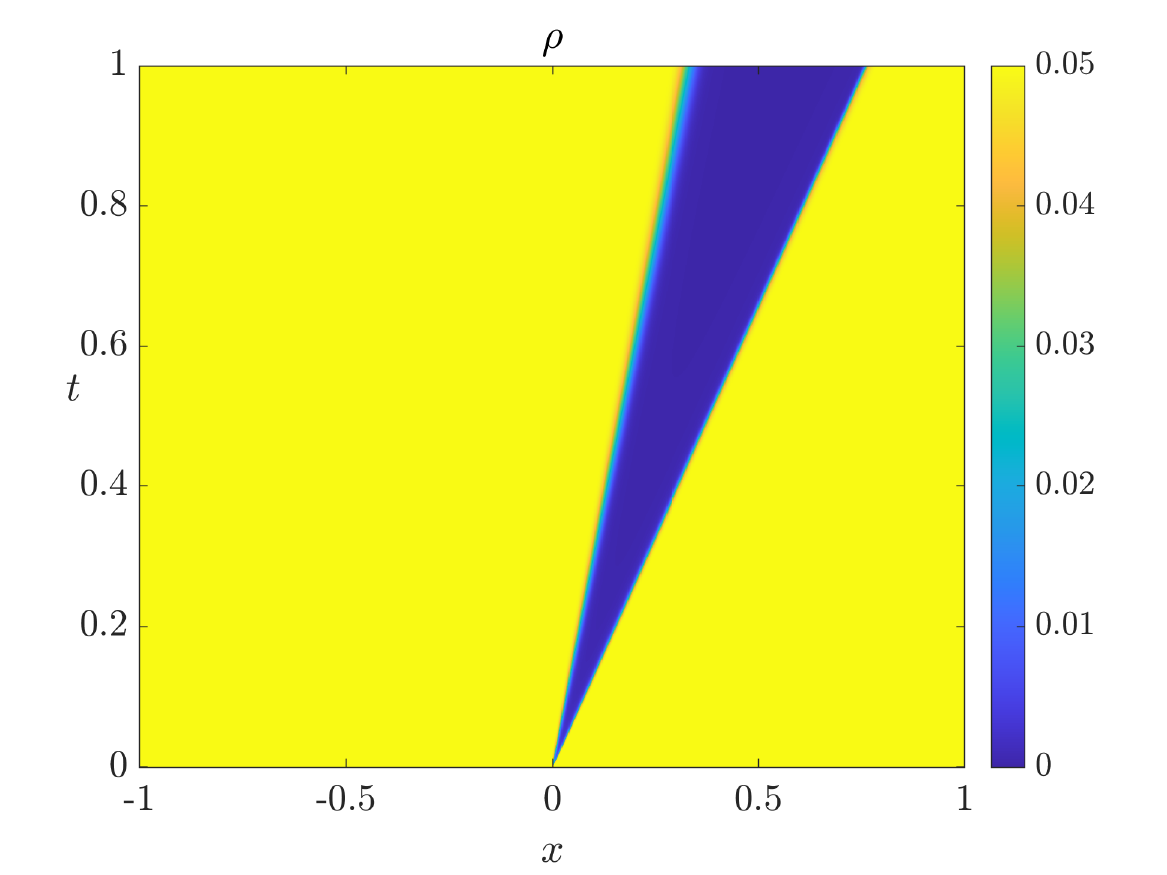}
\includegraphics[scale=0.25, trim={26mm 0mm 6mm 0mm}, clip]{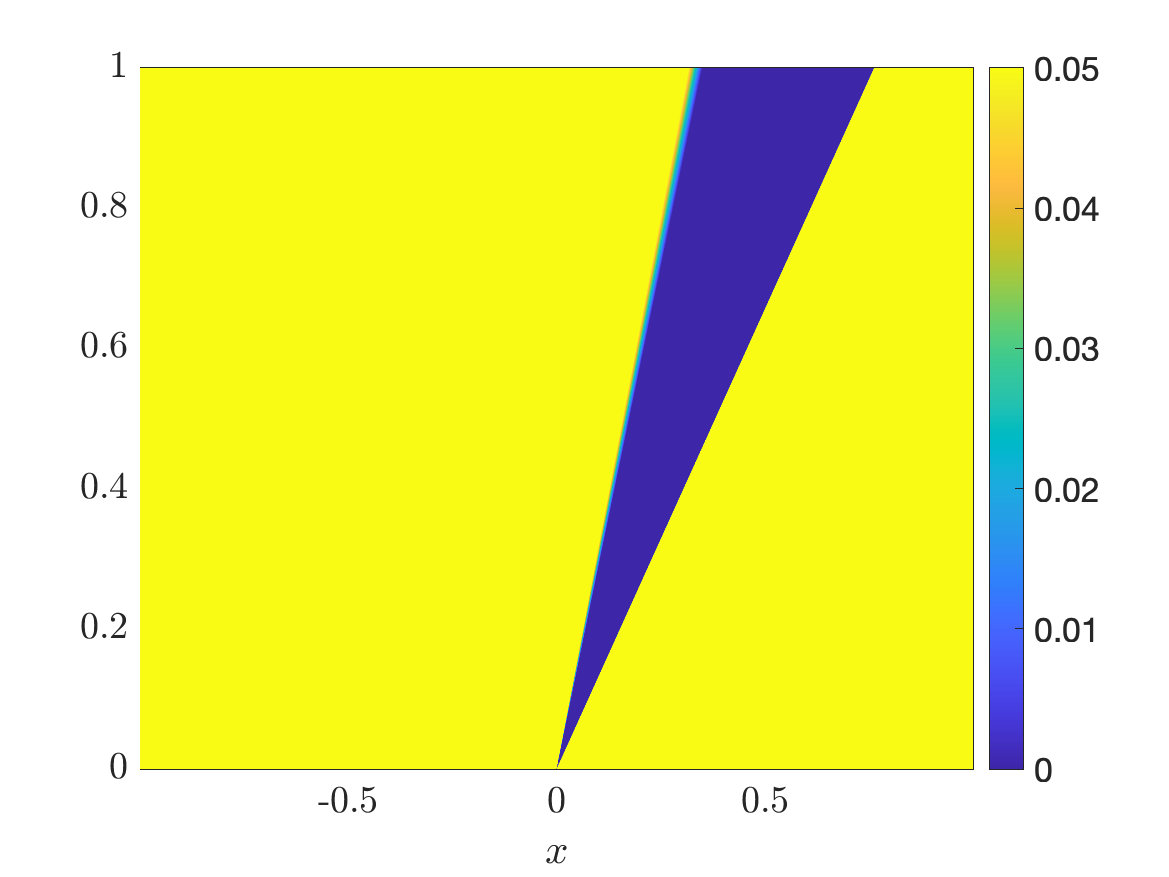}
\caption{$(t,x)$-plots of the density $\rho$ for different values of $\eta.$ 
From left to right, $\eta= 0.1$,\ $\eta=0.01$,\ $\eta=0.001$
and the exact solution. 
\label{fig:ex8}}
\end{figure}
\begin{figure}[!ht]
\centering
\includegraphics[scale=0.25, trim={6mm 0mm 31mm 0mm}, clip]{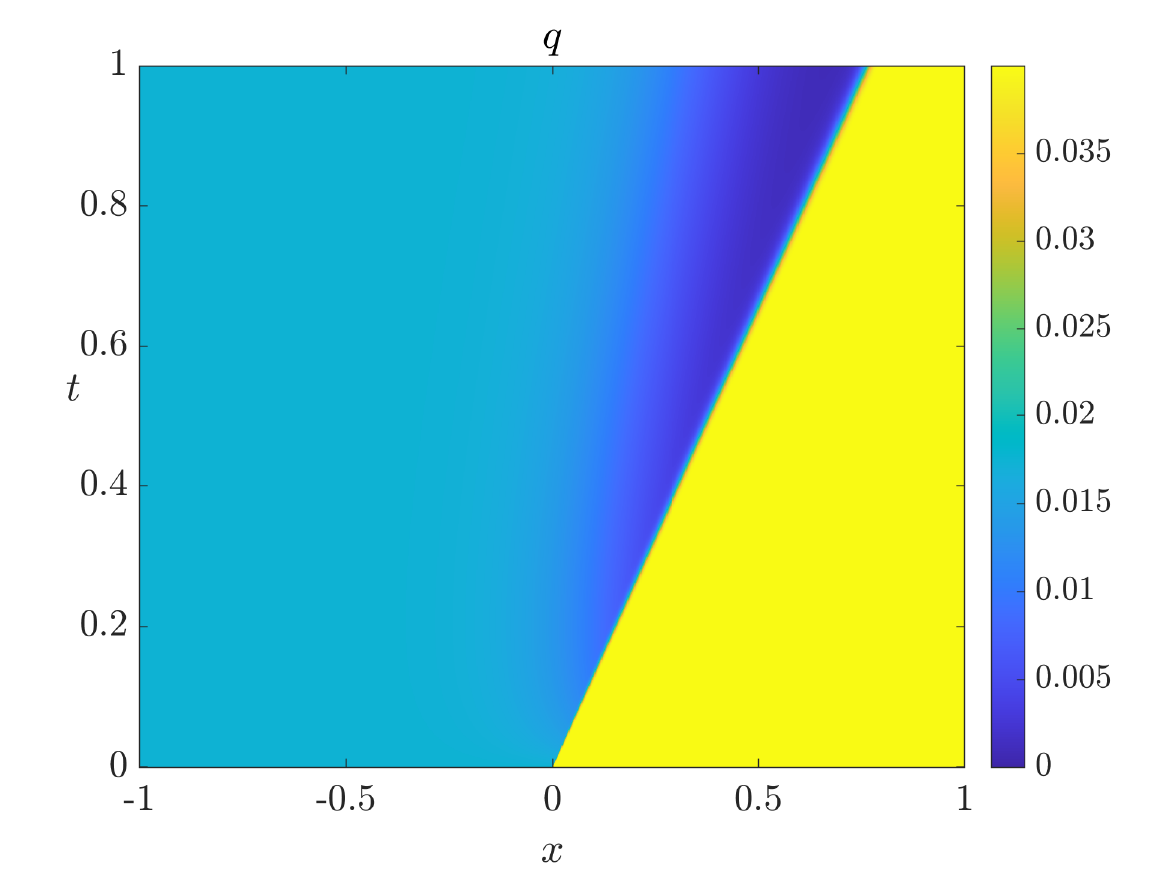}
\includegraphics[scale=0.25, trim={26mm 0mm 31mm 0mm}, clip]{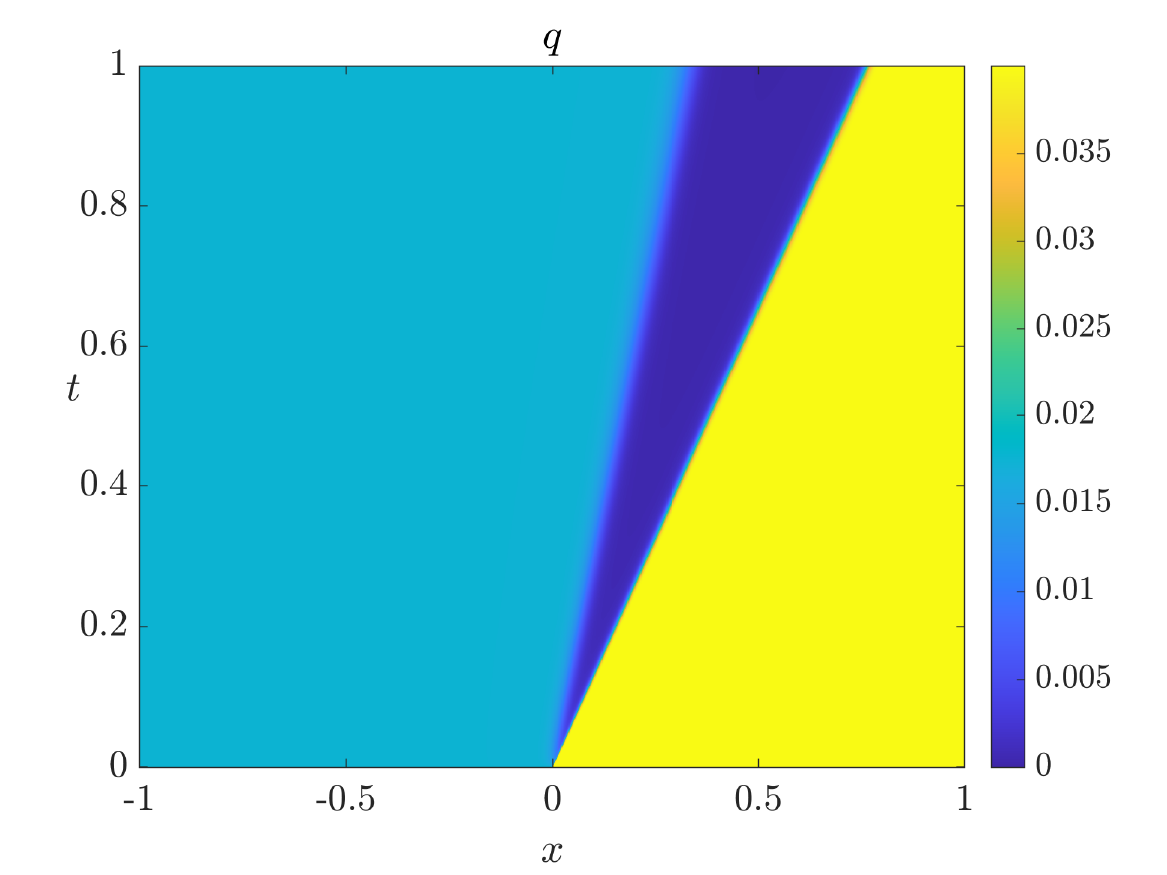}
\includegraphics[scale=0.25, trim={26mm 0mm 31mm 0mm}, clip]{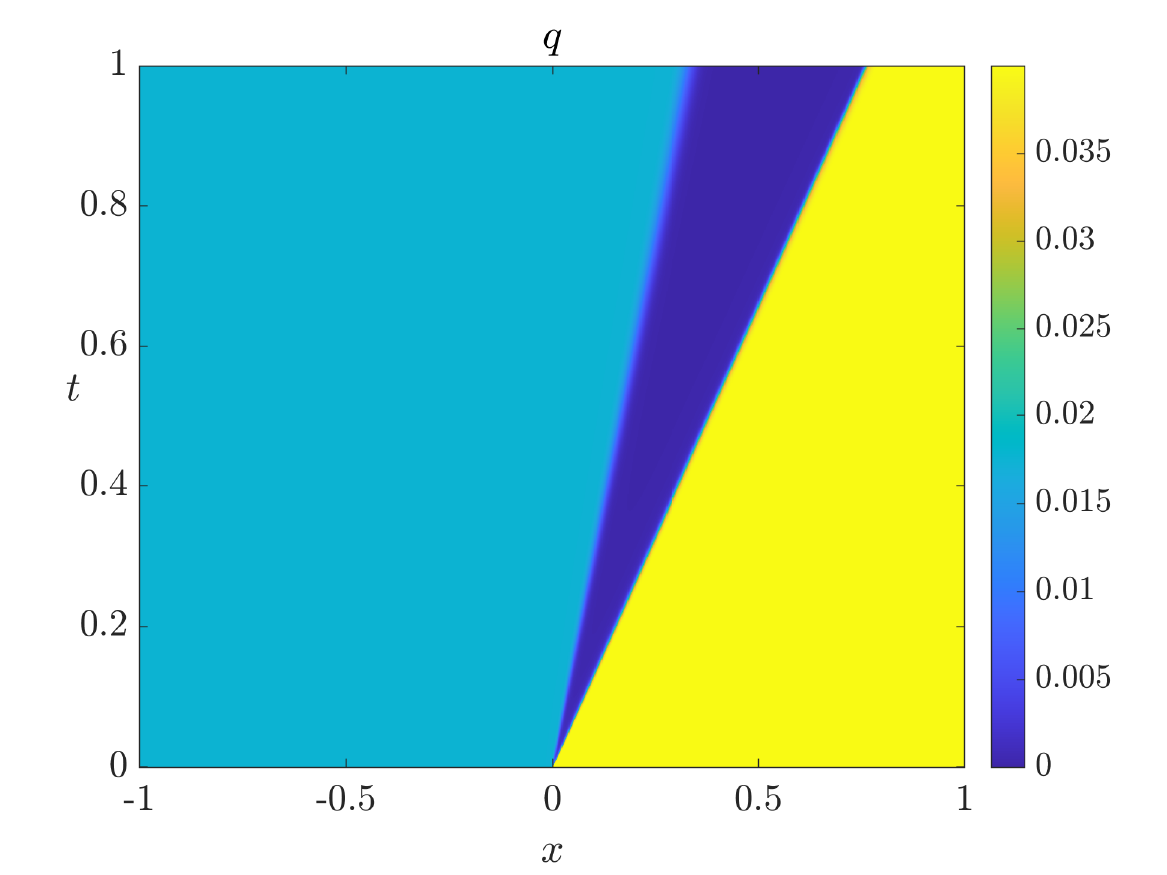}
\includegraphics[scale=0.25, trim={26mm 0mm 6mm 0mm}, clip]{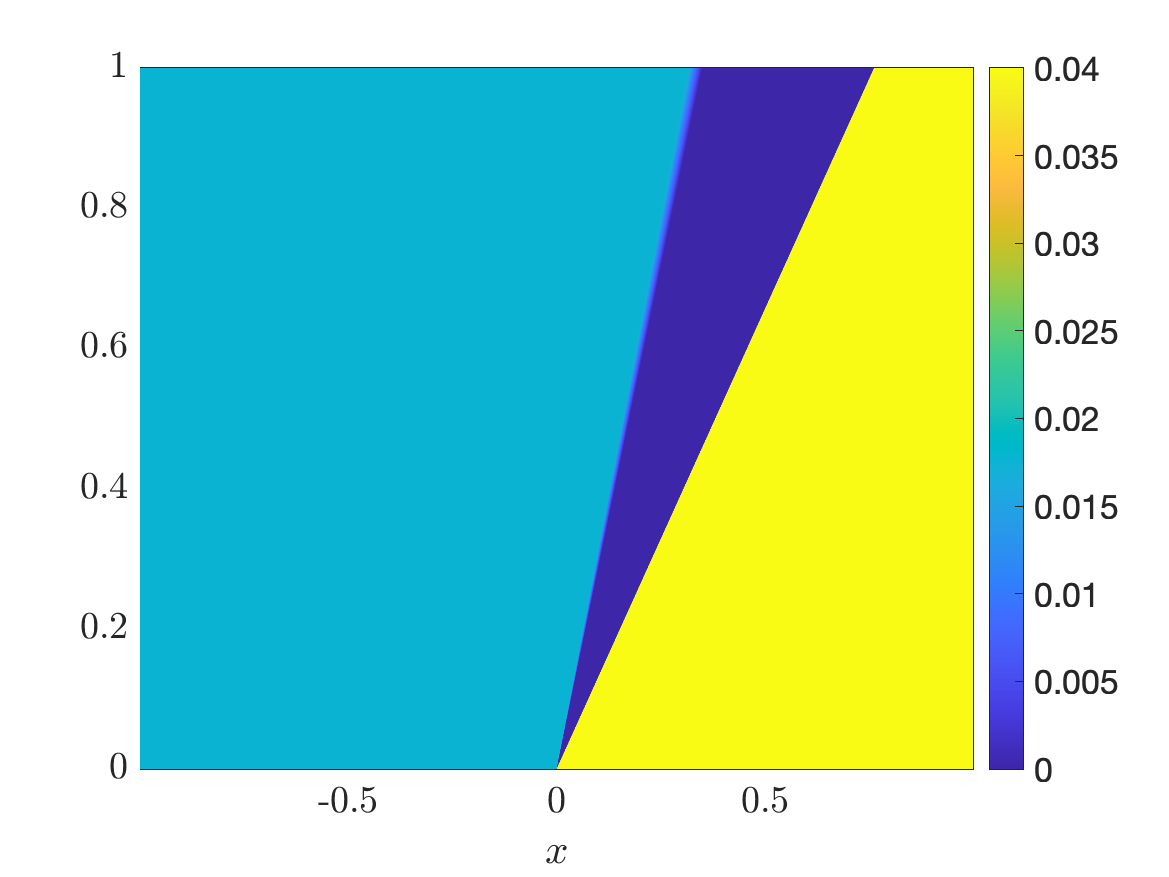}
\caption{$(t,x)$-plots of the momentum $q.$ 
From left to right, $\eta= 0.1$,\ $\eta=0.01$,\ $\eta=0.001$
and the exact solution. 
}
\label{fig:ex9}
\end{figure}

\section{Open problems}\label{sec:conclusion}
The following topics would be interesting to consider in future work:
\begin{itemize}
    \item The singular limit convergence result requires \cref{eq:requirement_TV_diminishing}, which restricts the choice of initial data and velocity function. Thus, a natural question is whether this condition can be circumvented.
   
    \item The singular limit result has only been proven for the exponential kernel, but particularly in traffic applications, a kernel with finite support might be more reasonable.

    \item We obtain the convergence to a weak solution also for the vacuum case, but at present, no results seem to be available concerning the uniqueness of corresponding entropy solutions in the local case.

   \item Recently, the convergence of the singular limit was shown for scalar conservation laws with sign-restricted derivative of the velocity without a priori established \(TV\) bounds on the nonlocal terms by means of compensated compactness \cite{coclite2025singular}. A future attempt to prove the singular limit convergence might follow the same steps, potentially avoiding \cref{eq:requirement_TV_diminishing}. 
   \item Similarly, taking advantage of the corresponding Hamilton--Jacobi equation, it was demonstrated in \cite{keimer2026nonlocal} that solutions of general scalar nonlocal conservation laws converge to the corresponding local entropy solution, again without taking advantage of \(TV\) estimates. This is not directly applicable to the system case, but as the same nonlocal velocity in this contribution appears in both equations, it might be possible to apply these arguments to the nonlocal velocity directly.
\end{itemize}

\bibliographystyle{elsarticle-harv}

\end{document}

%% file: definitions.tex
\newcommand{\x}{\boldsymbol x}

\newcommand{\e}{\mathrm{e}}

\newcommand{\R}{\mathbb R}
\newcommand{\N}{\mathbb N}

\newcommand{\eps}{\epsilon}

\renewcommand{\phi}{\varphi}

\newcommand{\cV}{\mathcal{V}}

\newcommand{\cF}{{\mathcal F}}

\newcommand{\bF}{\mathbf{F}}
\renewcommand{\:}{\mathrel{\coloneqq}}

\newcommand{\OT}{{\Omega_{T}}}

\newcommand{\br}{\boldsymbol{r}}

\newcommand{\loc}{\textnormal{loc}}

\newcommand{\bH}{\boldsymbol{H}}

\newcommand{\dd}{\ensuremath{\,\mathrm{d}}}

\DeclareMathOperator{\sgn}{sgn}
\DeclareMathOperator*{\esssup}{ess-sup}
\DeclareMathOperator{\supp}{supp}
\DeclareMathOperator*{\essinf}{ess-inf}

\renewcommand{\epsilon}{\varepsilon}

\crefname{ass}{Asm.}{Assumptions}
\crefname{defi}{Def.}{Definitions}
\crefname{theo}{Thm.}{Theorems}
\crefname{lem}{Lem.}{Lemmata}
\crefname{cor}{Cor.}{Corollaries}
\crefname{rem}{Rmk.}{Remarks}

%% file: main_final.bbl
\begin{thebibliography}{48}
\expandafter\ifx\csname natexlab\endcsname\relax\def\natexlab#1{#1}\fi
\providecommand{\url}[1]{\texttt{#1}}
\providecommand{\href}[2]{#2}
\providecommand{\path}[1]{#1}
\providecommand{\DOIprefix}{doi:}
\providecommand{\ArXivprefix}{arXiv:}
\providecommand{\URLprefix}{URL: }
\providecommand{\Pubmedprefix}{pmid:}
\providecommand{\doi}[1]{\href{http://dx.doi.org/#1}{\path{#1}}}
\providecommand{\Pubmed}[1]{\href{pmid:#1}{\path{#1}}}
\providecommand{\bibinfo}[2]{#2}
\ifx\xfnm\relax \def\xfnm[#1]{\unskip,\space#1}\fi
\bibitem[{Amadori et~al.(2004)Amadori, Gosse and Guerra}]{Amadori-Gosse-Guerra_2004}
\bibinfo{author}{Amadori, D.}, \bibinfo{author}{Gosse, L.}, \bibinfo{author}{Guerra, G.}, \bibinfo{year}{2004}.
\newblock \bibinfo{title}{Godunov-type approximation for a general resonant balance law with large data}.
\newblock \bibinfo{journal}{J. Differ. Equ.} \bibinfo{volume}{198}, \bibinfo{pages}{233--274}.
\newblock \DOIprefix\doi{10.1016/j.jde.2003.10.004}.
\bibitem[{Aw and Rascle(2000)}]{aw}
\bibinfo{author}{Aw, A.}, \bibinfo{author}{Rascle, M.}, \bibinfo{year}{2000}.
\newblock \bibinfo{title}{Resurrection of ``second order'' models of traffic flow}.
\newblock \bibinfo{journal}{SIAM J. Appl. Math.} \bibinfo{volume}{60}, \bibinfo{pages}{916--938}.
\newblock \DOIprefix\doi{10.1137/S0036139997332099}.
\bibitem[{Baiti and Bressan(1997)}]{Baiti1997}
\bibinfo{author}{Baiti, P.}, \bibinfo{author}{Bressan, A.}, \bibinfo{year}{1997}.
\newblock \bibinfo{title}{The semigroup generated by a {T}emple class system with large data}.
\newblock \bibinfo{journal}{Differ. Integral Equ.} \bibinfo{volume}{10}, \bibinfo{pages}{401--418}.
\newblock \DOIprefix\doi{10.57262/die/1367525659}.
\bibitem[{Bayen et~al.(2022)Bayen, Friedrich, Keimer, Pflug and Veeravalli}]{KeimerMultilane2022}
\bibinfo{author}{Bayen, A.}, \bibinfo{author}{Friedrich, J.}, \bibinfo{author}{Keimer, A.}, \bibinfo{author}{Pflug, L.}, \bibinfo{author}{Veeravalli, T.}, \bibinfo{year}{2022}.
\newblock \bibinfo{title}{Modeling multilane traffic with moving obstacles by nonlocal balance laws}.
\newblock \bibinfo{journal}{SIAM J. Appl. Dyn. Syst.} \bibinfo{volume}{21}, \bibinfo{pages}{1495--1538}.
\newblock \DOIprefix\doi{10.1137/20M1366654}.
\bibitem[{Blandin and Goatin(2015)}]{blandin2016well}
\bibinfo{author}{Blandin, S.}, \bibinfo{author}{Goatin, P.}, \bibinfo{year}{2015}.
\newblock \bibinfo{title}{Well-posedness of a conservation law with non-local flux arising in traffic flow modeling}.
\newblock \bibinfo{journal}{Numer. Math.} \bibinfo{volume}{132}, \bibinfo{pages}{217–241}.
\newblock \DOIprefix\doi{10.1007/s00211-015-0717-6}.
\bibitem[{Bressan(2000)}]{Bressan2000}
\bibinfo{author}{Bressan, A.}, \bibinfo{year}{2000}.
\newblock \bibinfo{title}{Hyperbolic systems of conservation laws. The one-dimensional Cauchy problem}. volume~\bibinfo{volume}{20} of \textit{\bibinfo{series}{Oxford Lect. Ser. Math. Appl.}}
\newblock \bibinfo{publisher}{Oxford University Press, Oxford}.
\bibitem[{Bressan and Guerra(2024)}]{Bressan-Guerra_2024}
\bibinfo{author}{Bressan, A.}, \bibinfo{author}{Guerra, G.}, \bibinfo{year}{2024}.
\newblock \bibinfo{title}{Unique solutions to hyperbolic conservation laws with a strictly convex entropy}.
\newblock \bibinfo{journal}{J. Differ. Equ.} \bibinfo{volume}{387}, \bibinfo{pages}{432--447}.
\newblock \DOIprefix\doi{10.1016/j.jde.2024.01.005}.
\bibitem[{Bressan and Shen(2020)}]{bressan-shen_2019traffic}
\bibinfo{author}{Bressan, A.}, \bibinfo{author}{Shen, W.}, \bibinfo{year}{2020}.
\newblock \bibinfo{title}{On traffic flow with nonlocal flux: A relaxation representation}.
\newblock \bibinfo{journal}{Arch. Ration. Mech. Anal.} \bibinfo{volume}{237}, \bibinfo{pages}{1213--1236}.
\newblock \DOIprefix\doi{10.1007/s00205-020-01529-z}.
\bibitem[{Bressan and Shen(2021)}]{bressan-shen2021entropy}
\bibinfo{author}{Bressan, A.}, \bibinfo{author}{Shen, W.}, \bibinfo{year}{2021}.
\newblock \bibinfo{title}{Entropy admissibility of the limit solution for a nonlocal model of traffic flow}.
\newblock \bibinfo{journal}{Commun. Math. Sci.} \bibinfo{volume}{19}, \bibinfo{pages}{1447--1450}.
\newblock \DOIprefix\doi{10.4310/CMS.2021.v19.n5.a12}.
\bibitem[{Cheng(2025)}]{Cheng2025}
\bibinfo{author}{Cheng, J.}, \bibinfo{year}{2025}.
\newblock \bibinfo{title}{Uniqueness \& weak-{BV} stability in the large for isothermal gas dynamics}.
\newblock \bibinfo{journal}{J. Differ. Equ.} \bibinfo{volume}{446}, \bibinfo{pages}{Paper No. 113599, 31}.
\newblock \DOIprefix\doi{10.1016/j.jde.2025.113599}.
\bibitem[{Chiarello et~al.(2020)Chiarello, Friedrich, Goatin and G\"{o}ttlich}]{chiarello2020micro}
\bibinfo{author}{Chiarello, F.A.}, \bibinfo{author}{Friedrich, J.}, \bibinfo{author}{Goatin, P.}, \bibinfo{author}{G\"{o}ttlich, S.}, \bibinfo{year}{2020}.
\newblock \bibinfo{title}{Micro-macro limit of a nonlocal generalized {Aw-Rascle} type model}.
\newblock \bibinfo{journal}{SIAM J. Appl. Math.} \bibinfo{volume}{80}, \bibinfo{pages}{1841–1861}.
\newblock \DOIprefix\doi{10.1137/20m1313337}.
\bibitem[{Chiarello et~al.(2019)Chiarello, Friedrich, Goatin, G\"ottlich and Kolb}]{chiarello2019non}
\bibinfo{author}{Chiarello, F.A.}, \bibinfo{author}{Friedrich, J.}, \bibinfo{author}{Goatin, P.}, \bibinfo{author}{G\"ottlich, S.}, \bibinfo{author}{Kolb, O.}, \bibinfo{year}{2019}.
\newblock \bibinfo{title}{A non-local traffic flow model for 1-to-1 junctions}.
\newblock \bibinfo{journal}{Eur. J. Appl. Math.} \bibinfo{volume}{31}, \bibinfo{pages}{1029–1049}.
\newblock \DOIprefix\doi{10.1017/s095679251900038x}.
\bibitem[{Chiarello and Goatin(2018)}]{chiarello}
\bibinfo{author}{Chiarello, F.A.}, \bibinfo{author}{Goatin, P.}, \bibinfo{year}{2018}.
\newblock \bibinfo{title}{Global entropy weak solutions for general non-local traffic flow models with anisotropic kernel}.
\newblock \bibinfo{journal}{ESAIM: Math. Model. Numer. Anal.} \bibinfo{volume}{52}, \bibinfo{pages}{163–180}.
\newblock \DOIprefix\doi{10.1051/m2an/2017066}.
\bibitem[{Chiarello and Keimer(2024)}]{Chiarello2024}
\bibinfo{author}{Chiarello, F.A.}, \bibinfo{author}{Keimer, A.}, \bibinfo{year}{2024}.
\newblock \bibinfo{title}{On the singular limit problem in nonlocal balance laws: Applications to nonlocal lane-changing traffic flow models}.
\newblock \bibinfo{journal}{J. Math. Anal. Appl.} \bibinfo{volume}{537}, \bibinfo{pages}{128358}.
\newblock \DOIprefix\doi{10.1016/j.jmaa.2024.128358}.
\bibitem[{Chiarello et~al.(2025)Chiarello, Keimer and Pflug}]{chiarello2025pnorm}
\bibinfo{author}{Chiarello, F.A.}, \bibinfo{author}{Keimer, A.}, \bibinfo{author}{Pflug, L.}, \bibinfo{year}{2025}.
\newblock \bibinfo{title}{Nonlocal conservation laws with p-norm, the singular limit problem and applications to traffic flow}.
\newblock \URLprefix \url{https://arxiv.org/abs/2512.18701}, \href{http://arxiv.org/abs/2512.18701}{{\tt arXiv:2512.18701}}.
\bibitem[{Coclite et~al.(2024)Coclite, Colombo, Crippa, {De Nitti}, Keimer, Marconi, Pflug and Spinolo}]{coclite2023oleinik}
\bibinfo{author}{Coclite, G.M.}, \bibinfo{author}{Colombo, M.}, \bibinfo{author}{Crippa, G.}, \bibinfo{author}{{De Nitti}, N.}, \bibinfo{author}{Keimer, A.}, \bibinfo{author}{Marconi, E.}, \bibinfo{author}{Pflug, L.}, \bibinfo{author}{Spinolo, L.}, \bibinfo{year}{2024}.
\newblock \bibinfo{title}{Oleinik-type estimates for nonlocal conservation laws and applications to the nonlocal-to-local limit}.
\newblock \bibinfo{journal}{J. Hyperbolic Differ. Equ.} \bibinfo{volume}{21}, \bibinfo{pages}{681--705}.
\newblock \DOIprefix\doi{10.1142/S021989162440006X}.
\bibitem[{Coclite et~al.(2022)Coclite, Coron, {De Nitti}, Keimer and Pflug}]{coclite2022general}
\bibinfo{author}{Coclite, G.M.}, \bibinfo{author}{Coron, J.M.}, \bibinfo{author}{{De Nitti}, N.}, \bibinfo{author}{Keimer, A.}, \bibinfo{author}{Pflug, L.}, \bibinfo{year}{2022}.
\newblock \bibinfo{title}{A general result on the approximation of local conservation laws by nonlocal conservation laws: The singular limit problem for exponential kernels}.
\newblock \bibinfo{journal}{Ann. Inst. H. Poincar\'{e} C Anal. Non Lin\'{e}aire} \bibinfo{volume}{40}, \bibinfo{pages}{1205–1223}.
\newblock \DOIprefix\doi{10.4171/aihpc/58}.
\bibitem[{Coclite and {De Nitti}(2026)}]{Coclite2026nonlocal}
\bibinfo{author}{Coclite, G.M.}, \bibinfo{author}{{De Nitti}, N.}, \bibinfo{year}{2026}.
\newblock \bibinfo{title}{On a nonlocal regularization of a non-strictly hyperbolic system of conservation laws}.
\newblock \bibinfo{journal}{Nonlinear Anal. Real World Appl.} \bibinfo{volume}{92}, \bibinfo{pages}{104447}.
\newblock \DOIprefix\doi{10.1016/j.nonrwa.2025.104447}.
\bibitem[{Coclite et~al.(2025)Coclite, {De Nitti} and Huang}]{coclite2025singular}
\bibinfo{author}{Coclite, G.M.}, \bibinfo{author}{{De Nitti}, N.}, \bibinfo{author}{Huang, K.}, \bibinfo{year}{2025}.
\newblock \bibinfo{title}{Singular limit for a class of nonlocal conservation laws via compensated compactness}.
\newblock \bibinfo{type}{Technical Report}. https://arxiv.org/abs/2511.15631.
\newblock \href{http://arxiv.org/abs/2511.15631}{{\tt arXiv:2511.15631}}.
\bibitem[{Colombo et~al.(2021)Colombo, Crippa, Marconi and Spinolo}]{ColomboCrippaMarconiSpinolo2021}
\bibinfo{author}{Colombo, M.}, \bibinfo{author}{Crippa, G.}, \bibinfo{author}{Marconi, E.}, \bibinfo{author}{Spinolo, L.V.}, \bibinfo{year}{2021}.
\newblock \bibinfo{title}{Local limit of nonlocal traffic models: Convergence results and total variation blow-up}.
\newblock \bibinfo{journal}{Ann. Inst. H. Poincaré C Anal. Non Linéaire} \bibinfo{volume}{38}, \bibinfo{pages}{1653--1666}.
\newblock \DOIprefix\doi{https://doi.org/10.1016/j.anihpc.2020.12.002}.
\bibitem[{Colombo et~al.(2023)Colombo, Crippa, Marconi and Spinolo}]{Marconi2023}
\bibinfo{author}{Colombo, M.}, \bibinfo{author}{Crippa, G.}, \bibinfo{author}{Marconi, E.}, \bibinfo{author}{Spinolo, L.V.}, \bibinfo{year}{2023}.
\newblock \bibinfo{title}{Nonlocal traffic models with general kernels: Singular limit, entropy admissibility, and convergence rate}.
\newblock \bibinfo{journal}{Arch. Ration. Mech. Anal.} \bibinfo{volume}{247}.
\newblock \DOIprefix\doi{10.1007/s00205-023-01845-0}.
\bibitem[{Coron et~al.(2010)Coron, Kawski and Wang}]{wang}
\bibinfo{author}{Coron, J.M.}, \bibinfo{author}{Kawski, M.}, \bibinfo{author}{Wang, Z.}, \bibinfo{year}{2010}.
\newblock \bibinfo{title}{Analysis of a conservation law modeling a highly re-entrant manufacturing system}.
\newblock \bibinfo{journal}{Discrete Contin. Dyn. Syst. Ser. B} \bibinfo{volume}{14}, \bibinfo{pages}{1337--1359}.
\newblock \DOIprefix\doi{10.3934/dcdsb.2010.14.1337}.
\bibitem[{Dafermos(2016)}]{Dafermos2016}
\bibinfo{author}{Dafermos, C.}, \bibinfo{year}{2016}.
\newblock \bibinfo{title}{Hyperbolic conservation laws in continuum physics}. volume \bibinfo{volume}{325} of \textit{\bibinfo{series}{Grundlehren der mathematischen Wissenschaften}}.
\newblock \bibinfo{edition}{Fourth} ed., \bibinfo{publisher}{Springer-Verlag, Berlin}.
\newblock \DOIprefix\doi{10.1007/978-3-662-49451-6}.
\bibitem[{Dafermos and Geng(1991)}]{DG91}
\bibinfo{author}{Dafermos, C.M.}, \bibinfo{author}{Geng, X.}, \bibinfo{year}{1991}.
\newblock \bibinfo{title}{Generalized characteristics uniqueness and regularity of solutions in a hyperbolic system of conservation laws}.
\newblock \bibinfo{journal}{Ann. Inst. H. Poincar\'e{} C Anal. Non Lin\'eaire} \bibinfo{volume}{8}, \bibinfo{pages}{231--269}.
\newblock \DOIprefix\doi{10.1016/S0294-1449(16)30263-3}.
\bibitem[{Danskin(1967)}]{Danskin1967}
\bibinfo{author}{Danskin, J.}, \bibinfo{year}{1967}.
\newblock \bibinfo{title}{The Theory of Max-Min and its Application to Weapons Allocation Problems}.
\newblock \bibinfo{publisher}{Springer Berlin Heidelberg}.
\newblock \DOIprefix\doi{10.1007/978-3-642-46092-0}.
\bibitem[{Fan et~al.(2014)Fan, Herty and Seibold}]{FanHertySeibold2014}
\bibinfo{author}{Fan, S.}, \bibinfo{author}{Herty, M.}, \bibinfo{author}{Seibold, B.}, \bibinfo{year}{2014}.
\newblock \bibinfo{title}{Comparative model accuracy of a data-fitted generalized {A}w-{R}ascle-{Z}hang model}.
\newblock \bibinfo{journal}{Netw. Heterog. Media} \bibinfo{volume}{9}, \bibinfo{pages}{239--268}.
\newblock \DOIprefix\doi{10.3934/nhm.2014.9.239}.
\bibitem[{Friedrich et~al.(2024)Friedrich, G\"ottlich, Keimer and Pflug}]{friedrich2022conservation}
\bibinfo{author}{Friedrich, J.}, \bibinfo{author}{G\"ottlich, S.}, \bibinfo{author}{Keimer, A.}, \bibinfo{author}{Pflug, L.}, \bibinfo{year}{2024}.
\newblock \bibinfo{title}{Conservation laws with nonlocal velocity: The singular limit problem}.
\newblock \bibinfo{journal}{SIAM J. Appl. Math.} \bibinfo{volume}{84}, \bibinfo{pages}{497--522}.
\newblock \DOIprefix\doi{10.1137/22M1530471}.
\bibitem[{Friedrich et~al.(2018)Friedrich, Kolb and G\"{o}ttlich}]{friedrich2018godunov}
\bibinfo{author}{Friedrich, J.}, \bibinfo{author}{Kolb, O.}, \bibinfo{author}{G\"{o}ttlich, S.}, \bibinfo{year}{2018}.
\newblock \bibinfo{title}{A {G}odunov type scheme for a class of {LWR} traffic flow models with non-local flux}.
\newblock \bibinfo{journal}{Netw. Heterog. Media} \bibinfo{volume}{13}, \bibinfo{pages}{531–547}.
\newblock \DOIprefix\doi{10.3934/nhm.2018024}.
\bibitem[{Goatin and Scialanga(2016)}]{scialanga}
\bibinfo{author}{Goatin, P.}, \bibinfo{author}{Scialanga, S.}, \bibinfo{year}{2016}.
\newblock \bibinfo{title}{Well-posedness and finite volume approximations of the {LWR} traffic flow model with non-local velocity}.
\newblock \bibinfo{journal}{Netw. Heterog. Media} \bibinfo{volume}{11}, \bibinfo{pages}{107--121}.
\newblock \DOIprefix\doi{10.3934/nhm.2016.11.107}.
\bibitem[{Hamori and Tan(2026)}]{Hamori-Tan_2026}
\bibinfo{author}{Hamori, T.}, \bibinfo{author}{Tan, C.}, \bibinfo{year}{2026}.
\newblock \bibinfo{title}{On the {A}w-{R}ascle-{Z}hang traffic models with nonlocal look-ahead interactions}.
\newblock \bibinfo{journal}{Nonlinear Anal.} \bibinfo{volume}{262}, \bibinfo{pages}{Paper No. 113930, 20}.
\newblock \DOIprefix\doi{10.1016/j.na.2025.113930}.
\bibitem[{Heibig(1994)}]{H94}
\bibinfo{author}{Heibig, A.}, \bibinfo{year}{1994}.
\newblock \bibinfo{title}{Existence and uniqueness of solutions for some hyperbolic systems of conservation laws}.
\newblock \bibinfo{journal}{Arch. Ration. Mech. Anal.} \bibinfo{volume}{126}, \bibinfo{pages}{79--101}.
\newblock \DOIprefix\doi{10.1007/BF00375697}.
\bibitem[{Josephy(1981)}]{Josephy1981}
\bibinfo{author}{Josephy, M.}, \bibinfo{year}{1981}.
\newblock \bibinfo{title}{Composing functions of bounded variation}.
\newblock \bibinfo{journal}{Proc. Amer. Math. Soc.} \bibinfo{volume}{83}, \bibinfo{pages}{354--356}.
\newblock \DOIprefix\doi{10.2307/2043527}.
\bibitem[{Keimer and Pflug(2017)}]{pflug}
\bibinfo{author}{Keimer, A.}, \bibinfo{author}{Pflug, L.}, \bibinfo{year}{2017}.
\newblock \bibinfo{title}{Existence, uniqueness and regularity results on nonlocal balance laws}.
\newblock \bibinfo{journal}{J. Differ. Equ.} \bibinfo{volume}{263}, \bibinfo{pages}{4023–4069}.
\newblock \DOIprefix\doi{10.1016/j.jde.2017.05.015}.
\bibitem[{Keimer and Pflug(2019)}]{pflug4}
\bibinfo{author}{Keimer, A.}, \bibinfo{author}{Pflug, L.}, \bibinfo{year}{2019}.
\newblock \bibinfo{title}{On approximation of local conservation laws by nonlocal conservation laws}.
\newblock \bibinfo{journal}{J. Math. Anal. Appl.} \bibinfo{volume}{475}, \bibinfo{pages}{1927–1955}.
\newblock \DOIprefix\doi{10.1016/j.jmaa.2019.03.063}.
\bibitem[{Keimer and Pflug(2023)}]{Keimer2023}
\bibinfo{author}{Keimer, A.}, \bibinfo{author}{Pflug, L.}, \bibinfo{year}{2023}.
\newblock \bibinfo{title}{Discontinuous nonlocal conservation laws and related discontinuous {ODEs} – {Existence}, uniqueness, stability and regularity}.
\newblock \bibinfo{journal}{C. R. Math.} \bibinfo{volume}{361}, \bibinfo{pages}{1723–1760}.
\newblock \DOIprefix\doi{10.5802/crmath.490}.
\bibitem[{Keimer and Pflug(2025)}]{keimer42}
\bibinfo{author}{Keimer, A.}, \bibinfo{author}{Pflug, L.}, \bibinfo{year}{2025}.
\newblock \bibinfo{title}{On the singular limit problem for nonlocal conservation laws: A general approximation result for kernels with fixed support}.
\newblock \bibinfo{journal}{J. Math. Anal. Appl.} \bibinfo{volume}{547}, \bibinfo{pages}{129307}.
\newblock \DOIprefix\doi{10.1016/j.jmaa.2025.129307}.
\bibitem[{Keimer and Pflug(2026)}]{keimer2026nonlocal}
\bibinfo{author}{Keimer, A.}, \bibinfo{author}{Pflug, L.}, \bibinfo{year}{2026}.
\newblock \bibinfo{title}{Nonlocal approximation principle for entropy solutions of scalar conservation laws}.
\newblock \URLprefix \url{https://arxiv.org/abs/2605.00635}, \href{http://arxiv.org/abs/2605.00635}{{\tt arXiv:2605.00635}}.
\bibitem[{Leoni(2017)}]{leoni}
\bibinfo{author}{Leoni, G.}, \bibinfo{year}{2017}.
\newblock \bibinfo{title}{A first course in {S}obolev spaces}. volume \bibinfo{volume}{181} of \textit{\bibinfo{series}{Graduate Studies in Mathematics}}.
\newblock \bibinfo{edition}{Second} ed., \bibinfo{publisher}{American Mathematical Society, Providence, RI}.
\newblock \DOIprefix\doi{10.1090/gsm/181}.
\bibitem[{Lighthill and Whitham(1955)}]{lwr_1}
\bibinfo{author}{Lighthill, M.}, \bibinfo{author}{Whitham, G.}, \bibinfo{year}{1955}.
\newblock \bibinfo{title}{On kinematic waves. {I}. {Flood} movement in long rivers}.
\newblock \bibinfo{journal}{Proc. R. Soc. Lond. A} \bibinfo{volume}{229}, \bibinfo{pages}{281--316}.
\newblock \DOIprefix\doi{10.1098/rspa.1955.0088}.
\bibitem[{Marconi and Spinolo(2025)}]{MS2025-1}
\bibinfo{author}{Marconi, E.}, \bibinfo{author}{Spinolo, L.V.}, \bibinfo{year}{2025}.
\newblock \bibinfo{title}{Nonlocal generalized {A}w-{R}ascle-{Z}hang model: Well-posedness and singular limit}.
\newblock \URLprefix \url{https://arxiv.org/abs/2505.10102}, \DOIprefix\doi{10.48550/ARXIV.2505.10102}.
\bibitem[{Marconi and Spinolo(2026)}]{MS2025-2}
\bibinfo{author}{Marconi, E.}, \bibinfo{author}{Spinolo, L.V.}, \bibinfo{year}{2026}.
\newblock \bibinfo{title}{Well-posedness results for the generalized {A}w-{R}ascle-{Z}hang model}.
\newblock \bibinfo{journal}{Adv. Contin. Discrete Models} \bibinfo{volume}{2026}.
\newblock \DOIprefix\doi{10.1186/s13662-026-04079-y}.
\bibitem[{Ramadan et~al.(2021)Ramadan, Rosales and Seibold}]{Ramadan2020}
\bibinfo{author}{Ramadan, R.}, \bibinfo{author}{Rosales, R.R.}, \bibinfo{author}{Seibold, B.}, \bibinfo{year}{2021}.
\newblock \bibinfo{title}{Structural properties of the stability of jamitons}, in: \bibinfo{booktitle}{Mathematical descriptions of traffic flow: micro, macro and kinetic models}. \bibinfo{publisher}{Springer, Cham}. volume~\bibinfo{volume}{12} of \textit{\bibinfo{series}{ICIAM 2019 SEMA SIMAI Springer Ser.}}, pp. \bibinfo{pages}{35--62}.
\newblock \URLprefix \url{https://doi.org/10.1007/978-3-030-66560-9_3}, \DOIprefix\doi{10.1007/978-3-030-66560-9\_3}.
\bibitem[{Richards(1956)}]{lwr_2}
\bibinfo{author}{Richards, P.}, \bibinfo{year}{1956}.
\newblock \bibinfo{title}{Shock waves on the highway}.
\newblock \bibinfo{journal}{Oper. Res.} \bibinfo{volume}{4}, \bibinfo{pages}{42--51}.
\newblock \URLprefix \url{https://www.jstor.org/stable/167515}.
\bibitem[{Rosini(2019)}]{RosiniAnnales}
\bibinfo{author}{Rosini, M.D.}, \bibinfo{year}{2019}.
\newblock \bibinfo{title}{Systems of conservation laws with discontinuous fluxes and applications to traffic}.
\newblock \bibinfo{journal}{Ann. Univ. Mariae Curie-Sklodowska Sect. A} \bibinfo{volume}{73}, \bibinfo{pages}{135--173}.
\bibitem[{Simon(1987)}]{Simon1986}
\bibinfo{author}{Simon, J.}, \bibinfo{year}{1987}.
\newblock \bibinfo{title}{Compact sets in the space {$L^p(0,T;B)$}}.
\newblock \bibinfo{journal}{Ann. Mat. Pura Appl. (4)} \bibinfo{volume}{146}, \bibinfo{pages}{65--96}.
\newblock \DOIprefix\doi{10.1007/BF01762360}.
\bibitem[{Sopasakis and Katsoulakis(2006)}]{Sopasakis2006}
\bibinfo{author}{Sopasakis, A.}, \bibinfo{author}{Katsoulakis, M.A.}, \bibinfo{year}{2006}.
\newblock \bibinfo{title}{Stochastic modeling and simulation of traffic flow: Asymmetric single exclusion process with {A}rrhenius look-ahead dynamics}.
\newblock \bibinfo{journal}{SIAM J. Appl. Math.} \bibinfo{volume}{66}, \bibinfo{pages}{921–944}.
\newblock \DOIprefix\doi{10.1137/040617790}.
\bibitem[{Temple(1983)}]{Temple01}
\bibinfo{author}{Temple, B.}, \bibinfo{year}{1983}.
\newblock \bibinfo{title}{Systems of conservation laws with invariant submanifolds}.
\newblock \bibinfo{journal}{Trans. Amer. Math. Soc.} \bibinfo{volume}{280}, \bibinfo{pages}{781--795}.
\newblock \DOIprefix\doi{10.1090/s0002-9947-1983-0716850-2}.
\bibitem[{Zhang(2002)}]{zhang2002non}
\bibinfo{author}{Zhang, H.}, \bibinfo{year}{2002}.
\newblock \bibinfo{title}{A non-equilibrium traffic model devoid of gas-like behavior}.
\newblock \bibinfo{journal}{Transp. Res. Part B Methodol.} \bibinfo{volume}{36}, \bibinfo{pages}{275–290}.
\newblock \DOIprefix\doi{10.1016/s0191-2615(00)00050-3}.

\end{thebibliography}
